\documentclass[a4paper, 12pt, twoside]{amsart}
\usepackage{calrsfs}
\usepackage{amssymb}
\usepackage{amsthm}
\usepackage{amscd}
\usepackage{amsopn}
\usepackage{eucal}
\usepackage{etex}
\usepackage{pictex}
\usepackage[all,cmtip]{xy}
\usepackage{tikz}

\usetikzlibrary{patterns}

\newtheoremstyle{Teorema}{10pt}{10pt}{\it}{}{\sc}{. }{ }{}
\theoremstyle{Teorema}
\newtheorem{Theorem}{Theorem}[subsection]

\newtheorem{Corollary}[Theorem]{Corollary}

\newtheorem{Proposition}[Theorem]{Proposition}

\newtheorem{PropositionDefinition}[Theorem]{Proposition and Definition}
\newtheorem{DefinitionProposition}[Theorem]{Definition and Proposition}
\newtheorem{Definition}[Theorem]{Definition}
\newtheorem{Lemma}[Theorem]{Lemma}

\newtheoremstyle{PseudoHeading}{10pt}{10pt}{}{}{\sc}{. }{ }{}
\theoremstyle{PseudoHeading}

\newtheorem*{AbelianVarieties}{Abelian Varieties}
\newtheorem*{AbelianVarietiesUpToIsogeny}{Abelian Varieties up to Isogeny}
\newtheorem*{Acknowledgements}{Acknowledgements}
\newtheorem*{AffineMaps}{Affine Maps}
\newtheorem*{AffineMapsasQCMaps}{Affine Maps as Quasiconformal Maps}
\newtheorem*{AlgebraicCurves}{Algebraic Curves}

\newtheorem*{AThmofSchmid}{A Theorem of Schmid and its Consequences}
\newtheorem*{BaseChange}{Base Change}
\newtheorem*{BeingDefinedoveraSubfield}{Being Defined over a Subfield}
\newtheorem*{BeltramiForms}{Beltrami Forms}
\newtheorem*{Billiards}{Billiards}

\newtheorem*{Comparison}{Comparison}
\newtheorem*{ConePoints}{Cone Points}
\newtheorem*{ConstructingRMFromASubsystem}{Constructing Real Multiplication from a Subsystem}
\newtheorem*{Constructions}{Constructions}

\newtheorem*{DecompositionOfJacobian}{Decomposition of the Jacobian}

\newtheorem*{DefinitionOfRM}{Definition of Real Multiplication}
\newtheorem*{Dictionary}{Dictionary}
\newtheorem*{DirectionsandFoliations}{Directions and Foliations}
\newtheorem*{Eigenforms}{Eigenforms}
\newtheorem*{Eigenvalues}{Eigenvalues}
\newtheorem*{ElementaryConstructions}{Elementary Constructions}
\newtheorem*{ElementarySubgroups}{Elementary Subgroups}
\newtheorem*{EuclideanMetric}{Euclidean Metric}
\newtheorem*{FamiliesOfRiemannSurfacesAndAlgebraicCurves}{Families of Riemann Surfaces and Algebraic Curves}

\newtheorem*{FundamentalGroupsandDeckTrafos}{Fundamental Groups and Deck Transformations}

\newtheorem*{GaloisConjugationofPiKModules}{Galois Conjugation of Local Systems}
\newtheorem*{GenusTwo}{Genus Two}
\newtheorem*{HodgeSubbundles}{Hodge Subbundles}

\newtheorem*{HolonomyVectors}{Saddle Connections and Holonomy Vectors}
\newtheorem*{HSofSmoothCurves}{Hodge Structures of Smooth Curves}
\newtheorem*{HSonCohomology}{Hodge Structures on Cohomology}
\newtheorem*{ImplicationsForThePeriodMatrix}{Implications for the Period Matrix}
\newtheorem*{JSDirections}{Jenkins-Strebel Directions and Cylinder Decompositions}
\newtheorem*{KobayashiHyperbolicSpaces}{Kobayashi Hyperbolic Spaces}
\newtheorem*{KobayashiIsometries}{Kobayashi Isometries}
\newtheorem*{Leaves}{Leaves}
\newtheorem*{LevelStructures}{Level Structures}
\newtheorem*{LSandMonodromyReps}{Local Systems and Monodromy Representations}
\newtheorem*{MappingClassGroups}{Mapping Class Groups}
\newtheorem*{MFofCanonicalSubvariation}{The Moduli Field of the Canonical Subvariation}
\newtheorem*{ModuliofAD}{Moduli of Abelian Differentials}

\newtheorem*{ModuliSpacesOfRiemannSurfaces}{Moduli Spaces of Riemann Surfaces}
\newtheorem*{ModuliSpacesOfAlgebraicCurves}{Moduli Spaces of Algebraic Curves}
\newtheorem*{MonodromyGovernsTheHD}{Monodromy Governs the Hodge Decomposition}
\newtheorem*{MonodromyRepresentation}{Monodromy Representation}
\newtheorem*{NonelementaryWithParabolic}{Nonelementary Veech Group Containing a Parabolic Element}
\newtheorem*{Origamis}{Origamis}
\newtheorem*{OrthogonalDecomposition}{Orthogonal Decomposition}
\newtheorem*{ParabolicsInTheVG}{Parabolics in the Veech Group}
\newtheorem*{PeriodMatrices}{Period Matrices}
\newtheorem*{PeriodMatricesOfJacobians}{Period Matrices of Jacobians}

\newtheorem*{Polarizations}{Polarizations}
\newtheorem*{PureRealHS}{Pure Real Hodge Structures}
\newtheorem*{PureAHS}{Pure $A$-Hodge Structures}

\newtheorem*{QuasiconformalMappings}{Quasiconformal Mappings}
\newtheorem*{QuaternionAlgebras}{Quaternion Algebras}
\newtheorem*{RationalIsoscelesTriangles}{Rational Isosceles Triangles}
\newtheorem*{RegularPolygonsI}{Regular Polygons I}
\newtheorem*{RegularPolygonsII}{Regular Polygons II}
\newtheorem*{RiemannSurfaces}{Riemann Surfaces}
\newtheorem*{RMandTCurves}{Real Multiplication and Teichm\"{u}ller Curves}
\newtheorem*{RMbytheTraceField}{Real Multiplication by the Trace Field}
\newtheorem*{Semisimplicity}{Semisimplicity}
\newtheorem*{SemisimplicityandBaseChange}{Semisimplicity and Base Change}
\newtheorem*{SemisimplicityofMonodromy}{Semisimplicity of Monodromy Representations}
\newtheorem*{SemisimplicityofVHS}{Semisimplicity of Variations of Hodge Structure}
\newtheorem*{SimplicityandBaseChange}{Simplicity and Base Change}
\newtheorem*{SomeRemarksonGaloisTheory}{Galois Theory of Number Fields}
\newtheorem*{TDisksViaBeltramiForms}{Teichm\"{u}ller Disks via Beltrami Forms}
\newtheorem*{TDisksViaSLtwoR}{Teichm\"{u}ller Disks via $\SL_2(\bbr )$-Orbits}
\newtheorem*{TeichmuellerMarkings}{Teichm\"{u}ller Markings}
\newtheorem*{TeichmuellerMetric}{The Teichm\"{u}ller Metric}
\newtheorem*{TeichmuellerSpaces}{Teichm\"{u}ller Spaces}
\newtheorem*{TheAffineGroupActingontheTDisk}{The Affine Group Acting on the Teichm\"{u}ller Disk}
\newtheorem*{TheCanonicalSubspace}{The Canonical Subspace}
\newtheorem*{TheCanonicalSubvariation}{The Canonical Subsystem}
\newtheorem*{TheGeneralCase}{The General Case}
\newtheorem*{TheGLtwoRAction}{The $\GL_2^+(\bbr )$-Action}
\newtheorem*{TheFamilyOverATC}{The Family of Curves over a Teichm\"{u}ller Curve}
\newtheorem*{TheCrossRatioField}{The Field of Cross-Ratios}
\newtheorem*{TheFieldofModuli}{The Field of Moduli of a Simple Local System}
\newtheorem*{TheImageInMCG}{The Image in the Mapping Class Group}
\newtheorem*{TheImageInMS}{The Image in Moduli Space}
\newtheorem*{TheKobayashiSemimetric}{The Kobayashi Semimetric}
\newtheorem*{TheMirrorVG}{The Mirror Veech Group}
\newtheorem*{TheNumberField}{Reminder on Cyclotomic Fields}
\newtheorem*{TheVeechGroup}{The Veech Group}
\newtheorem*{TheVeechGroupALattice}{The Veech Group a Lattice}
\newtheorem*{TheSCFormula}{The Schwarz-Christoffel Formula}
\newtheorem*{TheWimanCurvesasVeechSurfaces}{The Wiman Curves as Counterexamples}
\newtheorem*{TraceFields}{Trace Fields}
\newtheorem*{TraceFieldsOfVeechGroups}{Trace Fields of Veech Groups}
\newtheorem*{TranslationSurfaces}{Translation Surfaces}
\newtheorem*{TransportingRealMultiplication}{Transporting Real Multiplication}
\newtheorem*{TrichotomyForElementsOfTheVG}{Trichotomy for Elements of the Veech Group}
\newtheorem*{TrivializationofTangentSpaces}{Trivialization of Tangent Spaces}
\newtheorem*{VariationsofPsHS}{Variations of Pseudo-Hodge Structure}
\newtheorem*{VeechSurfaces}{Veech Surfaces}
\newtheorem*{WimanCurvesAndRegularPolygons}{Wiman Curves and Regular Polygons}

\renewcommand{\to}{\longrightarrow}

\def\Aff{\operatorname{Aff}}
\def\Spec{\operatorname{Spec}}

\def\Hom{\operatorname{Hom}}
\def\Aut{\operatorname{Aut}}
\def\Mod{\mathrm{Mod}}
\def\Out{\operatorname{Out}}
\def\End{\operatorname{End}}

\def\Im{\operatorname{Im}}
\def\Re{\operatorname{Re}}
\def\Gal{\operatorname{Gal}}
\def\SL{\operatorname{SL}}

\def\SO{\operatorname{SO}}
\def\GL{\operatorname{GL}}
\def\PSL{\operatorname{PSL}}
\def\pr{\operatorname{pr}}
\def\Gr{\operatorname{Gr}}
\def\tr{\operatorname{tr}}
\def\bbz{\mathbb{Z}}
\def\bbr{\mathbb{R}}
\def\bbq{\mathbb{Q}}
\def\bbc{\mathbb{C}}
\def\bbh{\mathbb{H}}

\def\de{\mathrm{d}}

\begin{document}
\title{Real Multiplication on Jacobian Varieties}
\author{Robert A. Kucharczyk}
\address{Universit\"{a}t Bonn\\
Mathematisches Institut\\
Endenicher Allee 60\\
D-53115 Bonn\\
Germany}
\email{rak@math.uni-bonn.de}

\begin{abstract}
This is a slightly revised version of the author's 2010 diploma thesis. It is concerned with the interplay between real multiplication on Jacobian varieties, as the title suggests, and complex geodesics in the moduli space of curves. Large parts are expository and may hopefully serve as a very incomplete introduction to Teichm\"{u}ller disks and curves, the moduli space of abelian differentials with its $\SL_2(\bbr )$-operation and variations of Hodge structure.
\end{abstract}
\maketitle
\tableofcontents

\newpage
\section{Introduction}

\noindent One of the most traditional and popular subjects in pure mathematics is the study of algebraic curves and their deformation theory. One reason for this is without doubt what David Mumford in \cite[Lecture I]{Mumford75} calls an ``AMAZING SYNTHESIS, which surely overwhelmed each of us as graduate students and should really not be taken for granted''. Namely function fields, algebraic curves (over $\bbc$) and compact Riemann surfaces are essentially the same (in the sense of canonical equivalences of categories). It is easy to \textit{theoretically} pass from one realm to the other, but it is sometimes quite hard to describe for instance a given Riemann surface, say given as a quotient of some Fuchsian group, as an algebraic curve inside some projective space. In each of these realms there are some preferred and ``simple'' ways to explicitly give an object, but these presentations to not always fit together smoothly. But precisely this oddity makes the theory of Riemann surfaces or algebraic curves so rich.

Starting from the point of view of Riemann surfaces, one of the easiest and most intuitive ways to construct such a surface is certainly that of gluing Euclidean polygons along parallel sides to a closed surface. The surface then inherits more than a complex structure: a translation structure, i.e. an atlas (apart from some ``singular'' points coming from the vertices of the polygon) all of whose chart transition maps are translations of the complex plane. One can now rebuild a great part of Euclidean geometry in such a translation surface.

Translation surface have been much studied, mainly for two reasons: first, a translation surface provides a nice and explicit, but not too complicated, model for dynamical systems theory. The dynamics of geodesics on such surfaces is particularly well understood. Second, one of the visually most pleasant ways to deform a Riemann surface is to take some translation surface and then change its complex structure by changing the complex structure \textit{linearly} on the Euclidean plane to which the charts map. This provides then a family of Riemann surfaces, and thus also a family of algebraic curves. In certain special cases this is an \textit{algebraic} family of algebraic curves, parameterized by an algebraic curve. This yields an algebraic curve in the moduli space $\mathrsfs{M}_g$; curves obtained this way are called \textit{abelian Teichm\"{u}ller curves}. Now the following two questions seem natural:
\begin{enumerate}
\item From the point of view of Riemann surface theory: can one give a criterion whether a given translation structure generates a Teichm\"{u}ller curve, only using the nice and intuitive Euclidean geometry of the translation surface?
\item From the point of view of algebraic geometry: can one say something about Teichm\"{u}ller curves in terms that an algebraic geometer feels more comfortable with, for example its family of Jacobian varieties?
\end{enumerate}
It turns out that the answer to both questions is ``yes''. As to (i), one defines a certain discrete subgroup of $\SL_2(\bbr )$ for each translation surface, called the \textit{Veech group} of that surface. It roughly consists of those linear maps by which one can ``distort'' the translation structure and obtain a translation surface which is isomorphic to the original one. Then a translation surface is called a \textit{Veech surface} if its Veech group is a lattice. One can then show (cf. Prop. 3.5.6) that a translation surface generates an abelian Teichm\"{u}ller curve if and only if it is a Veech surface. Being a Veech surface is also an interesting property from the point of view of dynamical systems theory, see e.g. \cite{HubertSchmidt00}, \cite{Veech89}.

To answer question (ii) one needs considerably more machinery. To begin with, a translation structure is ``the same'' as an abelian differential, i.e. a nonzero holomorphic one-form: given a translation structure, identify the Euclidean plane with $\bbc$ and pull back the one-form $\de z$ on $\bbc$ along the translation charts. These glue to an abelian differential on the surface. Conversely, for an abelian differential $\omega$, the set of charts $z$ for which $\omega$ becomes $\de z$ defines a translation structure on the underlying surface. One has to work a bit more detailed at the singularities, which then correspond to zeros of the differential. Thus a translation surface is ``the same'' as a pair $(X,\omega )$, where $X$ is a complete algebraic curve and $\omega$ is a non-zero regular one-form  (K\"{a}hler differential) on $X$.

It then turns out that there are subtle and deep relations between properties of the Veech group of $(X,\omega )$ (which was, recall, defined in terms of Euclidean geometry of translation surfaces) and properties of the Jacobian of $X$. Namely if $(X,\omega )$ is a Veech surface, we obtain real multiplication by a certain number field on an abelian subvariety of the Jacobian. The number field is the trace field of the Veech group, i.e. the subfield of $\bbr$ generated by all traces of elements of the Veech group. The form $\omega$ is then an \textit{eigenform} for this real multiplication structure in the following sense: the cotangent space of the Jacobian can be canonically identified with the space $\Omega^1(X)$ of holomorphic one-forms on $X$, hence real multiplication on a subtorus of the Jacobian induces a family of endomorphisms of a subspace of $\Omega^1(X)$. Hence it makes sense to say that $\omega$ is an eigenform, i.e. a simultaneous eigenvector, for this family of endomorphisms. Finally, all these statements hold along the whole family of curves determined by $(X,\omega )$; we briefly say that the structures are ``preserved by the $\SL_2(\bbr )$-operation''. These result can basically be found in two research articles, one by Curtis McMullen \cite{McMullen03} and one by Martin M\"{o}ller \cite{Moeller05b}. We should, however, note that these articles both contain much more material than just the parts which are quoted and paraphrased in the present thesis.

The proofs of these statements use deep theorems of Deligne \cite{Deligne87} and Schmid \cite{Schmid73} which are only true under very special conditions. In particular the proofs cannot be carried over to the non-Veech case. In genus two, however, there are also more elementary proofs which do not make use of those theorems and hence hold also in the non-Veech case. The purpose of this thesis is now to construct a series of counterexamples in all higher genera. More precisely, we find in every genus $g\ge 3$ a translation surface $(X,\omega )$ with real multiplication by the trace field of the Veech group on a subtorus of the Jacobian, such that the involved structures share many properties with those that occur in the Veech case, and yet they are not preserved by the $\SL_2(\bbr )$-operation. The counterexamples have already been used in \cite{McMullen03} and \cite{Moeller05b} to show less precise statements.

Now we give a quick overview of what happens in the different chapters, together with some remarks on which parts are original work and which parts merely paraphrase work of others.

The second chapter just assembles some standard facts about Riemann surfaces to be used later and fixes notations. In section 1.3 we give a quick summary of properties of the Kobayashi metric that will turn out useful.

The third chapter summarizes the theory of translation surfaces (or, equivalently, abelian differentials), with an emphasis on the Euclidean-geometric side. We also introduce our later counterexamples, the \textit{Wiman curves}. As translation surfaces, they can be very neatly constructed from triangles whose angles are rational multiples of $\pi$. With the help of the Schwarz-Christoffel formula, a gem of 19th century complex analysis, we can describe the Wiman curves as algebraic curves with K\"{a}hler differentials. These computations are original, but the idea is well-known to the experts, cf. \cite{AurellItzykson}.

The fourth chapter then discusses the deformation theory of abelian differentials and the related theory of Veech groups. There is also nothing new in this chapter except for the presentation.

In the fifth chapter we set up the machinery necessary to understand the Jacobian of a translation surface. For technical reasons we prefer to work with Hodge structures instead of abelian varieties. In particular we need to change coefficients of Hodge structures, even into rings which are only contained in $\bbc$ and not in $\bbr$. It seems that this, although not difficult, has never been explicitly done in the literature, so we do this here, at some places certainly more detailed than strictly necessary for the proof of the main theorem. We also discuss the action of the affine group of a translation surface on the cohomology of the surface, which gives nontrivial arithmetic information about the eigenvalues of elements of the Veech group (also previously known material, but scattered around the literature).

Then in the sixth chapter we discuss trace fields of Veech groups. First we give an overview about their properties and their significance (again material which has been known, but distributed over many different research papers). Then we compute the trace fields of our example curves, using only elementary trigonometry; this is entirely new.

The seventh chapter then draws together the work of chapters 5 and 6. First, local systems and their behaviour under base extension and Galois conjugation are discussed. It also seems that there is no account of this theory in the literature. The consequences that we really need for our counterexamples are implicit in the literature, though. Then we summarize the Deligne-Schmid theory on variations of Hodge structure and explain how it enlightens the behaviour of Jacobians under the $\SL_2(\bbr )$-operation.

The eigth chapter is the actual heart of the present work. We introduce a modified notion of real multiplication, better suited for this theory than the classical one, and reformulate the results of McMullen and M\"{o}ller in this language, with proofs. The thesis concludes with Theorem 8.3.3 which shows that our counterexamples are actually counterexamples.

\begin{Acknowledgements}
The author wishes to stress his most sincere gratitude to his thesis advisor, Ursula Hamenst\"{a}dt.
\end{Acknowledgements}

\newpage
\section{Families of Algebraic Curves}

\noindent In this chapter we discuss some generalities about moduli spaces of curves, curves in moduli spaces of curves and families of curves. Working over the complex numbers, we make excessive use of the equivalence between algebraic curves over the complex numbers and finite type Riemann surfaces.

There does not exist a fine moduli space of curves. Since we do not to want to work with stacks or orbifolds, we consider the moduli space $\mathrsfs{M}_g$ as a coarse moduli space. The price we have to pay is that whenever we want to use some correspondence between maps to moduli space and families of curves, we have to pass to a cover of $\mathrsfs{M}_g$. There are basically two possibilities, both of which are important: moduli spaces with level structures $\mathrsfs{M}_g^{[n]}$ for $n\ge 3$, and Teichm\"{u}ller spaces $\mathrsfs{T}_g$.

We are interested in one-dimensional families of curves which are in some sense maximally apart from being isotrivial. This can be made precise using metric terms. We shall find that all one-dimensional families of curves which are not isotrivial must have a hyperbolic base. Hence we have the hyperbolic metric on the base, and luckily this behaves nicely when compared to the Teichm\"{u}ller metric on moduli (or Teichm\"{u}ller) space. These metrics can be conceptually subsumed as Kobayashi metrics. An analysis of the properties of Kobayashi metrics then hints that one should regard precisely those families as maximally non-isotrivial whose classifying map is a local isometry.

\subsection{Conventions and Setup}

We begin by recalling some basic ideas and constructions in moduli theory of algebraic curves. The primary purpose of this section is to fix notation and precise definitions, hence it is not self-contained.

\begin{RiemannSurfaces} A Riemann surface is a one-dimensional complex manifold. We always assume Riemann surfaces to be connected, unless explicitly stated otherwise. The universal covering space of a Riemann surface $X$ is, up to isomorphism, exactly one of the following three Riemann surfaces: the Riemann sphere $\mathbb{P}^1(\bbc )$, the complex plane $\bbc$ and the unit disk $\Delta =\{ z\in\bbc :|z|<1\}$. In the first case $X$ is called \textit{elliptic}, in the second case \textit{parabolic} and in the third case \textit{hyperbolic}. These namings correspond to natural geometries on such surfaces which are only of relevance for us in the hyperbolic case: the \textit{Poincar\'{e} metric} on the unit disk is the metric tensor
\begin{equation*}
\varrho =\frac{\de z\cdot\de\overline{z}}{(1-|z|^2)^2}
\end{equation*}
which gives a well-defined Hermitian metric on $\Delta$; this metric is invariant under the holomorphic automorphism group of $\Delta$ and hence descends to every hyperbolic Riemann surface. It has constant curvature $-4$. For every pair of points $x,y\in\Delta$ there exists a holomorphic automorphism of $\Delta$ which sends $x$ to $0$ and $y$ to some $t\in [0,1[$; the distance with respect to $\varrho$ from $x$ to $y$ is then
\begin{equation*}
\frac{1}{2}\log\frac{1+t}{1-t}.
\end{equation*}

A Riemann surface is called of \textit{finite type} if it is either elliptic, parabolic or hyperbolic with finite hyperbolic volume. Every Riemann surface $X$ of finite type can be written as $\overline{X}\smallsetminus D$ where $\overline{X}$ is a compact Riemann surface of genus $g$ and $D\subseteq\overline{X}$ is a finite set; conversely every such Riemann surface is of finite type. The cardinality of $D$, which we often denote by $n$, and the genus $g$ of $\overline{X}$ are uniquely determined by $X$; we call $(g,n)$ the \textit{type} of the Riemann surface $X$.
\end{RiemannSurfaces}

\begin{AlgebraicCurves}
Let $k$ be an algebraically closed field. An \textit{algebraic curve} over $k$ is a scheme $C$ of finite type over $k$ which is integral, reduced and one-dimensional. It is called \textit{smooth} if the structure morphism $C\to\Spec k$ is smooth; usually we will only consider smooth curves. For $k=\bbc$ there is an ``analytification functor'' from smooth algebraic curves over $\bbc$ with $\bbc$-morphisms to Riemann surfaces with holomorphic maps; its essential image consists precisely of the finite type Riemann surfaces. Two complex algebraic curves are isomorphic precisely if the associated finite type Riemann surfaces are biholomorphic.

An algebraic curve $C$ over $k$ is called \textit{complete} if the structure morphism $C\to\Spec k$ is proper or, equivalently, projective. Every smooth algebraic curve $C$ can be written as $\overline{C}\smallsetminus D$, where $\overline{C}$ is a smooth complete algebraic curve and $D$ is a finite set of closed points in $\overline{C}$. For $k=\bbc$ the analytification functor from smooth complete algebraic curves over $\bbc$ with $\bbc$-scheme morphisms to compact Riemann surfaces with holomorphic maps is an equivalence of categories. For this reason we will often not distinguish between compact (or even finite type) Riemann surfaces and complex algebraic curves.
\end{AlgebraicCurves}

\begin{MappingClassGroups}
Let $\Sigma$ be a compact orientable (topological) surface of genus $g$. The \textit{mapping class group} $\mathrm{Mod}(\Sigma )$ is the quotient of the group of orientation-preserving self-homeomorphisms of $\Sigma$ modulo the subgroup of those homeomorphisms which are isotopic to the identity. In case $\Sigma$ comes equipped with a $\mathrsfs{C}^{\infty }$ structure, we may as well only take diffeomorphisms, the resulting group is the same. Once and for all we fix a ``model surface'' $\Sigma_g$ and then write $\mathrm{Mod}_g$ instead of $\mathrm{Mod}(\Sigma_g)$.

The mapping class group can also be defined group-theoretically\footnote{The following only holds for $g>0$, but the case $g=0$ is trivial anyway}. Namely the fundamental group $\pi_g$ of $\Sigma$ is a priori only well-defined up to inner automorphisms, but so its outer automorphism group $\Out\pi_g$ only depends on $\Sigma_g$ and not on the choice of a base point. There is an obvious group homomorphism $\Mod_g\to\Out\pi_g$; this is injective, and the image is the index two subgroup $\Out^+\pi_g$ of those outer automorphisms that act trivially on $H^2(\pi_g,\bbz )$. The fundamental group $\pi_g$ has the following presentation:
\begin{equation*}
\pi_g=\langle a_1,\ldots ,a_g,b_1,\ldots ,b_g\, |\, a_1b_1a_1^{-1}b_1^{-1}\cdot a_2b_2a_2^{-1}b_2^{-1}\cdots a_gb_ga_g^{-1}b_g^{-1}=1\rangle .
\end{equation*}
\end{MappingClassGroups}

\begin{TeichmuellerMarkings}
Let $X$ be a compact Riemann surface of genus $g$. A \textit{(Teichm\"{u}ller) marking} of $X$ is an isotopy class of diffeomorphisms $\Sigma\to X$. Equivalently it can be described as an orientation-preserving\footnote{In the sense that it preserves the canonical orientations on group cohomology $H^2(\cdot ,\bbz )$} outer isomorphism $\pi_g\to\pi_1(X)$.
\end{TeichmuellerMarkings}

\begin{LevelStructures}
Let $X$ be a compact Riemann surface of genus $g$ and let $n\ge 1$ be an integer. Then a \textit{level-$n$-structure} on $X$ is a symplectic basis $(a_1,\ldots ,a_g,b_1,\ldots ,b_g)$ of $H^1(X,\bbz /n\bbz)$ as $\bbz /n\bbz$-module. Equivalently we can define a level-$n$-structure as a symplectic isomorphism $(\bbz /n\bbz )^{2g}\to H^1(X,\bbz /n\bbz )$; here the standard basis of the former is written as
$$(a_1,\ldots ,a_g,b_1,\ldots ,b_g),$$
with symplectic form
\begin{equation*}
\langle a_i,b_j\rangle =-\langle b_j,a_i\rangle =\delta_{ij},\quad \langle a_i,a_j\rangle =\langle b_i,b_j\rangle =0,
\end{equation*}
and the cohomology being endowed with the Poincar\'{e} duality form.
\end{LevelStructures}

\begin{FamiliesOfRiemannSurfacesAndAlgebraicCurves}
Let $C$ be a complex space. A \textit{family of Riemann surfaces over $C$} is a proper morphism of complex spaces $p:X\to C$ of relative dimension one with connected fibres. In other words, all fibres are compact Riemann surfaces of some genus; if $C$ is connected, this genus is the same for every fibre and called the \textit{genus} of the family.

Every family of Riemann surfaces is topologically a fibre bundle, hence the fundamental, homology and cohomology groups of the fibres can be organized into local systems. A \textit{Teichm\"{u}ller marking} of a family of Riemann surfaces $p:X\to C$ can be described in two equivalent ways. First it can be defined by giving a Teichm\"{u}ller marking on every fibre $X_c$ in a ``locally constant'' way; second, it can be defined as an outer isomorphism from the constant local system of groups $\pi_g$ on $C$ to the local system of the $\pi_1(X_c)$, defining a Teichm\"{u}ller marking on each fibre. For a rigorous definition see \cite[D\'{e}finition 2.2]{Grothendieck1960}.

The local system of first cohomology groups with $\bbz /n\bbz$ coefficients can be described more easily: this is the local system $R^1p_{\ast}(\bbz /n\bbz )$ on $C$. A \textit{level-$n$-structure} on the family $p:X\to C$ is then a symplectic isomorphism from the constant local system $(\bbz /n\bbz )^{2g}$ to $R^1p_{\ast}(\bbz /n\bbz )$.

Let $C$ be a noetherian scheme. Then a \textit{family of algebraic curves over $C$} is a proper smooth morphism of schemes $p:X\to C$ such that for every geometric point $\overline{c}\in C(\Omega )$ the fibre $X_{\overline{c}}$ is a smooth complete algebraic curve over $\Omega$. If $C$ is connected, again the genus of $X_{\overline{c}}$ is constant and called the \textit{genus} of the family. Teichm\"{u}ller markings do not have a direct analogue, but level structures do, using either Jacobians (as abelian varieties) or \'{e}tale cohomology.
\end{FamiliesOfRiemannSurfacesAndAlgebraicCurves}

\begin{TeichmuellerSpaces}
Consider the following moduli functor $T_g$ on the category of complex spaces: $T_g(C)$ is the set of isomorphism classes of families of Riemann surfaces of genus $g$ with Teichm\"{u}ller marking over $C$, with the action on morphisms being given by pullbacks. This functor is representable by a complex space $\mathrsfs{T}_g$ called \textit{Teichm\"{u}ller space of genus $g$}. By the universal property of $\mathrsfs{T}_g$ there exists a universal Teichm\"{u}ller marked family $\mathrsfs{U}_g\to\mathrsfs{T}_g$, briefly called the \textit{universal family over Teichm\"{u}ller space}. This is topologically (but of course not complex-analytically) a trivial fibre bundle $\Sigma_g\times\mathrsfs{T}_g$. Assigning to every point $t\in\mathrsfs{T}_g$ the class of the fibre $(\mathrsfs{U}_g)_t$ sets up a bijection between points of $\mathrsfs{T}_g$ and isomorphism classes of Teichm\"{u}ller marked Riemann surfaces of genus $g$. We will often identify these two sets by this bijection.

The complex space $\mathrsfs{T}_g$ is indeed a complex manifold, biholomorphic to a bounded domain in $\bbc^{3g-3}$ and homeomorphic to a $(6g-6)$-dimensional ball.

The mapping class group $\Mod_g$ acts from the left on $\mathrsfs{T}_g$: for $t\in\mathrsfs{T}_g$ represented by a Riemann surface $X$ with a Teichm\"{u}ller marking $f:\Sigma_g\to X$ and for $[\gamma ]\in\Mod_g$ represented by $\gamma :\Sigma_g\to \Sigma_g$, define the element $[\gamma ]\cdot t\in\mathrsfs{T}_g$ as represented by the same Riemann surface $X$ but now with the marking
\begin{equation*}
\Sigma_g\overset{\gamma^{-1}}{\to }\Sigma_g\overset{f}{\to }X.
\end{equation*}
This defines a faithful holomorphic and properly discontinuous action of $\Mod_g$ on $\mathrsfs{T}_g$.
\end{TeichmuellerSpaces}

\begin{ModuliSpacesOfRiemannSurfaces} Now consider the following moduli functor $M_g$, also on complex spaces: $M_g(C)$ is the set of isomorphism classes of families of Riemann surfaces over $C$. This functor is \textit{not} representable but admits a coarse moduli space $\mathrsfs{M}_g$. Hence the points of $\mathrsfs{M}_g$ are in bijection with isomorphism classes of compact Riemann surfaces of genus $g$. There is also a tautological family\footnote{Not a family of Riemann surfaces as we defined it!} $\mathrsfs{X}_g\to\mathrsfs{M}_g$; its fibre over a point of $\mathrsfs{M}_g$ represented by a Riemann surface $X$ is the surface $X/\Aut (X)$.

The coarse moduli space $\mathrsfs{M}_g$ can be constructed as the quotient $\Mod_g\backslash\mathrsfs{T}_g$; since the action of the mapping class group on Teichm\"{u}ller space is properly discontinuous, this quotient is a normal complex space. If one is willing to work with orbifolds (we shall avoid this in this paper) one can also view the orbifold quotient $\Mod_g\backslash\!\backslash\mathrsfs{T}_g$ as a fine moduli stack. There is a natural lift of the action of $\Mod_g$ on $\mathrsfs{T}_g$ to an action on $\mathrsfs{U}_g$; then the tautological family over $\mathrsfs{M}_g$ can be constructed as $\Mod_g\backslash\mathrsfs{U}_g$.

We rather work with rigidifications by level structures: they are close enough to moduli spaces for our purposes but have the advantage that they represent some functor and carry a universal family. So let $n\ge 1$ be an integer. The moduli functor $M_g^{[n]}$ on complex spaces sends a complex space $C$ to the set of all families of Riemann surfaces of genus $g$ with level-$n$-structure over $C$, up to isomorphism. For every $n$ this admits a coarse moduli space $\mathrsfs{M}_g^{[n]}$ (of course $\mathrsfs{M}_g^{[1]}=\mathrsfs{M}_g$), and for $n\ge 3$ this is even a \textit{fine} moduli space and a complex manifold. In particular for these $n$ there is a universal family $\mathrsfs{X}_g^{[n]}\to\mathrsfs{M}_g^{[n]}$. The forgetful map $\mathrsfs{M}_g^{[n]}\to\mathrsfs{M}_g$ is finite.

These moduli spaces with level structures can be constructed as follows: denote by $\Mod_g^{[n]}$ the subgroup of $\Mod_g$ consisting of those mapping classes which operate trivially on the cohomology group $H^1(\Sigma_g,\bbz /n\bbz )=H^1(\pi_g,\bbz /n\bbz )$. The quotient $\Mod_g^{[n]}\backslash\mathrsfs{T}_g$ is then $\mathrsfs{M}_g^{[n]}$; for $n\ge 3$ the group $\Mod_g^{[n]}$ is torsion-free, hence operates without fixed points, and so the quotient turns out to be a complex manifold.
\end{ModuliSpacesOfRiemannSurfaces}

\begin{ModuliSpacesOfAlgebraicCurves}
In contrast to Teichm\"{u}ller spaces, moduli spaces with or without level structures can also be constructed in the algebraic category. Since we only work in characteristic zero, we restrict ourselves to defining them over the rationals. So let $M_g^{\mathrm{alg}}$ be the functor on noetherian $\bbq$-schemes which to such a scheme $X$ associates the set of all families of algebraic curves over $X$, up to isomorphism. This functor is not representable but admits a coarse moduli space $\mathrsfs{M}_g^{\mathrm{alg}}$; this is a quasi-projective variety over $\bbq$. Its analytification is canonically isomorphic to $\mathrsfs{M}_g$.

For any positive integer $n$, define in analogy to the above a functor $M_g^{\mathrm{alg},[n]}$ on noetherian $\bbq$-schemes. This admits a coarse moduli space which for $n\ge 3$ is even a fine moduli space. In particular it then admits a universal family $\mathrsfs{C}_g^{\mathrm{alg},[n]}\to\mathrsfs{M}_g^{\mathrm{alg},[n]}$. The analytification of all these objects yield the corresponding analytic ones.
\end{ModuliSpacesOfAlgebraicCurves}

\begin{QuasiconformalMappings}
Let $U\subseteq\bbc$ be open and let $0\le k<1$. Then $f:U\to\bbc$ is called \textit{$k$-quasiconformal} if the following holds:
\begin{enumerate}
\item $f$ is open,
\item $f$ is a homeomorphism onto its image,
\item $f$ has partial distributional derivatives locally in $L^2$ and
\item almost everywhere one has
\begin{equation*}
\left\lvert\frac{\partial f}{\partial\overline{z}}\right\rvert\le k\cdot\left\lvert\frac{\partial f}{\partial z}\right\rvert.
\end{equation*}
\end{enumerate}
The smallest $k$ for which this holds is called the \textit{constant of quasiconformality} of $f$.
Note that some authors also call $K=\frac{1+k}{1-k}$ the quasiconformal constant; we prefer to call $K$ the \textit{quasiconformal dilatation}.

A map is $0$-quasiconformal (or, equivalently, quasiconformal with dilatation $1$) if and only if it is biholomorphic.

The composition of a quasiconformal map with a biholomorphic map (from the left or the right) gives again a quasiconformal map with the same constant of quasiconformality. In particular $k$-quasiconformal maps between Riemann surfaces are well-defined. A map between Riemann surfaces is $0$-quasiconformal (or, equivalently, quasiconformal with dilatation $1$) if and only if it is biholomorphic.
\end{QuasiconformalMappings}

\begin{BeltramiForms}
Let $X$ be a Riemann surface. A \textit{Beltrami form} on $X$ is an $L^1$ complex
antilinear endomorphism of the tangent bundle $TX$. Denote the space of Beltrami
forms on $X$ by $\mathcal{B}(X)$; this is a complex vector space. It can be endowed with a norm as follows: let $\mu\in\mathcal{B}(X)$, this defines for every point $x\in X$ a complex antilinear map $\mu_x :T_xX\to T_xX$. Since $T_xX$ is a one-dimensional complex vector space, the operator norm $\lVert\mu_x\rVert\in [0,\infty [$ is well-defined, i.e. does not depend on a choice of norm on $T_xX$. Take then $\lVert\mu\rVert$ to be the essential supremum of all $\lVert\mu_x\rVert$, where $x$ ranges over all points of $X$. With this norm $\mathcal{B}(X)$ becomes a Banach space.

Denote by $\mathcal{B}^{<1}(X)$ the open unit ball in $\mathcal{B}(X)$, i.e. the set of all Beltrami forms of norm less than one. Every element $\mu\in\mathcal{B}^{<1}(X)$ defines a new complex structure on $X$, as follows: consider a holomorphic atlas $(\varphi_i:U_i\overset{\simeq}{\to}V_i\subseteq\bbc )_i$ of $X$. For every $i$ then $\mu |_{U_i}$ corresponds to a Beltrami form
\begin{equation*}
\mu_i(z)\frac{\de\overline{z}}{\de z}
\end{equation*}
on $V_i$. Then let $\psi_i:V_i\to\bbc $ be some quasi-conformal homeomorphism with Beltrami coefficient $\mu_i$, i.e. such that
\begin{equation*}
\frac{\partial\psi_i(z)}{\partial \overline{z}}=\mu_i(z)\frac{\partial\psi_i(z)}{\partial z}.
\end{equation*}
Then the collection $(\psi_i\circ\varphi_i :U_i\to\psi_i(V_i))_i$ is a holomorphic atlas on the smooth surface underlying $X$; this is \textit{not} necessarily compatible with the old atlas, so we get a possibly new Riemann surface structure on $X$. We call this new Riemann surface $X_{\mu}$. Since its underlying smooth surface is identified with that of $X$, it automatically inherits a marking. Thus we have defined a map
\begin{equation}\label{BeltramiToTeichmueller}
\Phi :\mathcal{B}^{<1}(X)\to\mathrsfs{T}_{g,0},\quad \mu\mapsto X_{\mu }.
\end{equation}
This is a holomorphic submersion of complex Banach manifolds, in particular surjective.
\end{BeltramiForms}

\subsection{Curves in Moduli Space and Families of Curves}

The moduli space of curves (or Riemann surfaces) only being a coarse moduli space, there is a subtle difference between curves in the moduli space of curves and families of curves over a curve.

First let us explain how to pass from a family to a curve. So let $C$ be an algebraic curve over the complex numbers with at most ordinary double-points as singularities, and let $p: X\to C$ be a family of curves of genus $g\ge 2$. This yields, by definition of $\mathrsfs{M}_g$, a regular map $C\to\mathrsfs{M}_g$. It does not, however, always suffice to know this map in order to reconstruct the family. Hence it is useful to choose a rigidification, either a level structure or a Teichm\"{u}ller marking. But this forces us to pass to a cover of the original family. Let $C'\to C$ be the normalization of $C$; then $C'$ is a smooth algebraic curve, or a finite type Riemann surface. By pullback
\begin{equation*}
\xymatrix{
X'\ar[r]^{p'}
\ar[d]
&C' \ar[d]\\
X\ar[r]_{p}
&C
}
\end{equation*}
we obtain a family of curves $X'\to C'$. Now let $\widetilde{C'}\to C'$ be the universal cover of $C'$; again by pullback we get a family of curves $\widetilde{p'}:\widetilde{X'}\to\widetilde{C'}$. Now $\widetilde{C'}$ being simply connected we may choose a Teichm\"{u}ller marking of this family; this in turn defines a map $\widetilde{C'}\to\mathrsfs{T}_g$, by definition of Teichm\"{u}ller space. Clearly the following diagram commutes:
\begin{equation*}
\xymatrix{
\widetilde{C'}\ar[r]
\ar[d]
&\mathrsfs{T}_g \ar[d]\\
C\ar[r]_{p}
&\mathrsfs{M}_g,
}
\end{equation*}
and the family $\widetilde{X'}\to \widetilde{C'}$ can be reconstructed from the classifying map $\widetilde{C'}\to\mathrsfs{T}_g$.

Now let us consider the case of level structures. So choose some $n\ge 3$ and consider the monodromy representation of the local system $R^1p_{1,\ast }(\bbz /n\bbz )$ on $C'$; this is a homomorphism $\pi_1(C',c)\to\operatorname{Sp}_{2g}(\bbz /n\bbz )$. Since the latter group is finite, the kernel of this map is a finite index normal subgroup $\Pi$ of $\pi_1(C',c)$. This corresponds to a finite cover $T\to C'$, and the pullback of the family $p'$ along this covering map yields a family $q:Y\to T$ whose monodromy representation on $H^1(\cdot ,\bbz /n\bbz )$ is trivial. Hence there exists a level-$n$-structure of the family $q$, and this defines a regular map $T\to\mathrsfs{M}_g^{[n]}$. Now again this map allows to reconstruct the original family, but it has two advantages over the Teichm\"{u}ller marked family constructed above. First, it keeps much more of the topological information of the original family; second, it remains entirely within the realm of algebraic geometry. Our construction of course only works for complex algebraic geometry, but using \'{e}tale fundamental groups and \'{e}tale local systems instead of the ``classical'' ones, the construction can be carried out over an arbitrary base field (after possibly passing to a finite extension of the base field).

There is --- for the moment --- less to say about the other direction. Namely let $C$ be an algebraic curve with at most double point singularities and let $f:C\to\mathrsfs{M}_g$ be a regular map. This does not directly yield a family of algebraic curves over $C$, only over a finite (possibly ramified) cover of it: consider the forgetful map $\mathrsfs{M}_g^{[n]}\to\mathrsfs{M}_g$, for $n\ge 3$. Pulling back $f$ along this map gives a cartesian diagram
\begin{equation*}
\xymatrix{
C'\ar[r]^{f'}
\ar[d]
&\mathrsfs{M}_g^{[n]} \ar[d]\\
C\ar[r]_{f}
&\mathrsfs{M}_g .
}
\end{equation*}
Since the pullback of finite morphisms is finite, the map $C'\to C$ is also finite. In particular $C'$ is again an algebraic curve. Now by the universal property of $\mathrsfs{M}_g^{[n]}$ we get from $f'$ a family $X'\to C'$ which comes equipped with a level-$n$-structure. This family, even forgetting the level structure, does not necessarily descend to a family on $C$.

\subsection{The Kobayashi Metric}

It turns out that interesting families of algebraic curves can only occur over algebraic curves which are themselves hyperbolic. ``Interesting'' here means ``not isotrivial'':

\begin{Definition}
Let $C$ be a complex space (or an algebraic variety). A family $p:X\to C$ of curves is called
\begin{itemize}
\item \textup{isotrivial} if all its fibres are abstractly isomorphic;
\item \textup{trivial} if it is isomorphic, as a family, to $\operatorname{pr}_2 :X_0\times C\to C$ for some curve $X_0$.
\end{itemize}
\end{Definition}

Every trivial family is isotrivial, but the converse does not hold. A family $p:X\to C$ is isotrivial if and only if the corresponding classifying map $f:C\to\mathrsfs{M}_g$ is constant. We shall henceforth only consider non-isotrivial families.

In order to determine over which curves non-isotrivial families are possible at all, we have to take a closer look at the differential geometry of the involved spaces.

\begin{TheKobayashiSemimetric}
On any complex space $X$ there is canonical semimetric, the Kobayashi semimetric. In many cases, in particular most of the cases interesting to us, it is in fact a metric. The naturality of its definition implies that it is invariant under biholomorphic maps and has some other nice properties. It is obtained by considering all holomorphic mappings from the unit disk to $X$ and using the Poincar\'{e} metric on the former.

To begin with, let $X$ be a complex space and let $p$ and $q$ be points of $X$. By a \textit{holomorphic chain} from $p$ to $q$ (on $X$) we mean a finite sequence of holomorphic maps $f_{\nu }:\Delta\to X$ together with points $z_{\nu },w_{\nu}\in\Delta$, where $\nu$ varies from $1$ to $n$, subject to the following conditions:
\begin{equation*}
f_1(z_1)=p,\quad f_{\nu}(w_{\nu})=f_{\nu +1}(z_{\nu +1}),\quad f_n(w_n)=q.
\end{equation*}
It is easy to see that if $X$ is a path-connected complex manifold and $p,q\in X$, there is always a holomorphic chain from $p$ to $q$. Namely cover $X$ by open subsets which are biholomorphic to the polydisk $\Delta^r$ and choose a path from $p$ to $q$. The image of the path is compact, so it can be covered by finitely many of the given subsets, which reduces the problem to the
case $X=\Delta^r$ where it is clear.
\end{TheKobayashiSemimetric}
\begin{Definition}
The \textup{Kobayashi semimetric} on a complex space $X$ is the function $k_X:X\times X\to [0,\infty ]$ given by
\begin{equation*}
k_X(p,q)=\inf\sum_{\nu =1}^n\varrho (z_{\nu },w_{\nu }).
\end{equation*}
Here $\varrho$ is the Poincar\'{e} distance on $\Delta$ and the infimum is taken over all holomorphic chains from $p$ to $q$ as above.
\end{Definition}
It is readily checked that $k_X$ is a semimetric, i.e. it satisfies all the usual axioms for a metric with the exception that $k_X(p,q)$ is allowed to be zero also for $p\neq q$, and also to be infinity. Note that taking holomorphic chains instead of just holomorphic maps $\Delta\to X$ enforces the triangle inequality.
On complex manifolds there is also an infinitesimal form of the Kobayashi semimetric, the Koba\-ya\-shi seminorm. This is a norm on the tangential space $T_pX$, smoothly varying with $p$. Such allows to transfer rudiments of Riemannian geometry.
\begin{Definition}
Let $X$ be a complex manifold. The \textup{Kobayashi seminorm} on $X$ is the function $K_X:TX\to [0,\infty [$ whose value at a tangent vector $v\in T_pX$ is
\begin{equation*}
K_X(v)=\inf \frac{1}{r}
\end{equation*}
where the infimum ranges over all $r>0$ such that there exists a holomorphic map $f:\Delta_r\to X$ (here $\Delta_r$ is the disk of radius $r$ around zero) with $f(0)=p$ and $f'(0)=v$. For a nonzero tangent vector $v$ we can also describe this as the infimum of all $\| f_{\ast}^{-1}(v)\|$ where $f:\Delta\to X$ is a holomorphic function with $f(0)=p$ and $f'(0)=v$.
\end{Definition}
We will not dwell here on the smoothness properties of $K_X$ and other subtleties. The important point is that Royden has shown that from this infinitesimal metric one can retrieve the global metric. Namely for a piecewise smooth path $\gamma :[a,b]\to X$ we can define the \textit{Kobayashi length} as
\begin{equation*}
\ell (\gamma )=\int_a^bK_X(\gamma '(t))\de t;
\end{equation*}
this is unchanged under reparametrization of $\gamma$. Then $k_X(p,q)$ is the infimum of $\ell (\gamma )$ where $\gamma$ ranges over all piecewise smooth paths from $p$ to $q$. This is Theorem 3 in \cite{Royden70}.

The following is now easy:
\begin{Proposition}\label{CoveringMapsAreLocalKIsoms}
Let $f:X\to Y$ be a holomorphic map between complex manifolds, and let $v\in TX$. Then $K_Y(f_{\ast}(v))\le K_X(v)$. If $f$ is a covering map, equality holds, so that $f$ is a local Kobayashi isometry.\hfill $\square$
\end{Proposition}
This means the following for the global Kobayashi semimetric:
\begin{Corollary}\label{HolomorphicMapsAreKobayashiContractions}
Let $f:X\to Y$ be a holomorphic map between complex manifolds, and let $x_1,x_2\in X$. Then one has $k_Y(f(x_1),f(x_2)\le k_X(x_1,x_2)$. If $f$ is a covering map one gets more precisely
\begin{equation*}
k_Y(y_1,y_2)=\inf k_X(x_1,x_2)
\end{equation*}
with the infimum ranging over all preimages $x_1$ of $y_1$ and $x_2$ of $y_2$.\hfill $\square$
\end{Corollary}
For later reference we also note how to compute the Kobayashi semimetric of a product:
\begin{Proposition}\label{KobayashiMetricOnProduct}
Let $X$ and $Y$ be complex spaces. Then the Kobayashi semimetric on $X\times Y$ is given by the formula
\begin{equation*}
k_{X\times Y}((x_1,y_1),(x_2,y_2))=\max (k_X(x_1,x_2),k_Y(y_1,y_2)).
\end{equation*}
\end{Proposition}
\begin{proof}
This is \cite[Theorem (3.1.9)]{Kobayashi98}.
\end{proof}
\begin{KobayashiHyperbolicSpaces}
A complex space $X$ is called \textit{Kobayashi hyperbolic} if its Kobayashi semimetric is actually a metric, i.e. if $k_X(x,y)>0$ whenever $x\neq y$. In this case we call $k_X$ the \textit{Kobayashi metric.}
First note that Proposition \ref{CoveringMapsAreLocalKIsoms} implies:
\begin{Proposition}
Let $X$ be a complex manifold and let $\tilde{X}$ be a universal covering space of $X$. Then $X$ is Kobayashi hyperbolic if and only if $\tilde{X}$ is.\hfill $\square$
\end{Proposition}
The notion of Kobayashi hyperbolicity can now be justified by the following simple observation.
\end{KobayashiHyperbolicSpaces}
\begin{Proposition}
Let $X$ be a Riemann surface. If the universal covering space of $X$ is $\bbc$ or $\mathbb{P}^1(\bbc )$, then the Kobayashi semimetric on $X$ is identically zero. If the universal covering space of $X$ is $\Delta$, then the Kobayashi semimetric on $X$ equals the Poincar\'{e} metric descending from $\Delta$.

In particular a Riemann surface is Kobayashi hyperbolic if and only if it is hyperbolic in the usual sense.\hfill $\square$
\end{Proposition}
We will now collect some easy properties of Kobayashi hyperbolic manifolds. To begin with, if $X$ is an complex subspace of $Y$, then the inclusion $X\to Y$ is holomorphic, hence distance non-increasing for the Kobayashi semimetrics. So any complex subspace of a Kobayashi hyperbolic space is again Kobayashi hyperbolic.

From Proposition \ref{KobayashiMetricOnProduct} we also see that products of Kobayashi hyperbolic spaces are again Kobayashi hyperbolic. Hence for instance all polydisks $\Delta^n$ are Kobayashi hyperbolic, and we see:
\begin{Proposition}
Every bounded domain in $\bbc^n$ is Kobayashi hyperbolic.\hfill $\square$
\end{Proposition}
\begin{TeichmuellerMetric}
Each Teichm\"{u}ller space $\mathrsfs{T}_g$ is biholomorphic to a bounded symmetric domain in $\bbc^{3g-3}$, e.g. by Bers' embedding. Hence $\mathrsfs{T}_g$ is Kobayashi hyperbolic. This allows us to partially answer the question raised at the beginning of this section:
\end{TeichmuellerMetric}
\begin{Corollary}
Let $C$ be Riemann surface such that there exists a non-isotrivial family $X\to C$ of curves of genus $g$. Then $C$ is hyperbolic.
\end{Corollary}
\begin{proof}
Assume that $C$ is not hyperbolic. Then its universal covering space $\widetilde{C}$ is isomorphic to either $\bbc$ or $\mathbb{P}^1(\bbc )$. In particular the Kobayashi semimetric on $\widetilde{C}$ is identically zero.

Let $X\to C$ be any family of curves of genus $g$; this lifts to a family over $\widetilde{C}$ which admits a Teichm\"{u}ller marking. Hence a holomorphic map $f:\widetilde{C}\to\mathrsfs{T}_g$. This is distance-non-increasing for the Kobayashi metric, and so for any $c,d\in\widetilde{C}$ one has
\begin{equation*}
0=k_{\widetilde{C}}(c,d)\ge k_{\mathrsfs{T}_g}(f(c),f(d)),
\end{equation*}
hence $f(c)=f(d)$. This means that the map $f$ is constant, hence the family is isotrivial.
\end{proof}

If $f:X\to C$ is a non-isotrivial family of curves, then the universal covering space of $C$ is isomorphic to the unit disk. After fixing an isomorphism we may thus interpret the associated classifying map as a nonconstant holomorphic map $\Delta \to\mathrsfs{T}_g$. By Corollary \ref{HolomorphicMapsAreKobayashiContractions} this is distance-non-increasing for the Kobayashi metric. So one might say that the family is maximally apart from being isotrivial when the classifying map is an isometry, and study isometric embeddings $\Delta\to\mathrsfs{T}_g$.

\begin{Definition}
A \textup{Teichm\"{u}ller disk} is a holomorphic Kobayashi isometry $f:\Delta\to\mathrsfs{T}_g$. Sometimes we also call images of such maps Teichm\"{u}ller disks.
\end{Definition}
It turns out that Teichm\"{u}ller disks exhaust Teichm\"{u}ller space in the following precise sense: for every two distinct points in $\mathrsfs{T}_g$ there exists a unique\footnote{meaning that it is unique as a subspace of Teichm\"{u}ller space} Teichm\"{u}ller disk containing both of them.\footnote{This follows from Teichm\"{u}ller's Theorem, see \cite{Hubbard06}} The picture becomes very different when one requires the map $f:\Delta\to\mathrsfs{T}_g$ to be the classifying map of a family as above, with $C$ of finite type. This is equivalent to requiring that the image of the Teichm\"{u}ller disk in $\mathrsfs{M}_g$ (or in $\mathrsfs{M}_g^{[n]}$) is closed, or is an algebraic curve. Such curves are called \textit{Teichm\"{u}ller curves} and provide a spectacularly rich geometry and arithmetic, of which this work discusses just some aspects.

On Teichm\"{u}ller spaces the Kobayashi metric also has a more concrete description, which helps us in understanding Teichm\"{u}ller disks and curves.
\begin{Definition}
Let $g\ge 2$. Then the \textup{Teichm\"{u}ller metric} $d_T$ on $\mathrsfs{T}_g$ is defined as follows: let $x,y\in\mathrsfs{T}_g$ be represented by marked Riemann surfaces $\Sigma_g\to X$ and $\Sigma_g\to Y$. Then we set
\begin{equation*}
d_T(x,y)=\inf_{f}\frac{1}{2}\log K(f),
\end{equation*}
where $f$ ranges over all quasi-conformal maps $X\to Y$ such that the diagram
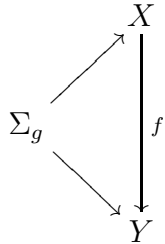
\begin{equation*}
\xymatrix{
&X\ar[dd]^f\\
\Sigma_g\ar[ur] \ar[dr]\\
&Y
}
\end{equation*}
commutes up to isotopy, and $K(f)$ is the quasiconformal dilatation of $f$.
\end{Definition}
One can show that this defines a complete metric on $\mathrsfs{T}_g$, see \cite{Hubbard06}. Now:
\begin{Theorem}[Royden]\label{KobayashiEqualsTeichmueller}
Let $g\ge 1$. Then the Kobayashi metric on Teichm\"{u}ller space $\mathrsfs{T}_g$ equals the Teichm\"{u}ller metric.
\end{Theorem}
\begin{proof}
This is \cite[Theorem 3]{Royden71}.
\end{proof}
Now the mapping class group $\Mod_g$ acts on $\mathrsfs{T}_g$ by holomorphic automorphisms, hence by Kobayashi isometries. In particular the Teichm\"{u}ller metric descends to any quotient $\Gamma\backslash\mathrsfs{T}_g$ for a subgroup $\Gamma$ of $\Mod_g$. If $\Gamma$ is torsion free, this map is an unramified covering, hence a local Kobayashi isometry. Note that there are finite index torsion-free subgroups of the mapping class group, e.g. $\Mod_g^{[n]}$ for $n\ge 3$; this is \cite[Corollary 1.5]{Ivanov92}. On $\mathrsfs{M}_g=\Mod_g\backslash\mathrsfs{T}_g$ we also get a well-defined quotient metric, but this is not necessarily equal to the Kobayashi metric of $\mathrsfs{M}_g$ itself! For example for $g=1$ we have $\mathrsfs{M}_1\simeq\bbc$ which is clearly not Kobayashi hyperbolic.

By the virtue of Theorem \ref{KobayashiEqualsTeichmueller} we can easily describe the quotient metric on $\mathrsfs{M}_g$: the distance between two points in $\mathrsfs{M}_g$, represented by Riemann surfaces $X$ and $Y$, is
\begin{equation*}
\inf_f\frac{1}{2}\log K(f)
\end{equation*}
where $f$ now runs over \textit{all} quasiconformal maps $X\to Y$.

\newpage

\section{Abelian Differentials: Local Theory}

\noindent In the third chapter we will describe a method, geometrically very explicit, to construct families of curves which are maximally non-isotrivial in the sense of the previous chapter, i.e. which correspond to holomorphic isometric embeddings of the hyperbolic plane into Teichm\"{u}ller space. The idea is simple: a complex structure on a topological surface can be thought of as a conformal equivalence class of metrics. Take then some particularly nice metric in that conformal class, and deform it in the simplest possible way. One usually takes a metric with constant curvature, in our case hence the hyperbolic metric. But hyperbolic geometry is more difficult to deform than euclidean geometry. Now for topological reasons\footnote{i.e. because of the Gauss-Bonnet Theorem} we cannot find a flat metric on a surface of genus greater than one --- unless we allow singularities. So we take a flat metric which has finitely many possibly mild singularities.

There are several possible formalizations of this intuition; the one most suitable for our purposes is that of \textit{translation structures with cone points}. Such a structure on a surface $\Sigma$ consists of an atlas defined outside a finite set of ``cone points'' $P\subset\Sigma$ (thought of as the set of singularities), where all transition maps are affine translations $\bbc\to\bbc$, together with certain conditions on what happens at points in $P$. Clearly transitions are biholomorphic, so we can also view the atlas as a holomorphic atlas. The extra conditions ensure that this holomorphic atlas extends to a holomorphic atlas on all of $\Sigma$. On the other hand it also induces a flat euclidean metric on $\Sigma\smallsetminus P$ in the conformal class of this holomorphic atlas.

The plan is then to deform such a translation structure by postcomposing the atlas with a fixed element of $\SL_2(\bbr )$. Here we let $\SL_2(\bbr )$ act on $\bbc$ via the usual identification $\bbc =\bbr^2$ (with $x+\mathrm{i}y$ corresponding to $(x,y)$). This will then give a map $\SL_2(\bbr )\to\mathrsfs{T}_g$ which is the technical heart of the desired construction of families of curves.

But before we can do this, we have to analyze the geometry of one single translation structure. Namely the details in the definition of a translation structure with cone points are made in a way that such structures correspond bijectively to abelian differentials, i.e. nonzero holomorphic one-forms. It is easy to explain how this looks like outside the singularities: the holomorphic one-form corresponding to a translation structure is the unique form which in all charts $z$ of the translation atlas becomes $\de z$. To phrase it the other way, the translation charts of a holomorphic one-form are obtained by locally integrating the holomorphic one-form. The cone points of the translation structure are then precisely the zeros of the abelian differential.

We also discuss in detail our most important examples: the \textit{Wiman curves of type I}. These are curves which have an exceptionally large automorphism group, in the following sense: by a classical theorem of Wiman, an automorphism of a compact Riemann surface of genus $g\ge 2$ has order at most $4g+2$. Moreover in each genus $g\ge 2$ there is a unique compact Riemann surface which admits an automorphism of order equal to $4g+2$. This can most easily be described as an algebraic curve $W_g$: it is the hyperelliptic curve of genus $g$ which is branched over all $n$-th roots of unity and $\infty$ (here $n=2g+1$). In other words, it is the projective algebraic curve which is birationally equivalent to the affine curve with equation\footnote{Note that this is \textit{not} its projective closure $y^2z^{n-2}=x^n-z^n$, but rather the normalization of that curve}
\begin{equation}\label{AffineEquationOfWimanCurve}
y^2=x^n-1.
\end{equation}
For proofs and further information about these curves, see \cite{Kulkarni97}.

A basis of $\Omega^1(W_g)$ is given by the abelian differentials $\omega_1,\ldots ,\omega_g$ which can in affine coordinates be represented as
\begin{equation}\label{WimanDifferential}
\omega_k=\frac{x^{k-1}\de x}{y}.
\end{equation}
Now each pair $(W_g,\omega_k)$ can, as indicated above, be interpreted as a surface with a translation structure. We explicitly compute that $(W_g,\omega_1)$ can be viewed as two copies of a regular $n$-gon with corresponding sides glued, and that $(W_g,\omega_g)$ can be viewed as a regular $2n$-gon with opposite sides glued. Also the ``intermediate'' surfaces $(W_g,\omega_k)$ with $1\le k\le g$ admit an interpretation in terms of Euclidean geometry: they are obtained by glueing rotated copies of an isosceles triangle with apex angle $k\pi /n$.

\subsection{The Geometry of Abelian Differentials}

\begin{TranslationSurfaces}
We now give the technical details for a correspondence between abelian differentials and translation structures. We begin with the easiest special case: no singularities, no zeros.
\end{TranslationSurfaces}
\begin{Definition}\label{DefinitionTranslationSurfaceWithoutAnything}
Let $\Sigma$ be a (topological) surface. A \textup{translation structure} on $\Sigma$ is an atlas
	\begin{equation*}
	\mathfrak{T}=(\varphi_i:U_i\overset{\simeq}{\to}V_i)_i\text{ with }V_i\subseteq\bbc\text{ open,}
	\end{equation*}
	such that
	\begin{enumerate}
	\item all chart transitions are of the form $z\mapsto z+c$ for some $c\in\bbc $, i.e. just affine translations;
	\item the atlas $\mathfrak{T}$ is maximal amongst atlases satisfying (i).
	\end{enumerate}
	A surface together with a translation structure is also called a \textup{translation surface.}
	\end{Definition}

We claim that a translation structure on a surface is ``the same'' as a complex structure together with a nowhere vanishing holomorphic (with respect to that complex structure) one-form. To start with let $\Sigma$ be a surface with a translation structure $\mathfrak{T}$. Since all chart transition maps are biholomorphic, clearly $\mathfrak{T}$ is a holomorphic atlas and hence determines a complex structure on $\Sigma$. Since the canonical holomorphic one-form $\mathrm{d}z$ on $\bbc $ has no zeros and is invariant under affine translations, its pullbacks along all charts in $\mathfrak{T}$ glue together to a holomorphic one-form without zeros on $\Sigma$.
	
For the other direction let $X$ be a Riemann surface and let $\omega$ be a holomorphic one-form on $X$ which is nowhere zero. We consider all holomorphic charts $z$ of $X$ in which $\omega$ obtains the form $\mathrm{d}z$; we claim that they form a translation structure $\mathfrak{T}_{\omega}$. Clearly all chart transition maps are affine translations; the only thing which has to be checked is that these charts cover $X$. But locally around $z_0\in X$ one can define
\begin{equation*}
X\supseteq U\overset{\simeq}{\to}V\subseteq\bbc,\quad z\mapsto\int_{z_0}^z\omega ,
\end{equation*}
where $U$ is a neighborhood of $z_0$, small enough so that the integral is well-defined; this is a chart in $\mathfrak{T}_{\omega }$.

We summarize:	
	\begin{Proposition}
	Let $\Sigma$ be a topological surface. The following two sets are in bijection via the construction described above:
	\begin{enumerate}
	\item The set of all translation structures on $\Sigma$;
	\item The set of all pairs $(\mathrsfs{O},\omega )$, where $\mathrsfs{O}$ is a complex structure on $\Sigma$, and $\omega$ is a holomorphic one-form without zeros on the Riemann surface $(\Sigma ,\mathrsfs{O})$.\hfill $\square$
	\end{enumerate}
	\end{Proposition}
\begin{ConePoints}
Translation structures in this sense cannot exist on compact surfaces of genus greater than one. This can be seen using either the Gauss-Bonnet formula (which implies that the Euler characteristic of a closed surface with a flat metric must be zero) or the fact that the degree of the canonical bundle on a Riemann surface of genus $g$ has degree $2g-2$, hence if $g>1$, every abelian differential has to have some zero. Hence we need to allow certain exceptional points in the interior. Let us first explain this with a local model: let $n$ be a positive integer and let $\varepsilon> 0$. Set $\delta =\sqrt[n]{\varepsilon }$, and consider the map
\begin{equation*}
\pi :\Sigma_n(\varepsilon )=\Delta_{\delta }\to\Delta_{\varepsilon },\quad z\mapsto z^n.
\end{equation*}
When removing $0$ both in the domain and the target, this defines a cyclic cover of degree $n$. Define a translation structure (in the old sense) on $\Sigma_n(\varepsilon )\smallsetminus \{ 0\}$ by pulling back that on $\Delta_{\varepsilon }$ by $\pi$; i.e. let $\mathfrak{T}_n(\varepsilon )$ be the set of all $\pi|_U:U\to\bbc$, where $U$ is an open subset of $S\smallsetminus\{ 0\}$ such that the restriction $\pi|_U$ is injective.
One can also imagine $\Sigma_n(\varepsilon )$ as $n$ copies of a ``slit disk'' $\Delta_{\varepsilon }\smallsetminus [0,\varepsilon [$, indexed by $\bbz /n\bbz$, where ``upper bank'' of the slit in the $k$-th copy is glued to the ``lower bank'' of the slit in the $(k+1)$-st copy.
\end{ConePoints}
\begin{Definition}\label{DefTranslationStructureWithConePoints}
Let $\Sigma$ be a topological surface with boundary. A \textup{translation structure with cone points} on $\Sigma$ consists of the following data:
\begin{itemize}
\item a discrete subset $P\subseteq\Sigma$, the set of ``cone points'', and
\item an atlas
\begin{equation*}
\mathfrak{T}=(\varphi_i:U_i\overset{\simeq}{\to }V_i)_i
\end{equation*}
on $\Sigma\smallsetminus P$ with $V_i\subseteq\bbc$ open,
\end{itemize}
such that the following conditions are satisfied:
\begin{enumerate}
\item The atlas $\mathfrak{T}$ on $\Sigma\smallsetminus P$ defines a translation structure on the surface $\Sigma\smallsetminus P$;
\item for every cone point $p\in P$ there exist a positive integer $n$, some $\varepsilon>0$, a neighbourhood $U$ of $p$ in $\Sigma$ and a homeomorphism $U\to \Sigma_n(\varepsilon )$ which sends $p$ to $0$ and such that the restriction of the atlas $\mathfrak{T}$ to $U\smallsetminus\{ p\}$ corresponds via this homeomorphism to the atlas $\mathfrak{T}_n(\varepsilon )$ on $\Sigma_n(\varepsilon )\smallsetminus\{ 0\}$.
\end{enumerate}
A surface together with a translation structure with cone points is called a \textup{translation surface with cone points}.
\end{Definition}
For every cone point $p$, the integer $n$ as in (ii) is uniquely determined and called the \textit{multiplicity} of the cone point. It can be described more intuitively as follows: when walking once around $p$ (in the topological sense, i.e. along a little circle embedded around $p$) one travels by an angle of $2\pi n$, instead of $2\pi$ as would be the case around a non-singular point.

Note that a cone point of multiplicity one has a particularly simple description; in fact the atlas of a translation structure can be uniquely extended into such a cone point, turning it into a regular point. For this reason we shall usually assume that all cone points have multiplicity at least two.

This definition of cone points is somewhat ad hoc, but serves our purpose. In the literature on this subject one can find quite diverse definitions.
\begin{Dictionary}
Now comes the promised relation with Riemann surface theory. We claim that on a given topological surface there is a natural one-to-one correspondence between complex structures together with an abelian differential on the one hand and translation structures with cone points on the other hand. In order to describe this correspondence first start with a translation surface $\Sigma$ with cone points, with notation as in the definition. We define a holomorphic atlas on $\Sigma$ as containing all charts of the following form:
\begin{itemize}
\item the charts in the translation structure $\mathfrak{T}$ on $\Sigma\smallsetminus P$;
\item for every cone point $p\in P$ of multiplicity $n$, choose some homeomorphism $U\to\Sigma_n(\varepsilon )$ as in Definition \ref{DefTranslationStructureWithConePoints} and note that $\Sigma_n(\varepsilon )$ was defined as $\Delta_{\delta }$, where $\delta =\sqrt[n]{\varepsilon }$, hence as an open subset of $\bbc$; so we define the composed map
\begin{equation*}
U\to \Sigma_n(\varepsilon )=\Delta_{\delta }\subset\bbc
\end{equation*}
as a chart of our atlas.
\end{itemize}
It is easily seen that all chart transition maps of this atlas are biholomorphic, so it can be uniquely extended to a maximal holomorphic atlas, defining a Riemann surface structure on $\Sigma$. On $\Sigma\smallsetminus P$ this is of course just the Riemann surface structure defined above. Also by the above construction we get a nowhere vanishing holomorphic one-form $\omega $ on $\Sigma\smallsetminus P$ (recall this was defined by pasting together the $\de z$, where $z$ is a chart in $\mathfrak{T}$), and we claim that it can be uniquely extended to a holomorphic one-form on all of $\Sigma$. Of course it suffices to show this locally at a cone point, in other words, in the ``model surface'' $\Sigma_n(\varepsilon )$. Now the translation structure on $\Sigma_n(\varepsilon )\smallsetminus \{ 0\} =\Delta_{\delta }^{\times }$ is the pullback along $\pi$ of the canonical one on $\Delta_{\varepsilon }^{\times }$, hence the corresponding one-form is
\begin{equation*}
\pi^{\ast} (\de z )=\de\pi (z)=\de (z^n)=nz^{n-1}\,\de z.
\end{equation*}
This of course can be uniquely extended to a holomorphic one-form on $S_n(\varepsilon )=\Delta_{\delta }$, by the very same formula. Note that this has a zero of order $n-1$ at $p$.

Vice versa let $\Sigma$ be a Riemann surface and let $\omega$ be a holomorphic one-form on $\Sigma$ which is not identically zero. Let $P$ be the set of zeros of $\omega$; this is a discrete subset. Then the restriction of $\omega $ to $S\smallsetminus P$ is a nowhere vanishing holomorphic one-form, defining a translation structure on $\Sigma$. It remains to be shown that this translation structure has the desired behaviour around the zeros of $\omega$.

To see this, let $p\in P$ be a zero of $\omega $ of order $n-1$ and choose some simply connected neighbourhood $U$ of $p$ which contains no other zeros of $\omega$. Then the function $f:U\to\bbc$ with
\begin{equation*}
f(q)=\int_p^q\omega
\end{equation*}
(the integral running over any path from $p$ to $q$ entirely within $U$) is holomorphic and has an isolated zero of order $n$ at $p$. Hence there exists some local coordinate $g :U\to\bbc$ with $f(q)=g(q)^n$. By shrinking $U$ we may assume that the image of $g$ is a disk $\Delta_{\delta }$ for some $\delta>0$; then setting $\varepsilon =\delta^n$ we get the following commutative diagram:
\begin{equation*}
\xymatrix{
U \ar[r]_{\simeq }^g \ar[rd]_f
&\Delta_{\delta } \ar[d]^{(\cdot )^n}\\
&\Delta_{\varepsilon }
}
\end{equation*}
But here $f$ is a holomorphic map which outside $p$ is a finite cover and such that the pullback of the canonical one-form $\de z$ along $f$ is $\omega $; so this diagram means that indeed the whole setup is locally isomorphic to the standard one, and hence that $\Sigma$ is a translation surface with cone points. We sum up:
\end{Dictionary}
\begin{Proposition}
Let $\Sigma$ be a topological surface (without boundary). The above constructions define mutually inverse bijections between the following two sets:
\begin{enumerate}
\item the set of all translation structures with cone points on $\Sigma$;
\item the set of all pairs $(\mathrsfs{O},\omega )$, where $\mathrsfs{O}$ is a complex structure and $\omega$ is an abelian differential on $S$ (with respect to this complex structure).
\end{enumerate}
The cone points (of multiplicity greater than one) of a translation structure are precisely the zeros of the associated abelian differential, a cone point of multiplicity $n$ corresponding to a zero of order $n-1$.\hfill $\square$
\end{Proposition}

\subsection{Some Examples}

A large class of examples of translation surfaces with cone points can be produced as follows: take some (compact, simply connected) polygon in the plane, and partition its set of edges into pairs, such that each two edges in a pair are parallel and have the same length, but with reverse orientation. Then for each pair there exists exactly one translation taking one edge to the other; glue the edges along these translations. This gives a compact surface with a translation atlas on an open dense subset; this atlas extends to a translation structure with cone points on the whole surface.

\begin{Origamis}
We begin with \textit{origamis.} These are in principle purely combinatorial objects. The idea is that one only glues copies of one and the same polygon (not even rotated), which is as simple as it may be: the unit square.
\end{Origamis}
\begin{Definition}
An \textup{origami} consists of the following data:
\begin{enumerate}
\item a finite\footnote{One may also drop this condition and study infinite origamis. This would lead out of the realm of algebraic geometry, so we omit this} set $I$ of copies of the unit square in $\bbc$, and
\item a partition of the set of edges of squares in $I$ into pairs, such that every left side is paired with a right side (possibly of the same square!) and every upper side with a lower side.
\end{enumerate}
Gluing the paired sides defines a closed translation surface with cone points, called the translation surface associated with the origami.
\end{Definition}
\begin{Billiards}
A large class of examples of closed translation surfaces with cone points is obtained from a classical toy model in dynamical systems theory: billiards in a rational polygon.

To explain this, start with an arbitrary\footnote{We always assume polygons to be compact} euclidean polygon $P\subseteq\bbc$ and imagine a point-shaped massless billiard ball moving inside the polygon without any acceleration, except that it is reflected in the sides of the polygon.\footnote{Since our interest in this work is primarily in differential and algebraic geometry rather than dynamical systems, we refrain from formalizing this} Now if one wishes to study the behaviour of such billiard balls, the non-differentiability at the boundary of the polygon poses some problems, and so one might instead do the following: reflect the whole polygon in each of its sides, glue the thus obtained copies to the ``old'' polygon along the corresponding sides, and repeat indefinitely. One then obtains a huge translation surface with some singularities (coming from the vertices of the polygon), and instead of reflecting the ball in the sides of the polygon we can let it move along a straight line in the big translation surface. But we have to bear in mind that there are infinitely many copies of the original polygon in this surface, in other words we have to divide by the action of some infinite group\footnote{To be precise, this the group with generators $\sigma_e$ parameterized by the edges $e$ of $P$, modulo the relation $\sigma_e^2=1$ for every $e$} to get back the polygon. Hence important dynamical properties of the polygon and the big translation surface may be totally different, and we have to look for a less generous way of ``unfolding'' the billiard table.

Indeed we can identify those copies of $P$ obtained by consecutive reflection which are translation equivalent, i.e. which can be transformed into each other by a translation of $\bbc$. If we are lucky there is only a finite number of equivalence classes, and we can glue finitely many polygons (each congruent to $P$, but pairwise non-translation-equivalent) and obtain a closed translation surface with cone points; in other words, we end up in the realm of algebraic geometry!

Let us now formalize this intuition. For every side $e$ of $P$, consider the reflection in $e$; this is a real affine map, whose linear part we denote by $\sigma_e\in\mathrm{O}(2)$. Since we are working with translation structures, so everything should be invariant under translations in $\bbc$, it is only this linear part which is of interest for us. Now let $G_P\subset\mathrm{O}(2)$ be the subgroup generated by the $\sigma_e$. We now wish to construct a translation surface as follows: for every $\tau\in G_P$ take the polygon $\tau (P)\subset\bbc$, so that you obtain a family of translation surfaces with corners, indexed by $G$; glue the side $e$ of $\tau (P)$ to the side $\sigma_e(e)$ of $(\sigma_e\circ\tau )(P)$. For this to be possible of course we have to check the following:
\end{Billiards}
\begin{Lemma}\label{SymmetryGroupOfAPolyhedronDoesNotChangeUnderReflection}
If $\tau\in G_P$, then $G_{\tau (P)}=G_P$.
\end{Lemma}
\begin{proof}
This comes down to the formula $\sigma_{\sigma_e(f)}=\sigma_e\circ\sigma_f\circ\sigma_e$.
\end{proof}
But there are more serious problems: it is not clear what happens at the vertices of the polygon, and it is not ensured that there be only finitely many copies of $P$ involved. So we have to restrict ourselves to a certain class of polygons:
\begin{Definition}
Let $P\subset\bbc$ be a compact polygon. Then $P$ is \textup{rational} if the group $G_P$ is finite.
\end{Definition}
So only for rational polygons we have a chance of getting a \textit{compact} translation surface (possibly with cones). The naming is justified:
\begin{Proposition}
Let $P\subset\bbc$ be a compact polygon all of whose angles are rational multiples of $\pi$. Then $P$ is rational.
\end{Proposition}
\begin{proof}
Fix some $N$ such that all angles of $P$ are integral multiples of $\frac{\pi}{N}$. Write $G_P^+=G_P\cap\SO (2)$. Then there is a short exact sequence
\begin{equation*}
1\to G_P^+\to G_P\overset{\det}{\to}\{\pm 1\}\to 1
\end{equation*}
(note that reflections have determinant $-1$, so the third map is indeed surjective), so it suffices to show that $G_P^+$ is finite. We claim that $G_P^+$ is contained in the cyclic subgroup of $\SO (2)$ generated by the rotation with angle $\frac{2\pi }{N}$. Namely $G_P^+$ is generated by the elements $\sigma_e\sigma_f$ where $e$ and $f$ are sides of $P$. The angle between $e$ and $f$ (or rather the euclidean lines they define) is of the form $\frac{m\pi}{N}$ for some integer $m$, and hence $\sigma_e\sigma_f$ is the rotation with angle $\frac{2m\pi}{N}$.
\end{proof}
From this proof we also see that $G_P$ has at most order $2N$.

We now construct the closed translation surface with cones associated with a rational polygon. For simplicity we shall assume that $P$ satisfies the condition of the preceding proposition. So let $P\subset\bbc$ be a compact polygon all of whose angles are rational multiples of $\pi$. For every $\tau\in G_P$ let $P_{\tau}$ be the polygon $\tau (P)\subset\bbc$, considered as a translation surface with corners (this notation is just for emphasizing that we wish to consider the $P_{\tau }$ as disjoint, distinct abstract surfaces, even though they may overlap as subsets of the complex plane). The desired translation surface $X_P$ is then the quotient of the non-connected translation surface with corners
\begin{equation*}
X'_P=\coprod_{\tau\in G_P}P_{\tau }
\end{equation*}
modulo the following equivalence relation: for every $\tau\in G_P$ and every side $e$ of $P_{\tau }=\tau (P)$, identify $e$ via the identity (!) with the edge $\sigma_e(e)$ of $\sigma_e(P_{\tau })=P_{\sigma_e\tau }$. Here we use Lemma \ref{SymmetryGroupOfAPolyhedronDoesNotChangeUnderReflection}. This defines indeed a closed translation surface with cone points; we denote the underlying Riemann surface also by $X_P$ and the holomorphic one-form corresponding to the translation structure with cones by $\omega_P$.

\begin{RationalIsoscelesTriangles}
Let $g=2n+1$ be odd, and let $1\le k\le g$. We shall construct a translation surface $(X(g,k),\omega (g,k))$. Denote once and for all $\zeta_p=\exp\frac{2\pi\mathrm{i}}{p}$, for any positive integer $p$. Take a triangle $T$ in the complex plane with vertices
\begin{equation*}
(0,\zeta_{2n}^{-k},\zeta_{2n}^k).
\end{equation*}
This is an isosceles triangle with angles
\begin{equation*}
\left(\frac{2k\pi }{n},\frac{(n-2k)\pi }{2n},\frac{(n-2k)\pi }{2n}\right) ,
\end{equation*}
and it can be subdivided into two right triangles $T^+$ and $T^-$, where $T^+$ has vertices
\begin{equation*}
(0,\cos\frac{\pi k}{n},\zeta_{2n}^k)
\end{equation*}
and $T^-$ has vertices
\begin{equation*}
(0,\zeta_{2n}^{-k },\cos\frac{\pi k}{n}).
\end{equation*}
We show this for $n=7$ and $k=1,2,3$ (from left to right):
\end{RationalIsoscelesTriangles}
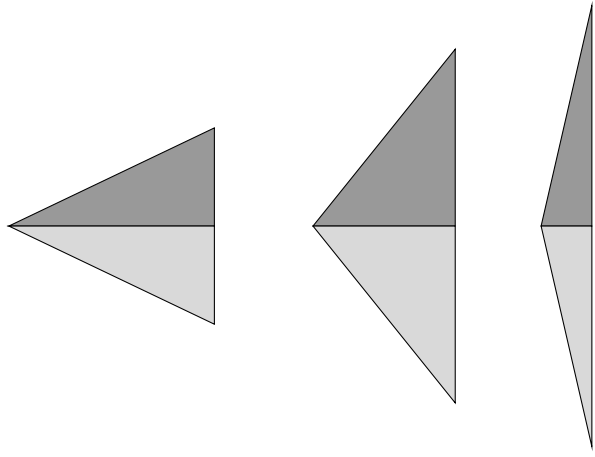
\begin{center}
\begin{tikzpicture}
\filldraw[fill=gray!80!white] (0,0) -- ++(2.7030,0) -- ++(0,1.3017) -- cycle
    (4,0) -- ++(1.8705,0) -- ++(0,2.3454) -- cycle
    (7,0) -- ++(0.6675,0) -- ++(0,2.9247) -- cycle;
\filldraw[fill=gray!30!white] (0,0) -- ++(2.7030,0) -- ++(0,-1.3017) -- cycle
    (4,0) -- ++(1.8705,0) -- ++(0,-2.3454) -- cycle
    (7,0) -- ++(0.6675,0) -- ++(0,-2.9247) -- cycle;
\end{tikzpicture}
\end{center}
Now we want to glue a family of isometric copies of $T$. For $m\in\bbz /n\bbz$ and $\varepsilon\in\{\pm 1\}$, let $T_{m,\varepsilon }$ be (a copy of)
\begin{equation*}
\varepsilon\zeta_n^{km}\cdot T\subset\bbc .
\end{equation*}
Then let $X(g,k)$ be the surface
\begin{equation*}
X'(g,k)=\coprod_{m\in\bbz /n\bbz \atop \varepsilon =\pm 1} T_{m,\varepsilon }
\end{equation*}
modulo the following equivalence relation: for each $m$ and $\varepsilon$, we identify:
\begin{enumerate}
\item The segment from $0$ to $\varepsilon\zeta_{2n}^{k(2m+1)}$ is a side of both $T_{m,\varepsilon }$ and $T_{m+1,\varepsilon}$; glue these two triangles along these sides.
\item The segment from $\zeta_{2n}^{k(2m-1)}$ to $\zeta_{2n}^{k(2m+1)}$ is a side of $T_{m,1}$, and the segment from $-\zeta_{2n}^{k(2m-1)}$ to $-\zeta_{2n}^{k(2m+1)}$ is a side of $T_{m,-1}$. These sides are parallel and have the same length, hence there is a unique translation that takes the former to the latter. Identify these two sides by this translation.
\end{enumerate}
The result is a closed Riemann surface $X(g,k)$; the translation structures on the interiors of the $T_{m,\varepsilon }$, induced by the inclusion $T_{m,\varepsilon }\hookrightarrow\bbc$, glue to a translation structure with cone-points on $X(g,k)$, whose associated one-form we denote by $\omega (g,k)$.

The following is an elementary observation, but very important for what follows:
\begin{Proposition}\label{WimanCurvesAsBilliardSurfaces}
Assume that $k$ and $n=2g+1$ are coprime. Then $(X(g,k),\omega (g,k))$ is the billiard surface associated with the triangular ``billiard table'' $T$ with angles
\begin{equation*}
\frac{2k\pi }{n},\frac{(n-2k)\pi }{2n},\frac{(n-2k)\pi }{2n}.
\end{equation*}
It is also the billiard surface associated with the triangle $T^+$ having angles
\begin{equation*}
\frac{k\pi }{n},\frac{\pi }{2},\frac{(n-2k)\pi }{2n}.
\end{equation*}
\end{Proposition}
If $k$ and $n$ are not coprime, $(X(g,k),\omega (g,k))$ can still be interpreted as a finite cover of such a billiard table, see section 5.2.

\subsection{Translation Surfaces as Algebraic Curves}

We have shown that closed translation surfaces with cone points are ``the same'' as compact Riemann surfaces together with abelian differentials, hence ``the same'' as algebraic curves over $\bbc$ with nonzero algebraic one-forms. In this section we will thus describe our examples of translation surfaces in these terms.

\begin{TheSCFormula}
Let $P\subset\bbc$ be a simply connected compact Euclidean polygon in the complex domain, and let $p_1,p_2,\ldots ,p_n,p_1$ be its boundary points in counter-clockwise order. Then by the Riemann mapping theorem there is a conformal map $f:\bbh\to P^{\circ }$ which extends to a homeomorphism $\bbh\cup\mathbb{P}^1(\mathbb{R})\to\overline{P}$. The preimages of the $p_i$ are then certain $a_i\in\mathbb{P}^1(\bbr )$ about which one cannot say very much in general, except for that one can move them around by precomposition with a M\"{o}bius transformation. In particular we can assume that $\infty$ is amongst them or not, as we please. But once we know the $a_i$ we can give a formula for $f$:
\end{TheSCFormula}
\begin{Proposition}
With the notation as above, let $\gamma_i\pi$ be the interior angle of $P$ at $b_i$ and assume that $\gamma_i\neq 1$. If none of the $a_i$ is equal to $\infty$, then the map $f$ is given by the \textup{Schwarz-Christoffel formula}
\begin{equation}
f(z)=c_1\int_0^z(\zeta -a_1)^{\gamma_1-1}(\zeta -a_2)^{\gamma_2-1}\cdots (\zeta -a_n)^{\gamma_n-1}\,\de\zeta +c_0
\end{equation}
for suitable constants $c_0,c_1$. In case that say $a_n=\infty$, the formula gets modified to
\begin{equation}
f(z)=c_1\int_0^z(\zeta -a_1)^{\gamma_1-1}(\zeta -a_2)^{\gamma_2-1}\cdots (\zeta -a_{n-1})^{\gamma_{n-1}-1}\,\de\zeta +c_0.
\end{equation}
The map $f$ is then also called a \textup{Schwarz-Christoffel map}, and the $a_i$ are called its \textup{accessory parameters.}
\end{Proposition}
\begin{proof}
See \cite[p. 431 et seqq.]{HurwitzCourant64}.
\end{proof}

\begin{center}
\begin{tikzpicture}
\clip (0,-1) rectangle (15,4);
\shade [shading=radial] (0,-3) rectangle (5,3);
\fill[white] (0,-3) rectangle (5,0);
\draw (0,0) -- (5,0);
\draw [->] (6,1.5) -- (7,1.5) node[above=2pt] {$f$} -- (8, 1.5);

\filldraw[black] (0.5,0) circle (2pt)
                 (0.9,0) circle (2pt)
                 (1.3,0) circle (2pt)
                 (2,0) circle (2pt)
                 (2.6,0) circle (2pt)
                 (2.9,0) circle (2pt)
                 (3.4,0) circle (2pt)
                 (4.3,0) circle (2pt);

\filldraw[fill=gray!30!white] (9,1) -- ++(-10:62pt)  -- ++(-25:24pt) -- ++(42:47pt) -- ++(101:56pt) -- ++(206:63pt) -- ++(157:43pt) -- cycle;
\filldraw[black]
    (9,1) circle (2pt)
     ++(-10:62pt) circle (2pt)
     ++(-25:24pt) circle (2pt)
     ++(42:47pt) circle (2pt)
     ++(101:56pt) circle (2pt)
     ++(206:63pt) circle (2pt)
     ++(157:43pt) circle (2pt);
\end{tikzpicture}
\end{center}

In general there is still the task to determine the accessory parameters if one wishes a fully explicit formula; in important special cases one can get around that.

For example if $P$ is a triangle, then it is --- up to similarity --- uniquely determined by the two angles $\gamma_1\pi$ and $\gamma_2\pi$. Furthermore we can map the accessory parameters by a M\"{o}bius transformation to $0,1,\infty$. Hence:
\begin{Corollary}
Let $T$ be the interior of a triangle in the complex domain, with vertices $p_1,p_2,p_3$ (in counter-clockwise order) and with internal angle $\gamma_i\pi$ at $p_i$. Then there is a unique conformal map $s:\bbh\to T$ extending to a homeomorphism $\bbh\cup\mathbb{P}^1(\bbr )\to\overline{T}$ (which we also denote by $s$) with $s(0)=p_1$, $s(1)=p_2$ and $s(\infty )=p_3$. It is given by the formula
\begin{equation}\label{SchwarzChristoffelTriangle}
s(z)=c\int_0^z\zeta^{\gamma_1-1}(\zeta -1)^{\gamma_2-1}\,\de\zeta +p_1.
\end{equation}
Furthermore $c$ can be computed by using that
\begin{equation*}
p_2-p_1=s(1)-s(0)=c\int_0^1\zeta^{\gamma_1-1}(\zeta -1)^{\gamma_2-1}\,\de\zeta
\end{equation*}
and the classical formula for Euler's beta function
\begin{equation*}
\mathrm{B}(\gamma_1,\gamma_2)=\int_0^1\zeta^{\gamma_1-1}(1-\zeta )^{\gamma_2-1}\,\de\zeta =\frac{\Gamma (\gamma_1)\Gamma (\gamma_2)}{\Gamma (\gamma_1+\gamma_2)}.
\end{equation*}
\end{Corollary}
\begin{center}
\begin{tikzpicture}
\fill[gray!80!white] (0,0) rectangle (4,3);
\draw (0,0) -- (4,0);
\draw [->] (5,1.5) -- (6,1.5) node[above=2pt] {$s$} -- (7, 1.5);
\draw (2,3.5) node {$\infty$};
\filldraw[fill=gray!80!white] (8,0.5) -- (10.3,1) -- (9.5,2.8) -- cycle;
\filldraw[black] (1,0) circle (2pt) node[below=2pt] {$0$}
                 (3,0) circle (2pt) node[below=2pt] {$1$}
                 (8,0.5) circle (2pt) node[below=2pt] {$s(0)=p_1$}
                 (10.3,1) circle (2pt) node[right=2pt] {$s(1)=p_2$}
                 (9.5,2.8) circle (2pt) node[above=2pt] {$s(\infty )=p_3$};
\draw (8,0.5) -- +(12.26:12pt) arc(12.26:56.89:12pt)  node[right=2pt] {$\gamma_1\pi$} -- cycle;
\draw (10.3,1) -- +(113.96:12pt) arc (113.96:192.26:12pt) -- cycle;
\draw (9.5,1.3) node {$\gamma_2\pi$};
\end{tikzpicture}
\end{center}

These formul\ae\ help us to explicitly describe geometrically constructed translation surfaces as algebraic curves.

We now show that the curves $X(g,k)$ defined above are hyperelliptic, with ramification locus precisely $\mu_n(\bbc )\cap\{\infty\}$. For this, we first construct a holomorphic $2:1$ map $X(g,k)\to\mathbb{P}^1(\bbc )$.

Let $s:\overline{\mathbb{H}}\overset{\simeq}{\to}T^+$ be the unique Schwarz-Christoffel map with
\begin{equation*}
s(0)=0,\, s(1)=\cos\frac{k\pi}{n}\text{ and }s(\infty )=\zeta_{2n}^k.
\end{equation*}
Let $S\subset\bbc$ denote the open sector defined by $0<\operatorname{arg}z<\pi /n$, and let
\begin{equation*}
\varphi :S\to T^+,\quad z\mapsto s(z^n);
\end{equation*}
this is a biholomorphic map which extends to a homeomorphism of the closures (in $\mathbb{P}^1(\bbc )$!).

\begin{center}
\begin{tikzpicture}
\fill[color=gray!80!white] (7,-4) -- (11,-4) arc (0:36:4cm);
\filldraw[fill=gray!80!white] (0,-1) -- (3.0777,-1) -- ++(90:1cm) -- cycle;
\fill[fill=gray!80!white] (7,1) rectangle (13,3);
\draw (11,-4) -- (7,-4) -- ++(36:4cm);
\draw (7,1) -- (13,1);
\draw (7.7,-4) arc (0:36:0.7cm);
\draw (8,-3.6) node {$\frac{\pi }{n}$};
\filldraw[black] (0,-1) circle (2pt) node[left=2pt] {$s(0)$}
                 (3.0777,-1) circle (2pt) node[below=2pt] {$s(1)$}
                 (3.0777,0) circle (2pt) node[above=2pt] {$s(\infty )$}
                 (9.5,1) circle (2pt) node[below=2pt] {$0$}
                 (10.5,1) circle (2pt) node[below=2pt] {$1$}
                 (7,-4) circle (2pt) node[below=2pt] {$0$}
                 (8,-4) circle (2pt) node[below=2pt] {$1$};
\draw [->] (8,-2.5) -- (8,0);
\draw (8.5,-1.25) node {$(\cdot )^n$};
\draw [->] (7,0) .. controls (6,-0.5) .. (4.1,-0.2);
\draw (5.5,0) node {$s$};
\draw [->] (7.6,-2.9) -- (4.1,-0.8);
\draw (5.35,-2.25) node {$\varphi$};
\end{tikzpicture}
\end{center}
\begin{Proposition}
The map $\varphi^{-1}:T^+\to S$ extends, by continued Schwarz reflection, to a finite holomorphic map $\Psi :X(g,k)\to\mathbb{P}^1(\bbc )$ of degree $2$; its ramification locus is precisely $\mu_n(\bbc )\cup\{\infty\}$. Hence $\Psi$ identifies $X(g,k)$ with the Wiman curve $W_g$.
\end{Proposition}
\begin{proof}
The main point to show is that $\Psi$ is well-defined. So we investigate what it must do on each $T_{m,\varepsilon}$ and then check that these maps glue to a global holomorphic map.

To begin with, $\Psi$ maps the interior of $T^-$ conformally to the sector $-\pi /n<\operatorname{\arg }z<0$, and hence it maps the interior of $T_{1,1}=T$ biholomorphically to the set
\begin{equation*}
\{ z\in\bbc^{\times }\mid -\frac{\pi}{n}<\operatorname{arg}z<\frac{\pi}{n}\}\smallsetminus [1,\infty [.
\end{equation*}
This is illustrated in the following picture:

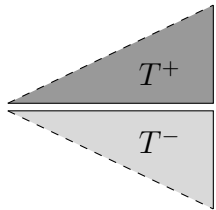
\begin{center}
\begin{tikzpicture}
\fill[gray!80!white] (0,0) -- ++(2.7030,0) -- ++(0,1.3017) -- cycle;
\fill[gray!30!white] (0,-0.1) -- ++(2.7030,0) -- ++(0,-1.3017) -- cycle;
\draw (0,0) -- ++(2.7030,0) -- ++(0,1.3017)
    (0,-0.1) -- ++(2.7030,0) -- ++(0,-1.3017);
\draw[dashed] (0,0) -- ++(2.7030,1.3017)
    (0,-0.1) -- ++(2.7030,-1.3017);
\draw (2,0.4) node {$T^+$}
    (2,-0.5) node {$T^-$};
\end{tikzpicture}
\end{center}
Here the little distance between $T^+$ and $T^-$ is only drawn for sake of clarity, actually they are glued together. Now $\Psi$ maps the upper dashed line to the half-line $\operatorname{arg }z=\pi /n$ and the upper full line to the positive real axis. The lower full line is also sent to the positive real axis, with the branch point being sent to $1$. Finally the lower dashed line is sent to the half-line $\operatorname{arg}z=-\pi /n$.

We have a very similar pattern for $T_{m,1}$; in fact we can use the same picture (up to a rotation which we ignore, in order to use words like ``above''), but now $T^+$ is replaced by $\zeta_n^{mk}T^+$, and $T^-$ is replaced by $\zeta_n^{mk}T^-$. Here the upper dashed line goes to $\operatorname{arg }z=(2m+1)\pi /n$, and the lower dashed line goes to $\operatorname{arg }z=(2m-1)\pi /n$. Both full lines are sent to $\operatorname{arg}z=2m\pi /n$, with the branch point being sent to $\zeta_n^m$. So $\Psi$ sends the interior of $T_{m,1}$ conformally to the sector $(2m-1)\pi /n<\operatorname{arg }z<(2m+1)\pi /n$, minus the straight segment from $0$ to $\zeta_n^m$.

What happens on $T_{m,-1}$? Again we have the same picture up to an ignored rotation. In particular it is the same as that for $T_{m,1}$, up to a rotation by $\pi$, and up to this rotation $\Psi$ does \textit{exactly the same} on both triangles.

We can now check that these maps do glue to a continuous map $\Psi :X(g,k)\to\mathbb{P}^1(\bbc )$. This is clearly holomorphic on the interiors of the triangles $T_{m,\varepsilon }$. Since for each identification of two sides involved in the construction of $X(g,k)$, the map on one bank is the Schwarz mirror image of the map on the other bank, $\Psi$ is also holomorphic everywhere except possibly for the cone points. But these form a finite set, and $\Psi$ is continuous, so by Riemann's theorem on removable singularities, $\Psi$ has to be holomorphic everywhere.

Since both $X(g,k)$ and $\mathbb{P}^1(\bbc )$ are compact and $\Psi$ is nonconstant, it has to be a finite holomorphic map. Its degree is easily seen to be two. The branch points are then precisely those points in $\mathbb{P}^1(\bbc )$ whose preimage has cardinality one, so by construction the branch locus is $\mu_n(\bbc )\cup\{\infty\}$.
\end{proof}
Next we compute the abelian differential $\omega (g,k)$:
\begin{Proposition}
Under the identification $X(g,k)=W_g$ induced by $\Psi$, the differential $\omega (g,k)$ on $X(g,k)$ corresponds to a nonzero multiple of $\omega_k$ on $W_g$.
\end{Proposition}
\begin{proof}
Recall that $\omega_k$ is in affine coordinates $(x,y)$ with $y^2=x^n-1$ given as $x^{k-1}\de x/y$. By the identity theorem, we only need to check that these two one-forms agree on some nonempty open subset, say the interior of $T^+$. Here we can describe $\Psi$ as the inverse of $\varphi :S\to T^+$. Letting $x$ be the standard coordinate on $\mathbb{P}^1(\bbc )$, we compute $\varphi^{\ast }\omega (g,k)=\varphi^{\ast }\de z$. For this, recall that $\varphi (x)=s(x^n)$, where
\begin{equation*}
c\cdot\int_0^xu^{\frac{k}{n}-1}(u-1)^{\frac{1}{2}-1}\de u
\end{equation*}
for a suitable $c\in\bbc^{\times }$.
\begin{equation*}
\begin{split}
\varphi^{\ast }\de z
&=\varphi '(x)\de x\\
&=s'(x^n)\cdot nx^{n-1}\de x\\
&=nc\cdot\frac{x^{k-n}}{\sqrt{x^n-1}}\cdot x^{n-1}\de x.\\
\end{split}
\end{equation*}
Locally choosing a square-root means locally choosing a leaf of $W_g\to\mathbb{P}^1(\bbc )$. Hence for a suitable constant $c'\in\bbc^{\times }$ we have $\omega (g,k)=c'\omega_k$.
\end{proof}
\begin{WimanCurvesAndRegularPolygons}
The special cases $(W_g,\omega_1)$ and $(W_g,\omega_g)$ have also simpler descriptions. Start with $(W_g,\omega_1)$; up to a scaling factor, this is isomorphic to $(X(g,1),\omega (g,1))$. Using the notation from above, we see that the triangles $T_{1,1},\ldots ,T_{n,1}$ glue together to a regular $n$-gon (as always, $n=2g+1$) with the $n$-th roots of unity in $\bbc$ as vertices. Similarly, $T_{1,-1},\ldots ,T_{n,-1}$ glue together to a regular $n$-gon with the $-\zeta$ as vertices, where $\zeta$ runs through all $n$-th roots of unity in $\bbc$. In other words, we can describe $(W_g,\omega_1)$ as two regular $n$-gons, one being the pointwise mirror of the other, glued along parallel sides. Here we show a picture for $g=2$:
\end{WimanCurvesAndRegularPolygons}
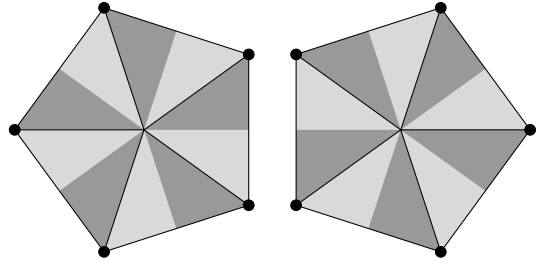
\begin{center}
\begin{tikzpicture}
\fill[gray!30!white] (1.3764,-1) -- ++(90:2cm) -- ++(162:2cm) -- ++(234:2cm) -- ++(306:2cm) -- cycle
    (2,-1) -- ++(90:2cm) -- ++(18:2cm) -- ++(306:2cm) -- ++(234:2cm) -- cycle;
\fill[gray!80!white]
    (0,0) -- ++(0:1.3764cm) -- ++(90:1cm) -- cycle
    (0,0) -- ++(72:1.3764cm) -- ++(162:1cm) -- cycle
    (0,0) -- ++(144:1.3764cm) -- ++(234:1cm) -- cycle
    (0,0) -- ++(216:1.3764cm) -- ++(306:1cm) -- cycle
    (0,0) -- ++(288:1.3764cm) -- ++(18:1cm) -- cycle
    (3.3764,0) -- ++(180:1.3764cm) -- ++(270:1cm) -- cycle
    (3.3764,0) -- ++(252:1.3764cm) -- ++(342:1cm) -- cycle
    (3.3764,0) -- ++(324:1.3764cm) -- ++(54:1cm) -- cycle
    (3.3764,0) -- ++(36:1.3764cm) -- ++(126:1cm) -- cycle
    (3.3764,0) -- ++(108:1.3764cm) -- ++(198:1cm) -- cycle;
\draw (1.3764,-1) -- ++(90:2cm) -- ++(162:2cm) -- ++(234:2cm) -- ++(306:2cm) -- cycle
    (2,-1) -- ++(90:2cm) -- ++(18:2cm) -- ++(306:2cm) -- ++(234:2cm) -- cycle
    (3.3764,0) -- ++(0:1.7013cm)
    (3.3764,0) -- ++(72:1.7013cm)
    (3.3764,0) -- ++(144:1.7013cm)
    (3.3764,0) -- ++(216:1.7013cm)
    (3.3764,0) -- ++(288:1.7013cm)
    (0,0) -- ++(180:1.7013cm)
    (0,0) -- ++(252:1.7013cm)
    (0,0) -- ++(324:1.7013cm)
    (0,0) -- ++(36:1.7013cm)
    (0,0) -- ++(108:1.7013cm);
    \filldraw[black]
    (1.3764,-1) circle (2pt)
    ++(90:2cm) circle (2pt)
    ++(162:2cm) circle (2pt)
    ++(234:2cm) circle (2pt)
    ++(306:2cm) circle (2pt)
    (2,-1) circle (2pt)
    ++(90:2cm) circle (2pt)
    ++(18:2cm) circle (2pt)
    ++(306:2cm) circle (2pt)
    ++(234:2cm) circle (2pt);
\end{tikzpicture}
\end{center}
Now consider the translation surface $(W_g,\omega_g)$. We claim that this is a regular $2n$-gon with opposite sides identified. Namely, recall from Proposition \ref{WimanCurvesAsBilliardSurfaces} that $(W_g,\omega_g)$ is the billiard surface obtained from $T^+$, where $T^+$ is a triangle with angles
\begin{equation*}
\frac{g\pi }{n},\frac{\pi }{2},\frac{\pi }{2n}.
\end{equation*}
Now take a regular $2n$-gon and join its center with all vertices and with all midpoints of its sides. This gives a tesselation of the $2n$-gon by triangles similar to $T^+$, and turns this $2n$-gon with opposite sides identified into the billiard surface of $T^+$. Again a picture for $g=2$:

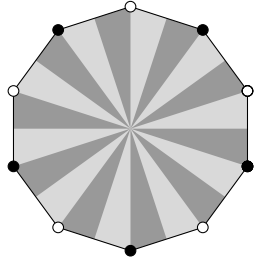
\begin{center}
\begin{tikzpicture}[scale=0.5]
\fill[fill=gray!80!white] (3.0777,-1) -- ++(90:2cm) -- ++(126:2cm) -- ++(162:2cm) -- ++(198:2cm) -- ++(234:2cm) -- ++(270:2cm) -- ++(306:2cm) -- ++(342:2cm) -- ++(18:2cm) -- ++(54:2cm) -- cycle;
\fill[gray!30!white] (0,0) -- (3.0777,0) -- (3.0777,1) -- cycle
             (0,0) -- ++(36:3.0777cm) -- ++(126:1cm) -- cycle
             (0,0) -- ++(72:3.0777cm) -- ++(162:1cm) -- cycle
             (0,0) -- ++(108:3.0777cm) -- ++(198:1cm) -- cycle
             (0,0) -- ++(144:3.0777cm) -- ++(234:1cm) -- cycle
             (0,0) -- ++(180:3.0777cm) -- ++(270:1cm) -- cycle
             (0,0) -- ++(216:3.0777cm) -- ++(306:1cm) -- cycle
             (0,0) -- ++(252:3.0777cm) -- ++(342:1cm) -- cycle
             (0,0) -- ++(288:3.0777cm) -- ++(18:1cm) -- cycle
             (0,0) -- ++(324:3.0777cm) -- ++(54:1cm) -- cycle;
\draw (3.0777,-1) -- ++(90:2cm) -- ++(126:2cm) -- ++(162:2cm) -- ++(198:2cm) -- ++(234:2cm) -- ++(270:2cm) -- ++(306:2cm) -- ++(342:2cm) -- ++(18:2cm) -- ++(54:2cm) -- cycle;
\filldraw[fill=black]
    (3.0777,-1) circle (4pt)
    ++(90:2cm)
    ++(126:2cm) circle (4pt)
    ++(162:2cm)
    ++(198:2cm) circle (4pt)
    ++(234:2cm)
    ++(270:2cm) circle (4pt)
    ++(306:2cm)
    ++(342:2cm) circle (4pt)
    ++(18:2cm)
    ++(54:2cm) circle (4pt)
    ++(90:2cm);
\filldraw[fill=white]
    (3.0777,-1)
    ++(90:2cm) circle (4pt)
    ++(126:2cm)
    ++(162:2cm) circle (4pt)
    ++(198:2cm)
    ++(234:2cm) circle (4pt)
    ++(270:2cm)
    ++(306:2cm) circle (4pt)
    ++(342:2cm)
    ++(18:2cm) circle (4pt)
    ++(54:2cm)
    ++(90:2cm) circle (4pt);
\end{tikzpicture}
\end{center}

We record for future reference:
\begin{Proposition}
The translation surface $(W_g,\omega_1)$ can be obtained by taking a regular $n$-gon and its point mirror image, and gluing the two along their parallel sides. The translation surface $(W_g,\omega_g)$ can be obtained by gluing a regular $2n$-gon along opposite sides. Here $n=2g+1$.\hfill $\square$
\end{Proposition}

\subsection{Some More Geometry of Abelian Differentials}\label{SectionOnMoreGeoOfAbDiff}

Because of the very restricted form of the chart transition maps, we can transfer many features of Euclidean geometry to translation surfaces. This observation is more or less trivial, but is important to be spelled out because it has very nontrivial consequences. To keep the language simpler, we will now identify translation surfaces with cone points and Riemann surfaces with holomorphic one-forms.

\begin{EuclideanMetric}
To begin with, translations in $\bbc$ are isometries. Hence if $(X,\omega)$ is a translation surface with cone points, pulling back the Euclidean metric on $\bbc$ gives a well-defined Euclidean (even Hermitian) metric on $X$ minus the zeros of $\omega$. Since the Euclidean metric on $\bbc$ can be described as $\de x\otimes\de y$, the metric on $X$ is $\operatorname{Re}\omega\otimes\operatorname{Im}\omega$. Note that this extends, by the same formula, to a tensor on the whole of $X$, degenerate at the zeros of $\omega$. The volume form of the Euclidean metric in $\bbc$ is $\frac{\mathrm{i}}{2}\de z\wedge\de\overline{z}$, hence the volume form on $X$ is $\frac{\mathrm{i}}{2}\omega\wedge\overline{\omega }$. This is a two-form which is degenerate precisely at the zeros of $\omega$.
\end{EuclideanMetric}
\begin{TrivializationofTangentSpaces}
Also we can canonically identify every tangent space $T_z\bbc$ with $\bbc$, and every translation in $\bbc$ has (with this identification) derivative equal to the identity of $\bbc$. This means that we can identify every tangent space $T_xX$, where $x$ is \textit{not} a zero of $\omega$, with $\bbc$. Note that this can in general not be extended to the zeros of $\omega$, since this would yield a trivialization of the tangent bundle of $X$, which is not possible for closed surfaces of genus at least two.
\end{TrivializationofTangentSpaces}
\begin{DirectionsandFoliations}
Translations also preserve rays and their directions. Hence for every $\vartheta\in\bbr /2\pi\bbz$ it makes sense to study the ray emanating from a given point $x\in X$, with direction $\vartheta$. Sometimes this can be indefinitely continued, sometimes it comes back to $x$ and sometimes it hits a cone point (in which case we adopt the convention that it just stops there). Going both forward and backward we obtain a line through $x$ in direction $\vartheta$ (with the restriction that it may be cut off at one or two ends and that it may in fact be a circle). To make this more formal, define a foliation $\mathrsfs{F}_{\vartheta }(\omega )$ on $X\smallsetminus Z(\omega )$ as follows: recall that every tangent space is canonically identified with $\bbc$, and let $\mathrsfs{F}_{\vartheta }(\omega )_x\subset T_xX=\bbc$ be the real one-dimensional subspace spanned by $\exp (\mathrm{i}\vartheta )$. Note that this foliation might not be extendable to a foliation on $X$. We can also describe this foliation in a way closer to abelian differentials: first note that $\mathrsfs{F}_{\vartheta }(\omega )=\mathrsfs{F}_0(\mathrm{e}^{\mathrm{i}\vartheta }\omega )$, so it suffices to describe $\mathrsfs{F}_0(\omega )$. But $\mathrsfs{F}_0(\omega )$, as a real subbundle of the tangential bundle, is the kernel of the real-valued one-form $\Im\omega $. Similarly $\mathrsfs{F}_{\pi /2}(\omega )$ is the kernel of $\Re\omega $.
\end{DirectionsandFoliations}
\begin{Leaves}
Assume now $X$ closed. The leaves of $\mathrsfs{F}_{\vartheta }(\omega )$ are what we have inadequately described as the ``lines in direction $\vartheta$ through a point in $X$'' above. Since these leaves are smooth one-dimensional manifolds, there are not too many possibilities of what they can look like. Let $\ell$ be a leaf of $\mathrsfs{F}_{\vartheta }(\omega )$, then there exists an interval $I\subseteq\bbr$ and a local isometry $I\to\ell$, maximal with respect to inclusion of intervals. Exactly one of the following statements then holds:
\end{Leaves}
\begin{enumerate}
\item $I=\bbr$, and $\ell$ is diffeomorphic to a circle, in which case we call it \textit{periodic.} Intuitively this means that if we start from some $x\in\ell$ and begin to walk in direction $\vartheta$, we arrive again at $x$ after some time.
\item $I=\bbr$ and the map $I\to\ell$ is injective, but \textit{not} a diffeomorphism onto its image. The leaf $\ell$ is not closed in $X$. This case can be quite complicated.
\item $I=]-\infty ,a[$ for some $a\in \bbr$. This means that if we start from some $x\in\ell$ and walk in direction $\vartheta$, we can walk backwards indefinitely without ever arriving at the same point again, but upon walking forwards we run into a cone point after finite time.
\item $I=]a,\infty [$. The ``time mirror'' of case (iii).
\item $I=]a,b[$ for $a,b\in\bbr$. This means that both forward and backward we run into a cone point. Such leaves are parametrized by saddle connections, see below.
\end{enumerate}
\begin{HolonomyVectors}
A \textit{saddle connection} on $(X,\omega )$ is a smooth injective path $\gamma :[a,b]\to X$ such that $\gamma (a)$ and $\gamma (b)$ are cone points, $\gamma (t)$ is not a cone point for $a<t<b$, and the restriction of $\gamma$ to $]a,b[$ is a geodesic. The \textit{holonomy vector} of $\gamma$ is the complex number $\int_{\gamma }\omega$. For $X$ compact, the set of all saddle connections of $(X,\omega )$ is always a discrete subset of $\bbc$, see \cite[section 1.3.2]{HubertSchmidt04}.
\end{HolonomyVectors}
\begin{JSDirections}
Since the foliation ``breaks down'' at the cone points, directions in which they do not cause to much trouble are of special interest:
\end{JSDirections}
\begin{Definition}
Let $(X,\omega )$ be a closed translation surface. A direction $\vartheta\in\bbr /2\pi\bbz$ is called a \textup{Jenkins-Strebel direction} if in the foliation $\mathrsfs{F}_{\vartheta }(\omega )$ there are only leaves of type (i) and (v).
\end{Definition}
Jenkins-Strebel directions are very helpful in understanding translation surfaces since they lead to complete cylinder decompositions. To explain this, let $\vartheta\in\bbr /2\pi\bbz$ and let $a,b>0$. We then define a translation surface with boundary\footnote{By this we mean a topological surface with boundary, together with a translation structure on its interior} $C_{\vartheta }(a,b)$ as follows: take a rectangle in $\bbc$ with two edges in direction $\vartheta$ and two edges in direction $\vartheta +\frac{\pi }{2}$, such that the edges in direction $\vartheta$ have length $a$ and the other two edges have length $b$. This is uniquely determined up to a translation. Identify the two edges in direction $\vartheta +\frac{\pi }{2}$, and the obtained translation surface with boundary, which is homeomorphic to $[0,1]\times S^1$, is called a \textit{cylinder in direction $\vartheta$ of circumference $a$ and width $b$} and denoted by $C_{\vartheta }(a,b)$.

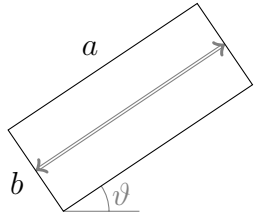
\begin{center}
\begin{tikzpicture}
\draw[gray] (0,0) -- (1,0)
            (0.6,0) arc (0:34:0.6cm);
\draw[gray] (0,0) ++(17:0.8cm) node {$\vartheta$};
\draw[gray, <->, double] (0,0) ++(124:0.65cm) -- ++(34:3cm);
\draw[black] (0,0) -- ++(34:3cm) -- ++(124:1.3cm) -- ++(214:3cm) -- cycle;
\draw (0,0) ++(124:1.3cm) ++(34:1.5cm) ++(124:0.3cm) node {$a$}
      (0,0) ++(124:0.65cm) ++(214:0.3cm) node {$b$};
\end{tikzpicture}
\end{center}

A perhaps better definition of $C_{\vartheta }(a,b)$ is the following: take an infinite closed strip $S$ of width $b$ in direction $\vartheta$; i.e. take two lines $\ell_1$, $\ell_2$ in $\bbc$ in direction $\vartheta$ whose distance is $b$, then $S$ is the convex closure of $\ell_1\cup\ell_2$. Then the additive cyclic subgroup $G$ of $\bbc$ generated by $a\exp (\mathrm{i}\vartheta )$, operating on $\bbc$ by translations, takes $S$ into itself; set $C_{\vartheta }(a,b)=S/G$.

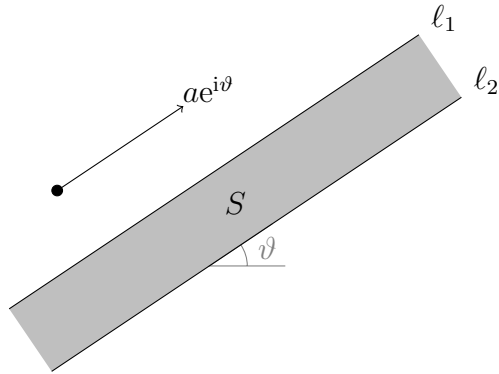
\begin{center}
\begin{tikzpicture}
\fill[gray!50!white] (0,0) -- ++(34:4cm) -- ++(124:1cm) -- ++(214:6.5cm) -- ++(304:1cm) -- cycle;
\draw (0,0) ++(124:0.5cm) ++(34:0.75cm) node {$S$};
\draw[gray] (1,0) -- (0,0) -- ++(34:1cm)
            (0.5,0) arc (0:34:0.5cm)
            (0,0) ++(17:0.8cm) node {$\vartheta$};
\draw (0,0) -- ++(34:4cm) ++(34:0.4cm) node {$\ell_2$}
      (0,0) -- ++(214:2.5cm) ++(124:1cm) -- ++(34:6.5cm) ++(34:0.4cm) node {$\ell_1$};
\draw[->] (-2,1) -- ++(34:2cm);
\draw (-2,1) ++(34:2.4cm) node {$a\mathrm{e}^{\mathrm{i}\vartheta}$};
\filldraw[black] (-2,1) circle (2pt);
\end{tikzpicture}
\end{center}

Now assume that $(X,\omega )$ is a closed translation surface, and let $\vartheta\in\bbr /2\pi\bbz$. The \textit{critical graph} in direction $\vartheta$ is the union of the zeros of $\omega$ and the saddle connections of $\mathrsfs{F}_{\vartheta }(\omega )$. If $\vartheta$ is a Jenkins-Strebel direction for $\omega$, the critical graph is a compact graph embedded into $X$. Cut $X$ along this graph; this yields a finite union of compact translation surfaces with boundary (but no corners!). Its entire boundary goes in direction $\vartheta$, and all leaves in the foliation in direction $\vartheta$ in its interior are periodic. Hence every component is isomorphic (as translation surface) to a cylinder in direction $\vartheta$. In other words we can write $(X,\omega )$ as a finite union of cylinders $C_{\vartheta }(a_i,b_i)$ (where the $a_i$ and $b_i$ might be different for different $i$), glued along their boundaries in some way. Conversely for every translation surface which is glued from cylinders in the same direction, this direction is Jenkins-Strebel.

We now discuss these concepts for some examples.
\begin{Origamis}
Let us start with one of the simplest translation surfaces: a unit square with opposite sides identified, in other words, $\bbc /\bbz [\mathrm{i}]$ with $\de z$. Here the foliation $\mathrsfs{F}_{\vartheta }(\omega )$ behaves differently according to whether the slope $\mu =\tan\vartheta$ is rational or not.\footnote{Here we are a bit sloppy. We actually view $\mu=(\sin\vartheta :\cos\vartheta )$ as an element of $\mathbb{P}^1(\bbr )$, and the question is whether $\mu$ lies in $\mathbb{P}^1(\bbq )$ or not} If it is rational, then all leaves of $\mathrsfs{F}_{\vartheta }(\omega )$ are periodic, in particular $\vartheta$ is a Jenkins-Strebel direction. If $\mu$ is irrational, however, every leaf is of type (ii), dense in the torus, and the space of leaves bears the indiscrete topology. Thus the Jenkins-Strebel directions are precisely the directions with rational slope.

This observation extends to origamis: on an origami surface the Jenkins-Strebel directions are also precisely those with rational slope. Of course the complete cylinder decompositions in the vertical and horizontal directions can be directly read off the origami itself.
\end{Origamis}
\begin{RegularPolygonsI}
Next consider the surface $(W_g,\omega_g)$; in other words, consider a regular $2n$-gon, $n=2g+1$, with opposite sides identified. Without loss of generality we may assume that there is a side in the horizontal direction. Then the horizontal direction\footnote{as well as any other direction perpendicular or parallel to a side} is a Jenkins-Strebel direction. In the picture we show a complete cylinder decomposition in the case $n=9$:
\end{RegularPolygonsI}
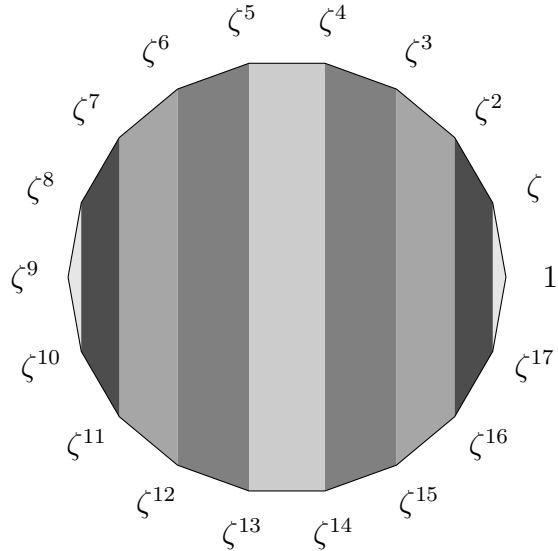
\begin{center}
\begin{tikzpicture}
\fill[fill=gray!20!white]
    (0,0) coordinate (A)
    -- ++(100:1cm) coordinate (B)
    -- ++(120:1cm) coordinate (C)
    -- ++(140:1cm) coordinate (D)
    -- ++(160:1cm) coordinate (E)
    -- ++(180:1cm) coordinate (F)
    -- ++(200:1cm) coordinate (G)
    -- ++(220:1cm) coordinate (H)
    -- ++(240:1cm) coordinate (I)
    -- ++(260:1cm) coordinate (J)
    -- ++(280:1cm) coordinate (K)
    -- ++(300:1cm) coordinate (L)
    -- ++(320:1cm) coordinate (M)
    -- ++(340:1cm) coordinate (N)
    -- ++(0:1cm) coordinate (O)
    -- ++(20:1cm) coordinate (P)
    -- ++(40:1cm) coordinate (Q)
    -- ++(60:1cm) coordinate (R)
    -- cycle;
\fill[fill=gray!40!white]
    (E) -- (F) -- (N) -- (O) -- cycle;
\fill[fill=gray]
    (D) -- (E) -- (O) -- (P) -- cycle
    (F) -- (G) -- (M) -- (N) -- cycle;
\fill[fill=gray!70!white]
    (C) -- (D) -- (P) -- (Q) -- cycle
    (G) -- (H) -- (L) -- (M) -- cycle;
\fill[fill=gray!60!black]
    (B) -- (C) -- (Q) -- (R) -- cycle
    (H) -- (I) -- (K) -- (L) -- cycle;
\draw (A) -- (B) -- (C) -- (D) -- (E) -- (F) -- (G) -- (H) -- (I) -- (J) -- (K) -- (L) -- (M) -- (N) -- (O) -- (P) -- (Q) -- (R) -- cycle;
\draw (0.5848,0) node {$1$}
    ++(100:1.2cm) node {$\zeta$}
    ++(120:1.2cm) node {$\zeta^2$}
    ++(140:1.2cm) node {$\zeta^3$}
    ++(160:1.2cm) node {$\zeta^4$}
    ++(180:1.2cm) node {$\zeta^5$}
    ++(200:1.2cm) node {$\zeta^6$}
    ++(220:1.2cm) node {$\zeta^7$}
    ++(240:1.2cm) node {$\zeta^8$}
    ++(260:1.2cm) node {$\zeta^9$}
    ++(280:1.2cm) node {$\zeta^{10}$}
    ++(300:1.2cm) node {$\zeta^{11}$}
    ++(320:1.2cm) node {$\zeta^{12}$}
    ++(340:1.2cm) node {$\zeta^{13}$}
    ++(0:1.2cm) node {$\zeta^{14}$}
    ++(20:1.2cm) node {$\zeta^{15}$}
    ++(40:1.2cm) node {$\zeta^{16}$}
    ++(60:1.2cm) node {$\zeta^{17}$};
\end{tikzpicture}
\end{center}
Here the trapezia of the same shade glue together to cylinders. For general odd $n=2g+1$ we will have one rectangle in the middle which glues to a cylinder $C_0$, and $g$ pairs of trapezia with vertices among the vertices of the $2n$-gon, where the leftmost glues with the rightmost to a cylinder $C_1$, the second one from the left with the second one from the right to $C_2$, and so on till $C_g$. Let us compute the circumferences and widths of these cylinders.

For this, assume that the vertices of the $2n$-gons are precisely the $2n$-th roots of unity in $\bbc$ (another choice will amount to multiplying everything with a constant factor). Write $\zeta =\exp (\frac{\mathrm{i}\pi}{n})$. Then clearly $C_0$ is obtained from the rectangle with vertices $\zeta^g$, $\zeta^{g+1}$, $-\zeta^g$ and $-\zeta^{g+1}$. Hence it has width
\begin{equation*}
\Re \zeta^g-\Re\zeta^{g+1}=2\Re\zeta^g=2\cos\frac{g\pi}{n}
\end{equation*}
and circumference
\begin{equation*}
2\Im\zeta^g=2\sin\frac{g\pi}{n}.
\end{equation*}
The cylinder $C_k$ is glued from two trapezia; one of them has vertices $\zeta^{k-1}$, $\zeta^k$ and their complex conjugates, the other is obtained from this by multiplication with $-1$. Hence this cylinder has width equal to the width of each trapezium, which is
\begin{equation*}
\Re\zeta^{k-1}-\Re\zeta^k=\cos\frac{(k-1)\pi}{n}-\cos\frac{k\pi}{n},
\end{equation*}
and it has circumference equal to the sum of the lengths of the two vertical sides of each trapezium, i.e.
\begin{equation*}
2\Im\zeta^{k-1}+2\Im\zeta^k=2\left(\sin\frac{(k-1)\pi}{n}+\sin\frac{k\pi}{n}\right) .
\end{equation*}
\begin{RegularPolygonsII}
Next consider $(W_g,\omega_1)$, i.e. the translation surface obtained from gluing two regular $n$-gons, with $n=2g+1$. We may assume that one of the $n$-gons has as vertices the $n$-th roots of unity, the other one has minus the $n$-th roots of unity (or, if you so wish, some large enough number minus the $n$-th roots of unity, in order that the two polygons do not intersect). Write $\xi=\zeta^2 =\exp\frac{2\pi}{n}$. Again the vertical direction is a Jenkins-Strebel direction; in the example $n=9$ the cylinder decomposition looks as follows:
\end{RegularPolygonsII}
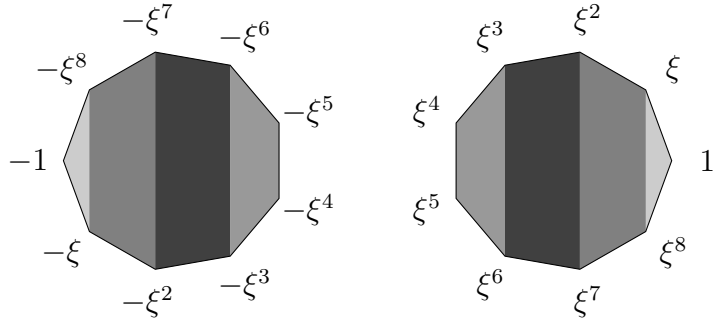
\begin{center}
\begin{tikzpicture}
\fill[color=gray!40!white]
    (0,0) coordinate (A1)
    -- ++(70:1cm) coordinate (A2)
    -- ++(30:1cm) coordinate (A3)
    -- ++(350:1cm) coordinate (A4)
    -- ++(310:1cm) coordinate (A5)
    -- ++(270:1cm) coordinate (A6)
    -- ++(230:1cm) coordinate (A7)
    -- ++(190:1cm) coordinate (A8)
    -- ++(150:1cm) coordinate (A9)
    -- cycle
    (8,0) coordinate (B1)
    -- ++(250:1cm) coordinate (B2)
    -- ++(210:1cm) coordinate (B3)
    -- ++(170:1cm) coordinate (B4)
    -- ++(130:1cm) coordinate (B5)
    -- ++(90:1cm) coordinate (B6)
    -- ++(50:1cm) coordinate (B7)
    -- ++(10:1cm) coordinate (B8)
    -- ++(330:1cm) coordinate (B9)
    -- cycle;
\fill[color=gray]
    (A2) -- (A3) -- (A8) -- (A9) -- cycle
    (B2) -- (B3) -- (B8) -- (B9) -- cycle;
\fill[color=gray!50!black]
    (A3) -- (A4) -- (A7) -- (A8) -- cycle
    (B3) -- (B4) -- (B7) -- (B8) -- cycle;
\fill[color=gray!80!white]
    (A4) -- (A5) -- (A6) -- (A7) -- cycle
    (B4) -- (B5) -- (B6) -- (B7) -- cycle;
\draw (A1) -- (A2) -- (A3) -- (A4) -- (A5) -- (A6) -- (A7) -- (A8) -- (A9) -- cycle
      (B1) -- (B2) -- (B3) -- (B4) -- (B5) -- (B6) -- (B7) -- (B8) -- (B9) -- cycle;
\draw (-0.4667,0) node {$-1$}
    ++(70:1.3cm) node {$-\xi^8$}
    ++(30:1.3cm) node {$-\xi^7$}
    ++(350:1.3cm) node {$-\xi^6$}
    ++(310:1.3cm) node {$-\xi^5$}
    ++(270:1.3cm) node {$-\xi^4$}
    ++(230:1.3cm) node {$-\xi^3$}
    ++(190:1.3cm) node {$-\xi^2$}
    ++(150:1.3cm) node {$-\xi$}
    (8.4667,0) node {$1$}
    ++(110:1.3cm) node {$\xi$}
    ++(150:1.3cm) node {$\xi^2$}
    ++(190:1.3cm) node {$\xi^3$}
    ++(230:1.3cm) node {$\xi^4$}
    ++(270:1.3cm) node {$\xi^5$}
    ++(310:1.3cm) node {$\xi^6$}
    ++(350:1.3cm) node {$\xi^7$}
    ++(30:1.3cm) node {$\xi^8$};
\end{tikzpicture}
\end{center}
Here we have $g$ cylinders $C_1,\ldots ,C_g$, where $C_k$ is glued from two trapezia (possibly degenerate): one of them has vertices $\xi^{k-1}$, $\xi^k$ and their complex conjugates, the other is minus one times the first. Hence we see that $C_k$ has width
\begin{equation*}
\Re\xi^{k-1}-\Re\xi^k=\cos\frac{2(k-1)\pi}{n}-\cos\frac{k\pi }{n}
\end{equation*}
and circumference
\begin{equation*}
2\Im\xi^{k-1}+2\Im\xi^k=2\left(\sin\frac{2(k-1)\pi }{n}+\sin\frac{2k\pi}{n}\right) .
\end{equation*}

\newpage
\section{Abelian Differentials: Global Theory}

In the previous chapter we have studied the geometry of one single abelian differential; now we go on to see how it can be deformed. We exploit the fact that abelian differentials can be understood in two ways, with different intuitions behind them: as objects of complex or algebraic geometry, or as objects of euclidean geometry. We begin by constructing moduli spaces for abelian differentials $\Omega\mathrsfs{M}_g$, together with variants analogous to moduli spaces with level structures and Teichm\"{u}ller spaces. On all such spaces, the interpretation in terms of euclidean geometry gives a natural $\SL_2(\bbr )$-operation.

Now let $X$ be a compact Riemann surface of genus $g$, equipped with a Teichm\"{u}ller marking and an abelian differential $\omega$. We may consider $(X,\omega )$ as an object of $\Omega\mathrsfs{T}_g$. Its $\SL_2(\bbr )$-orbit in $\Omega\mathrsfs{T}_g$ will not be a complex submanifold, but the image of this orbit in $\mathrsfs{T}_g$ will be. Even better, it will be an isometrically embedded hyperbolic plane. In order to understand the details and to be able to do explicit computations in examples, we have to go through a painfully pedantic analysis of the different avatars of the hyperbolic plane. Recall that the hyperbolic plane can be constructed as ``the symmetric space of $\SL_2(\bbr )$'', the unit disk $\Delta\subset\bbc$ or the upper half plane $\bbh\subset\bbc$. These can be identified by several ``canonical'' isomorphisms. This is common wisdom, but we have to fix some compatible system of isomorphisms and check that they are indeed compatible.

Once this is settled, one can indeed construct for every $(X,\omega )$ as above a Teichm\"{u}ller disk $f^{\omega }:\Delta\to\mathrsfs{T}_g$. Since such disks are obtained with the help of abelian differentials, we call them \textit{abelian Teichm\"{u}ller disks}. The next question is then: how does the image of an abelian Teichm\"{u}ller disk in moduli space look like? In general this will be very complicated, but in some rare cases it happens that the image is an algebraic curve. Such curves in moduli space are called \textit{(abelian) Teichm\"{u}ller curves}. We take a closer look at the associated families of algebraic curves.

Whether a translation surface gives rise to a Teichm\"{u}ller curve or not can also be rephrased in terms closer to Euclidean geometry. For a translation surface $(X,\omega )$ one defines a discrete subgroup of $\SL_2(\bbr )$, called the \textit{Veech group} of the surface. Then for some conjugate $\Gamma$ of the Veech group, called the \textit{mirror Veech group}, the composition
\begin{equation*}
\bbh\overset{f^{\omega }}{\to }\mathrsfs{T}_g\to\mathrsfs{M}_g
\end{equation*}
factors over the quotient $\Gamma\backslash\bbh$. Hence $(X,\omega )$ will give rise to a Teichm\"{u}ller curve if the Veech group is a lattice in $\SL_2(\bbr )$. The converse is also true; surfaces that satisfy this condition are called \textit{Veech surfaces}.

We also sketch a helpful method for the construction of nontrivial elements in the Veech group. This will enable us to determine (finite index subgroups of) the Veech groups of our favourite examples $(W_g,\omega_1)$ and $(W_g,\omega_g)$. Both are Veech surfaces, with hyperbolic triangle groups as (mirror) Veech groups.

\subsection{The $\SL_2(\bbr )$-Action}

\begin{ModuliofAD}
Let once again $p:X\to C$ be a family of algebraic curves of genus $g$. For such a family we have the sheaf of \textit{relative holomorphic one-forms}, $\Omega_{X|C}^1$. This is an invertible $\mathrsfs{O}_X$-module; its restriction to every fibre $X_c=p^{-1}(c)$ can be identified with the canonical bundle on this fibre. A section $\omega$ of this sheaf hence defines a holomorphic one-form on each fibre $X_c$.
\end{ModuliofAD}
\begin{Definition}
Let $C$ be a complex space. A \textup{family of abelian differentials of genus $g$} over $C$ consists of the following data:
\begin{enumerate}
\item a family of curves $p:X\to C$ of genus $g$,
\item a section $\omega\in\Gamma (X,\Omega_{X|C}^1)$ such that for every $c\in C$ the induced holomorphic one-form $\omega|_{X_c}$ on the fibre $X_c$ is not identically zero, i.e. an abelian differential by our definition.
\end{enumerate}
\end{Definition}
There is an obvious notion of isomorphism of families of abelian differentials, giving rise to different moduli functors on the category of complex spaces. First consider the moduli functor $\Omega M_g$ whose value on a complex space $C$ is the set of isomorphism classes of families of abelian differentials over $C$. We would like to represent this moduli functor, but since the moduli functor of curves itself is not representable, chances are bad that this is possible. So we introduce rigidifications, as for curves.

Again there are two options: Teichm\"{u}ller markings and level structures. Hence we define the moduli functor $\Omega T_g$ on complex spaces which for a complex space $C$ gives the set of isomorphism classes of families of abelian differentials together with a Teichm\"{u}ller marking of the underlying family of curves. We impose no compatibility conditions on the Teichm\"{u}ller markings and the abelian differentials. This functor is representable, and the representing space can be constructed as follows: let $p:\mathrsfs{U}_g\to\mathrsfs{T}_g$ be the universal curve over Teichm\"{u}ller space. Consider the associated sheaf of relative holomorphic one-forms $\Omega_{\mathrsfs{U}_g|\mathrsfs{T}_g}^1$ and take its direct image by $p$; this gives a locally free $\mathrsfs{O}_{\mathrsfs{T}_g}$-module $p_{\ast}\Omega_{\mathrsfs{U}_g|\mathrsfs{T}_g}^1$ of rank $g$. The associated holomorphic vector bundle $H\to\mathrsfs{T}_g$ is (sometimes) called the \textit{Hodge bundle}. Define then $\Omega\mathrsfs{T}_g$ to be this bundle minus its zero section. So this is a complex manifold together with a holomorphic map to $\mathrsfs{T}_g$ which turns it into a holomorphic fibre bundle with fibre $\bbc^{g}\smallsetminus \{ 0\}$. Clearly it represents the functor $\Omega T_g$. In particular the points of $\Omega\mathrsfs{T}_g$ are in canonical one-to-one correspondence with the isomorphism classes of marked genus $g$ Riemann surfaces together with an abelian differential.

There is a natural action of the mapping class group $\Mod_g$ on $\Omega\mathrsfs{T}_g$: for a mapping class $[\gamma ]$ represented by a diffeomorphism $\gamma$ of $\Sigma_g$ and a point $t\in\Omega\mathrsfs{T}_g$ represented by a marked Riemann surface $f:\Sigma_g\to X$ and an abelian differential $\omega$ on $X$, let $[\gamma ]\cdot t$ be the point of $\Omega\mathrsfs{T}_g$ represented by the same Riemann surface and the same abelian differential, but now with the marking $f\circ\gamma^{-1}$. This action is continuous and properly discontinuous, and the map $\Omega\mathrsfs{T}_g\to\mathrsfs{T}_g$ is $\Mod_g$-equivariant.

The quotient $\Omega\mathrsfs{M}_g=\Mod_g\backslash\Omega\mathrsfs{T}_g$ is now, in an obvious way, a coarse moduli space for $\Omega M_g$. It is called the \textit{moduli space of abelian differentials}. The map $\Omega\mathrsfs{T}_g\to\mathrsfs{T}_g$, being equivariant for the action of the mapping class group, descends to a holomorphic map $\Omega\mathrsfs{M}_g\to\mathrsfs{M}_g$. This is, however, no longer a fibre bundle. Instead its fibre over the point of $\mathrsfs{M}_g$ represented by a Riemann surface $X$ is
\begin{equation*}
(H^0(X,\omega_X)\smallsetminus \{ 0\})/\Aut (X).
\end{equation*}
We may also introduce level structures: the quotient $\Omega\mathrsfs{M}_g^{[n]}=\Mod_g^{[n]}\backslash\Omega\mathrsfs{T}_g$ is, for $n\ge 3$, a fine moduli space for Riemann surfaces with an abelian differential and a level-$n$-structure (again no compatibility between the latter two is demanded). And here the map $\Omega\mathrsfs{M}_g^{[n]}\to\mathrsfs{M}_g^{[n]}$ is indeed a fibre bundle with fibre $\bbc^g\smallsetminus \{ 0\}$.

\begin{TheGLtwoRAction}
These moduli spaces come equipped with some important and interesting structure. Abelian differentials can also be interpreted as translation structures; hence we also have two points of view on their moduli spaces, linked to different intuitions and with highly nontrivial interactions.

The starting point is the following observation: Translation structures can be distorted by real-linear maps. This is not very surprising, but this operation is quite difficult to understand from the point of view of algebraic curves and one-forms.

As usual, we identify $\bbr^2$ with $\bbc$. Hence a matrix $A\in M_2(\bbr )$ operates $\bbr$-linearly on $\bbc$; we denote this operation by $z\mapsto Az$ (please do not confuse this with M\"{o}bius transformations!). This map is holomorphic if and only if $A\in\bbr\cdot\SO (2)$.
\end{TheGLtwoRAction}
\begin{Definition}
Let $\Sigma$ be a translation surface with corners and cone points, and let $A\in\GL_2(\bbr )$. Then the \textup{distortion of $\Sigma$ by $A$}, denoted by $A\cdot\Sigma$, is the following translation surface with corners and cone points: the underlying topological surface, the set of corners and the set of cone points is the same. If $\mathfrak{T}$ is the atlas defining the translation structure on $\Sigma$, the atlas of $A\cdot \Sigma$ consists of all $A\circ\varphi$ where $\varphi$ is a chart in $\mathfrak{T}$. Here $A$ is understood as a real linear self-map of $\bbc$.
\end{Definition}
This can of course also be applied to Riemann surfaces: let $X$ be a Riemann surface with an abelian differential $\omega$. Interpreting $(X,\omega )$ as a closed translation surface with cone points, we can apply the previous definition and obtain a new Riemann surface with abelian differential $A\cdot (X,\omega )$. In order to avoid nasty orientation problems we restrict ourselves to the case $A\in\GL_2^+(\bbr )$ and, in a moment, in fact $A\in\SL_2(\bbr )$.

Write $A\cdot (X,\omega )=(X',\omega ')$. Note that $X'$ has the same underlying topological (and, in fact, smooth) surface as $X$, but possibly a different complex structure. It is in general quite difficult to determine this new complex structure as an algebraic curve. But it is easy to describe $\omega '$ from the point of view of real analysis: this is the one-form
\begin{equation*}
\begin{pmatrix}
1&\mathrm{i}
\end{pmatrix}
\cdot A \cdot
\begin{pmatrix}
\Re\omega\\
\Im\omega
\end{pmatrix} .
\end{equation*}
More conceptually: under our identification $\bbr^2=\bbc$, we can interpret $\omega$ as a vector-valued real one-form. Then $\omega'$ is the vector-valued real one-form obtained by applying the matrix $A$ to $\omega$.

The assignment $(X,\omega )\mapsto A\cdot (X,\omega )$ as defined above is an action of the group $\GL_2^+(\bbr )$ on $\Omega\mathcal{T}_g$. This action commutes with the action of the mapping class group $\Mod_g$, and therefore descends to $\GL_2^+(\bbr )$-actions on the moduli spaces $\Omega\mathrsfs{M}_g$ and $\Omega\mathrsfs{M}_g^{[n]}$.

\subsection{The Various Guises of the Hyperbolic Plane}

One often identifies the following four spaces:
\begin{itemize}
\item the unit disk $\Delta =\{ t\in\bbc : |t|<1\}$,
\item the upper half plane $\bbh =\{\tau \in\bbc :\operatorname{Re}\tau >0\}$,
\item the left coset space $\SO (2)\backslash\SL_2(\bbr )$ and
\item the right coset space $\SL_2(\bbr )/\SO (2)$.
\end{itemize}
\noindent Considering the ambiguities involved we need to spell out explicitly which identifications we take and how they relate with each other. More precisely, we construct a commutative diagram of homeomorphisms
\begin{equation}\label{DiagramForGuisesOfDisk}
\xymatrix{
\SO (2)\backslash\SL_2(\bbr ) \ar[r]^{\iota} \ar[d]_{\mu } \ar[rd]^{\lambda }
&\SL_2(\bbr )/\SO (2) \ar[d]^{\sigma}\\
\Delta &\bbh . \ar[l]^{C}
}
\end{equation}
Let us now define these maps:
\begin{itemize}
\item $\iota :\SO (2)\backslash\SL_2(\bbr )\to\SL_2(\bbr )/\SO (2)$ is induced by an anti-involution on the Lie group $\SL_2(\bbr )$. Namely consider $\SL_2(\bbr )$ as a normal subgroup of $\GL_2(\bbr )$, then conjugation with any element of $\GL_2(\bbr )$ is a well-defined automorphism of $\SL_2(\bbr )$. Hence we can define
    \begin{equation*}
    \iota_0:\SL_2(\bbr )\to\SL_2(\bbr ),\quad A\mapsto
    \begin{pmatrix}
    -1&0\\
    0&1
    \end{pmatrix}
    A^{-1}
        \begin{pmatrix}
    -1&0\\
    0&1
    \end{pmatrix}.
    \end{equation*}
    Since $\iota_0(\SO (2))=\SO (2)$ and $\iota_0(AB)=\iota_0(B)\iota_0(A)$, this descends to a homeomorphism
    \begin{equation*}
    \iota :\SO (2)\backslash\SL_2(\bbr )\to\SL_2(\bbr )/\SO (2),\quad \SO (2)\cdot A\mapsto \iota_0(A)\cdot\SO (2).
    \end{equation*}
    Considering the entries of a matrix, it is easy to see that
    \begin{equation}\label{ElementaryFormulaForIota}
    \iota_0
        \begin{pmatrix}
    a&b\\
    c&d
    \end{pmatrix}
    =
    \begin{pmatrix}
    d&b\\
    c&a
    \end{pmatrix}.
    \end{equation}
\item $\sigma :\SL_2(\bbr )/\SO (2)\to\bbh$ is the standard identification: $\SL_2(\bbr )$ acts transitively on $\bbh$ by M\"{o}bius transformations, in formul\ae :
\begin{equation*}
\begin{pmatrix}
a&b\\
c&d
\end{pmatrix}
\cdot \tau =\frac{a\tau +b}{c\tau +d},
\end{equation*}
and the stabilizer of the point $\mathrm{i}\in\bbh$ is $\SO (2)$. Hence we get
\begin{equation*}
\sigma :\SL_2(\bbr )/\SO (2)\overset{\simeq }{\to }\bbh ,\quad A\cdot\SO (2)\mapsto A\cdot\mathrm{i}.
\end{equation*}
\item $\lambda :\SO (2)\backslash\SL_2(\bbr )\to\bbh$ is obtained by considering the natural action of $\SL_2(\bbr )$ on $\bbr^2=\bbc$.
    Viewing $A\in\SL_2(\bbr )$ as an $\bbr$-linear self map of $\bbc$, we ask: how far is it apart from being $\bbc$-linear? Certainly it is $\bbc$-linear if and only if $A(\mathrm{i})=\mathrm{i}A(1)$, so a good measure for this deviation would be the complex number
    \begin{equation*}
    \lambda_0(A)=\frac{A(\mathrm{i})}{A(1)}
    \end{equation*}
    which is in the upper half plane since $A$ has determinant one. Now every orthogonal matrix $T\in\SO (2)$ operates on $\bbc$ by multiplication with some complex number $\zeta$ of norm one, hence
    \begin{equation*}
    \lambda_0(TA)=\frac{T(A(\mathrm{i}))}{T(A(1))}=\frac{\zeta\cdot A(\mathrm{i})}{\zeta\cdot A(1)}=\frac{A(\mathrm{i})}{A(1)}=\lambda_0(A).
    \end{equation*}
    So $\lambda_0$ descends to a map $\lambda :\SO (2)\backslash\SL_2(\bbr )\to\bbh$.
\item In order to define $\mu :\SO(2)\backslash\SL_2(\bbr )\to\Delta$ we take the same setup, and again we ask: how far apart is $A\in\SL_2(\bbr )$, as an $\bbr$-linear map on $\bbc$, from being conformal? Since it is linear and orientation-preserving, it is quasi-conformal for some constant Beltrami coefficient. This means that it satisfies the Beltrami equation
    \begin{equation*}
    \frac{\partial A}{\partial\overline{z}}=\mu_0(A)\cdot\frac{\partial A}{\partial z}
    \end{equation*}
    for a unique $\mu_0(A)\in\Delta$. Again, orthogonal matrices $T\in\SO (2)$ operate on $\bbc$ by conformal maps, so that $\mu_0(TA)=\mu_0(A)$. This again means that $\mu_0$ descends to a well-defined map $\mu :\SO (2)\backslash\SL_2(\bbr )\to\Delta$.
\item Finally $C:\bbh\to\Delta$ is the M\"{o}bius transformation
\begin{equation*}
C:\bbh\to\Delta ,\quad\tau\mapsto\frac{\mathrm{i}-\tau }{\mathrm{i}+\tau }.
\end{equation*}
with inverse
\begin{equation*}
C^{-1}:\Delta\to\bbh ,\quad t\mapsto\mathrm{i}\cdot\frac{1-t}{1+t}
\end{equation*}
This differs from the classical Cayley map by $t\mapsto -t$ on $\Delta $ or, equivalently, by $\tau\mapsto -1/\tau$ on $\bbh$.
\end{itemize}
\begin{Proposition}
The five maps $\iota$, $\sigma$, $\lambda$, $\mu$ and $C$ defined above are all homeomorphisms, and the diagram (\ref{DiagramForGuisesOfDisk}) commutes.
\end{Proposition}
\begin{proof}
We have already seen that $\iota$, $\sigma$ and $C$ are homeomorphisms, so if the diagram commutes, $\lambda$ and $\mu$ are also homeomorphisms.

Hence it remains to show that the diagram commutes, which amounts to the two identities $\lambda =\sigma\circ\iota$ and $\mu =C\circ\lambda$. As to the first one, take $A\in\SL_2(\bbr )$ and write
\begin{equation*}
A=
\begin{pmatrix}
a&b\\
c&d
\end{pmatrix}. \end{equation*}
Then
\begin{equation*}
\lambda (A\cdot\SO (2))=\lambda_0(A)=\frac{A(\mathrm{i})}{A(1)}=\frac{b+\mathrm{i}d}{a+\mathrm{i}c}.
\end{equation*}
On the other hand, recalling (\ref{ElementaryFormulaForIota}), we can write
\begin{equation*}
\sigma (\iota (\SO (2)\cdot A))=\sigma_0(\iota_0(A))=
\begin{pmatrix}
d&b\\
c&a
\end{pmatrix}
\cdot\mathrm{i}=\frac{d\mathrm{i}+b}{c\mathrm{i}+a},
\end{equation*}
which is the same. Hence $\lambda =\sigma\circ\iota$.

The commutativity of the lower left triangle can be seen as follows: take again some $A\in\SL_2(\bbr )$. In order to compute
its Beltrami coefficient, we need to write it as a sum of a $\bbc$-linear map and a $\bbc$-antilinear map:
\begin{multline}
A (z)=A(\operatorname{Re }z+\mathrm{i}\operatorname{\Im }z)=A(1)\operatorname{Re }z+A(\mathrm{i})\mathrm{Im }z\\
=A(1)\cdot\frac{z+\overline{z}}{2}+A(\mathrm{i})\cdot\frac{z-\overline{z}}{2\mathrm{i}}
=\frac{A(1)-\mathrm{i}A(\mathrm{i})}{2}z+\frac{A(1)+\mathrm{i}A(\mathrm{i})}{2}\overline{z},
\end{multline}
so the Beltrami coefficient of $A$ is
\begin{equation*}
\frac{\partial A/\partial\overline{z}}{\partial A/\partial z}=\frac{A(1)+\mathrm{i}A(\mathrm{i})}{A(1)-\mathrm{i}A(\mathrm{i})}
=\frac{\mathrm{i}-\frac{A(\mathrm{i})}{A(1)}}{\mathrm{i}+\frac{A(\mathrm{i})}{A(1)}}=C(\lambda_0(A)).
\end{equation*}
In other words, $\mu_0(A)=C(\lambda_0(A))$ and thus $\mu =C\circ\lambda$, what was to be shown.
\end{proof}

\subsection{Affine maps and the Veech Group}

\begin{AffineMaps}
In this section we introduce an important class of self-maps of translation surfaces.
\end{AffineMaps}
\begin{Definition}
Let $S_1$, $S_2$ be translation surfaces with cone points. An \textup{affine map} from $S_1$ to $S_2$ is an orientation-preserving homeomorphism $f:S_1\to S_2$ with the following properties:
\begin{enumerate}
\item Let $P_i$ be the set of cone-points of $S_i$, then $f(P_1)=P_2$.
\item For every pair of charts on $S_1$ and $S_2$ in the atlases defining the translation structures, $f$ obtains the form $z\mapsto Az+b$ for some $A\in \GL_2^+(\bbr )$ and $b\in\bbc$. Here we identify $\bbc =\bbr^2$ in the usual way.
\end{enumerate}
When $f$ is an affine map and $S_1$ is connected, the matrix $A$ in (ii) is the same in every chart; it is called the \textup{linear part} of $f$.
\end{Definition}

How do affine maps behave from the point of view of complex analysis?

\begin{Proposition}
Let $X_1$ and $X_2$ be Riemann surfaces, let $\omega_i$ be an abelian differential on $X_i$ and regard $(X_i,\omega_i)$ as translation surfaces with cone points.
\begin{enumerate}
\item Let $f:X_1\to X_2$ be an affine map. Then $f$ is holomorphic (or, equivalently, biholomorphic) if and only if the affine part of $f$ is an element of $\SO (2)$.
\item Let $f:X_1\to X_2$ be a homeomorphism. Then $f$ is an affine map with linear part equal to the identity if and only if $f$ is holomorphic and $f^{\ast}\omega_2=\omega_1$.\hfill $\square$
\end{enumerate}
\end{Proposition}

From now on, to ease formulations, we identify closed translation surfaces with cone points and compact Riemann surfaces with abelian differentials. We also restrict our attention to the case where the domain and the target agree. So let $X$ be a compact Riemann surface and let $\omega$ be an abelian differential on $X$. Then the set of affine self-maps of $(X,\omega )$ forms a group under composition, which we denote by $\Aff^+(X,\omega)$. This is often called the \textit{affine group} of $(X,\omega )$.  Although we will not make use of it now, we anticipate the following important statement (see Proposition \ref{AffineGroupInjectsIntoModg}): the canonical homomorphism from $\Aff^+(X,\omega )$ to the mapping class group of $X$ is injective. This explains one of the motivations for studying affine groups of translation surfaces: they provide more or less easily accessible subgroups of mapping class groups. From this point of view it is natural to look for ``big'' affine groups.

\begin{TheVeechGroup}
There is a natural homomorphism $D:\Aff^+(X,\omega )\to\GL_2^+(\bbr )$ which sends every affine map to its linear part. If $X$ is compact, as we assumed, the image of this homomorphism is even contained in $\SL_2(\bbr )$:
\end{TheVeechGroup}
\begin{Lemma}
Let $X$ be a compact Riemann surface, let $\omega$ be a nonzero holomorphic one-form on $X$ and let $f\in\Aff^+(X,\omega )$. Then the determinant of $D\varphi$ is one.
\end{Lemma}
\begin{proof}
Let $\de V=\frac{\mathrm{i}}{2}\cdot\omega\wedge\overline{\omega }$ be the volume element of the Riemannian metric on $X$, see above. Then $\varphi^{\ast}\de V=\det (D\varphi )\cdot\de V$, and using that the total volume of $X$ is finite we find
\begin{equation*}
\operatorname{vol}(X)=\int_X\de V=\int_X\varphi^{\ast}\de V=\det (D\varphi )\int_X\de V=\det (D\varphi )\cdot\operatorname{vol}X,
\end{equation*}
so that $\det (D\varphi )=1$.
\end{proof}
So indeed $D$ defines a group homomorphism $\Aff^+(X,\omega )\to\SL_2(\bbr )$, and it is straightforward to see that this is a group homomorphism. It is also clear that the kernel of this map is precisely $\Aut (X,\omega )$. The image is harder to describe and hence deserves a name of its own:
\begin{Definition}
Let $X$ and $\omega $ as above. The image of the group homomorphism $D:\Aff^+(X,\omega )\to\SL_2(\bbr )$ is called the \textup{Veech group} of $(X,\omega )$ and denoted by $\SL (X,\omega )$.
\end{Definition}
The Veech group is also closely connected with the $\SL_2(\bbr )$-action on moduli spaces discussed in the previous section.
\begin{Proposition}
Let $(X,\omega )$ be a closed Riemann surface with an abelian differential. Then the Veech group $\SL (X,\omega )$ is equal to the stabilizer of the point in $\Omega\mathrsfs{M}_g$ represented by $(X,\omega )$ in $\SL_2(\bbr )$.
\end{Proposition}
\begin{proof}
To ease notation, denote this stabilizer by $S$. Hence
\begin{equation*}
S=\{ A\in\SL_2(\bbr )\, |\, A\cdot (X,\omega )\simeq (X,\omega )\text{ as Riemann surfaces with abelian differentials}\} .
\end{equation*}
We first show that $\SL (X,\omega )$ is contained in $S$. Let $A\in\SL (X,\omega )$, then there is some affine self-map $\varphi$ of $(X,\omega )$ with derivative $A$. A local coordinate computation\footnote{The map $A: A\cdot (\bbc ,\de z)\to (\bbc ,\de z)$ is an isomorphism of translation surfaces} confirms that the map
\begin{equation*}
A\cdot (X,\omega )\overset{\varphi }{\to }(X,\omega )
\end{equation*}
is an isomorphism of translation surfaces. Hence $A\in S$.

For the other inclusion let $A\in S$. This means that there is some isomorphism of translation surfaces
\begin{equation*}
\varphi :A\cdot (X,\omega )\to (X,\omega ).
\end{equation*}
In other words, this is an affine map with derivative equal to the identity. On the other hand we claim that the identity map
\begin{equation}\label{IdentityBetweenAffineStructures}
\operatorname{id} :(X,\omega )\to A\cdot (X,\omega )
\end{equation}
is an affine map with derivative $A$. This can again be seen by a local coordinate computation: look at the identity map from $(\bbc ,\de z)$ to $A\cdot (\bbc ,\de z)$. In the following diagram, the vertical maps are charts of the respective translation structures:
\begin{equation*}
\xymatrix{
(\bbc ,\de z) \ar[r]^{\mathrm{id}} \ar[d]_{\mathrm{id}}
& A\cdot (\bbc ,\de z) \ar[d]^A\\
(\bbc ,\de z) \ar[r]_A
& (\bbc ,\de z).
}
\end{equation*}
Thus indeed the identity map from $(\bbc ,\de z)$ to $A\cdot (\bbc ,\de z)$ is affine with derivative $A$, which proves our statement about the map (\ref{IdentityBetweenAffineStructures}).

Hence the composition
\begin{equation*}
(X,\omega )\overset{\mathrm{id}}{\to }A\cdot (X,\omega )\overset{\varphi }{\to }(X,\omega )
\end{equation*}
is an affine self-map of $(X,\omega )$ with derivative $A$. Thus $A$ is in the Veech group.
\end{proof}
Summarizing we get a short exact sequence
\begin{equation}\label{SESforVeechGroups}
1\to\Aut (X,\omega )\to\Aff^+(X,\omega )\overset{D}{\to}\SL (X,\omega )\to 1.
\end{equation}
Since the group $\Aut (X,\omega )$ is finite, the Veech group and the affine group are not terribly far apart. The Veech group however, as a group of matrices, is easier to describe. It has some nice properties:
\begin{Proposition}\label{VeechGroupIsFuchsianNonCocompact}
The Veech group $\SL (X,\omega )$ is a discrete and non-cocompact subgroup of $\SL_2(\bbr )$. Operating upon $\bbc =\bbr^2$, it permutes the holonomy vectors of $(X,\omega )$.
\end{Proposition}
\begin{proof}
The second statement is clear; the discreteness statement is originally due to Veech \cite{Veech89}. A much more elegant proof can be found in \cite[section 1.3.2]{HubertSchmidt04}. In that same section one also finds a proof for non-cocompactness. We briefly sketch the ideas.

As to discreteness, the set of holonomy vectors of $(X,\omega )$ is a discrete subset of $\bbc$. Hence $\SL (X,\omega )$ also has to be discrete.

Now Hubert and Schmidt consider the function $\Lambda :\SL_2(\bbr )\to ]0,\infty [$ with $\Lambda (A)$ being the length of the shortest holonomy vector of $A\cdot (X,\omega )$. This is continuous and descends to a function on the quotient $\SL_2(\bbr )/\SL (X,\omega )$. Hence if this quotient were compact, the image of $\Lambda$ would have to be compact. But
one can make a given saddle connection arbitrarily short by $\SL_2(\bbr )$-transformations, hence the image of $\Lambda$ comes arbitrarily close to zero, contradiction.
\end{proof}

\begin{ParabolicsInTheVG}
Now we describe a method to find nontrivial elements in the affine group, and hence the Veech group, of a translation surface. Let $X$ be a compact Riemann surface with an abelian differential $\omega$. Assume that $\vartheta$ is a Jenkins-Strebel direction for $\omega$, so that we get a complete cylinder decomposition. The idea is now to define some simple affine maps on the cylinders and investigate when they can be glued together to an affine self-map of the whole surface.

So let $a,b>0$ and consider the cylinder $C_{\vartheta }(a,b)$. We consider this, as described above, as the quotient of a closed strip $S$ in direction $\vartheta$ of width $b$ in the complex plane. Say this strip has boundary lines $\ell_1$ and $\ell_2$ where, when viewed in direction $\vartheta$, the line $\ell_1$ is the left boundary of $S$ and $\ell_2$ is the right boundary:
\end{ParabolicsInTheVG}
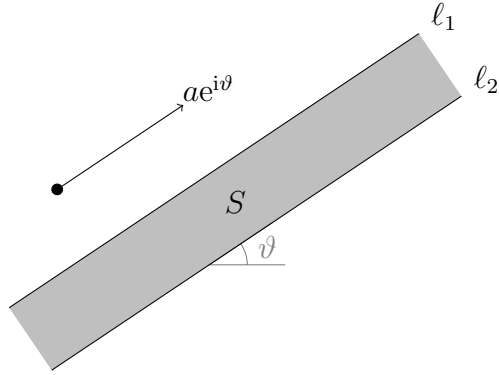
\begin{center}
\begin{tikzpicture}
\fill[gray!50!white] (0,0) -- ++(34:4cm) -- ++(124:1cm) -- ++(214:6.5cm) -- ++(304:1cm) -- cycle;
\draw (0,0) ++(124:0.5cm) ++(34:0.75cm) node {$S$};
\draw[gray] (1,0) -- (0,0) -- ++(34:1cm)
            (0.5,0) arc (0:34:0.5cm)
            (0,0) ++(17:0.8cm) node {$\vartheta$};
\draw (0,0) -- ++(34:4cm) ++(34:0.4cm) node {$\ell_2$}
      (0,0) -- ++(214:2.5cm) ++(124:1cm) -- ++(34:6.5cm) ++(34:0.4cm) node {$\ell_1$};
\draw[->] (-2,1) -- ++(34:2cm);
\draw (-2,1) ++(34:2.4cm) node {$a\mathrm{e}^{\mathrm{i}\vartheta}$};
\filldraw[black] (-2,1) circle (2pt);
\end{tikzpicture}
\end{center}
Then let $f$ be the unique real-affine map $\bbc\to\bbc$ with the following properties: $f$ sends each of the two lines $\ell_1$ and $\ell_2$ to itself; it is the identity on $\ell_2$ and translation by $a\exp (\mathrm{i}\vartheta )$ on $\ell_1$. This descends to a well-defined map on $C_{\vartheta }(a,b)$ which induces the identity on the boundary and is an affine map in the interior; it is called the \textit{affine Dehn twist} of $C_{\vartheta }(a,b)$.

Let us compute the linear part of this affine map for two important special cases. First let $\vartheta =0$, so that the direction of the cylinder is horizontal. Assume without loss of generality that $\ell_1$ is the real axis. Then we have $f(x)=x$ for real $x$, and $f(b\mathrm{i})=b\mathrm{i}+a$, hence $f(\mathrm{i})=\mathrm{i}+a/b$. In other words the affine part of $f$ is
\begin{equation*}
\begin{pmatrix}
1&\frac{a}{b}\\
0&1
\end{pmatrix}.
\end{equation*}
To ease terminology we introduce the following notion:
\begin{Definition}
The \textup{modulus} of a cylinder $C_{\vartheta }(a,b)$ is the positive real number $\mu =a/b$, i.e. its circumference divided by its width.
\end{Definition}
So we find: the affine part of the affine Dehn twist of a cylinder in direction $0$ of modulus $\mu$ is the matrix
\begin{equation}\label{ParabolicMatrix}
\begin{pmatrix}
1&\mu\\
0&1
\end{pmatrix}.
\end{equation}
Since a cylinder in direction $\vartheta$ can be obtained from one in direction zero by rotation with $\vartheta$, we see that the affine part of the affine Dehn twist of a cylinder in direction $\vartheta$ of modulus $\mu$ is the conjugate of the matrix (\ref{ParabolicMatrix}) by this rotation, in other words
\begin{equation}\label{ParabolicConjugate}
\begin{pmatrix}
\cos\vartheta&-\sin\vartheta\\
\sin\vartheta&\cos\vartheta
\end{pmatrix}
\begin{pmatrix}
1&\mu\\
0&1
\end{pmatrix}
\begin{pmatrix}
\cos\vartheta&\sin\vartheta\\
-\sin\vartheta&\cos\vartheta
\end{pmatrix}.
\end{equation}
One particular case deserves to be mentioned: for $\vartheta =\frac{\pi }{2}$, i.e. in the vertical direction, we get the matrix
\begin{equation}\label{ParabolicForVertical}
\begin{pmatrix}
1&0\\
-\mu&1
\end{pmatrix}.
\end{equation}

So let us return to Veech groups. Assume that $\vartheta$ is a Jenkins-Strebel direction for $\omega$; then we get a complete cylinder decomposition, say into $C_1,\ldots ,C_n$. For simplicity we assume that $\vartheta =0$; else we replace $\omega$ by $\exp (-\mathrm{i}\vartheta )\omega$. On every $C_k$ we have the affine Dehn twist $t_k$. If all the cylinders have the same moduli, all these Dehn twists have the same linear part and hence glue to an affine map $X\to X$. But also in the case where the moduli are commensurable (i.e. where they span a one-dimensional vector space over $\bbq$ or, equivalently, where they have some common integer multiple) we can do this trick: let the modulus of $C_k$ be $\mu_k$, then by assumption there are integers $m_1,\ldots ,m_n$ with
\begin{equation*}
m_1\mu_1=m_2\mu_2=\cdots =m_n\mu_n=\text{(say) }\mu .
\end{equation*}
Then on each $C_k$, the map $t_k^{m_k}$ has affine part
\begin{equation}\label{TrivialComputationsWithUnipotentMatrices}
\begin{pmatrix}
1&\mu_k\\
0&1
\end{pmatrix}^{m_k}=
\begin{pmatrix}
1&m_k\mu_k\\
0&1
\end{pmatrix}=
\begin{pmatrix}
1&\mu\\
0&1
\end{pmatrix}.
\end{equation}
Hence the maps $t_1^{m_1},\ldots ,t_n^{m_n}$ glue to an affine self-map of $(X,\omega )$, with affine part (\ref{TrivialComputationsWithUnipotentMatrices}). We summarize:
\begin{Proposition}
Assume that $0$ is a Jenkins-Strebel direction for $\omega$, and assume that the cylinders in the associated complete cylinder decomposition have commensurable moduli, where $\mu$ is a common integer multiple of these moduli. Then the matrix
\begin{equation}\label{ParabolicMatrixII}
\begin{pmatrix}
1&\mu\\
0&1
\end{pmatrix}
\end{equation}
is an element of $\SL (X,\omega )$.\hfill $\square$
\end{Proposition}
The proposition remains of course true if zero is replaced by an arbitrary $\vartheta\in\bbr /2\pi\bbz$ and the matrix (\ref{ParabolicMatrixII}) is replaced by its conjugate (\ref{ParabolicConjugate}).

So if one finds two different Jenkins-Strebel directions\footnote{meaning that they do not merely differ by a sign}, one finds two distinct parabolic elements of $\SL (X,\omega )$ with different eigenvectors, which means that $\SL (X,\omega )$ is quite large. Let us apply this method in order to find nontrivial elements in the Veech groups of our examples.

\begin{RegularPolygonsI}
First consider the Wiman curve $(W_g,\omega_g)$. Recall that this can be described as a regular $2n$-gon with opposite sides identified, where $n=2g+1$. Rescaling the $2n$-gon does not change the Veech group. We have shown that, after a suitable normalization, the vertical direction is a Jenkins-Strebel direction. Using the notation from above, let $\mu_k$ be the modulus of the cylinder $C_k$. Then using the trigonometrical formul\ae\footnote{These follow directly from the addition theorems of sine and cosine}
\begin{equation*}
\sin (\alpha +\beta )+\sin (\alpha -\beta )=2\sin\alpha\cos\beta ,
\end{equation*}
\begin{equation*}
\cos (\alpha -\beta )-\cos (\alpha +\beta )=2\sin\alpha\sin\beta
\end{equation*}
and
\begin{equation*}
\cot \alpha =\tan\left(\frac{\pi}{2}-\alpha\right) ,
\end{equation*}
we get
\begin{equation*}
\mu_0=\frac{2\sin\dfrac{g\pi}{n}}{2\cos\dfrac{g\pi}{n}}=\tan\frac{g\pi}{n}=\cot\left(\frac{\pi}{2}-\frac{g\pi}{n}\right)=\cot\frac{\pi}{2n},
\end{equation*}
and
\begin{equation*}
\mu_k=\frac{2\left(\sin\dfrac{(k-1)\pi}{n}+\sin\dfrac{k\pi}{n}\right)}{\cos\dfrac{(k-1)\pi}{n}-\cos\dfrac{k\pi}{n}}
=\frac{4\sin\dfrac{(2k-1)\pi}{2n}\cos\dfrac{\pi}{2n}}{2\sin\dfrac{(2k-1)\pi}{2n}\sin\dfrac{\pi}{2n}}
=2\cot\frac{\pi}{2n}.
\end{equation*}
Hence we see that $\SL (X,\omega )$ contains the element
\begin{equation*}
\begin{pmatrix}
1&0\\
2\cot\dfrac{\pi}{2n}&1
\end{pmatrix}.
\end{equation*}
Also rotation around the center of the regular $2n$-gon by an angle of $\frac{\pi}{n}$ clearly defines an affine self-map of $X$, so we get that the element
\begin{equation*}
\begin{pmatrix}
\cos\frac{\pi}{n}&-\sin\frac{\pi}{n}\\
\sin\frac{\pi}{n}&\cos\frac{\pi}{n}
\end{pmatrix}
\end{equation*}
is contained in $\SL (X,\omega )$. These two matrices already generate a finite covolume Fuchsian group, which must then be a finite index subgroup of $\SL (X,\omega )$, see section 3.5.
\end{RegularPolygonsI}
\begin{RegularPolygonsII}
Next we turn to $(W_g,\omega_1)$, i.e. two regular $n$-gons with parallel sides identified, again with $n=2g+1$. With the notation from section 2.4, we find that the modulus of the cylinder $C_k$ is
\begin{equation*}
\frac{2\left(\sin\dfrac{2(k-1)\pi}{n}+\sin\dfrac{2k\pi}{n}\right) }{\cos\dfrac{2(k-1)\pi}{n}-\cos\dfrac{2k\pi }{n}}
=\frac{4\sin\dfrac{(2k-1)\pi}{n}\cos\dfrac{\pi}{n}}{2\sin\dfrac{(2k-1)\pi}{n}\sin\dfrac{\pi}{n}}
=2\cot\frac{\pi}{n}.
\end{equation*}
Hence the matrix
\begin{equation*}
\begin{pmatrix}
1&0\\
2\cot\frac{\pi}{n}&1\\
\end{pmatrix}
\end{equation*}
is in the Veech group, as well as
\begin{equation*}
\begin{pmatrix}
\cos\frac{2\pi}{n}&-\sin\frac{2\pi}{n}\\
\sin\frac{2\pi}{n}&\cos\frac{2\pi}{n}
\end{pmatrix}
\end{equation*}
coming from an obvious rotation of the $n$-gons. Again we will see in section 3.5 that these two generate a finite covolume Fuchsian group.
\end{RegularPolygonsII}
\begin{AffineMapsasQCMaps}
Let us now consider affine maps between different translation surfaces. These provide important examples of quasiconformal maps. In the previous section we noted that every matrix $A\in\SL_2(\bbr )$, interpreted as a self-map of $\bbc =\bbr^2$, is quasiconformal with constant Beltrami coefficient, which we computed in that section. Hence we can also compute its quasiconformal dilatation as
\begin{equation*}
K(A)=\frac{1+k(A)}{1-k(A)}=\frac{1+|\mu_0(A)|}{1-|\mu_0(a)|}.
\end{equation*}
As a corollary we get:
\end{AffineMapsasQCMaps}
\begin{Proposition}
Let $X_1$, $X_2$ be closed Riemann surfaces, let $\omega_i$ be an abelian differential on $X_i$ and let $f:(X_1,\omega_1)\to (X_2,\omega_2)$ be an affine map with linear part $A\in\SL_2(\bbr )$. Then $f$ is quasiconformal with dilatation $K(A)$. \hfill $\square $
\end{Proposition}
Affine maps have minimal quasiconformal dilatation in their homotopy class:
\begin{Theorem}\label{AffineMapsHaveMinimalDilatation}
Let $X_1$, $X_2$ be closed Riemann surfaces, let $\omega_i$ be a quadratic differential on $X_i$ and let $f:(X_1,\omega_1)\to (X_2,\omega_2)$ be an affine map with derivative $A\in\SL_2(\bbr )$ and hence quasiconformal dilatation $K(A)$. Let $g:X_1\to X_2$ be any quasiconformal homeomorphism homotopic to $f$, with dilatation $K(g)$. Then $K(g)\ge K(A)$, with equality precisely for $g=f$.
\end{Theorem}
\begin{proof}
In the case where $A$ is diagonal, this statement is a special case of \textit{Teichm\"{u}ller's theorem}.\footnote{Teichm\"{u}ller's theorem is the analogous statement for quadratic differentials} There are many proofs of this theorem in the literature, see e.g. \cite[Thm. 5.3.8]{Hubbard06}, hence we only show how to reduce the proof to this special case.

Write $A$ in the following form:
\begin{equation*}
A=U_2\cdot D\cdot U_1
\end{equation*}
where $U_i\in\SO (2)$ and $D$ is diagonal. We can then find real numbers $\vartheta_1$ and $\vartheta_2$ such that
\begin{equation*}
U_i=
\begin{pmatrix}
\cos\vartheta_i &-\sin\vartheta_i\\
\sin\vartheta_i &\cos\vartheta_i
\end{pmatrix}.
\end{equation*}
Now we use the following elementary observations:
\begin{enumerate}
\item Let $B\in\SL_2(\bbr )$, and let $(X,\omega )$ be a Riemann surface with an abelian differential. Then the identity map $(X,\omega )\to B\cdot (X,\omega )$ has derivative $B$, and the identity map $B\cdot (X,\omega )\to (X,\omega )$ has derivative $B^{-1}$.
\item Let
\begin{equation*}
U=
\begin{pmatrix}
\cos\vartheta &-\sin\vartheta\\
\sin\vartheta &\cos\vartheta
\end{pmatrix},
\end{equation*}
then $U\cdot (X,\omega )=(X,\mathrm{e}^{-\mathrm{i}\vartheta }\omega )$.
\end{enumerate}
Applying this to our situation, consider the following composition:
\begin{equation*}
U_1\cdot (X_1,\omega_1)\overset{\mathrm{id}_{X_1}}{\to }(X_1,\omega_1)\overset{f}{\to }(X_2,\omega_2)\overset{\mathrm{id}_{X_2}}{\to} U_2^{-1}\cdot (X_2,\omega_2).
\end{equation*}
This is an affine map with derivative $U_2^{-1}AU_1=D$ from $(X_1,\mathrm{e}^{-\mathrm{i}\vartheta_1}\omega_1)$ to $(X_2,\mathrm{e}^{\mathrm{i}\vartheta_2}\omega_2)$, hence we can apply Teichm\"{u}ller's theorem to it. But now forgetting the abelian differentials, this is the \textit{same} quasiconformal map from $X_1$ to $X_2$ as $f$, hence the minimality statement transfers to $f$.
\end{proof}
\begin{TheImageInMCG}
The affine group of a translation surface can also be interpreted as a subgroup of the mapping class group. Namely if $X$ is a Teichm\"{u}ller marked surface, then every orientation-preserving self-diffeomorphism of $X$ can be pulled back along the Teichm\"{u}ller marking to give a self-diffeomorphism of $\Sigma_g$, which is only well-defined up to isotopy. Hence a group homomorphism $\Aff^+(X,\omega )\to\Mod_g$.
\end{TheImageInMCG}
\begin{Proposition}\label{AffineGroupInjectsIntoModg}
Let $X$ be a Teichm\"{u}ller marked closed Riemann surface of genus $g$, and let $\omega$ be an abelian differential on it.
Then the map $\Aff^+(X,\omega )\to\Mod_g$ is injective.
\end{Proposition}
\begin{proof}
Let $f$ be an affine self-map of $(X,\omega )$ which maps to the identity in $\Mod_g$, in other words, which is isotopic to the identity. Now the identity is conformal, hence its quasiconformal dilatation is minimal in its isotopy class. By Theorem \ref{AffineMapsHaveMinimalDilatation} this is only possible if $f$ \textit{is} the identity.
\end{proof}

\subsection{Abelian Teichm\"{u}ller Disks}

Now let $X$ be a compact Riemann surface of genus $g$ together with a Teichm\"{u}ller marking and an abelian differential $\omega$. In this section we define an embedding of the hyperbolic plane into $\mathrsfs{T}_g$ associated with these data; it has different descriptions, using the different models of the hyperbolic plane discussed in the previous section.

To begin with, the marking allows us to identify $X$ with ``the'' model surface $\Sigma_g$ of genus $g$, and hence to identify $\mathrsfs{T}_g$ with the Teichm\"{u}ller space of $X$; so a point in $\mathrsfs{T}_g$ is represented by a compact Riemann surface $Y$ of genus $g$ together with an isotopy class of diffeomorphisms $X\to Y$.

\begin{TDisksViaSLtwoR}
The $\SL_2(\bbr )$-action on $\Omega\mathrsfs{T}_g$ defines a smooth map
\begin{equation*}
\SL_2(\bbr )\to\Omega\mathrsfs{T}_g,\quad A\mapsto A\cdot (X,\omega );
\end{equation*}
composing this with the forgetful map $\Omega\mathrsfs{T}_g\to\mathrsfs{T}_g$ defines a map
\begin{equation}\label{DefinitionOfGZeroOmega}
g_0^{\omega}:\SL_2(\bbr )\to\mathrsfs{T}_g .
\end{equation}
We shall see in a minute that this factors over the left quotient by $\SO (2)$. But from this construction it is not clear at all that we arrive at something complex analytic. So let us first introduce another way to look at this map.
\end{TDisksViaSLtwoR}
\begin{TDisksViaBeltramiForms}
The quotient $\overline{\omega }/\omega$ is a well-defined Beltrami form on $X$ of norm
one. Hence for any $t\in\Delta$ the Beltrami form $t\overline{\omega }/\omega$ will be an element of $\mathcal{B}^{<1}(X)$. So using the map
(\ref{BeltramiToTeichmueller}) we obtain a holomorphic map
	\begin{equation}\label{TDiskViaBeltrami}
	f^{\omega }:\Delta\to\mathcal{B}^{<1}(X)\overset{\Phi}{\to}\mathrsfs{T}_g,\quad t\mapsto\Phi\left( t\frac{\overline{\omega }}{\omega }\right)
=[X\overset{\mathrm{id}}{\to} X_{t\overline{\omega }/\omega }].
	\end{equation}
\end{TDisksViaBeltramiForms}
\begin{Comparison}
These two constructions give the same image in Teichm\"{u}ller space:
\end{Comparison}
\begin{Proposition}
Let $X$ be a marked compact Riemann surface of genus $g$ and let $\omega$ be a nonzero holomorphic one-form on $X$. Let $f^{\omega }:\Delta\to\mathrsfs{T}_g$ be the map defined in (\ref{TDiskViaBeltrami}), let $g_0^{\omega }$ as in (\ref{DefinitionOfGZeroOmega}) and let $\mu_0 :\SO (2)\backslash\SL_2(\bbr )$ be as in the previous section. Then the following diagram commutes:
\begin{equation*}
\xymatrix{
\SL_2(\bbr ) \ar[dd]_{\mu_0 }\ar[rd]^{g_0^{\omega }}\\
&\mathrsfs{T}_g\\
\Delta \ar[ru]_{f^{\omega }}
}
\end{equation*}
\end{Proposition}
\begin{proof}
Let $A\in\SL_2(\bbr )$. Then $g^{\omega }$ sends $A$ to the marked Riemann surface underlying $A\cdot (X,\omega )$, which is the surface $X$ with a new complex structure. For the moment we denote it by $AX$ (although this notation is in general dangerous since it suppresses the dependence on $\omega$). On the other hand, $f^{\omega }\circ\mu_0$ sends $A$ to the marked Riemann surface $X_{\mu_A\cdot \overline{\omega}/\omega}$ which is again $X$ with a new complex structure. Hence we will be done if we show that the identity map $AX\to X_{\mu_A\cdot \overline{\omega }/\omega}$ is conformal.

We first show the corresponding local statement. So consider the complex plane $\bbc$ with two new complex structures:
\begin{itemize}
\item Consider the holomorphic one-form $\de z$ and the associated translation structure. Distort this by $A$ to obtain a new translation structure and hence a new complex structure on $\bbc$ which we denote, in accordance with our general notation, as $A\bbc$. By definition the map $A: A\bbc\to\bbc$ is a translation, in particular holomorphic, chart (in fact it is an isomorphism of Riemann surfaces).
\item Now consider the Beltrami form $\mu =\mu_A\cdot\de\overline{z}/\de z$ and associated new complex structure on $\bbc$; we denote $\bbc$ endowed with this new complex structure by $\bbc_{\mu }$. Now the map $A:\bbc_{\mu }\to\bbc$ satisfies the Beltrami equation
\begin{equation*}
\frac{\partial A}{\partial\overline{z}}=\mu_A\cdot\frac{\partial A}{\partial z}
\end{equation*}
and is hence a holomorphic chart (again it is even an isomorphism of Riemann surfaces).
\end{itemize}
We wish to show that the identity map $A\bbc\to\bbc_{\mu }$ is a holomorphic map: checking this in holomorphic charts boils down to the trivial statement that the diagram
\begin{equation*}
\xymatrix{
A\bbc\ar[r]^{\mathrm{id}} \ar[d]_A
& \bbc_{\mu } \ar[d]^A \\
\bbc \ar[r]_{\mathrm{id}} &\bbc
}
\end{equation*}
commutes.

Now on $X'=X\smallsetminus Z(\omega)$ we have natural flat coordinates determined by $\omega$. In these coordinates $\omega$ becomes $\de z$ and $\overline{\omega }/\omega$ consequently becomes $\de\overline{z}/\de z$. Thus in these coordinates the inclusion $AX'\to X_{\mu_A\cdot \overline{\omega }/\omega}$ becomes the identity map from (an open subset of) $A\bbc$ to (an open subset of) $\bbc_{\mu }$ which is, as just shown, holomorphic. But by Riemann's theorem on removable singularities, the whole identity map $AX\to X_{\mu_A\cdot \overline{\omega }/\omega}$ must then be holomorphic and thus, since it is a homeomorphism, biholomorphic.
\end{proof}

This has several interesting implications. For example the map $g_0^{\omega }$ factors over $\mu_0$, but this means nothing else than that it is constant on left $\SO (2)$-cosets. In other words, it descends to a map $g^{\omega }:\SO (2)\backslash\SL_2(\bbr )\to\mathrsfs{T}_g$, and the following diagram commutes:
\begin{equation*}
\xymatrix{
\SO (2)\backslash\SL_2(\bbr ) \ar[dd]_{\mu }^{\simeq }\ar[rd]^{g^{\omega }}\\
&\mathrsfs{T}_g\\
\Delta \ar[ru]_{f^{\omega }}
}
\end{equation*}
Also the two maps $f^{\omega }\circ C$ and $g^{\omega }\circ\lambda^{-1}$ from $\bbh$ to $\mathrsfs{T}_g$ agree, so we may denote either of them by $h^{\omega }$; this $h^{\omega }$ is then a holomorphic map. It will turn out that this is the easiest one to understand.
\begin{KobayashiIsometries}
The maps constructed above give Kobayashi isometric holomorphic embeddings of $\Delta$ into Teichm\"{u}ller spaces.
\end{KobayashiIsometries}
\begin{Theorem} Let $X$ be a marked compact Riemann surface of genus $g$, and let $\omega$ be an abelian differential on $X$. Then the map $f^{\omega}:\Delta\to\mathrsfs{T}_g$ is a Teichm\"{u}ller disk, i.e. a holomorphic Kobayashi isometry.
\end{Theorem}
\begin{proof}
The map is holomorphic by construction; to prove the other statement recall that the Kobayashi metric on the unit disk is equal to the Poincar\'{e} metric, and the Kobayashi metric on Teichm\"{u}ller space is equal to the Teichm\"{u}ller metric. Hence we have explicit formul\ae\ for these metrics.

Next, note that if $(X,\omega )$ and $(X',\omega ')$ belong to the same $\PSL_2(\bbr )$-orbit, then $f^{\omega }$ is a Kobayashi isometry if and only if $f^{\omega '}$ is a Kobayashi isometry. This is because there is some automorphism $\varphi $ of $\Delta$ such that $f^{\omega '}=f^{\omega }\circ\varphi$. Vice versa for every automorphism $\varphi$ of $\Delta$ there is some pair $(X',\omega ')$ which is $\SL_2(\bbr )$-equivalent to $(X,\omega )$ and such that $f^{\omega }\circ\varphi =f^{\omega '}$. From this we see that it suffices to show the following:

\textit{Let $0<t<1$. Then the hyperbolic distance from $0$ to $t$ is the same as the Teichm\"{u}ller distance from $f^{\omega }(0)$ to $f^{\omega }(t)$.} But both distances can be directly computed: the hyberbolic distance from $0$ to $t$ is
\begin{equation}\label{HyberbolicDistanceExplicit}
\frac{1}{2}\log\frac{1+t}{1-t}.
\end{equation}
To compute the Teichm\"{u}ller distance, note that the identity map $X\to X_{t\overline{\omega }/\omega }$ is an affine map with quasiconformal constant $t$, hence quasiconformal dilatation $(1+t)/(1-t)$. So by Theorem \ref{AffineMapsHaveMinimalDilatation} we get that the Teichm\"{u}ller distance between $f^{\omega }(0)$ and $f^{\omega }(t)$ is also given by (\ref{HyberbolicDistanceExplicit}).
\end{proof}
\begin{Definition}
An \textup{abelian Teichm\"{u}ller disk} is a map of the form $f^{\omega }:\Delta\to\mathrsfs{T}_g$ for some $(X,\omega )\in Q\mathrsfs{T}_g$. Sometimes we also call the image $f(\Delta )\subseteq\mathrsfs{T}_g$ an abelian Teichm\"{u}ller disk.
\end{Definition}
Also note that pulling back $f^{\omega }$ or $g^{\omega }$ along any appropriate composition of maps in the commutative diagram (\ref{DiagramForGuisesOfDisk}) we also obtain maps from the other two incarnations of the hyperbolic plane into Teichm\"{u}ller space. In particular we obtain a map $h^{\omega }=g^{\omega }\circ\lambda^{-1}=f^{\omega }\circ C :\bbh\to\mathrsfs{T}_g$. This is a again holomorphic Kobayashi isometry; we will also call the maps $g^{\omega }$ and $h^{\omega}$ abelian Teichm\"{u}ller disks since there is no danger of confusion.

\subsection{The Mirror Veech Group}

Let again $X$ be a marked compact Riemann surface of genus $g$ and let $\omega$ be an abelian differential on $X$. We now take a closer look at the image of the Teichm\"{u}ller disk $f^{\omega }(\Delta )$ in moduli space. It turns out that this is closely connected with the Veech group.
\begin{TheAffineGroupActingontheTDisk}
Recall that every affine self-map $\varphi $ of $(X,\omega )$ induces an element $[\varphi ]$ of the mapping class group $\Mod_g$.
\end{TheAffineGroupActingontheTDisk}
\begin{Proposition}\label{ActionOfVeechGroupOnVariousThings}
Let $\varphi\in\Aff^+(X,\omega)$ with derivative $A=D\varphi \in\SL (X,\omega )\subset\SL_2(\bbr )$. Then $[\varphi ]\in\Mod_g$, acting on $\mathrsfs{T}_g$, sends the image of $f^{\omega }$ to itself. Hence pulling its action back along $f^{\omega }$, $g^{\omega }$ and our system of compatible homeomorphisms $\iota$, $\sigma$, $\lambda$, $\mu$ and $C$, we obtain automorphisms of $\Delta$, $\bbh$ and the two quotients of $\SL_2(\bbr )$ by $\SO (2)$. These are as follows:
\begin{enumerate}
\item On $\SO (2)\backslash\SL_2(\bbr )$ it acts by right multiplication with $A^{-1}$.
\item On $\SL_2(\bbr )/\SO (2)$ it acts by left multiplication with
    \begin{equation}\label{FConjugate}
    \begin{pmatrix}
    -1&0\\
    0&1
    \end{pmatrix}
    A
    \begin{pmatrix}
    -1&0\\
    0&1
    \end{pmatrix}.
    \end{equation}
\item On $\bbh$ it acts by the matrix (\ref{FConjugate}) interpreted as a M\"{o}bius transformation.
\item On $\Delta$ it acts by the matrix
    \begin{equation}
    \begin{pmatrix}
    1&\mathrm{i}\\
    -1&\mathrm{i}
    \end{pmatrix}
    A
    \begin{pmatrix}
    1&\mathrm{i}\\
    -1&\mathrm{i}
    \end{pmatrix}^{-1},
\end{equation}
again interpreted as M\"{o}bius transformation.

\end{enumerate}
\end{Proposition}
\begin{proof}
First we show that $[\varphi ]$ sends the Teichm\"{u}ller disk to itself. For this purpose we view the Teichm\"{u}ller disk as the image of the map $g_0^{\omega }:\SL_2(\bbr )\to\mathrsfs{T}_g$. Then for any $B\in\SL_2(\bbr )$ we have
\begin{equation*}
\begin{split}
[\varphi ](g_0^{\omega }(B))
&=[\varphi ][X\overset{\mathrm{id}}{\to}B\cdot (X,\omega )]\\
&=[X\overset{\varphi^{-1}}{\to}B\cdot (X,\omega )]\\
&=[X\overset{\varphi^{-1}}{\to}B\cdot (X,\omega )\overset{\varphi }{\to }BA^{-1}\cdot (X,\omega )]\\
&=[X\overset{\mathrm{id}}{\to }BA^{-1}\cdot (X,\omega )]\\
&=g_0^{\omega }(BA^{-1}).
\end{split}
\end{equation*}
The third equality is justified by the map
\begin{equation*}
\varphi :B\cdot (X,\omega )\to BA^{-1}\cdot (X,\omega )
\end{equation*}
being conformal, in fact an isomorphism of translation surfaces, because
\begin{equation*}
\varphi : (X,\omega )\to A^{-1}\cdot (X,\omega )
\end{equation*}
is an isomorphism of translation surfaces, as remarked earlier.

By this computation we have at the same time shown (i). Concerning (ii), the action of $[\varphi ]$ on $\SL_2(\bbr )/\SO (2)$ is the pullback of the action (i) along the homeomorphism $\iota$. Hence:
\begin{equation*}
\begin{split}
[\varphi ](A\cdot\SO (2))
&=\iota ([\varphi ](\iota^{-1}(A\cdot\SO (2))))\\
&=\iota ([\varphi ](\SO (2)\cdot\iota_0(A)))\\
&=\iota (\SO (2)\cdot \iota_0(A)F^{-1})\\
&=\iota_0(\iota_0(A)F^{-1})\cdot\SO (2)\\
&=\iota_0(F^{-1})A\cdot\SO (2)\\
&=
\begin{pmatrix}
-1&0\\
0&1
\end{pmatrix}
F
\begin{pmatrix}
-1&0\\
0&1
\end{pmatrix}
\cdot A\cdot\SO (2).
\end{split}
\end{equation*}
Now (iii) easily follows from (ii), and (iv) follows by a straightforward computation, which we omit, from (iii).
\end{proof}
\begin{Proposition}\label{InclusionOfAffineGroupIntoMCG}
Consider once again the inclusion $\Aff^+(X,\omega )\to\Mod_g$. Its image is the setwise stabilizer of the Teichm\"{u}ller disk $f^{\omega }(\Delta )$, and the image of the subgroup $\Aut (X,\omega )$ is contained in the pointwise stabilizer of the Teichm\"{u}ller disk.
\end{Proposition}
\begin{proof}
As to the first statement, it follows from the proposition above that the image of $\Aff^+(X,\omega )$ is contained in the stabilizer. The other inclusion is proved in \cite[Thm. 1]{EarleGardiner97}.

Now an automorphism of the pair $(X,\omega )$ has the identity as derivative, so it acts trivially on the Teichm\"{u}ller disk.
\end{proof}
\begin{TheMirrorVG}
Now recall the short exact sequence (\ref{SESforVeechGroups}). This can be re-interpreted in our context as follows: by Proposition \ref{InclusionOfAffineGroupIntoMCG}, the quotient $\Aff^+(X,\omega )/\Aut (X,\omega )$ acts on the Teichm\"{u}ller disk $f^{\omega }(\Delta )$; the short exact sequence identifies this quotient with the Veech group $\SL (X,\omega )$. Pulling back along $h^{\omega }:\bbh\to\mathrsfs{T}_g$ we obtain an action of $\SL (X,\omega )$ on the upper half plane. This is, up to a conjugation by $\operatorname{diag }(-1,1)$, the usual action by M\"{o}bius transformations. Hence we define
\begin{equation*}
\Gamma (X,\omega)=
\begin{pmatrix}
-1&0\\
0&1
\end{pmatrix}
\cdot\SL (X,\omega )\cdot
\begin{pmatrix}
-1&0\\
0&1
\end{pmatrix}
\subset\SL_2(\bbr ).
\end{equation*}
This is a non-cocompact Fuchsian group by Proposition \ref{VeechGroupIsFuchsianNonCocompact}, called the \textit{mirror Veech group} of $(X,\omega )$. We will sometimes also consider its image in $\PSL_2(\bbr )$, which we denote by $\mathrm{P}\Gamma (X,\omega )$ and call the \textit{projective mirror Veech group} of $(X,\omega )$.
\end{TheMirrorVG}
\begin{Corollary} Consider the composition
\begin{equation*}
\bbh\overset{h^{\omega }}{\to }\mathrsfs{T}_g\to \mathrsfs{M}_g;
\end{equation*}
this factors over the quotient of $\bbh $ by the Fuchsian group
$\Gamma (X,\omega )$, giving a holomorphic immersion $\Gamma (X,\omega )\backslash\bbh\to\mathrsfs{M}_g.$\hfill $\square$
\end{Corollary}

Let us make some remarks on how the Veech group changes under some simple operations on the translation surface.
\begin{Proposition}
Let $(X,\omega )$ and $(X',\omega ')$ be equivalent under the $\GL_2^+(\bbr )$-operation on $\Omega\mathrsfs{M}_g$. Then the groups $\SL (X,\omega )$ and $\SL (X',\omega ')$ are conjugate within $\SL_2(\bbr )$. Hence also the groups $\Gamma (X,\omega )$ and $\Gamma (X',\omega ')$ are conjugate within $\SL_2(\bbr )$.\hfill $\square$
\end{Proposition}
\begin{TheImageInMS}
We are primarily interested in the case where this group is as big as possible:
\end{TheImageInMS}
\begin{Definition}
Let $X$ be a closed Riemann surface and let $\omega$ be an abelian differential on $X$. Then $(X,\omega )$ is called a \textup{Veech surface} if the group $\Gamma (X,\omega )\subset\PSL_2(\bbr )$ (or, equivalently, $\PSL (X,\omega )\subset\PSL_2(\bbr )$) is a lattice.
\end{Definition}
We recall that a Fuchsian group $\Gamma$ is a lattice if it has finite covolume or, equivalently, if the quotient surface $\Gamma\backslash\bbh$ is of finite type, hence an algebraic curve.
\begin{Theorem}
Let $X$ be a closed Riemann surface and let $\omega$ be an abelian differential. Denote the image of $f^{\omega }(\Delta )\subset\mathrsfs{T}_g$ in $\mathrsfs{M}_g$ by $C$. Then the following are equivalent:
\begin{enumerate}
\item $(X,\omega )$ is a Veech surface;
\item $C\subset\mathrsfs{M}_g$ is an algebraic subvariety;
\item there exists a complete algebraic curve $\overline{C}\subset\overline{\mathrsfs{M}_g}$, where $\overline{\mathrsfs{M}_g}$ is the Deligne-Mumford compactification of $\mathrsfs{M}_g$, such that $\overline{C}\cap\partial\mathrsfs{M}_g$ is finite and $C=\overline{C}\cap\mathrsfs{M}_g$.
\end{enumerate}
\end{Theorem}
\begin{proof}
The implications (iii) $\Rightarrow$ (ii) $\Rightarrow$ (i) are clear. It is shown in
 \cite[section 4.2]{HerrlichSchmithuesen07}
that (i) implies the ``analytic version'' of (iii), i.e. (iii) with ``complete algebraic curve'' replaced by ``complete analytic subspace''. But using the GAGA theorem \cite{Serre56}, one sees that this is equivalent to (iii).
\end{proof}
Note that this Theorem justifies the name ``Teichm\"{u}ller curve''.
\begin{Theorem}
Let $(X,\omega )$ be a Veech surface, then the algebraic curve $C$ and its embedding into $\mathrsfs{M}_g$ are defined over $\overline{\bbq }$. Furthermore, for any automorphism $\sigma\in\operatorname{Gal}(\overline{\bbq }|\bbq )$, the curve $C^{\sigma }$ with the embedding $C^{\sigma }\to\mathrsfs{M}_g^{\sigma }=\mathrsfs{M}_g$ is also a Teichm\"{u}ller curve.
\end{Theorem}
\begin{proof}
The first part was first shown by McMullen in \cite{McMullen09} by a rigidity argument; the second part was shown under some extra condition by M\"{o}ller in \cite{Moeller05b}. We give a combination of both proofs which is, as we think, simpler, but still relies on a result in \cite{Moeller05b} that we will not discuss in the present work, in order to keep it shorter.

First we cite Proposition 2 in \cite{McMullen09}; actually a weakening which says that there exist at most countably many Teichm\"{u}ller curves in each $\mathrsfs{M}_g$. The proof is simple: each Teichm\"{u}ller curve is, by analytic continuation, completely determined by giving one closed (real) geodesic $\gamma\subset\mathrsfs{M}_g$ it contains. But these geodesics are in bijection with pseudo-Anosov elements of the mapping class group $\operatorname{Mod}_g$, which is clearly countable.

Then we use Theorem 5.3 in \cite{Moeller05b}: it tells us that we can detect whether a curve in $\mathrsfs{M}_g$ is a Teichm\"{u}ller curve by looking at its Higgs bundle. We do not need to know what that precisely means; the crucial point is that M\"{o}ller's criterion is purely algebraic-geometric, making no use of complex analysis (in the proof, of course, it does, but not in the formulation). So the property of being a Teichm\"{u}ller curve is invariant under the operation of $\Aut\bbc$.

Finally let us take a Teichm\"{u}ller curve $C\subset\mathrsfs{M}_g$ and consider it as a point in a suitable Hilbert scheme. Combining the two observations drawn so far, its orbit in this Hilbert scheme under $\Aut\bbc$ must be countable. But this actually means that the orbit must be \emph{finite}, and the curve defined over a number field.
\end{proof}

\begin{TheFamilyOverATC}
Let $X$ be a closed Riemann surface of genus $g\ge 2$ together with a Teichm\"{u}ller marking, and let $\omega$ be an abelian differential on $X$. We get a Teichm\"uller disk $h^{\omega }:\bbh\to\mathrsfs{T}_g$. By the universal property of Teichm\"{u}ller space, this corresponds to a family of algebraic curves $\tilde{f}:\tilde{\mathrsfs{X}}\to\bbh$. Topologically this family is trivial, i.e. can be interpreted as the constant family $\bbh\times X\to\bbh$, but the complex structure on $\tilde{\mathrsfs{X}}_{\tau }$ for $\tau$ in $\bbh$ is a new one, namely the complex structure underlying $A\cdot (X,\omega )$ for any $A\in\SL_2(\bbr )$ with $A\cdot\mathrm{i}=\tau$, where $A$ operates on $\bbh$ as in Proposition \ref{ActionOfVeechGroupOnVariousThings}.

It is also interesting to see what happens when going down to a quotient by a subgroup of the mirror Veech group. There is a finite index torsion-free subgroup $\Gamma$ of $\Gamma (X,\omega )$ such that $\bbh\to\mathrsfs{M}_g^{[3]}$ factors over $\Gamma\backslash\bbh$. This gives a family of curves $\mathrsfs{X}\to\Gamma\backslash\bbh$. This is particularly interesting when $(X,\omega )$ is a Veech surface. Then $\Gamma$ is also a lattice, so that the quotient $C=\Gamma\backslash\bbh$ becomes an algebraic curve. Since the classifying map of the family $\mathrsfs{X}\to C$ is algebraic, the whole family is algebraic and, by McMullen's result, even defined over a number field.

Let us now consider our examples.
\end{TheFamilyOverATC}
\begin{Origamis}
The Veech group, and hence also the mirror Veech group, of the unit torus $(\bbc /\bbz [\mathrm{i}],\de z)$ is $\SL_2(\bbz )$. The induced Teichm\"{u}ller curve is the classical isomorphism $\SL_2(\bbz )\backslash\bbh \to\mathrsfs{M}_1$. In particular any translation surface in genus one is a Veech surface. In a later section (see Theorem \ref{VeechGroupsCommensurable}) we will see how Veech groups behave under finite covers; this will entail that the Veech group of any Origami is a finite index subgroup of $\SL_2(\bbz )$. Hence every origami is a Veech surface.
\end{Origamis}
Before turning to the more complicated examples, a helpful observation in hyperbolic geometry:
\begin{Lemma}\label{LemmaInHyperbolicGeometry}
Let $0<\vartheta <\frac{\pi}{2}$, and let $\ell\subset\bbh$ be the unique hyperbolic geodesic with the following properties:
\begin{enumerate}
\item $\mathrm{i}\in \ell$;
\item $\ell$ is a euclidean semi-circle with centre $a>0$;
\item $\ell$ meets the imaginary axis with an angle of $\vartheta$.
\end{enumerate}
Let further $b\in\bbr$ the unique positive endpoint of $\ell$. Then $a=\cot\vartheta$ and $b=\cot\frac{\vartheta}{2}$.
\end{Lemma}
\begin{proof}
Consider the following sketch:
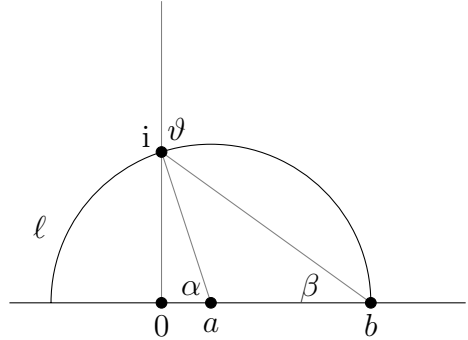
\begin{center}
\begin{tikzpicture}[scale=2]
\draw
    (1.3764,0) coordinate (B)
    (0.3249,0) coordinate (A)
    (0,1) coordinate (I)
    (0,0) coordinate (O);
\draw[gray] (0,0) -- (0,2)
    (A) -- (I)
    (B) -- (I);
\draw[black]
    (-1,0) -- (2,0)
    (1.3764,0) arc (0:180:1.0515cm);
\draw (-0.8,0.5) node {$\ell$}
    (A) ++(0,-0.15) node {$a$}
    (B) ++(0,-0.15) node {$b$}
    (O) ++(0,-0.15) node {$0$}
    (I) ++(-0.1,0.1) node {$\mathrm{i}$}
    (A) ++(-0.13,0.1) node {$\alpha$}
    (B) ++(-0.4,0.1) node {$\beta$}
    (I) ++(0.1,0.15) node {$\vartheta$};
\filldraw[black]
    (A) circle (1pt)
    (B) circle (1pt)
    (I) circle (1pt)
    (O) circle (1pt);
\end{tikzpicture}
\end{center}
Then $\alpha=\vartheta$, since both of them add up with the angle $\angle (a,\mathrm{i},0)$ to $\frac{\pi}{2}$. By the central angle theorem we then get $\beta=\alpha /2=\vartheta /2$. Hence by definition of the cotangent, the lemma follows.
\end{proof}
\begin{RegularPolygonsI}
Let us compute the (projective) mirror Veech group $\mathrm{P}\Gamma (W_g,\omega_g)$, i.e. of a regular $2n$-gon with opposite sides identified, where $n=2g+1$ is odd. In section 3.3 we have identified two elements of $\SL (W_g,\omega_g)$; conjugating them by $\operatorname{diag}(-1,1)$ gives the following two elements of the mirror Veech group:
\begin{equation}\label{GeneratorsOfTriangleGroupTwoNInfty}
\begin{pmatrix}
1&0\\
-2\cot\frac{\pi}{2n}&1
\end{pmatrix}
\text{ and }
\begin{pmatrix}
\cos\frac{\pi}{n}&\sin\frac{\pi}{n}\\
-\sin\frac{\pi}{n}&\cos\frac{\pi}{n}
\end{pmatrix}.
\end{equation}
\begin{Proposition}
The image of the two matrices (\ref{GeneratorsOfTriangleGroupTwoNInfty}) in $\PSL_2(\bbr )=\operatorname{Isom}\bbh$ generate a triangle group of signature $(n,\infty ,\infty )$. In particular the projective mirror Veech group $\mathrm{P}\Gamma (W_g,\omega_1)$ contains (a conjugate of) $\Delta (n,\infty ,\infty )$ as a finite index subgroup.
\end{Proposition}
\begin{proof}
The triangle group $\Delta (n,\infty ,\infty )$ is well-defined up to conjugation in $\PSL_2(\bbr )$. We shall give generators for a specific incarnation and show that they are simultaneously conjugate to the matrices (\ref{GeneratorsOfTriangleGroupTwoNInfty}). Let $H$ be the unique horocycle which hits $\mathrm{i}\in\bbh$ with angle $\pi /n$, and let $b$ be the endpoint of $H$ which, considered as a real number, is positive. Then the ideal triangle with vertices $\mathrm{i}$, $b$ and $\infty$ has internal angles $\pi /n$ at $\mathrm{i}$, $0$ at $b$ and $0$ at $\infty$. So we get an incarnation of $\Delta (2,n,\infty )$ which is generated by the rotation around $\mathrm{i}$ with angle $2\pi /n$ and the translation of twice the width of this triangle to the right. But the width of this triangle is $\cot (\pi /2n)$, as one sees from Lemma \ref{LemmaInHyperbolicGeometry}. These two hyperbolic isometries are the M\"{o}bius transformations with matrices
\begin{equation*}
\begin{pmatrix}
\cos\frac{\pi }{n}&-\sin\frac{\pi }{n}\\
\sin\frac{\pi }{n}&\cos\frac{\pi }{n}
\end{pmatrix}
\text{ and }
\begin{pmatrix}
1&2\cot\frac{\pi }{2n}\\
0&1
\end{pmatrix}.
\end{equation*}
Conjugating both by
\begin{equation*}
\begin{pmatrix}
0&-1\\
1&0
\end{pmatrix},
\end{equation*}
one obtains exactly the inverses of the matrices (\ref{GeneratorsOfTriangleGroupTwoNInfty}). Hence the projective mirror Veech group of $(W_g,\omega_g)$ contains a copy of $\Delta (n,\infty ,\infty )$. Since this triangle group has finite covolume, the index must be finite.
\end{proof}
So we see that $(W_g,\omega_g)$ is a Veech surface, which gives rise to a Teichm\"{u}ller curve
\begin{equation*}
\Delta (n,\infty ,\infty )\backslash\bbh\simeq\bbc^{\times }\to\mathrsfs{M}_g.
\end{equation*}
This is no contradiction to our earlier remark that one only has isotrivial families of curves over $\bbc^{\times }$, since triangle groups always have fixed points. Hence if one wishes to construct a family of curves out of this map, one has to pass to a finite index subgroup of $\Delta (n,\infty, \infty )$.
\end{RegularPolygonsI}
\begin{RegularPolygonsII}
A very similar game can be played with the surfaces $(W_g,\omega_1)$, i.e. with two $n$-gons with corresponding opposite sides identified.
\begin{Proposition}
The projective mirror Veech group $\mathrm{P}\Gamma (W_g,\omega_1)$ contains a conjugate of the triangle group $\Delta (2,n,\infty )$ as a finite index subgroup.
\end{Proposition}
\begin{proof}
As above, we construct explicit generators of $\Delta (2,n,\infty )$. Take the unique horocycle $H\subset\bbh$ which contains $\mathrm{i}$, has positive euclidean centre and meets the imaginary axis with angle $\pi /n$. Let $p$ be the point in $H$ with maximal imaginary part. Then $\mathrm{i}$, $p$ and $\infty$ are the vertices of an ideal triangle with angles $\pi /n$, $\pi/2$ and $0$ (in that order). Hence we get a copy of the triangle group which is generated by the rotation around $\mathrm{i}$ with angle $\pi /n$, and the translation by twice the width of the triangle to the right. The matrix of the rotation is as above, but the matrix of the translation now has (by Lemma (\ref{LemmaInHyperbolicGeometry}) $2\cot\pi /n$ instead of $2\cot\pi /2n$. Conjugating with the same matrix as above, we can identify $\Delta (2,n,\infty )$ with a subgroup of $\mathrm{P}\Gamma (W_g,\omega_1)$. Again, this triangle group has finite covolume, so the index must be finite.
\end{proof}
Since for our constructions we will often have to pass to finite index subgroups of a given Veech group, this information suffices for our purposes. Yet we note that the Veech groups just discussed have been determined exactly (and not only their commensurability classes). There is no Fuchsian group properly containing $\Delta (2,n,\infty )$, see \cite[Thm. 3B]{Greenberg63}, so we really have
\begin{equation*}
\mathrm{P}\Gamma (W_g,\omega_1)=\Delta (2,n,\infty ).
\end{equation*}
On the other hand, a hyperbolic ideal triangle with angles $\pi /n$, $0$ and $0$ can be cut into two ideal triangles with angles $\pi /2$, $\pi /2n$ and $0$ which are mirror images of each other. So we have an inclusion up to conjugation
\begin{equation*}
\Delta (n,\infty ,\infty )\subset\Delta (2,2n,\infty )
\end{equation*}
of index two, and one can show (see e.g. \cite{Veech89}) that indeed
\begin{equation*}
\mathrm{P}\Gamma (W_g,\omega _g)=\Delta (2,2n,\infty ).
\end{equation*}
Pierre Lochak has determined the thus obtained map $\Delta (2,2n,\infty )\backslash\bbh\simeq\bbc\to\mathrsfs{M}_g$: the complex number $t$ is sent to the hyperelliptic curve with affine equation
\begin{equation*}
y^2=\prod_{\nu =1}^{2g+1}(x-\zeta_n-\frac{t}{\zeta_n}),
\end{equation*}
where $\zeta_n=\exp\frac{2\pi\mathrm{i}}{n}$. This is \cite[Proposition 5.8]{LochakXY}.
\end{RegularPolygonsII}

\newpage
\section{Jacobians: Local Theory}

We wish to apply cohomological methods to the study of Teichm\"{u}ller disks and Teichm\"{u}ller curves. To see more clearly what happens, we again first consider just one abelian differential at a time. We first briefly recall the very classical notions of Jacobian varieties and period matrices. For our purposes it will be useful to think of them not as abelian varieties but as Hodge structures. Usually one only considers Hodge structures over one fixed coefficient ring (mostly $\bbz$, $\bbq$ or $\bbr$); for us it is important to distinguish between Hodge structures over different coefficient rings, and to change coefficient rings. Also we have to slightly extend the concept to coefficient rings that are contained in $\bbc$ but not necessarily in $\bbr$. Since there are different approaches in the literature how to do this, and since we have to go into the theory of decompositions of Hodge structures in quite a detailed way, we have to fix some notation, and therefore introduce the notion of a ``pseudo-Hodge structure''. With this we want to emphasize that Hodge symmetry, a very important tool in the theory of Hodge structures, gets sacrificed.

After these more abstract considerations, we return to translation surfaces and explain how affine automorphisms of the translation surface act on its cohomology. The relation with Hodge theory, which in this form still seems quite artificial, releases its full power only in the study of families, see chapter seven.

\subsection{Pure Hodge Structures}

We begin by considering the situation over a single point, i.e. for a single cohomology group instead of a local system. If $X$ is a closed Riemann surface, the cohomology group $H^1(X;\bbz )$ underlies a pure Hodge structure. This notion was introduced in \cite{Deligne71}. We recall the basic definitions and properties around pure Hodge structures, beginning with \textit{real} Hodge structures; these can be defined in several different but equivalent ways, of which we recall two.

\begin{PureRealHS}
A real Hodge structure is a ``real vector space with a Hodge decomposition'':
\end{PureRealHS}
\begin{Definition}
A \textup{pure real Hodge structure} is a finite-dimensional $\bbr$-vector space $M_{\bbr }$ together with a direct sum decomposition of $\bbc$-vector spaces
\begin{equation}\label{DirSumasinHS}
M_{\bbc }=M_{\bbr }\otimes_{\bbr }\bbc =\bigoplus_{p,q\in\bbz }M^{p,q}
\end{equation}
such that $\overline{M^{p,q}}=M^{q,p}$ for all $p,q\in\bbz$.
\end{Definition}
This decomposition is called the \textit{Hodge decomposition}, and we sometimes refer to the equation $\overline{M^{p,q}}=M^{q,p}$ as the ``Hodge symmetry''.

We shall only consider Hodge structures of one single weight (which is mostly $1$):
\begin{Definition}
Let $k$ be an integer. A pure real Hodge structure $M$ is called \textup{of weight $k$} if in the decomposition (\ref{DirSumasinHS}) one has $M^{p,q}=0$ whenever $p+q\neq k$.
\end{Definition}

Every pure real Hodge structure can be written in a canonical way as a direct sum of its weight components. Namely for a pure real Hodge structure $M_{\bbr }$ and an integer $k$ set
\begin{equation*}
M_{\bbc }^k=\bigoplus_{p+q=k}M^{p,q}.
\end{equation*}
Then by Hodge symmetry we find that $M_{\bbc }^k=\overline{M_{\bbc }^k}$, so it can be written uniquely as $M_{\bbr }^k\otimes_{\bbr }\bbc$ for some $\bbr$-sub-vector space $M_{\bbr }^k$ of $M_{\bbr }$.\footnote{This follows from the following elementary lemma: let $V_{\bbr }$ be a real vector space and let $W$ be a complex sub-vector space of $V_{\bbc }=V_{\bbr }\otimes_{\bbr }\bbc$ satisfying $W=\overline{W}$. Then there is a unique real sub-vector space $W_{\bbr }$ of $V_{\bbr }$ with $W_{\bbr }\otimes_{\bbr }\bbc =W$.

This can be seen as follows: the uniqueness of $W_{\bbr }$ is clear by $W_{\bbr }=W\cap V_{\bbr }$. That $W_{\bbr }\otimes_{\bbr }\bbc =W$ can then be seen by noting that for $w\in W$ one can write $w=\frac{1}{2}(w+\overline{w})+\frac{\mathrm{i}}{2}(\mathrm{i}w+\overline{\mathrm{i}w})$.

An analogous statement can be shown for any finite Galois extension of fields, by a similar trick. It also holds for infinite field extensions under some mild conditions, but the proof is more technical. See Lemma \ref{DescentForVectorSpaces}} Hence indeed $M_{\bbr}^k$ underlies a sub-Hodge structure of $M$ which we just denote by $M^k$. It is clear that we get a direct sum decomposition \textit{of Hodge structures} $M=\bigoplus_{k\in\bbz }M^k$. So there is no real restriction in only considering Hodge structures of some single weight.

There is a second way how one can define Hodge structures of a single weight. Namely let $M$ be a pure real Hodge structure of weight $k$, and define a descending filtration $F^{\bullet}M_{\bbc }$ on $M_{\bbc }$, indexed by all integers, as follows:
\begin{equation*}
F^pM_{\bbc }=\bigoplus_{r\ge p}M^{r,k-r}.
\end{equation*}
This filtration satisfies
\begin{equation}\label{ComplementarySubspacesForHF}
M_{\bbc }=F^pM_{\bbc }\oplus \overline{F^{k+1-p}M_{\bbc }}.
\end{equation}
Also it is \textit{exhausting}, i.e. $F^pM_{\bbc}=0$ for sufficiently large $p$ and $F^pM_{\bbc}=M_{\bbc}$ for sufficiently small $p$. This filtration is called the \textit{Hodge filtration}. Vice versa, start with a real vector space $M_{\bbr }$ and an exhausting descending filtration $F^{\bullet}M_{\bbc }$ by complex sub-vector spaces, indexed by the integers, such that (\ref{ComplementarySubspacesForHF}) holds. This defines a pure real Hodge structure of weight $k$ by setting
\begin{equation*}
M^{p,q}=F^pM_{\bbc }\cap\overline{F^qM_{\bbc }}.
\end{equation*}
These constructions are mutually inverse.

\begin{PureAHS}So far we have only considered real Hodge structures, but also Hodge structures over subrings of $\bbr$ (these will mostly be $\bbz$ or a number field) are important for what follows.
\end{PureAHS}
\begin{Definition}
Let $K\subseteq\bbr$ be a subring. A \textup{pure Hodge structure} $M$ over $K$ consists of the following data:
\begin{enumerate}
\item a finitely generated free\footnote{Some authors drop this condition} $K$-module $M_K$,
\item a direct sum decomposition of $K$-modules $M_K=\bigoplus_{k\in\bbz }M^k$ and
\item a real pure Hodge structure of weight $k$ on $M_{\bbr}^k=M_K^k=M_K\otimes_K\bbr$.
\end{enumerate}
A \textup{pure Hodge structure $M$ of weight $k$ over $K$} is a pure Hodge structure over $K$ satisfying $M_K^{\ell }=0$ for $\ell\neq k$; equivalently, this is a finitely generated free $K$-module $M_K$ together with a real pure Hodge structure of weight $k$ on $M_{\bbr }$.
\end{Definition}
We also talk of pure $K$-Hodge structures, and so on. An \textit{integral Hodge structure} is a $\bbz$-Hodge structure and a \textit{rational Hodge structure} is a $\bbq$-Hodge structure.

\begin{Constructions} Owing to the complicated definition, a great amount of abstract nonsense can be done with Hodge structures. We just list the most important topics:
\end{Constructions}
\begin{enumerate}
\item \textit{Morphisms:} For pure $K$-Hodge structures $M$ and $N$, let $\Hom (M,N)$ be the set of all $K$-module homomorphisms $\varphi :M_K\to N_K$ such that the $\mathbb{C}$-linear extension $\varphi_{\mathbb{C}}:M_{\mathbb{C}}\to N_{\mathbb{C}}$ satisfies $\varphi_{\mathbb{C}}(M^{p,q})\subseteq N^{p,q}$. Note this also ensures that $\varphi$ preserves the weight grading.

Alternatively we can define $\Hom (M,N)$ to be the set of all $K$-module homomorphisms $M_K\to N_K$ that respect the weight grading and on each weight component the Hodge filtration. These two definitions are equivalent.

    In this way we get the category of pure $K$-Hodge structures, denoted by $\mathrm{HS}_K$. If $K$ is a field, it is an abelian category\footnote{This is quite a nontrivial statement, see \cite{Deligne71}}.
\item \textit{Sub-Hodge-Structures:} Let $M$ be a pure $K$-Hodge structure. Then a sub-Hodge structure of $M$ is an $K$-Hodge structure $N$ whose underlying $K$-module $N_K$ is a submodule of $M_K$, and such that the inclusion $N_K\to M_K$ is in fact a morphism of $K$-Hodge structures. Clearly then
    \begin{equation*}
    N^{p,q}=N_{\bbc }\cap M^{p,q},
    \end{equation*}
    so on every submodule of $M_K$ there is at most one Hodge decomposition which turns it into a sub-Hodge structure.
    \begin{Lemma}\label{LemmaOnSubHS}
    Let $M$ be a pure $K$-Hodge structure of weight $k$, and let $N_K\subseteq M_K$ be a sub-$K$-module. Then, writing $N_{\bbc }=N_K\otimes_K\bbc$, we have
    \begin{equation}\label{InclusionFromLemmaOnSubHS}
    \bigoplus_{p+q=k}(N_{\bbc }\cap M^{p,q})\subseteq N_{\bbc }.
    \end{equation}
    Then the following are equivalent:
    \begin{enumerate}
    \item $N=(N_K,N^{p,q})$ is a sub-Hodge structure of $M$.
    \item The inclusion in (\ref{InclusionFromLemmaOnSubHS}) is in fact an equality.
    \item For every $x\in N_{\bbc }$, its components for the Hodge decomposition of $M_{\bbc }$ also lie in $N_{\bbc }$.\hfill $\square$
    \end{enumerate}
    \end{Lemma}
\item \textit{Tensor product:} let $M$ and $N$ be pure $K$-Hodge structures. Their tensor product is a pure Hodge structure $M\otimes N$ with $(M\otimes N)_K=M_K\otimes_KN_K$; in order to define its Hodge decomposition, note that there is a canonical identification
    \begin{equation*}
    (M_K\otimes_KN_K)\otimes_K\mathbb{C}=M_{\mathbb{C}}\otimes_{\mathbb{C}}N_{\mathbb{C}}
    \end{equation*}
    and the right hand side carries the tensor product decomposition, given by
    \begin{equation*}
    (M\otimes N)^{p,q}=\bigoplus_{p_1+p_2=p,\,
    q_1+q_2=q}(M^{p_1,q_1}\otimes N^{p_2,q_2})
    \end{equation*}
    This turns the category of pure Hodge structures into a tensor category. Note that if $M$ is pure of weight $m$ and $N$ is pure of weight $n$, then $M\otimes N$ is pure of weight $m+n$.
\item \textit{Tate twists:} For any $n\in\mathbb{Z}$, the pure Hodge structure $K(n)$ of weight $-2n$ is defined by
    \begin{equation*}
    K(n)_K=(2\pi\mathrm{i})^nK\subset\mathbb{C}=K(n)_{\mathbb{C}}=K(n)^{-n,-n}.
    \end{equation*}
    One checks that $K(n+m)\simeq K(n)\otimes K(m)$ and that $K(0)$ is a unit for the tensor product (i.e. for any pure Hodge structure $M$ one has canonical isomorphisms $M\otimes K(0)\simeq M\simeq K(0)\otimes M$. We also write $M(n)=M\otimes K(n)$; this is called the $n$-th Tate twist of $M$.
\item \textit{Internal $\Hom$s:} These can be explicitly described in the way that one would guess: let $M$ and $N$ be pure Hodge structures. Then $\underline{\Hom}(M,N)$ is described as follows: first, set $\underline{\Hom}(M,N)_K=\Hom_K (M_K,N_K)$; then we can identify $\underline{\Hom}(M,N)_{\mathbb{C}}=\Hom_{\bbc }(M_{\bbc},N_{\bbc})$. So the following definition is sensible:
    \begin{equation*}
    \underline{\Hom}(M,N)^{p,q}=\{\varphi :M_{\bbc }\to N_{\bbc}\, |\,\varphi (M^{r,s})\subseteq N^{r+p,s+q}\text{ for all }r,s\}
    \end{equation*}
    If $M$ is pure of weight $m$ and $N$ is pure of weight $n$, then $\underline{\Hom}(M,N)$ is pure of weight $n-m$. Finally these internal $\Hom$ objects relate to the ``ordinary'' groups of homomorphisms as follows: if $H$ is the Hodge structure $\underline{\Hom }(M,N)$, then $\Hom (M,N)$ is equal to the intersection $H^{0,0}\cap H_K$.
\item \textit{Duals:} As a special case of (iv) we get duals by $\check{M}=\underline{\Hom}(M, K(0))$. For example, $K(-n)$ is the dual of $K(n)$. For every $M$ there is a canonical morphism (``evaluation'') $M\otimes\check{M}\to K(0)$. If on the other hand $\varphi :M\otimes N\to K(0)$ is any morphism of pure Hodge structures, we get an associated morphism $N\to\check{M}$. We call $\varphi$ a \textit{perfect pairing} if this last map is an isomorphism (this condition is symmetric in $M$ and $N$).
\item \textit{Base extension:} Let $K\subseteq B\subseteq\bbr$ be subrings and let $M$ be a pure $K$-Hodge structure. We can define from this in a canonical way a pure $B$-Hodge structure which we denote by $M\otimes_KB$ or, if $K$ is clear from the context, $M\otimes B$. Its underlying $B$-module is $(M\otimes B)_B=M_K\otimes_KB$, with the weight components $M_B^k=M_K^k\otimes_KB$. Noting that we can then canonically identify $(M\otimes B)_{\bbr }$ with $M_{\bbr }$, we can transfer the real Hodge structure on $M_{\bbr }$ to a real Hodge structure on $(M\otimes B)_{\bbr }$. These data define in an obvious way a $B$-Hodge structure.

This construction defines an $K$-linear functor from $K$-Hodge structures to $B$-Hodge structures.
\end{enumerate}

\begin{HSonCohomology}
If $X$ is a projective complex manifold, its $k$-th cohomology can (and should) be considered not as a group or a vector space but rather as a pure integral Hodge structure $M^k(X)$ of weight $k$, with $M^k(X)_{\bbz }=H^k(X,\bbz )/\text{torsion}$. This Hodge structure can be constructed as follows:

Introduce a K\"{a}hler metric $g$ on $X$ (for example embed $X$ into some $\mathbb{P}^n(\bbc )$ and restrict the Fubini-Study metric to $X$). Identifying
    \begin{equation*}
    M^k(X)_{\bbc }=H^k(X,\bbc )=\mathrsfs{H}^k(X)
    \end{equation*}
    (harmonic $k$-forms with respect to $g$), the latter space decomposes as
    \begin{equation*}
    \mathrsfs{H}^k(X)=\bigoplus_{p+q=k\atop p,q\ge 0}\mathrsfs{H}^{p,q}(X)
    \end{equation*}
    where $\mathrsfs{H}^{p,q}(X)$ is the space of harmonic forms of type $(p,q)$. It of course depends on the choice of the K\"{a}hler metric, but its image in $H^k(X,\bbc )$ under the above identification does not. The conjugate of a form of type $(p,q)$ is a form of type $(q,p)$, so setting $M^k(X)^{p,q}=\mathrsfs{H}^{p,q}(X)$ we obtain a pure Hodge structure $M^k(X)$ of weight $k$. In fact this also works for an arbitrary K\"{a}hler manifold.

Let us discuss some examples of Hodge structures appearing as the cohomology of a complex projective variety.
\begin{enumerate}
\item The simplest example: $M^0(X)\simeq\mathbb{Z}(0)$ (provided that $X$ is connected).
\item If $X$ is connected and has complex dimension $n$, its ``top cohomology'' $M^{2n}(X)$ is isomorphic to $\mathbb{Z}(-n)$. If we demand in addition that the fundamental class\footnote{Note that every complex manifold has a canonical orientation} $[X]\in H^{2n}(X,\mathbb{Z})$ be mapped to $(2\pi\mathrm{i})^{-n}$, the isomorphism is uniquely determined. On $M^{2n}(X)_{\mathbb{C}}=\mathrsfs{H}^{2n}(X)$ it is given as
    \begin{equation*}
    [\omega ]\mapsto\left(\frac{1}{2\pi\mathrm{i}}\right)^n\int_X\omega .
    \end{equation*}
    This is Proposition 1.14 in  \cite{PetersSteenbrink}.
\item The cup product in cohomology gives rise to a morphism of pure Hodge structures $M^k(X)\otimes M^{\ell}(X)\to M^{k+\ell }(X)$.
\item Poincar\'{e} duality can also be formulated in the context of pure Hodge structures; it involves a Tate twist which is ``invisible'' in the classical formulation. Namely if $X$ is a connected projective complex manifold of complex dimension $n$ we get (by (ii) and (iii)) morphisms
    \begin{equation*}
    M^k(X)\otimes M^{2n-k}(X)\to M^{2n}(X)=\mathbb{Z}(-n);
    \end{equation*}
    applying a Tate twist we get
    \begin{equation*}
    M^k(X)\otimes M^{2n-k}(X)(n)\to\mathbb{Z}(0)
    \end{equation*}
    which sets up an isomorphism between $M^{2n-k}(X)(n)$ and the dual of $M^k(X)$.
\item There are canonical identifications $H_k(X,\mathbb{R})=\Hom (H^k(X,\mathbb{R}),\mathbb{R})$; hence we can just \textit{define} $M_k(X)$ as the dual to $M^k(X)$ --- it is thus a pure Hodge structure of weight $-k$.
\end{enumerate}
\end{HSonCohomology}
\begin{HSofSmoothCurves}
Let us now concentrate on the Hodge structures $M^{\ast}(X)$ for $X$ a smooth curve. Since there are canonical isomorphisms $M^0(X)\simeq\bbz (0)$ and $M^2(X)\simeq\bbz (-1)$ we only need to consider $M^1(X)$. We have $M^1(X)_{\bbz }=H^1(X,\bbz )$ which is a free abelian group of rank $2g$ where $g$ is the genus of $X$. The Hodge decomposition on $M^1(X)_{\bbc }=H_{\mathrm{dR}}^1(X)$ is given by $M^{1,0}=\Omega^1(X)$ (the space of holomorphic one-forms) and $M^{0,1}=\overline{\Omega^1(X)}$, the space of antiholomorphic one-forms.
\end{HSofSmoothCurves}
\begin{Polarizations}
We now introduce polarizations of Hodge structures. These can be roughly thought of as scalar products on Hodge structures. First we need another definition:
\end{Polarizations}
\begin{Definition}
Let $M$ be a pure Hodge structure of weight $k$. The \textup{Weil operator} is the unique $\bbc$-linear map $C:M_{\bbc}\to M_{\bbc}$ which respects the Hodge decomposition and acts as $\mathrm{i}^{p-q}$ times the identity on $M^{p,q}$.
\end{Definition}
The reason for introducing the Weil operator is basically the multilinear algebra one has to tackle when constructing something positive definite.
\begin{Definition}\label{DefinitionOfPolarizationForHS}
Let $K\subseteq\bbr$ be a subring and let $M$ be a pure $K$-Hodge structure of weight $k$.
\begin{enumerate}
\item A \textup{polarization} of $M$ is a morphism of $K$-Hodge structures
\begin{equation*}
S:M\otimes M\to K (-k)
\end{equation*}
which is $(-1)^k$-symmetric\footnote{i.e., symmetric when $k$ is even and antisymmetric when $k$ is odd} and such that the Hermitian form
\begin{equation}\label{HermitianFormOfPolarization}
H(x,y)=(2\pi\mathrm{i})^k\cdot S(Cx,\overline{y})
\end{equation}
on $M_{\bbc}$ is positive definite.
\item $M$ is called \textup{polarizable} if there exists a polarization of $M$.
\end{enumerate}
\end{Definition}
This definition is motivated by the Poincar\'{e} duality pairing on primitive cohomology. In the only case interesting to us, it takes the following form:
\begin{Proposition}
Let $X$ be a compact Riemann surface. Then the Poincar\'{e} duality pairing $S:M^1(X)\otimes M^1(X)\to M^2(X)=\bbz (-1)$ is a polarization.
\end{Proposition}
\begin{proof}
This is a special case of \cite[Th\'{e}or\`{e}me 6.32]{Voisin02}.
\end{proof}
Not every Hodge structure is polarizable, and in fact the existence of a polarization has strong consequences on a Hodge structure:
\begin{Proposition}\label{NicePropsOfPolarizableHS}
Let $K\subseteq\bbr$ be a subfield (!), let $M$ be a pure $K$-Hodge structure of weight $k$ and let $S:M\otimes M\to K(-k)$ be a polarization.
\begin{enumerate}
\item Let $x\in M^{p',q'}$ and $y\in M^{p'',q''}$, and assume that $(p',q')$ is different from $(q'',p'')$ (note the twist!). Then $S(x,y)=0$.
\item With respect to the Hermitian form $H$ as in (\ref{HermitianFormOfPolarization}), the Hodge decomposition of $M_{\bbc }$ is orthogonal.
\item Let $N\subseteq M$ be a sub-Hodge structure. Then the sub-$K$-vector space
\begin{equation*}
N_K^{\perp}=\{ x\in M_K : S(x,y)=0\text{ for all }y\in N_K\}
\end{equation*}
underlies a sub-Hodge structure $N^{\perp}$ such that $M=N\oplus N^{\perp }$.
\item In particular every polarizable pure $K$-Hodge structure is semisimple.
\end{enumerate}
\end{Proposition}
\begin{proof}
(i) follows from the fact that $S$ is a morphism of Hodge structures. Namely
\begin{equation*}
x\otimes y\in (M\otimes M)^{p'+p'',q'+q''}
\end{equation*}
and $K(-k)_{\bbc }$ only consists of elements of Hodge type $(k,k)$. So $S(x,y)$ only has a chance to be nonzero if
\begin{equation*}
p'+p''=q'+q''=k,\quad\text{i.e.}\quad p'=q'',\, p''=q'.
\end{equation*}

As to (ii), note that the Weil operator preserves the Hodge decomposition and that $\overline{M^{p,q}}=M^{q,p}$, so this follows from (i).

Concerning (iii), note that
\begin{equation*}
N_{\bbc }^{\perp}=N_K^{\perp}\otimes_K\bbc =\{ x\in M_{\bbc }: H(x,\overline{y})=0\text{ for all }y\in N_{\bbc }\} .
\end{equation*}
Bearing in mind the remarks just before Lemma \ref{LemmaOnSubHS} we have no other possibility than setting $(N^{\perp})^{p,q}=N_{\bbc}^{\perp}\cap M^{p,q}$. Now from (ii) we see that
\begin{equation*}
N_{\bbc }^{\perp}=\bigoplus_{p+q=k}N^{p,q},
\end{equation*}
so that by Lemma \ref{LemmaOnSubHS} $N$ is a sub-Hodge structure. By the explicit description we have given, $M=N\oplus N^{\perp}$ \textit{as Hodge structures}.

(iv) follows directly from (iii).
\end{proof}

Note that (i) only uses that $S$ is a morphism of Hodge structures. Hence we can rephrase the conditions on $S$ being a polarization as follows: the restriction of $H$ to $M^{p,q}$ is given by
\begin{equation*}
H(x,y)=(2\pi )^k\cdot (-1)^p\cdot S(x,\overline{y}).
\end{equation*}
Hence:
\begin{Corollary}\label{AlternativeDescriptionOfPolarization}
A $(-1)^k$-symmetric morphism $S:M\otimes M\to K(-k)$ is a polarization if and only if the form $S(x,\overline{y})$ on $M^{p,q}$ is positive definite whenever $p$ is even and negative definite whenever $p$ is odd.
\end{Corollary}

For later reference we state again the motivating example: the Hodge structure associated with a smooth compact curve.
\begin{Proposition}
Let $X$ be a smooth compact curve. Then the Poincar\'{e} pairing $M^1(X)\otimes M^1(X)\to\bbz (-1)$ is a principal polarization of $M^1(X)$. In particular it induces a canonical isomorphism $M_1(X)(-1)\simeq M^1(X)$.\hfill $\square$
\end{Proposition}
The latter isomorphism comes directly from the definition of $M_1(X)$ as $\underline{\Hom}(M^1(X),\bbz (0))$.

\subsection{Pseudo-Hodge Structures}

For every subring $K\subseteq\bbr$ we have defined Hodge structures with coefficients in $K$. They consist of a free $K$-module $M_K$ and some additional structure on the base change $M_{\bbc} =M_K\otimes_K\bbc$. The axioms for these additional structures make use of the inclusion $K\subseteq\bbr$ (via the existence of a canonical complex conjugation on $M_{\bbc }$). If $K\subseteq B\subseteq\bbr$ and $M$ is a Hodge structure defined over $K$, we can ``forget structure'' and obtain a Hodge structure defined over $B$. For what follows it will be necessary to even pass to objects which are only defined over $\bbc$. For this we have to sacrifice some of the axioms defining a Hodge structure since they would need a complex conjugation. Thus we introduce the following notion:
\begin{Definition}
Let $K\subseteq\bbc$ be a subring, and let $M_K$ be a finitely generated free $K$-module. Write $M_{\bbc }=M_K\otimes_K\bbc$. A \textup{pseudo-Hodge decomposition} on $M_K$ is a direct sum decomposition of complex vector spaces
\begin{equation*}
M_{\bbc }=\bigoplus_{p\in\bbz }M^p.
\end{equation*}
\end{Definition}
Note that both the weight and the additional conditions using complex conjugation are missing here. The omission of the weight is more or less irrelevant for the mathematics (at least for our purposes) and just a means to ease notation; the omission of the additional conditions, however, is necessary if we want to work over rings not contained in the reals, and it also drastically changes the behaviour of the considered objects.

To begin with, one can switch forth and back between Hodge decompositions and Hodge filtrations. This becomes false for their ``pseudo-Hodge'' versions, where there is only one direction (with a partial converse). Namely let $M_K$ be a finitely generated $K$-module and let $M_{\bbc }=\bigoplus_pM^p$ be a pseudo-Hodge decomposition on $M_K$. Then we can associate with this a filtration as follows:
\begin{equation}\label{PseudoHodgeFiltration}
F^pM_{\bbc }=\bigoplus_{r\ge p}M^r.
\end{equation}
Clearly there is no canonical way of reconstructing the pseudo-Hodge decomposition from this filtration. There is one as soon as a \textit{polarization} is added, see below.

\begin{Definition}
Let $K$ be a subring of $\bbc$. A \textup{pseudo-Hodge structure} $M$ over $K$ consists of a finitely generated free $K$-module $M_K$ together with a pseudo-Hodge decomposition on $M_K$.
\end{Definition}

So in particular a pseudo-Hodge structure over the complex numbers is merely a $\bbz$-graded finite-dimensional complex vector space. It will turn out that this notion is simple enough to provide strong decomposition theorems but close enough to Hodge structures to be valuable for us.

We now describe the forgetful functors from Hodge-structures to pseudo-Hodge structures. Let $K$ be a subring of $\bbr$ and let $M$ be a pure Hodge structure of weight $k$ over $K$, with Hodge decomposition
\begin{equation*}\label{HodgeDecompositionForPseudoHodgeStructure}
M_{\bbc }=\bigoplus_{p+q=k}M^{p,q}.
\end{equation*}
Then the associated pseudo-Hodge structure $M^{\psi}$ is defined by $(M^{\psi})_K=M_K$ and $M^p=M^{p,k-p}$. Note that if one fixes the weight there is no information lost in this process:
\begin{Proposition}
Fix an integer $k$ and a subring $K\subseteq\bbr$.
\begin{enumerate}
\item Let $P$ be a pseudo-Hodge structure over $K$. \textit{If} there exist a pure Hodge structure $M$ of weight $k$ over $K$ and an isomorphism $M^{\psi}\to P$, these data are unique up to unique isomorphism.
\item Let $M$ and $N$ be pure Hodge structures of weight $k$ over $K$. Then the canonical map $\Hom (M,N)\to\Hom (M^{\psi },N^{\psi})$ is a bijection.
\item Let $M$ be a pure Hodge structure of weight $k$ over $K$, and let $P$ be a sub-pseudo-Hodge structure of $M^{\psi}$. Then there exists a unique sub-Hodge structure $N\subseteq M$ with $N^{\psi }=P$.
\end{enumerate}
\end{Proposition}
\begin{proof}
The only non-obvious statement is (iii), but it follows easily from Lemma \ref{LemmaOnSubHS}.
\end{proof}

Finally one more notation: if $K\subseteq B\subseteq\bbc$ are subrings and $M$ is an $K$-pseudo-Hodge structure, we obtain a canonical $B$-pseudo-Hodge structure $M\otimes_KB$ by ``forgetting structure'': we let $(M\otimes_KB)_B=M_K\otimes_KB$; the base changes to the complex numbers are canonically isomorphic, and we demand the Hodge decompositions to be equal.

\begin{Definition}
Let $K\subseteq\bbc$ be a subring, stable under complex conjugation.
\begin{enumerate}
\item Let $V$ be an $K$-module. A \textup{hermitian form} on $K$ is a map $\varphi :V\times V\to K$ with the following properties:
\begin{enumerate}
\item $\varphi $ is $K$-linear in the first variable,
\item $\varphi $ is $K$-antilinear in the second variable (in the sense that it is additive and $\varphi (v,\alpha w)=\overline{\alpha }\varphi (v,w)$, where the bar denotes conjugation in $\bbc$), and
\item $\varphi (v,w)=\overline{\varphi (w,v)}$ for all $v,w\in V$.
\end{enumerate}
Note that for $K=\bbc$ this gives back the usual definition. If $\varphi :V\times V\to K$ is a hermitian form, we obtain a hermitian form on $V\otimes\bbc $, which we again denote by $\varphi $, by ``$\bbc$-sesquilinear extension'':
\begin{equation*}
\varphi (v\otimes\alpha ,w\otimes\beta )=\alpha\overline{\beta}\varphi (v,w).
\end{equation*}
\item Now let $M$ be a pseudo-Hodge structure over $K$. Then a \textup{polarization} of $M$ is a hermitian form $\varphi :M_K\times M_K\to K$ such that for the induced hermitian form on $M_{\bbc }$, the decomposition (\ref{HodgeDecompositionForPseudoHodgeStructure}) is orthogonal and such that $(-1)^p\varphi $ is positive definite on $M^p$.
\end{enumerate}
\end{Definition}
We extend the usual terminology: a pseudo-Hodge structure is polarizable if there exists some polarization of it, and so on. How do these notions relate to the corresponding notions for Hodge structures?

Let $K$ be a subring of $\bbr$, let $M$ be a pure Hodge structure of weight $k$ over $K$, and let $S:M\otimes M\to K(-k)$ be a $(-1)^k$-symmetric morphism of $K$-Hodge structures. Then \textit{as $K$-modules,} $M_K=M_K^{\psi}$, and we may consider
\begin{equation*}
\varphi :M_K\times M_K\to K,\quad (x,y)\mapsto S(x,\overline{y}).
\end{equation*}
Lemma \ref{AlternativeDescriptionOfPolarization} tells us that $S$ is a polarization of the Hodge structure $M$ if and only if $\varphi $ is a polarization of the pseudo-Hodge structure $M^{\psi}$. In particular $M$ is polarizable if and only if $M^{\psi }$ is. Of course the definition of a polarization of a pseudo-Hodge structure was just made such that this holds.

Here is the promised remark about pseudo-Hodge decompositions, filtrations and polarizations: let $(M,H)$ be a polarized pseudo-Hodge structure over $K$, with associated filtration $F^{\bullet}M_{\bbc}$ as in (\ref{PseudoHodgeFiltration}). Then one can reconstruct the pseudo-Hodge decomposition, given filtration and the polarization: $M^p$ is the $H$-orthogonal complement of $F^{p+1}M_{\bbc }$ in $F^pM_{\bbc }$.

\subsection{Hodge Structures of Jacobian Type}

The integral Hodge structure $M^1(X)$ for a smooth compact curve $X$ is a pure Hodge structure of weight one which comes equipped with a principal polarization and only has (possibly) nontrivial components of Hodge type $(1,0)$ and $(0,1)$. We will almost exclusively deal with such Hodge structures. Since their direct characterization gives quite a lengthy expression which we do not like to repeat again and again, we introduce the following notion:
\begin{Definition}
Let $K$ be a subring of $\bbr$. An \textup{$K$-Hodge structure of Jacobian type (of rank $g$)} is a pure $K$-Hodge structure $J$ of weight one together with a principal polarization $J\otimes J\to K(-1)$, such that $\dim_{\bbc }J^{1,0}=\dim_{\bbc }J^{0,1}=g$ and all other Hodge components are zero.
\end{Definition}
So if $X$ is a closed Riemann surface, $M^1(X)$ is an integral Hodge structure of Jacobian type. Also note that if $K\subseteq B\subseteq\bbr$ are subrings and $J$ is an $K$-Hodge structure of Jacobian type, then $J\otimes_KB$ is a $B$-Hodge structure of Jacobian type.

\begin{AbelianVarieties}
Integral Hodge structures of Jacobian type are closely related to complex abelian varieties; in fact the category of these Hodge structures is canonically equivalent to the category of polarized complex abelian varieties, and for many purposes the point of view of Hodge structures gives a clearer picture than that of abelian varieties. Although we will not need it, we record for the convenience of the reader how these two points of view are related.

So let $J$ be an integral Hodge structure of Jacobian type with polarization $S$. Let $p : J_{\bbc }\to J^{0,1}$ be the projection coming from the Hodge decomposition, then $p(J_{\bbz })$ is a full lattice in $J^{0,1}$. Hence we have a complex torus $T=J^{0,1}/p(J_{\bbz })$. We can identify $H_1(T,\bbz )$ with $J_{\bbz }$, hence $2\pi\mathrm{i}$ times the polarization gives an alternating two-form $\bigwedge_{\bbz }^2H_1(T,\bbz )\to \bbz$, i.e. an element of $H^2(T,\bbz )$. This is the Chern class of a holomorphic line-bundle, well defined up to analytic equivalence, defining a polarization of the torus $T$.

This construction defines a functor from integral Hodge structures of Jacobian type to polarized complex abelian varieties which is in fact an equivalence of categories.

Now let $X$ be a closed Riemann surface. The abelian variety associated with the Hodge structure $M^1(X)$ is the Jacobian variety of $X$ (whence the name). Note that the classical construction of the Jacobian is slightly different: this is obtained by first using the isomorphism $M^1(X)\simeq M_1(X)(-1)$ and then applying the above sketched equivalence of categories.

For further details and proofs see \cite[section 7.2.2]{Voisin02}.
\end{AbelianVarieties}
\begin{AbelianVarietiesUpToIsogeny}
In the same spirit we can re-interpret rational Hodge structures of Jacobian type. To explain this, first some notation. Let $\mathrm{HSJ}_{\bbz }$ denote the category of integral Hodge structures of Jacobian type, and let $\mathrm{HSJ}_{\bbq }$ denote the category of rational Hodge structures of Jacobian type. The former is a $\bbz$-linear abelian category, the latter is a $\bbq$-linear abelian category. There is an obvious functor
\begin{equation*}
T:\mathrm{HSJ}_{\bbz }\to\mathrm{HSJ}_{\bbq },\quad J\mapsto J\otimes\bbq,\quad f\mapsto f\otimes\mathrm{id}_{\bbq }.
\end{equation*}
By the above we can identify $\mathrm{HSJ}_{\bbz }$ with the category of polarized abelian varieties. Then a morphism $f$ of abelian varieties is an isogeny if and only if $T(f)$ is an isomorphism of Hodge structures. Also it is easily seen that every functor from $\mathrm{HSJ}_{\bbz }$ to an abelian category which sends all isogenies to isomorphisms factors uniquely over $T$. Hence $T$ is the \textit{localization} of the category of polarized abelian varieties in the class of isogenies, and we can say that $\mathrm{HSJ}_{\bbq }$ is the category of ``abelian varieties up to isogeny''.

This category, although less intuitive than the category of polarized abelian varieties, has some advantages. For example it is semisimple, whereas the latter is not. For this reason we shall almost all the time work only up to isogeny, i.e. with rational Hodge structures.

So for a closed Riemann surface $X$ we can view $M^1(X)\otimes\bbq$ as ``the Jacobian variety of $X$, up to isogeny''. Since we are going to use it very often, we shall denote this rational Hodge structure of Jacobian type by $J(X)$.
\end{AbelianVarietiesUpToIsogeny}
\begin{PeriodMatrices}
Explicit computations surrounding Hodge structures of Jacobian type are often done using period matrices. Again we need to fix notations. So let $K\subseteq\bbr$ be a subring of $\bbr$ and let $J$ be an $K$-Hodge structure of Jacobian type, say of rank $g$. Then we may choose some symplectic basis of $J_K$; to choose such a basis is the same as to give an isomorphism $J_K\simeq K^{2g}$ such that the polarization becomes the standard symplectic form. This means that if the canonical basis vectors of $K^{2g}$ are denoted by $a_1,\ldots ,a_g$, $b_1,\ldots ,b_g$, we have
\begin{equation*}
S(a_i,b_j)=-S(b_j,a_i)=\frac{\delta_{ij}}{2\pi\mathrm{i}},\quad S(a_i,a_j)=S(b_i,b_j)=0.
\end{equation*}
Then we can also identify $J_{\bbc }$ with $\bbc^{2g}$, and the only nontrivial information left is the complex subspace $J^{1,0}=F^1J_{\bbc }\subseteq\bbc^{2g}$. Denote by $\pr^{0,1}$ the projection $\bbc^{2g}=M_{\bbc }\to J^{0,1}$ coming from the Hodge decomposition.
\end{PeriodMatrices}
\begin{Lemma}
With the above notation, $(\pr^{0,1}a_1,\ldots ,\pr^{0,1}a_g)$ is a $\bbc$-basis of $J^{0,1}=\Gr_F^0\bbc^{2g}$.
\end{Lemma}
\begin{proof}
It suffices to show that $\eta_1,\ldots ,\eta_g$ generate $J^{0,1}$ as a $\bbc$-vector space, where $\eta_j=\pr^{0,1}a_j$. Let $V$ be the \textit{real} sub-vector space generated by $\eta_1,\ldots ,\eta_g$. We have to show that $J^{0,1}=V\oplus\mathrm{i}V$.

Indeed, $V\cap\mathrm{i}V$ is a \textit{complex} sub-vector space of $J^{0,1}$ on which $\Im H$ vanishes identically. But this means that $H$ itself vanishes identically on $V\cap\mathrm{i}V$. Since $H$ is positive definite, we must have $V\cap\mathrm{i}V=0$; for dimension reasons we then must have $J^{0,1}=V\oplus\mathrm{i}V$.
\end{proof}
So there is a unique $\bbc$-linear map $\psi :\bbc^{2g}\to\bbc^g$ with $\psi (a_j)=a_j$ and $\ker\psi =F^1\bbc^{2g}$. This has in the standard bases a matrix of the form $(1_g\,\tau )$ with some $\tau\in M_g(\bbc )$. In coordinates:
\begin{equation*}
\psi (a_j)=a_j,\quad \psi (b_j)=\sum_{j=1}^g\tau_{jk}a_k.
\end{equation*}
\begin{DefinitionProposition}
The matrix $\tau$ is called the \textup{period matrix} of $J$ with respect to the given basis. It is a symmetric matrix, and its imaginary part (taken entry-wise) is positive definite.
\end{DefinitionProposition}
\begin{proof}
This is a nasty matrix calculation. We follow an analogous calculation with different conventions in \cite[proof of Thm. 6.2]{Debarre}. Consider two different $\bbr$-bases of $\bbc^g$ and express the form $-2\Im H$ in each of these bases.
\begin{enumerate}
\item Note that $\psi$ establishes an isomorphism $\Gr_F^0\bbc^{2g}\simeq\bbc^g$. It further maps $M_{\bbr }=\bbr^{2g}$ bijectively to $\bbc^g$. Hence the image of $(a_1,\ldots ,a_g,b_1,\ldots ,b_g)$ under $\psi$ is an $\bbr$-basis of $\bbc^g$; it is explicitly given as
    \begin{equation*}
    (a_1,\ldots ,a_g,\sum_{k=1}^g\tau_{1k}a_k,\ldots ,\sum_{k=1}^g\tau_{gk}a_k).
    \end{equation*}
    In this basis, $-2\Im H$ is simply given by the matrix
    \begin{equation*}
    \begin{pmatrix}
    0_g&1_g\\
    -1_g&0_g
    \end{pmatrix}.
    \end{equation*}
\item Clearly also the following is an $\bbr$-basis of $\bbc^g$:
    \begin{equation*}
    (a_1,\ldots ,a_g,\mathrm{i}a_1,\ldots ,\mathrm{i}a_g).
    \end{equation*}
    In this basis, $-2\Im H$ has the matrix
    \begin{equation*}
    \begin{pmatrix}
    0_g&(\Im\tau )^{-1}\\
    -(\Im\tau )^{-1}&0_g
    \end{pmatrix}.
    \end{equation*}
\end{enumerate}
The base change matrix to pass from one to another is
\begin{equation*}
\begin{pmatrix}
1_g&\Re\tau\\
0_g&\Im\tau
\end{pmatrix}.
\end{equation*}
Hence
\begin{equation*}
\begin{pmatrix}
1_g&\Re\tau\\
0_g&\Im\tau
\end{pmatrix}^{\mathsf{T},-1}
\begin{pmatrix}
0_g&1_g\\
-1_g&0_g
\end{pmatrix}
\begin{pmatrix}
1_g&\Re\tau\\
0_g&\Im\tau
\end{pmatrix}^{-1}=
\begin{pmatrix}
0_g&(\Im\tau )^{-1}\\
-(\Im\tau )^{-1}&0_g
\end{pmatrix}.
\end{equation*}
Multiplying the matrices on the left hand side we find
\begin{equation*}
\begin{pmatrix}
0_g&(\Im\tau )^{-1}\\
-(\Im\tau )^{\mathsf{T},-1}&(\Im\tau )^{\mathsf{T},-1}(\Re\tau -(\Re\tau )^{\mathsf{T}})(\Im\tau )^{-1}
\end{pmatrix}=
\begin{pmatrix}
0&(\Im\tau )^{-1}\\
-(\Im\tau )^{-1}&0
\end{pmatrix}.
\end{equation*}
So comparing entries we find that both $\Im\tau$ and $\Re\tau$ are symmetric, hence $\tau$ is symmetric. Finally $(\Im\tau )^{-1}$ describes the Hermitian form $H$ in the \textit{complex} basis $e_1,\ldots ,e_g$ and hence is positive definite, so also $\Im\tau$ is positive definite.
\end{proof}
The set of all symmetric matrices in $M_g(\bbc )$ with positive definite imaginary part is called the \textit{Siegel upper half space} and denoted by $\mathrsfs{H}_g$. This naturally has a structure as a complex manifold of dimension $\frac{g(g+1)}{2}$. Note that $\mathrsfs{H}_1=\bbh$, the upper half plane.
\begin{PeriodMatricesOfJacobians}
Period matrices have a particularly nice description for the Jacobians of Riemann surfaces.
\end{PeriodMatricesOfJacobians}
\begin{Proposition}
Let $X$ be a closed Riemann surface of genus $g$, let $K\subseteq\bbr$ be a subring and let $(a'_j,b'_j)_{1\le j\le g}$ be a symplectic basis of $H^1(X,K)$. Let $(a_j,b_j)_{1\le j\le g}$ be the corresponding dual basis of $H_1(X,K)$, which is also symplectic. Then the period matrix $\tau$ of $M^1(X)\otimes K$ with respect to this basis is computed as follows: for every $1\le j\le g$ there exists a unique holomorphic one-form $\omega_j$ on $X$ such that
\begin{equation*}
\int_{a_k}\omega_j=\delta_{jk}.
\end{equation*}
Then
\begin{equation*}
\tau_{jk}=\tau_{kj}=\int_{b_k}\omega_j.
\end{equation*}
\end{Proposition}
\begin{proof}
We use the isomorphism $M^1(X)\simeq M_1(X)(-1)$ and rather describe $M_1(X)(-1)$ than $M^1(X)$.

So consider $M_K=H_1(X,K)$ with the symplectic basis $(a_1,\ldots ,a_g,b_1,\ldots ,b_g)$; this also defines a basis of
\begin{equation*}
M_{\bbc }=H_1(X,\bbc )=\Hom (H_{\mathrm{dR}}^1(X),\bbc )=\Hom (\Omega^1(X)\oplus\overline{\Omega^1(X)},\bbc ).
\end{equation*}
The projection $\pr^{0,1}$ is just the restriction to $\Omega^1(X)$:
\begin{equation*}
\Hom (\Omega^1(X)\oplus\overline{\Omega^1(X)},\bbc )\to\Hom (\Omega^1(X),\bbc).
\end{equation*}
This maps $(a_1,\ldots ,a_g)$ to a $\bbc$-basis of $\Hom (\Omega^1(X),\bbc )$; denote the dual basis of $\Omega^1(X)$ by $(\omega_1,\ldots ,\omega_g)$. It is thus uniquely characterized by the equations
\begin{equation*}
\int_{a_k}\omega_j=\delta_{jk}.
\end{equation*}
Now one of the remaining basis elements $b_k$ of $H_1(X,K)$ is mapped by $\pr^{0,1}$ to the linear form ``integration over the cycle $b_j$'' on $\Omega^1(X)$. But then it is mapped to the same element as $\sum_k\tau_{jk}a_k$, where
\begin{equation*}
\tau_{jk}=\int_{a_k}\omega_j.
\end{equation*}
Hence the period matrix of $M$ is $(\tau_{\alpha\beta })$.
\end{proof}

\subsection{The Action on Cohomology}

Let $X$ be a closed Riemann surface with an abelian differential $\omega$. The affine group $\Aff^+(X,\omega )$, consisting of (special) orientation-preserving homeomorphisms $X\to X$, acts on the first cohomology $H^1(X,\bbz )\simeq\bbz^{2g}$ by symplectic automorphisms. Let us take a closer look at this action.

\begin{TheCanonicalSubspace}
First we tensor with $\bbr$, to the effect that we can describe $H^1(X,\bbz )\otimes\bbr =H^1(X,\bbr )$ as de Rham cohomology, which in turn can be identified with the space of harmonic real-valued one-forms on $X$. Clearly the real part and the imaginary part of $\omega$ are harmonic; they are also linearly independent over $\bbr $. This can be checked locally in any chart of the translation structure, and that $\de x$ and $\de y$ on $\bbr^2=\bbc$ are linearly independent is a triviality. Hence they span a two-dimensional vector space which will be very important in what follows, so it deserves a name of its own.
\end{TheCanonicalSubspace}
\begin{Definition}
Let $X$ be a closed Riemann surface and let $\omega$ be an abelian differential on $X$. The \textup{canonical subspace} $S_{\omega }\subseteq H^1(X,\bbr )$ is the real sub-vector space generated by the de Rham cohomology classes of the one-forms $\Re\omega$ and $\Im\omega$.
\end{Definition}
The canonical subspace has several nice properties which are easily checked:
\begin{Proposition}
Let $X$ and $\omega$ as above. Then the canonical subspace $S_{\omega }$ is a real sub-Hodge structure of the real Hodge structure $J(X)\otimes\bbr$.
\end{Proposition}
\begin{proof}
Clearly $S_{\omega }\otimes\bbc$ is the complex sub-vector space of $H^1(X,\bbc )$ generated by $\omega$ and $\overline{\omega }$, i.e. a holomorphic one-form and its complex conjugate. Hence the Hodge components of elements of $S_{\omega }\otimes\bbc$ are again in $S_{\omega }\otimes\bbc$, and Lemma \ref{LemmaOnSubHS} can be applied.
\end{proof}
\begin{Proposition}
Let $X$ and $\omega$ as above and let $\varphi\in\Aff^+(X,\omega )$ with linear part $A=D\varphi\in\SL (X,\omega )$. Then
\begin{equation*}
\begin{pmatrix}
\varphi^{\ast}\Re\omega \\
\varphi^{\ast}\Im\omega
\end{pmatrix}
=A\cdot
\begin{pmatrix}
\Re\omega\\
\Im\omega
\end{pmatrix}.
\end{equation*}
So $\varphi^{\ast }:H^1(X,\bbr )\to H^1(X,\bbr )$ sends the canonical subspace $S_{\omega}$ to itself; the induced map $S_{\omega }\to S_{\omega }$ is, in the basis $(\Re\omega ,\Im\omega )$, given by the matrix $A^{\mathsf{T}}$.
\end{Proposition}
\begin{proof}
It suffices to check this locally in translation charts. But a simple computation confirms that for the real linear map $A:\bbc\to\bbc$ one has
\begin{equation*}
\begin{pmatrix}
A^{\ast}\de x\\
A^{\ast}\de y
\end{pmatrix}
=A\cdot
\begin{pmatrix}
\de x\\
\de y
\end{pmatrix},
\end{equation*}
which can of course be written shorter as $A^{\ast}\de z=A\cdot\de z$.

But this means that the matrix of $\varphi^{\ast }$ in the given basis is the transpose of $A$, not $A$ itself!
\end{proof}

\begin{Eigenvalues}
The aforementioned properties of the canonical subspace have strong implications for the eigenvalues of elements of the Veech group:
\end{Eigenvalues}
\begin{Corollary}\label{EigenvaluesOfAffineMapsAreAlgIntegers}
Let $X$ and $\omega$ as above, and let $A\in\SL (X,\omega )$. Then the eigenvalues of $A$ are algebraic integers.
\end{Corollary}
\begin{proof}
Write $A=D\varphi$ for $\varphi\in\Aff^+(X,\omega )$. By the above proposition we have
\begin{equation*}
\Spec A=\Spec (\varphi^{\ast}|S_{\omega })\subseteq\Spec (\varphi^{\ast}|H^1(X,\bbr ))=\Spec (\varphi^{\ast }|H^1(X,\bbz )).
\end{equation*}
Now clearly the characteristic polyonomial of $\varphi^{\ast }$ on $H^1(X,\bbz )$ is monic with integer coefficients, so all the eigenvalues of $\varphi^{\ast }$ are algebraic integers.
\end{proof}
We also note a general property of eigenvalues on cohomology, which has nothing particular to do with translation structures.
\begin{Proposition}\label{EigenvaluesComeInPairs}
Let $\varphi$ be an orientation-preserving diffeomorphism of $\Sigma_g$. Let $\lambda\neq\pm 1$ be an eigenvalue of $\varphi^{\ast }$ acting on $H^1(\Sigma_g,\bbz )$. Then $\lambda^{-1}$ is also an eigenvalue, with the same multiplicity.
\end{Proposition}
\begin{proof}
Choose a standard basis of $H^1(X,\bbz )$; then the intersection pairing on $H^1(X,\bbz )$ is represented by the matrix
\begin{equation*}
J=\begin{pmatrix}
0&-I_g\\
I_g&0
\end{pmatrix}.
\end{equation*}
Since $\varphi^{\ast }$ preserves the intersection pairing, it is represented by a matrix $M$ with $M^{\mathsf{T}}JM=J$. This implies that $M$ is conjugate to the transpose of $M^{-1}$, hence $M$ and $M^{-1}$ have the same characteristic polynomial. But then the characteristic polynomial of $M$ satisfies the following equation:
\begin{equation*}
\begin{split}
\chi (M,x)
&=\det (M-x\cdot I)\\
&=\det M\cdot \det (I-xM^{-1})\\
&=x^{2g}\det (M^{-1}-\frac{1}{x}I)\\
&=x^{2g}\chi (M,x).
\end{split}
\end{equation*}
This implies that if $\lambda$ is a root of multiplicity $n$ of $\chi (M,x)$, then so is $\lambda^{-1}$.
\end{proof}
\begin{TrichotomyForElementsOfTheVG}
Let us take a closer look at the eigenvalues of an element $A\in\SL (X,\omega )$ for a translation surface $(X,\omega )$, according to whether $A$ is elliptic, parabolic or hyperbolic.

If $A$ is elliptic, it must be of finite order since the Veech group $\SL (X,\omega )$ is Fuchsian. Hence it is conjugate to a matrix
\begin{equation*}
\begin{pmatrix}
\cos (2\pi /n) &-\sin (2\pi /n)\\
\sin (2\pi /n) &\cos (2\pi /n)
\end{pmatrix}
\end{equation*}
for some positive integer $n$. Hence the eigenvalues of $A$ are $\zeta$ and $\overline{\zeta }$, where $\zeta =\exp (2\pi\mathrm{i}/n)$. Now since the characteristic polynomial of $\varphi^{\ast }$ on $H^1(X,\bbz )$ has integer coefficients, all Galois conjugates of $\zeta$ must occur in the spectrum of $\varphi^{\ast}$. We fix this:
\end{TrichotomyForElementsOfTheVG}
\begin{Proposition}
Let $X$ and $\omega$ as above, and let $\varphi\in\Aff^+(X,\omega )$ be such that $A=D\varphi\in\SL (X,\omega )$ is elliptic. Then $A$ is of finite order say $n$, and all primitive $n$-th roots of unity occur as eigenvalues of $\varphi^{\ast }$ on $H^1(X,\bbz )$. \hfill $\square$
\end{Proposition}
Next we consider the parabolic case. If $A=D\varphi$ is parabolic, it has either a double eigenvalue $+1$ or a double eigenvalue $-1$. In any case there is a unique (up to scale) eigenvector $v\in\bbr^2$; let $\vartheta$ be the direction of that eigenvector. Then \cite[Prop. 2.4]{Veech89} says that $\vartheta$ is a Jenkins-Strebel direction for $\vartheta$, and that some positive power of $\varphi$ is a product of powers of Dehn twists along the cylinders in the associated complete cylinder decomposition. Now a Dehn twist operates on cohomology by a unipotent map; the Dehn twists about the different cylinders commute, and hence so do their induced maps on cohomology. A product of commuting unipotent linear maps is unipotent. So a positive power of $\varphi^{\ast }$ acting on $H^1(X,\bbz )$ is unipotent. Recall that a linear endomorphism is called quasi-unipotent if some positive power of it is unipotent or, equivalently, if all its eigenvalues are roots of unity.\footnote{The equivalence of the two definitions follows from the existence of the Jordan decomposition}
\begin{Proposition}
Let $X$ and $\omega$ as above, and let $\varphi\in\Aff^+(X,\omega )$ be such that $A=D\varphi\in\SL (X,\omega )$ is parabolic. Then the induced map $\varphi^{\ast }:H^1(X,\bbz )\to H^1(X,\bbz )$ is quasi-unipotent, but not of finite order.\hfill $\square$
\end{Proposition}

It remains to consider the hyperbolic case which is most interesting from an arithmetic point of view. So let $A=D\varphi$ be hyperbolic; then it has two distinct positive eigenvalues whose product is one, say $\lambda$ and $\lambda^{-1}$, where $\lambda >1$. Now by Corollary \ref{EigenvaluesOfAffineMapsAreAlgIntegers} both $\lambda$ and $\lambda^{-1}$ are algebraic integers. In particular $\lambda$ must be irrational: if it were rational, it could only be plus or minus one (these are the only rational integers whose inverses are also rational integers). But then the trace of $A$ would be plus or minus two, and $A$ would not be hyperbolic.

By what we have seen above, $\lambda$ and $\lambda^{-1}$ are also eigenvalues of $\varphi$ acting on the first cohomology of $X$. It turns out that they appear with multiplicity one and that $\lambda$ is the leading eigenvalue:
\begin{Proposition}\label{EigenvaluesOfHyperbolicInVGActingOnH}
Let $X$ and $\omega$ as above, and let $\varphi\in\Aff^+(X,\omega )$ be such that $A=D\varphi\in\SL (X,\omega )$ is hyperbolic. Let $\lambda$ and $\lambda^{-1}$ be the eigenvalues of the matrix $A$, where $\lambda >1$. Then $\lambda$ and $\lambda^{-1}$ each occur as simple eigenvalues of $\varphi^{\ast}$ acting on $H^1(X,\bbz )$, and all other eigenvalues $\lambda '\in\bbc$ satisfy $\lambda^{-1}<|\lambda '|<\lambda$.

The direct sum of the eigenspaces for $\lambda$ and $\lambda^{-1}$ in $H^1(X,\bbr )$ is precisely the canonical subspace $S_{\omega }$.
\end{Proposition}
\begin{proof}
This is \cite[Thm. 5.3]{McMullen03} together Proposition \ref{EigenvaluesComeInPairs}.
\end{proof}

\newpage
\section{Trace Fields and Primitivity}

We now introduce an important invariant of translation surfaces: the trace field of the Veech group. The trace field of a Fuchsian group $\Gamma$ is the subfield of $\bbr$ generated by all traces of elements of $\Gamma$. In general one cannot say very much about these fields. But if $\Gamma$ is the Veech group of a closed translation surface of genus $g$, then the trace field of $\Gamma$ is an algebraic number field of degree at most $g$. The trace field then also admits a more geometric interpretation: it is the subfield of $\bbr$ generated by all cross-ratios of directions of saddle connections. We quote several theorems about the behaviour of trace fields, in order to outline their significance in the theory of translation surfaces. For us, however, the trace field will be important for understanding the structure of the Jacobian of a Veech surface; this will be done in chapter 7.

Finally we compute the trace fields of the Veech groups of the Wiman curves. It turns out that the trace field of $\SL (W_g,\omega_k)$ is generated by $\cos\frac{2\pi k}{n}$, where $n=2g+1$, as always. This is done by using only elementary trigonometry.

\subsection{Trace Fields}

In this section we review some algebraic invariants of Fuchsian groups, with particular attention towards Veech groups.
\begin{QuaternionAlgebras}
The reference for this paragraph is \cite[Chapter IX]{Weil74}.

Let $k$ be any field of characteristic zero. In what follows, algebras are always tacitly assumed to be associative and unital. A \textit{central simple algebra} over $k$ is a finite-dimensional $k$-algebra $A$, such that the center of $A$ is precisely $k\cdot 1$ (which we identify with $k$) and such that $A$ has no both-sided ideals other than $(0)$ and $A$.

For every $n\ge 1$ the algebra $M_n(k)$ of $n\times n$ matrices is central simple. If $k$ is algebraically closed, then every central simple $k$-algebra is isomorphic to some $M_n(k)$. In general if $K|k$ is any extension of fields of characteristic zero and if $A$ is a finite-dimensional $k$-algebra, then $A$ is central simple if and only if $A\otimes K$ is central simple. Hence a finite-dimensional $k$-algebra $A$ is central simple if and only if there is a finite extension $K|k$ such that $A\otimes K\simeq M_n(K)$ as a $K$-algebra. In particular the dimension of a central simple algebra is always a square.

The classification of central simple algebras over $k$ is equivalent to the classification of skew field extensions of $k$:
\end{QuaternionAlgebras}
\begin{Theorem}[Wedderburn]
Let $A$ be a central simple algebra over a field $k$ of characteristic zero. Then there exist a finite-dimensional $k$-algebra $D$ which is a division algebra and has $k$ as its center, and an integer $n\ge 1$ such that $A\simeq M_n(D)$ as $k$-algebras. The algebra $D$ is unique up to isomorphism.

Vice versa, any such $M_n(D)$ is a central simple $k$-algebra.\hfill $\square$
\end{Theorem}

Since the dimension of a central simple algebras is always a square, the simplest nontrivial examples are those of dimension four; these are called \textit{quaternion algebras}. The name is derived from Hamilton's quaternions which form a quaternion algebra over $\bbr$.

From Wedderburn's Structure Theorem we easily deduce:
\begin{Lemma}\label{CriterionForRamification}
Let $k$ be a field of characteristic zero, and let $A$ be a quaternion algebra over $k$. Then exactly one of the following statements is true:
\begin{enumerate}
\item As a $k$-algebra, $A\simeq M_2(k)$.
\item $A$ is a division algebra.\hfill $\square$
\end{enumerate}
\end{Lemma}
Another fact about central simple algebras is worth mentioning.
\begin{Theorem}[Skolem-Noether]
Let $A$ be a central simple algebra over $k$, and let $\varphi$ be a $k$-algebra automorphism of $A$. Then $\varphi$ is an inner automorphism, i.e. there exists some invertible element $a\in A$ such that $\varphi (x)=a^{-1}xa$ for all $x\in A$.\hfill $\square$
\end{Theorem}

\begin{ElementarySubgroups}
A subgroup $\Gamma\subseteq\SL_2(\bbc )$ is called \textit{elementary} if there is a finite $\Gamma$-orbit in the closure of hyperbolic three-space $\mathbf{H}^3$, i.e. in $\mathbf{H}^3\cup\mathbb{P}^1(\bbc )$. For subgroups of $\SL_2(\bbr )$, this can also be characterized by their action on the hyperbolic plane $\mathbf{H}^2=\bbh$ (here we mean the usual action by M\"{o}bius transformations on the upper half plane): $\Gamma\subseteq\SL_2(\bbr )$ is elementary if and only if at least one of the following statements hold:
\begin{enumerate}
\item $\Gamma$ fixes a point in $\overline{\mathbf{H}^2}=\bbh\cup\mathbb{P}^1(\bbr )$;
\item $\Gamma$ preserves a geodesic in $\mathbf{H}^2$.
\end{enumerate}
A subgroup $\Gamma\subseteq\SL_2(\bbr )$ is elementary if and only if it is solvable as an abstract group. Every Fuchsian group of finite covolume is non-elementary. Every non-elementary subgroup of $\SL_2(\bbr )$ contains an infinite set of hyperbolic elements such that no two of them share a fixed point in the closure of $\mathbf{H}^3$.\footnote{The last statement is \cite[Thm. 5.1.3]{Beardon83}}
\end{ElementarySubgroups}
\begin{TraceFields}
Let $\Gamma\subseteq\SL_2(\bbr )$ be a subgroup. The \textit{trace field} of $\Gamma$, denoted $\bbq (\tr\Gamma )$, is the subfield of $\bbr$ generated by all traces $\tr\gamma$ for $\gamma\in\Gamma$.

Clearly the trace field is a conjugacy invariant. When $\Gamma$ is conjugate to a subgroup of $\SL_2(k)$ for some field $k\subseteq\bbr $, then the trace field will be contained in $k$. The converse is not necessarily true, but we can always conjugate a non-elementary subgroup into $\SL_2(k)$ for $k$ a finite extension of its trace field. In order to see this we have to introduce another algebraic object associated with a Fuchsian group. Consider the space of real $2\times 2$-matrices $M_2(\bbr )$ as a vector space (in fact, as an algebra) over the trace field $\bbq (\tr\Gamma )$, and consider $\Gamma$ as a subset of $M_2(\bbr )$. Then let $A_0\Gamma$ be the sub-$\bbq (\tr\Gamma )$-vector space of $M_2(\bbr )$ generated by $\Gamma$.
\end{TraceFields}
\begin{Theorem}\label{QuaternionAlgebraOfNonelementaryGroup}
Let $\Gamma$ be non-elementary. The space $A_0\Gamma$ is a sub-$\bbq (\tr\Gamma )$-algebra of $M_2(\bbr )$. As such, it is a quaternion algebra over $\bbq (\tr\Gamma )$.

Let $\gamma$ and $\delta$ be two elements of $\Gamma\smallsetminus \{\pm I\}$ without a common fixed point. Then the elements $1$, $\gamma$, $\delta$, $\gamma\delta$ form a $\bbq (\tr\Gamma )$-basis of $A_0\Gamma$.
\end{Theorem}
\begin{proof}
This is a special case of \cite[Cor. 3.2.3]{MaclachlanReid03}.
\end{proof}

\begin{Corollary}
Let $\Gamma\subseteq\SL_2(\bbr )$ be a non-elementary subgroup with trace field $k=\bbq (\tr\Gamma )$ and let $\gamma\in\Gamma$ be hyperbolic. Let the eigenvalues of $\gamma$ be $\lambda$ and $\lambda^{-1}$. Then $\Gamma$ is conjugate within $\SL_2(\bbr )$ to a subgroup of $\SL_2(k(\lambda ))$.
\end{Corollary}
The proof we give follows the argument in \cite[p. 115]{MaclachlanReid03}.
\begin{proof}
Let $\delta\in\Gamma$ be another hyperbolic element such that $\delta$ and $\gamma$ have no common fixed point. After conjugation we may assume that
\begin{equation*}
\gamma =\begin{pmatrix}
\lambda &0\\
0&\lambda^{-1}
\end{pmatrix},\quad
\delta =\begin{pmatrix}
a&b\\
c&d
\end{pmatrix}.
\end{equation*}
Here $b$ and $c$ are nonzero (else $\gamma$ and $\delta$ would have a common fixed point). Also after further conjugation with a diagonal matrix (which does not change $\gamma$) we may assume that $b=1$. By definition of the trace field, we find that $a+d\in k$ and (using $\tr (\gamma\delta )\in k$) that $\lambda a+\lambda^{-1}d\in k$. So $a$, $d$ and $c=ad-1$ are elements of $k(\lambda )$.

From this and the second part of Theorem\ref{QuaternionAlgebraOfNonelementaryGroup}, we see that after conjugation we may assume that $A_0\Gamma\subseteq M_2(k(\lambda ))$, and hence $\Gamma\subseteq \SL_2(k(\lambda ))$.
\end{proof}
Note that by the Cayley-Hamilton theorem, $\lambda$ satisfies a quadratic equation over $k$:
\begin{equation*}
\lambda^2-(\tr\lambda )\lambda +1=0,
\end{equation*}
so the degree of the extension $k(\lambda )|k$ is at most two.
\begin{Proposition}
Let $\Gamma\subseteq\SL_2(\bbr )$ be non-elementary and assume that it contains a parabolic element. Let $k=\bbq (\tr\Gamma )$. Then the quaternion algebra $A_0\Gamma$ is isomorphic to $M_2(k)$ as a $k$-algebra.
\end{Proposition}
\begin{proof}
Let $\gamma\in\Gamma$ be parabolic, then $1-\gamma$ is a nonzero, non-invertible element of $A_0\Gamma$. Apply Lemma \ref{CriterionForRamification}.
\end{proof}
\begin{Corollary}\label{GroupWithParabolicIntoSLTwoTraceField}
Let $\Gamma\subseteq\SL_2(\bbr )$ be a non-elementary subgroup containing a parabolic element. Then $\Gamma$ is conjugate within $\SL_2(\bbr )$ to a subgroup of $\SL_2(k)$, where $k=\bbq (\tr\Gamma )$.
\end{Corollary}
\begin{proof}
By the previous proposition we know that $A_0\Gamma$ is abstractly isomorphic to $M_2(k)$. By tensoring with $\bbr$ we obtain an isomorphism
\begin{equation*}
M_2(\bbr )=A_0\Gamma\otimes_k\bbr \overset{\simeq}{\to}M_2(k)\otimes_k\bbr =M_2(\bbr ).
\end{equation*}
But by the Skolem-Noether theorem, every automorphism of a matrix algebra is inner. In other words we get a matrix $M\in\GL_2(\bbr )$ with $M\cdot A_0\Gamma\cdot M^{-1}=M_2(k)$. Possibly replacing $M$ by
\begin{equation*}
\begin{pmatrix}
-1&0\\
0&1
\end{pmatrix}
\cdot M,
\end{equation*}
we can assume that $M\in\GL_2^+(\bbr )$. Then we can replace $M$ by
\begin{equation*}
\frac{1}{\sqrt{\det M}}\cdot M
\end{equation*}
and finally assume that $M\in\SL_2(\bbr )$.

But then also $M\Gamma M^{-1}\subset M\cdot A_0\Gamma\cdot M^{-1}=M_2(k)$, and thus
\begin{equation*}
M\Gamma^{-1}M^{-1}\subseteq M\cdot A_0\Gamma\cdot M^{-1}\cap\SL_2(\bbr )=\SL_2(k).
\end{equation*}
\end{proof}

\begin{TraceFieldsOfVeechGroups}
After these general considerations we now concentrate on Veech groups of translation surfaces. Let $X$ be a compact Riemann surface and let $\omega$ be an abelian differential on $X$. We have already seen that every eigenvalue of every element in the Veech group $\SL (X,\omega )$, and hence also its trace, is an algebraic integer. Therefore the trace field of $\SL (X,\omega )$ is an algebraic extension of $\bbq$. But we can say more:
\end{TraceFieldsOfVeechGroups}
\begin{Proposition}[McMullen]
Let $X$ be a closed Riemann surface of genus $g$ and let $\omega$ be an abelian differential on $X$. Then the trace field $\bbq (\tr\SL (X,\omega ))$ is a finite extension of $\bbq$, of degree at most $g$.
\end{Proposition}
\begin{proof}
Let $k$ be the sub-$\bbq$-vector space of $\bbr$ generated by the traces of $\Gamma$. We claim that $k$ is the trace field.

To start with, $k$ is a $\bbq$-algebra because of the formula
\begin{equation*}
(\tr\gamma )(\tr\delta )=\tr (\gamma\delta )+\tr (\gamma\delta^{-1}).
\end{equation*}
Also it is a finite-dimensional $\bbq$-vector space: it is contained in the image of the $\bbq$-linear map
\begin{equation*}
\{ f\in\End_{\bbq }H^1(X,\bbq )\, |\, f_{\bbr }(S_{\omega })\subseteq S_{\omega }\}\to\bbr,\quad T\mapsto\tr (T|S_{\omega }).
\end{equation*}
Hence $k$ is a finite field extension of $\bbq$, and equal to the trace field of $\SL (X,\omega )$. It remains to show the degree bound. By the theorem of the primitive element there is some $t\in k$ with $k=\bbq (t)$. We can write
\begin{equation*}
t=\sum_{i=1}^na_i\tr (D\varphi_i)
\end{equation*}
for $a_i\in\bbq$ and $\varphi_i\in\Aff^+(X,\omega )$. Define then an endomorphism $T$ of $H^1(X,\bbq )$ by
\begin{equation*}
T=\sum_{i=1}^na_i(\varphi_i^{\ast }+\varphi_i^{-1,\ast }).
\end{equation*}
Now in $\SL_2(\bbr )$ the following equation holds:
\begin{equation*}
A+A^{-1}=(\tr A)\cdot I.
\end{equation*}
Hence $T$ acts as $t$ times the identity on $S_{\omega }$. Hence $t$ is an eigenvalue of $T$ with multiplicity $\ge 2$. Thus any Galois conjugate of $t$ is also an eigenvalue with multiplicity $\ge 2$. Since the dimension of $H^1(X,\bbq )$ is $2g$, this means that $t$ can have at most $g$ Galois conjugates, i.e. it can have at most degree $g$ over $\bbq$.
\end{proof}
This result is special for Veech groups. By using Fricke coordinates on Teichm\"{u}ller space, we see that a ``generic'' Fuchsian group of finite covolume has transcendental traces. The situation changes in dimension three, i.e. for Kleinian groups acting on hyperbolic three-space: by Mostow Rigidity, every finite covolume Kleinian group has an algebraic number field as trace field, see
\cite[Thm. 3.1.2]{MaclachlanReid03}.

\begin{TheCrossRatioField}
Gutkin and Judge in \cite{GutkinJudge00} give an alternative description of the trace field of a Veech group, closer to the geometry of translation surfaces. Let $X$ be a closed Riemann surface with an abelian differential $\omega$, and consider the set of all slopes of saddle connections $\mathcal{D}\subset\mathbb{P}^1(\bbr )$. Recall that for four disctinct points $x_1,\ldots ,x_4\in\mathbb{P}^1(k )$, where $k$ is any field, there exists a unique M\"{o}bius transformation $\gamma\in\operatorname{PGL}_2( k)$ with $\gamma (x_1)=0$,
$\gamma (x_2)=1$ and $\gamma (x_3)=\infty$. The element $\gamma (x_4)\in\mathbb{P}^1(k )\smallsetminus \{ 0,1,\infty \}$, where we identify the latter with $k\smallsetminus\{ 0,1\}$, is called the cross-ratio of the four points and denoted by $\mathrm{CR} (x_1,\ldots ,x_4)$. Now let $K(X,\omega )$ be the subfield of $\bbr$ generated by all cross-ratios of families of four distinct elements in $\mathcal{D}$. Then Theorem 7.1 in op.cit. reads:
\end{TheCrossRatioField}
\begin{Theorem}[Gutkin, Judge]\label{CrossRatioEqualsTraceField}
The trace field of $\SL (X,\omega )$ is contained in $K(X,\omega )$. If $(X,\omega )$ is a Veech surface, the two fields are equal.
\end{Theorem}
The trace field of a lattice Veech group can be explicitly computed as soon as one knows a single hyperbolic element:
\begin{Proposition}\label{TraceFieldCanBeComputedByOneHyperbolic}
Let $(X,\omega )$ be a Veech surface and let $A$ be a hyperbolic element of $\SL (X,\omega )$. Then the trace field of $\SL (X,\omega )$ is generated by $\tr A$.
\end{Proposition}
\begin{proof}
This follows from the Theorem of Gutkin and Judge just quoted together with Theorem 28 in the appendix of \cite{KenyonSmillie00}.
\end{proof}
\begin{Corollary}\label{TraceFieldSameForFiniteIndex}
Assume that $(X,\omega )$ is a Veech surface, and let $\Gamma$ be a finite index subgroup of $\SL (X,\omega )$. Then the trace fields of $\SL (X,\omega )$ and of $\Gamma$ are the same.\hfill $\square$
\end{Corollary}
In general the trace field of a finite covolume Fuchsian group is not stable under passing to a finite index subgroup.\footnote{Here is a counterexample, taken from \cite[Exercise 3.3.6]{MaclachlanReid03}. Consider the Fuchsian triangle group $\Gamma =\Delta (2,3,8)$ and its subgroup $\Gamma^{(2)}=\langle\gamma^2|\gamma\in\Gamma\rangle$. This is a finite index subgroup by op.cit., Lemma 3.3.3. By formula (3.25) in op.cit. we see that
\begin{equation*}
\bbq (\tr\Gamma )=\bbq (2\cos\frac{2\pi}{2},2\cos\frac{2\pi}{3},2\cos\frac{2\pi}{8})=\bbq (-2,-1,2\sqrt{2})=\bbq(\sqrt{2} ).
\end{equation*}
By the remark right before formula (3.28) in op.cit., we see that
\begin{equation*}
\bbq (\tr\Gamma^{(2)})=\bbq (2\cos\frac{4\pi}{2},2\cos\frac{4\pi}{3},2\cos\frac{4\pi}{8})=\bbq (2,-1,0)=\bbq .
\end{equation*}
} Also worthwile to mention is the following:
\begin{Proposition}
Let $(X,\omega )$ be a Veech surface. Then the Veech group $\SL (X,\omega )$ is conjugate to a subgroup of $\SL_2(k)$, where $k$ is the trace field of $\SL (X,\omega )$.
\end{Proposition}
\begin{proof}
Since $\SL (X,\omega )$ is a non-cocompact lattice, it contains a parabolic element. Now apply Corollary \ref{GroupWithParabolicIntoSLTwoTraceField}.
\end{proof}

\subsection{Primitive and Imprimitive Abelian Differentials}

Given a translation surface $(X,\omega )$, one can construct translation surfaces in higher genera which do not differ very much in their behaviour under the $\SL_2(\bbr )$-operation from $(X,\omega )$. We start with a covering which is ramified over periodic points:
\begin{Definition}
Let $X$ be a compact Riemann surface and let $\omega$ be an abelian differential on $X$. A point $x\in X$ is called \textup{periodic} if its orbit under $\Aff^+(X,\omega )$ is finite.
\end{Definition}
Since every affine map permutes the zeros of $\omega$ and there are only finitely many zeros, we see that every zero of $\omega$ is periodic in this sense.

For a less trivial example, consider the Wiman curve $W_g$ with the abelian differential $\omega_g$. Recall that it can be interpreted as a regular $2n$-gon with opposite sides identified. By our determination of the Veech group, we find that the points above $\infty$ and the $n$-th roots of unity, i.e. the fixed points of the hyperelliptic involution, are periodic. The same holds for $(X_g,\omega_1)$, i.e. two regular $n$-gons with respective sides identified.
\begin{Definition}
Let $(X,\omega )$ and $(Y,\eta )$ be compact Riemann surfaces with abelian differentials. A finite holomorphic map $f:X\to Y$ is called a \textup{translation cover} if it satisfies the following conditions:
\begin{enumerate}
\item $f^{\ast }\eta =\omega$,
\item (only for $g(Y)>1$) the ramification locus of $f$ in $Y$ consists entirely of periodic points,
\item (only for $g(Y)=1$) after fixing a suitable base point in the elliptic curve $Y$, the ramification locus of $f$ only consists of torsion points.\footnote{equivalently: for all $y_1,y_2\in Y$, the difference $[y_1]-[y_2]$ in the Jacobian of $Y$ is torsion}
\end{enumerate}
\end{Definition}
Also here the Wiman curves provide a nice series of examples. Let as always $n=2g+1$, and let $t$ be a divisor of $n$ with $1\le t\le g$. Then the quotient $n/t$ is also odd, say equal to $2h+1$ for some nonnegative integer $h$. Consider then the following map between Wiman curves:
\begin{equation*}
f:W_g\to W_h,\quad (x,y)\mapsto (x^t,y).
\end{equation*}
This is easily checked to be well-defined. Pulling back our standard differentials, we obtain
\begin{equation*}
f^{\ast }\omega_1=\frac{\de (x^t)}{y}=t\cdot\omega_t\text{ and }f^{\ast }\omega_h=\frac{(x^t)^{h-1}\de (x^t)}{y}=t\cdot\omega_{th} .
\end{equation*}
Veech groups, and hence also mirror Veech groups, behave nicely under translation covers:
\begin{Theorem}[Gutkin, Judge]\label{VeechGroupsCommensurable}
Let $f:(X,\omega )\to (Y,\eta )$ be a translation cover. Then $\SL (X,\omega )$ and $\SL (Y,\eta )$ are commensurable.
\end{Theorem}
\begin{proof}
This is \cite[Thm. 4.9]{GutkinJudge00}.
\end{proof}
Hence if a translation surface is a translation cover of another translation surface, then one of them is Veech if and only if the other is. In particular we see:
\begin{Corollary}
Let $g\ge 2$, let $1\le t\le g$ be a divisor of $n=2g+1$. Write $n/t=2h+1$. Then $(W_g,\omega_t)$ and $(W_g,\omega_{th})$ are Veech surfaces.
\end{Corollary}
We shall see later that this gives a complete list of the pairs $(g,k)$ such that $(W_g,\omega_k)$ is a Veech surface.

\begin{Definition}
A translation surface $(X,\omega )$ is called \textup{geometrically primitive} if at least one of the following statements holds:
\begin{enumerate}
\item the genus of $X$ is one;
\item every translation cover $f:(X,\omega )\to (Y,\eta)$ is already an isomorphism (equivalently\footnote{by the Riemann-Hurwitz formula}: there does not exists a translation cover $f:(X,\omega )\to (Y,\eta )$ with $g(Y)<g(X)$).
\end{enumerate}
\end{Definition}
Every Veech surface covers a unique geometrically primitive translation surface, in the following sense:
\begin{Proposition}[McMullen]\label{PrimitiveVeechSurfaces}
Let $(X,\omega )$ be a Veech surface. Then there exist a geometrically primitive translation surface $(X_0,\omega_0)$ and a translation cover $f:(X,\omega )\to (X_0,\omega_0)$. For every such $(X_0,\omega_0)$, one has $\SL(X,\omega )\subseteq\SL (X_0,\omega_0 )$.

If $g(X_0)>1$, these data are unique up to isomorphism. If $g(X_0)=1$ and in addition one demands that $f$ induces a surjection on fundamental groups, they are also unique.
\end{Proposition}
\begin{proof}
This is a synthesis of Thm. 2.1, Cor. 2.2 and Cor. 2.3 from \cite{McMullen06}.
\end{proof}
The trace field of the Veech group, as discussed in the previous section, provides a powerful tool for deciding whether a given Veech surface is geometrically primitive. This is based in the following simple observation:
\begin{Corollary}
Let $f:(X,\omega )\to (Y,\eta )$ be a translation cover of Veech surfaces. Then $\SL (X,\omega )$ and $\SL (Y,\eta )$ have the same trace field.
\end{Corollary}
\begin{proof}
This follows from Theorem \ref{VeechGroupsCommensurable} and Corollary \ref{TraceFieldSameForFiniteIndex}.
\end{proof}
Recall that the trace field of a Veech group is always an algebraic number field, of degree at most the genus of the surface. Now if $(X,\omega )$ is a Veech surface and $(X_0,\omega_0)$ is as in Proposition \ref{PrimitiveVeechSurfaces}, then the degree of the trace field of $\SL(X,\omega )$ is at most the genus of $X_0$. So we get the following important chain of inequalities (always provided that $(X,\omega )$ is a Veech surface):
\begin{equation}\label{InequalitiesForVeechSurfaces}
1\le [\bbq (\tr\SL (X,\omega )):\bbq ]\le g(X_0)\le g(X).
\end{equation}
Note that $(X,\omega )$ is geometrically primitive if and only if in the last step equality holds. So the following definition is sensible:
\begin{Definition}
A translation surface $(X,\omega )$ is called \textup{algebraically primitive} if the degree of the trace field of $\SL (X,\omega )$ is equal to the genus of $X$.
\end{Definition}
For Veech surfaces, algebraic primitivity implies geometric primitivity by (\ref{InequalitiesForVeechSurfaces}); the converse is wrong. For a counterexample see \cite[section 2, example iii]{Moeller06}.

Consider again the Wiman curves $(W_g,\omega_1)$, i.e. consider two regular $g$-gons with respective sides identified, where $g=2n+1$ is odd. The genus of $W_g$ is $g$, and the degree of the trace field is $\varphi (n)/2$, where $\varphi$ is Euler's totient function. These two integers are equal if and only if $n=2g+1$ is prime. Hence $(W_g,\omega_1)$ is algebraically primitive if and only if $2g+1$ is prime.

We finally note that the case where equality holds in the first step of (\ref{InequalitiesForVeechSurfaces}) is particularly interesting: then it also holds in the second step.
\begin{Theorem}
Let $(X,\omega )$ be a Veech surface. Then the following are equivalent:
\begin{enumerate}
\item the trace field of $\SL (X,\omega )$ is $\bbq$;
\item the cross-ratio field $K(X,\omega )$ as discussed before Theorem \ref{CrossRatioEqualsTraceField} is $\bbq$;
\item the Veech group $\SL (X,\omega )$ is commensurable to $\SL_2(\bbz)$;
\item $(X,\omega)$ is in the $\GL_2^+(\bbr )$-orbit of an origami;
\item $(X,\omega )$ can be tiled by translations of one and the same Euclidean parallelogram;
\item the genus of $X_0$ as in Prop. \ref{PrimitiveVeechSurfaces} is one.
\end{enumerate}
\end{Theorem}
\begin{proof}
The equivalence of (i) and (ii) is a special case of Theorem \ref{CrossRatioEqualsTraceField}. The equivalence of (ii), (iii), (iv) and (v) is \cite[Theorem 5.5]{GutkinJudge00}. Finally the equivalence of (v) and (vi) is obvious.
\end{proof}

\subsection{The Trace Fields of the Wiman Curves}

Now we compute the trace fields of our examples. This can be done by using their explicit description as translation surfaces. The occurring fields will be the maximal real subfields of cyclotomic fields, hence we first summarize some elementary properties of these.

\begin{TheNumberField}
Denote by $\zeta_n$ a primitive $n$-th root of unity; we identify $\zeta_n$ with the complex number $\exp\frac{2\pi\mathrm{i}}{n}$. This is an algebraic integer of degree $\varphi (n)$ (Euler's totient function); its minimal polynomial over $\bbq$ is the $n$-th cyclotomic polynomial
\begin{equation*}
c_n(x)=\prod_{1\le k\le n \atop \gcd (k,n)=1}(x-\zeta_n^k)=\prod_{n=k\ell \atop 1\le k,\ell\le n}(x^k-1)^{\mu (\ell )}.
\end{equation*}
The number field generated by $\zeta_n$ is called the $n$-th cyclotomic field. Its embeddings into $\bbc$ are the
\begin{equation*}
\sigma_k:\bbq (\zeta_n)\to\bbc,\quad \zeta_n\mapsto\exp\frac{2k\pi\mathrm{i}}{n}
\end{equation*}
where $k$ runs through the integers in $\{ 1,\ldots ,n\}$ that are coprime to $n$. In particular we see that the extension $[\bbq (\zeta_n):\bbq ]$ is Galois, and that for $n\ge 3$ the field $\bbq (\zeta_n)$ is totally imaginary.

Assume from now on $n\ge 3$, and assume that $n$ is odd, $n=2g+1$. Then $\bbq (\zeta_n)$ contains the subfield $K_n$ generated by $\xi_n=\zeta_n+\zeta_n^{-1}$; we claim that the former is a quadratic field extension of the latter. Namely any field embedding $K_n\to\bbc$ extends to a field embedding $\bbq (\zeta_n)\to\bbc$, and we have just listed the latter. So every field embedding $K_n\to\bbc$ is given by sending $\xi_n=\zeta_n+\zeta_n^{-1}$ to a complex number of the form
\begin{equation*}
\exp\frac{2k\mathrm{i}}{n}+\exp\frac{-2k\mathrm{i}}{n}.
\end{equation*}
But this remains the same number if $k$ is replaced by $n-k$, so we only have need to consider those $k$ with $1\le k\le g$, coprime to $2g+1=n$. These do give pairwise different embeddings.

For later reference let us write down these embeddings explicitly. For every $1\le k\le g$ with $\gcd (k,2g+1)=1$ set
\begin{equation*}
\tau_k:K_n=\bbq (\xi_n)\to\bbc ,\xi_n\mapsto\exp\frac{2k\mathrm{i}}{2g+1}+\exp\frac{-2k\mathrm{i}}{2g+1}=2\cos\frac{2k\pi }{2g+1}.
\end{equation*}
This shows, among other things, that $2\cos\frac{2\pi}{n}$ is an algebraic integer of degree $\frac{\varphi (n)}{2}$ and that its Galois conjugates are the $2\cos\frac{2k\pi}{n}$. Thus we may identify $K_n$ with the number field generated by $\cos\frac{2\pi }{n}$. Also $K_n$ is totally real, $[K_n:\bbq ]=\frac{\varphi (n)}{2}$, and indeed the extension $\bbq (\zeta_n)|K_n$ is quadratic.
\end{TheNumberField}

Note that this implies the following identities of number fields: for a rational number $\frac{p}{q}$ in lowest terms, say with $0<\frac{p}{q}<1$, the field generated by $\exp (2\pi\mathrm{i}\cdot\frac{p}{q})=\zeta_q^p$ is equal to $\bbq (\zeta_q)$, and the field generated by $\cos (2\pi\cdot\frac{p}{q})$ is equal to
\begin{equation*}
\bbq (\zeta_q)\cap\bbr =\bbq (\cos\frac{2\pi}{q})=K_q.
\end{equation*}
\begin{Lemma}\label{RatiosOfSinesExpressedByCosines}
Let $p$ and $q$ be positive integers, with $q\ge 3$. Then
\begin{equation*}
\frac{\sin\frac{2\pi p}{q}}{\sin\frac{2\pi }{q}}\in\bbq (\cos\frac{2\pi }{q}).
\end{equation*}
\end{Lemma}
\begin{proof}
Recall de Moivre's formula
\begin{equation*}
\sin px=\sum_{k=0}^{\infty }(-1)^k{p\choose 2k+1}\cos^{p-2k-1}x\sin^{2k+1}x
\end{equation*}
(this is a finite sum since all terms with $2k+1>p$ are zero). This directly gives
\begin{equation*}
\frac{\sin px}{\sin x}=\sum_{k=0}^{\infty }(-1)^k{p\choose 2k+1}\cos^{p-2k-1}x\sin^{2k}x
=\sum_{k=0}^{\infty }(-1)^k{p\choose 2k+1}\cos^{p-2k-1}x(1-\cos^2x)^k.
\end{equation*}
Plugging in $x=\frac{2\pi }{q}$ we obtain the desired statement.
\end{proof}
\begin{Lemma}\label{LemmaOnCyclotomicFields}
Let $n\ge 3$ be an integer, and let $f:\bbc\to\bbc$ be an $\bbr$-linear map with the property $f(\bbq (\zeta_n))\subseteq\bbq (\zeta_n)$. Then the trace of $f$ is contained in the number field $\bbq (\cos 2\pi /n)$.
\end{Lemma}
\begin{proof}
Choosing the $\bbr$-basis $(1,\mathrm{i})$ of $\bbc$, we find that
\begin{equation*}
\tr f=\Re f(1)+\Im f(\mathrm{i}).
\end{equation*}
Now $f(1)$ is an element of $\bbq (\zeta_n)$, hence so is $\overline{f(1)}$ and hence also
\begin{equation*}
\Re f(1)=\frac{f(1)+\overline{f(1)}}{2}.
\end{equation*}
Since $\Re f(1)$ is evidently real, it must therefore be an element of $\bbq (\zeta_n)\cap\bbr =\bbq (\cos 2\pi /n)$. So it remains to show that $\Im f(\mathrm{i})$ is an element of $\bbq (\cos 2\pi /n)$.

We use that
\begin{equation*}
\zeta_n-\zeta_n^{-1}=\zeta_n-\overline{\zeta_n}=2\mathrm{i}\Im\zeta_n=2\mathrm{i}\sin\frac{2\pi}{n}
\end{equation*}
(which is nonzero!) and that $f$ is $\bbr$-linear to conclude
\begin{equation*}
f(\mathrm{i})=f\left(\frac{\zeta_n-\zeta_n^{-1}}{2\sin\frac{2\pi}{n}}\right)=\frac{1}{2\sin\frac{2\pi}{n}}f(\zeta_n-\zeta_n^{-1}).
\end{equation*}
By assumption, $f(\zeta_n-\zeta_n^{-1})$ is an element of $\bbq (\zeta_n)$ and can therefore be written as
\begin{equation*}
f(\zeta_n-\zeta_n^{-1})=\sum_{\nu =0}^{\varphi (n)-1}a_{\nu }\zeta_n^{\nu },\quad a_{\nu }\in\bbq .
\end{equation*}
Then
\begin{equation*}
\Im f(\mathrm{i})=\frac{1}{2\sin\frac{2\pi}{n}}\sum_{\nu =0}^{\varphi (n)-1}a_{\nu }\sin\frac{2\pi\nu }{n}=\sum_{\nu =0}^{\varphi (n)-1}a_k\cdot\frac{\sin\frac{2\pi\nu }{n}}{2\sin\frac{2\pi }{n}}.
\end{equation*}
By Lemma \ref{RatiosOfSinesExpressedByCosines} this is an element of $\bbq (\cos\frac{2\pi }{n})$.
\end{proof}
The connection to Wiman curves is given by the following observation:
\begin{Lemma}\label{HolonomyVectorsOfWimanCurves}
Let $n=2g+1$ be odd, and let $1\le k\le g$. Then for a suitable complex number $c\neq 0$, the sub-$\bbq$-vector space $V$ of $\bbc$ generated by all holonomy vectors of saddle connections of $(W_g,c\omega_k)$ is equal to the field $\bbq (\zeta_n^k)$.
\end{Lemma}
\begin{proof}
We choose $c$ such that $(W_g,c\omega_k)$ is isomorphic to $(X(g,k),\omega (g,k))$. Set $\zeta =\zeta_g^k$. The vertices of all triangles used in the construction of the surface are in $\bbq (\zeta )$; the translations which define the necessary identifications of the sides are also in $\bbq (\zeta )$. Hence all holonomy vectors of saddle connections are elements of $\bbq (\zeta )$, and we have $V\subseteq\bbq (\zeta )$.

For the other inclusion we note that for each $p\in\bbz$ there is a saddle connection with holonomy vector $\zeta^p-\zeta^{p-1}$. Thus $V$ is contained in
\begin{equation}
\operatorname{span}_{\bbq }\{ \zeta^p-\zeta^{p-1 }\mid p\in\bbz\} =(\zeta -1)\cdot\operatorname{span}_{\bbq }\{\zeta^t\mid t\in\bbz\} =(\zeta -1)\cdot\bbq (\zeta )=\bbq (\zeta ),
\end{equation}
and the lemma is shown.
\end{proof}
\begin{Proposition}
Let $n=2g+1$ be odd, let $1\le k\le g$ and write
\begin{equation*}
t=\gcd (n,k),\quad n_0=\frac{n}{t}.
\end{equation*}
Then the trace field of the Veech group $\SL (W_g,\omega_k)$ is equal to $\bbq (\cos\frac{2\pi }{n_0})$.
\end{Proposition}
\begin{proof}
First note that the Veech group does not change if $\omega_k$ is replaced by a constant multiple, hence we may replace $\omega_k$ by $c\cdot\omega_k$ such that the previous lemma holds, i.e. such that the translation surface in question is given by glueing triangles as described above. Then there is an obvious element in $\Aff^+(W_g,\omega_k)$ which sends $T_{m,\varepsilon }$ to $T_{m+1,\varepsilon }$; this element has derivative equal to the rotation by $2\pi k/n$. Thus
\begin{equation*}
\begin{pmatrix}
\cos\frac{2\pi k}{n} &-\sin\frac{2\pi k}{n}\\
\sin\frac{2\pi k}{n} & \cos\frac{2\pi k}{n}
\end{pmatrix}
\in\SL (W_g,\omega_k),\text{ hence }\bbq (\tr\SL (W_g,\omega_k))\supseteq\bbq (\cos\frac{2\pi k}{n})=\bbq (\cos\frac{2\pi }{n_0}).
\end{equation*}
It remains to show the converse inclusion.

Let $A\in\SL (W_g,\omega_k)$. Then $A$, considered as an $\bbr$-linear map $\bbc\to\bbc$, sends holonomy vectors to holonomy vectors. By Lemma \ref{HolonomyVectorsOfWimanCurves} this implies that $A\cdot\bbq (\zeta_n^k)=\bbq (\zeta_n^k)$. Note that $\bbq (\zeta_n^k)=\bbq (\zeta_{n_0})$. Then apply Lemma \ref{LemmaOnCyclotomicFields} to conclude that the trace of $A$ is contained in $\bbq (\cos 2\pi /n_0)$. Since this holds for every $A$ in the Veech group, we have the desired inclusion of fields.
\end{proof}
We could also have determined the trace fields of $\SL (W_g,\omega_1)$ and $\SL (W_g,\omega_g)$ with the help of Proposition \ref{TraceFieldCanBeComputedByOneHyperbolic}. But note that the method given here, apart from being elementary, gives the trace field of \textit{all} Wiman curves, regardless of whether or not they are Veech surfaces (it will turn out below that not all of them are).

\newpage
\section{Jacobians: Global Theory}

We now go on to globalize the considerations of chapter 4, i.e. we study how Jacobians of translation surfaces behave under the $\SL_2(\bbr )$-operation. Since we view Jacobians as a sort of ``enriched cohomology'', we first look at the underlying topological phenomena, i.e. the singular cohomology. Given a translation surface $(X,\omega )$, we get a family of Riemann surfaces $f:\mathrsfs{X}\to C$, where $C=\Gamma\backslash\bbh$ for a suitable finite index subgroup $\Gamma$ of the mirror Veech group of $(X,\omega )$. Then the first cohomology groups of the fibres form a local system on $C$. Since vector spaces behave considerably nicer than abelian groups, we take rational instead of integral cohomology (on the level of Jacobians, this means working with abelian varieties up to isogeny). So we have a local system of $\bbq$-vector spaces $R^1f_{\ast }\bbq$ on $C$. On the other hand, the canonical subspaces in real cohomology of each fibre form a local subsystem of $R^1f_{\ast }\bbr$. To reconcile these two points of view, we have to carefully examine extensions of the base field and direct sum decompositions for local systems. A crucial r\^{o}le is played by the notion of \textit{moduli fields} of local systems. In a sense to be made precise, the moduli field of a local system is the smallest field over which it can be defined. For simple subsystems of systems which are defined over $\bbq$, such as the system of cohomology groups introduced above, the moduli field is (under some mild condition) always a number field. This implies that for large enough Veech groups (i.e., such that $\Gamma$ is nonelementary) the canonical subsystem can already be defined over its moduli field. Moreover, this moduli field agrees with the trace field of the Veech group. Also we get Galois conjugates of the canonical subvariation, which is not evident at all from its definition.

Then we take the additional structures on the cohomology spaces, i.e. their Hodge decompositions, into account. This is equivalent to considering the Jacobians as abelian varieties (up to isogeny) and not merely tori. Using deep results of Deligne and Schmid, we find that in the case where the Veech group is a lattice, the canonical subsystem and its Galois conjugates are actually sub-variations of Hodge structure.

\subsection{Local Systems}

The material in this section is classical and well-known, but different authors use different conventions, so we need to make some things quite explicit.
\begin{LSandMonodromyReps}
Let $X$ be a topological space which is path-connected, locally path-connected and locally simply connected, and let $x_0\in X$ be some fixed point. By a local system of sets, groups, vector spaces, whatever ..., we mean a sheaf of sets, groups, vector spaces, whatever ..., which is locally isomorphic to a constant sheaf. To give such an object is equivalent to giving a set, group, vector space, whatever ... with an action of the fundamental group.

To state this well-known principle in a form suitable for explicit computations, introduce the following convention for the fundamental group: $\pi_1(X,x_0)$ consists of the homotopy classes of loops $\alpha :[0,1]\to X$ with $\alpha (0)=\alpha (1)=x_0$, this condition being preserved through the homotopies. If homotopy classes $a,b\in\pi_1(X,x)$ are represented by loops $\alpha ,\beta$, then $ab$ is the homotopy class of the loop $\alpha\beta$ which is defined as \textit{first $\beta$, then $\alpha$}.\footnote{To be very explicit: $\alpha\beta (t)=\beta (2t)$ for $0\le t\le \frac{1}{2}$ and $\alpha\beta (t)=\alpha (2t-1)$ for $\frac{1}{2}\le t\le 1$} This is opposite to the convention used by most topologists, but current in algebraic geometry.

Now let $\mathcal{L}$ be a local system on $X$, and let $a\in\pi_1(X,x_0)$ be represented by $\alpha :[0,1]\to X$. The pullback $\alpha^{\ast }\mathcal{L}$ is a constant sheaf on $[0,1]$, so we get an isomorphism
\begin{equation*}
[a]:\mathcal{L}_{x_0}=(\alpha^{\ast}\mathcal{L})_0\overset{\simeq}{\longleftarrow}\Gamma ([0,1],\mathcal{L})\overset{\simeq}{\longrightarrow}(\alpha^{\ast }\mathcal{L})_1=\mathcal{L}_{x_0}.
\end{equation*}
One checks that this only depends on the homotopy class $a$ of $\alpha$, and that $[ab]=[a]\circ [b]$. Hence we get a left action of $\pi_1(X,x_0)$ on $\mathcal{L}$, called the \textit{monodromy representation}.

This construction gives an equivalence of categories: for instance when talking about local systems of sets, the monodromy representation gives a functor from local systems of sets on $X$ to sets with a left $\pi_1(X,x_0)$-action. When talking about local systems of $k$-vector spaces, it gives a functor from local systems of $k$-vector spaces on $X$ to $k$-vector spaces with a linear left $\pi_1(X,x_0)$-action, and so on. All these functors are equivalences of categories.\footnote{The construction works for all categories of the kind ``sets with some extra structure'' where the morphisms are maps which respect this extra structure in some way. It is possible to make this mathematically sound using machinery from category theory, but this leads to far apart. We shall only need it for groups and vector spaces}
\end{LSandMonodromyReps}
\begin{FundamentalGroupsandDeckTrafos}
The local systems that we shall be interested in are as follows: let $C$ be a Riemann surface and let $p:X\to C$ be a family of Riemann surfaces of genus $g$. We consider the cohomologies of the fibres, i.e. the local systems $R^kp_{\ast }\bbz$ on $C$. The only nontrivial case here is $k=1$; any fibre of $R^1p_{\ast }\bbz$ is a free abelian group of rank $2g$. Hence we get some representation of the fundamental group of $X$ on a free abelian group; this representation will be of great importance later on. As argued in section 1.3, this setup will be particularly interesting in the case where $C$ is hyperbolic, hence where we can write $C=\Gamma\backslash\bbh$ for some torsion-free Fuchsian group $\Gamma\subset\SL_2(\bbr )$. It is common wisdom, but again with messy details, that we can identify the fundamental group of $C$ with $\Gamma$. Here come the messy details.

We need to fix base points; for this we take $\mathrm{i}\in\bbh$ and its image $c_0 =\Gamma\cdot\mathrm{i}\in C=\Gamma\backslash\bbh$. Then $(\bbh ,\mathrm{i})$ is the universal covering space of $(C,c_0)$. Let $a\in\pi_1(C,c_0)$ be represented by a path $\alpha :[0,1]\to C$, then there exists a unique continuous lift $\tilde{\alpha }:[0,1]\to\bbh$ with $\tilde{\alpha }(0)=\mathrm{i}$. Now there is a unique Deck transformation $D_a:\bbh\to\bbh$ with $D_a(\tilde{\alpha }(1))=\tilde{\alpha }(0)=\mathrm{i}$; but the Deck group of the covering $\bbh \to C$ \textit{is} precisely $\Gamma$. So we have $D_a\in\Gamma$. This depends only on $a$, so we get a well-defined map
\begin{equation*}
\pi_1(C,c_0)\to\Gamma ,\quad a\mapsto D_a.
\end{equation*}
This is seen to be an isomorphism of groups.\footnote{The only thing to check is that it is a group homomorphism. Let $a,b\in\pi_1(C,c_0)$ be represented by paths $\alpha$, $\beta$. Then $\tilde{\alpha\beta }=\bar{\alpha }\tilde{\beta }$ where $\bar{\alpha}$ is the unique continuous lift of $\alpha$ satisfying $\bar{\alpha }(0)=\tilde{\beta }(1)=D_b^{-1}(\mathrm{i})$. Since Deck transformations are determined by their action on a single point, we see that $\bar{\alpha }=D_b^{-1}\circ\tilde{\alpha }$. Now one checks that $D_{ab}^{-1}(\mathrm{i})=\tilde{\alpha\beta }(1)=D_b^{-1}(\tilde{\alpha }(1))=D_b^{-1}D_a^{-1}(\mathrm{i})$, hence by the same argument $D_{ab}^{-1}=D_b^{-1}D_a^{-1}$, which implies $D_{ab}=D_aD_b$} Hence composing with the inverse of that isomorphism, we can view the monodromy representation as a representation of the Fuchsian group $\Gamma$.
\end{FundamentalGroupsandDeckTrafos}
\begin{ElementaryConstructions}
Let $C$ be a nice\footnote{meaning path-connected, locally path-connected and locally simply connected} topological space with fixed base point $c_0\in C$; set $\pi =\pi_1(C,c_0)$. We regard this as an abstract group $\pi$ which need not be finite, but we assume it to be finitely generated. Take a field $k$ of characteristic zero, then by the above remarks, local systems of $k$-vector spaces on $X$ can be identified with finite-dimensional $k$-vector spaces $V$ on which $\pi$ operates linearly from the left. We call such spaces \textit{$(\pi ,k)$-modules}. Note that if $k[\pi ]$ is the group algebra, a $(\pi ,k)$-module is the same as a $k[\pi ]$-module whose underlying $k$-vector space is finite dimensional.

We begin with some formalities on $(\pi ,k)$-modules and their interpretation in terms of local systems. So let $C$ be a nice topological space with base point $c_0$ and fundamental group $\pi =\pi_1(C,c_0)$; we denote $(\pi ,k)$-modules by italic letters $A, B, C,\ldots $ and the corresponding local systems of $k$-vector spaces on $X$ by calligraphic letters $\mathcal{A},\mathcal{B},\mathcal{C},\ldots $.
\end{ElementaryConstructions}
\begin{itemize}
\item \textit{Direct sums:} These are defined in the obvious way, both on the representation side and on the local systems side.
\item \textit{Tensor products:} The tensor product of $V$ and $W$ is the vector space $V\otimes_kW$ together with the tensor product of the given actions, i.e. $p(v\otimes w)=pv\otimes pw$; the tensor product of $\mathcal{V}$ and $\mathcal{W}$ is the locally constant sheaf $\mathcal{V}\otimes_k\mathcal{W}$. These constructions correspond to each other under the above equivalence of categories.
\item \textit{Duals:} The dual of $V$ is the vector space $\check{V}=\Hom_k(V,k)$ together with the action $(p\lambda)(v)=\lambda (p^{-1}v)$, where $\lambda$ is a linear form on $V$. This corresponds to the local system $\check{\mathcal{V}}$ which is the sheaf associated with the presheaf
    \begin{equation*}
    U\mapsto \Hom_k(\Gamma (U,\mathcal{V}),k).
    \end{equation*}
\item \textit{Fixed part / global sections:} The fixed point set $V^{\pi }$ of the given action of $\pi$ is the largest submodule on which $\pi$ acts trivially; the space $H^0(X,\mathcal{V})$ of global sections of $\mathcal{V}$ is mapped isomorphically to $V^{\pi }$ by the map $H^0(X,\mathcal{V})\to V$ (``fibre at $x_0$'').
\item \textit{Internal $\Hom$ objects:} We set $\underline{\Hom }(V,W)=\check{V}\otimes W$ and, mutatis mutandis, on local systems. There is a more concrete description of this: $\underline{\Hom }(V,W)$ is the vector space $\Hom_k(V,W)$ with the $\pi$-operation $(p\lambda )(v)=p\lambda (p^{-1}v)$ for $\lambda :V\to W$. This corresponds to the local system $\underline{\Hom }(\mathcal{V},\mathcal{W})$ which is the sheaf associated with the presheaf
    \begin{equation*}
    U\mapsto \Hom_k(\Gamma (U,\mathcal{V}),\Gamma (U,\mathcal{W})).
    \end{equation*}
    It is readily checked that these really \textit{are} the internal $\Hom$-objects for the given tensor structure on the category of $(\pi ,k)$-modules or local systems.
\item \textit{Space of homomorphisms:} The space $\Hom_{\pi }(V,W)$ of $k$-linear and $\pi$-equivariant homomorphisms is equal to the fixed part $\underline{\Hom}(V,W)^{\pi }$ of the internal $\Hom$ defined above. In terms of local systems this means that $\Hom (\mathcal{V},\mathcal{W})=H^0(C,\underline{\Hom }(\mathcal{V},\mathcal{W}))$.
\end{itemize}
\begin{Semisimplicity}
Recall that if $\pi$ is a finite group, the category of $(\pi ,k)$-modules is semisimple. For infinite $\pi$ this need no longer be the case, as the example $\pi=\bbz$ shows.

Every $(\pi ,k)$-module has a largest semisimple submodule; this is the sum of all simple submodules, in particular (as soon as the module itself is not zero) it is not reduced to zero.

If $V$ is a semisimple $(\pi ,k)$-module and $S$ is a simple $(\pi ,k)$-module, denote by $V_S$ the \textit{$S$-isotypical component}, i.e. the sum of all submodules isomorphic to $S$. This is then isomorphic to $S^m$ for a unique positive integer $m=m(S,V)$, called the \textit{multiplicity} of $S$ in $V$. For two non-isomorphic simple modules $S_1$ and $S_2$ the isotypical components $V_{S_1}$ and $V_{S_2}$ have zero intersection, and for dimension reasons $V_S$ is nonzero only for finitely many isomorphism classes of simple modules $S$. Let these be $S_1,\ldots ,S_n$, so that we get a decomposition
\begin{equation}\label{IsotypicalDecomposition}
V=\bigoplus_{\nu =1}^nV_{S_{\nu }}
\end{equation}
and (non-canonical) isomorphisms
\begin{equation*}
V_{S_{\nu }}\simeq S_{\nu }^{m(S_{\nu },V)}.
\end{equation*}

Let us now examine the structure of endomorphism algebras of semisimple modules. To begin with, if $S$ be a simple $(\pi ,k)$-module, then its endomorphism algebra $D=\End_{\pi} S$ is a finite-dimensional skew field over $k$ (this is Schur's lemma). For $k$ algebraically closed it must therefore be equal to $k$. Then for a positive integer $m$ we have
\begin{equation*}
\End_{\pi }(S^m)=M_m(D).
\end{equation*}
Finally for a general semisimple $(\pi ,k)$-module $V$, every endomorphism respects the decomposition (\ref{IsotypicalDecomposition}), whence an isomorphism
\begin{equation*}
\End_{\pi }(V)=\prod_{\nu =1}^n\End_{\pi }(V_{S_{\nu }})\simeq\prod_{\nu =1}^nM_{m_{\nu }}(D_{\nu })
\end{equation*}
with $m_{\nu }=m(S_{\nu },V)$ and $D_{\nu }=\End_{\pi }S_{\nu }$.

From this we obtain an almost trivial characterization of simple modules over algebraically closed fields:
\end{Semisimplicity}
\begin{Proposition}\label{SimpleIffDimEndIsOne}
Let $k$ be an algebraically closed field and let $V$ be a semisimple $(\pi ,k)$-module. Then $V$ is simple if and only if $\dim_k\End_{\pi }(V)=1$.\hfill $\square $
\end{Proposition}
\begin{BaseChange}
We now study how properties of $(\pi ,k)$-modules behave under base change.
\end{BaseChange}
\begin{Proposition}\label{GoodbehaviourOfRepsUnderBC}
Let $k\subseteq K$ be a field extension, and let $V$ and $W$ be $(\pi ,k)$-modules with corresponding local systems $\mathcal{V}$ and $\mathcal{W}$. Denote a base extension $-\otimes_kK$ by a subscript ``$K$''.
\begin{enumerate}
\item The formation of direct sums, tensor products, duals and inner $\Hom$ objects is compatible with base extension, in the sense that there are natural identifications $(V\oplus W)_K=V_K\oplus W_K$ etc.
\item Taking fixed parts resp. global sections is also compatible with base extension, in the sense that the natural maps
    \begin{equation*}
    V^{\pi }\otimes_kK\to (V\otimes_kK)^{\pi }\quad\text{and}\quad H^0(X,\mathcal{V})\otimes_kK\to H^0(X,\mathcal{V}\otimes_kK)
    \end{equation*}
    are isomorphisms.
\item Finally also spaces of homomorphisms are compatible with base extension, meaning that the natural maps
    \begin{equation*}
    \Hom_{k[\pi ]}(V,W)\otimes_kK\to\Hom_{K[\pi ]}(V_K,W_K)\quad\text{and}\quad \Hom_k(\mathcal{V},\mathcal{W})\otimes_kK\to\Hom_K(\mathcal{V}_K,\mathcal{W}_K)
    \end{equation*}
    are isomorphisms.
\end{enumerate}
\end{Proposition}
\begin{proof}
In (i), the statement for direct sums, tensor products and duals is clear; from this the statement for iternal $\Hom$ objects follows using $\underline{\Hom }(V,W)=\check{V}\otimes W$.

In (ii) it suffices to prove the statement about representations (the one about local systems is equivalent). But this is clear if one uses a $k$-basis $\mathrsfs{B}$ of $K$ and uses that
\begin{equation*}
V_K=\bigoplus_{b\in\mathrsfs{B}}V\cdot b
\end{equation*}
\textit{as $k[\pi ]$-modules}.

Now (iii) follows from (ii) and (i) using that $\Hom (V,W)=\underline{\Hom }(V,W)^{\pi }$; again the statement about local systems is equivalent.
\end{proof}
\begin{GaloisConjugationofPiKModules}
In this paragraph let $K|k$ be a field extension in characteristic zero, and set $G=\Aut (K|k)$. There are two notions of Galois conjugations of $K$-local systems or, equivalently, of $(\pi ,K)$-modules: an ``abstract'' construction and a construction for submodules of a module which is defined over $k$. We begin with the ``abstract'' definition:
\end{GaloisConjugationofPiKModules}
\begin{Definition}
Let $V$ be a $K$-vector space and let $\sigma \in G$. Then by $V^{[\sigma ]}$ we denote the $K$-vector space whose underlying abelian group is the same as that of $V$, but with the ``new'' scalar multiplication $\ast$ given as $\lambda\ast v=\lambda^{\sigma^{-1} }\cdot v$ (where the dot denotes the multiplication in $V$).
\end{Definition}
It is easily checked that this indeed defines a vector space structure on $V$, which is different from the old one as soon as $\sigma$ is not the identity. More interesting is to observe what happens to linear maps. Namely if $V$ and $W$ are $K$-vector spaces and $f:V\to W$ is a $K$-linear map, then $f$ remains $K$-linear if interpreted as a map from $V^{[\sigma ]}$ to $W^{[\sigma ]}$:
\begin{equation*}
f(\lambda\ast v)=f(\lambda^{\sigma^{-1}}v)=\lambda^{\sigma^{-1}}f(v)=\lambda\ast f(v).
\end{equation*}
Let us now be a bit more concrete and assume that $V$ and $W$ are finite-dimensional, of dimensions $m$ and $n$, say. Introduce bases $v_1,\ldots ,v_m$ of $V$ and $w_1,\ldots ,w_n$ of $W$. Then we can write $f$ as a matrix, with entries $\varphi_{ij}$ given by
\begin{equation*}
f(v_i)=\sum_{j=n}\varphi_{ij}w_j.
\end{equation*}
But now upon interpreting $f$ as a $K$-linear map between $V^{[\sigma ]}$ and $W^{[\sigma ]}$, its matrix in the given bases becomes $(\varphi_{ij}^{\sigma })$:
\begin{equation*}
f(v_i)=\sum_{j=n}\varphi_{ij}^{\sigma }\ast w_j
\end{equation*}
by definition of $\ast$. This is the reason for defining $V^{[\sigma ]}$ the way we did.

So conjugation by an element of $G$ can change the isomorphism type of $f$. For example in the case $V=W$, the eigenvalues of $f:V^{[\sigma ]}\to V^{[\sigma ]}$ are the images under $\sigma$ of the eigenvalues of $f:V\to V$. Hence the following definition is sensible:
\begin{Definition}
Let $V$ be a $(\pi ,K)$-module and let $\sigma\in G$. Then the \textup{conjugate $(\pi ,K)$-module} $V^{[\sigma ]}$ is the $K$-vector space $V^{[\sigma ]}$ as defined above, together with the given action of $\pi$ (which is then again $K$-linear).
\end{Definition}
This may of course be restated in terms of local systems: if $V$ corresponds to the local system $\mathcal{V}$ of $K$-vector spaces, the local system $\mathcal{V}^{[\sigma ]}$ corresponding to $V^{[\sigma ]}$ is obtained by applying the ``conjugation of vector spaces'' construction fiberwise.

If we choose a basis of $V$ as a $K$-vector space, we can view the $(\pi ,K)$-module structure on $V=K^n$ as a group homomorphism $\varrho :\pi\to\GL_n(K)$; then the homomorphism $\varrho^{[\sigma ]}:\pi\to\GL_n(K)$ belonging in the same way to $V^{[\sigma ]}$ is described as follows: $\varrho^{[\sigma ]}(\gamma )=(\varrho (\gamma ))^{\sigma }$, with $\sigma$ being applied to each entry of the matrix $\varrho (\gamma )$.

\begin{Proposition}
Let $V$ be a $(\pi ,K)$-module and let $\sigma\in G$.
\begin{enumerate}
\item Let $W\subseteq V$ be a sub-$K$-vector space. Then $W$ is a sub-$(\pi ,K)$-module if and only if $W^{[\sigma ]}$ is a sub-$(\pi ,K)$-module of $V^{[\sigma ]}$. The assignment $W\mapsto W^{[\sigma ]}$ sets up a bijection between the submodules of $V$ and the submodules of $V^{[\sigma ]}$.
\item $V$ is simple if and only if $V^{[\sigma ]}$ is simple.
\item $V$ is semisimple if and only if $V^{[\sigma ]}$ is semisimple.
\end{enumerate}
\end{Proposition}
\begin{proof}
(i) is clear, and (ii) and (iii) evidently follow from (i).
\end{proof}
Hence $G$ permutes the isomorphism classes of simple $(\pi ,K)$-modules. This can be made more precise for simple summands occuring in a module defined over $k$:
\begin{Proposition}
Let $V_k$ be a $(\pi ,k)$-module and set $V_K=V_k\otimes_kK$. Let $S$ be a simple $(\pi ,K)$-module and let $\sigma\in G$. Then the multiplicity of $S$ in $V_K$ is the same as that of $S^{[\sigma ]}$ in $V_K$.
\end{Proposition}
\begin{proof}
How can we express the multiplicity of a simple module in a module in an invariant way? First look at $\End S$: this is a finite skew field extension of $K$, say of $K$-dimension $d^2$. Then the multiplicity of $S$ in $V_K$ can be expressed as
\begin{equation*}
m(S,V_K)=\frac{\dim_K\Hom (S,V_K)}{d^2}.
\end{equation*}
What happens if $S$ is replaced by $S^{[\sigma ]}$? The denominator does not change, since $\End S=\End S^{[\sigma ]}$ \textit{on the nose}, and the numerator also remains untouched: since $V_K$ is defined over $k$, we have $V_K^{[\sigma ] }\simeq V_K$ and hence
\begin{equation*}
\Hom (S,V_K)=\Hom (S^{[\sigma ]},V_K^{[\sigma ]})\simeq\Hom (S^{[\sigma ]},V_K).
\end{equation*}
Hence also the multiplicity stays the same.
\end{proof}

Now we introduce the second notion of Galois conjugation. Namely let $V_k$ be a $(\pi ,k)$-module, and let $U$ be a sub-$(\pi ,K)$-module of $V_K$. Then we can consider the sub-$(\pi ,K)$-module
\begin{equation*}
U^{\sigma }=\{ u^{\sigma }| u\in U\}\subseteq V_K,
\end{equation*}
where we let $\sigma$ act on the second factor of $V_K=V_k\otimes_kK$.
\begin{Proposition}
In this situation, the map
\begin{equation*}
\varphi : U\to U^{\sigma },\quad u\mapsto u^{\sigma }
\end{equation*}
defines an isomorphism of $(\pi ,K)$-modules $U^{[\sigma ]}\to U^{\sigma }$.
\end{Proposition}
\begin{proof}
Clearly $\varphi$ is additive, and it intertwines the $\pi$-operation because that is defined over $k$, so it only remains to be shown that it is a $K$-linear map from $U^{[\sigma ]}$ to $U^{\sigma }$. Denoting the scalar multiplication on $U$ by a dot and the one on $U^{[\sigma ]}$ by $\ast$, we find
\begin{equation*}
\varphi (\lambda\ast u)=(\lambda\ast u)^{\sigma }=(\lambda^{{\sigma }^{-1}}\cdot u)^{\sigma }=\lambda\cdot u^{\sigma }=\lambda\cdot\varphi (u).
\end{equation*}
Note that this is again a justification for defining $U^{[\sigma ]}$ the way we did.
\end{proof}
\begin{BeingDefinedoveraSubfield}
We extend the notion of Galois extensions to transcendental extensions, in the following sense:
\begin{Definition}\label{DefinitionofGalois}
A field extension $K|k$ is called \textup{Galois} if the fixed field of $\Aut (K|k)$ acting on $K$ is equal to $k$.
\end{Definition}
For algebraic extensions this agrees with the usual definition of Galois extensions. As to the transcendental case, note that every extension $K|k$ in characteristic zero with $K$ algebraically closed is Galois.

For a Galois extension of fields $K|k$ one often constructs algebraic objects which are ``defined over $K$'' and makes them descend to ``objects over $k$'' by showing that they are in a suitable sense invariant under $\Aut (K|k)$. We now make this general principle precise for local systems of vector spaces or, equivalently, $(\pi ,k)$-modules.
\end{BeingDefinedoveraSubfield}
\begin{Lemma}\label{DescentForVectorSpaces}
Let $k\subseteq K$ be a Galois field extension in the sense of Definition \ref{DefinitionofGalois}. Write $G=\Aut (K|k)$. Let $V_k$ be a finite-dimensional $k$-vector space and let $U$ be a sub-$K$-vector space of $V_K=V_k\otimes_kK$ such that $U^{\sigma }=U$ (as sets) for all $\sigma\in\Aut (K|k)$. Then $U$ ``is defined over $k$'', i.e. there is a unique sub-$k$-vector space $T_k\subseteq V_k$ with $T_k\otimes_kK=U$. Furthermore we have $T_k=U\cap V_k=U^G$.
\end{Lemma}
\begin{proof}
The hypotheses imply that $(V_K)^G=V_k$. It is clear that if $T_k$ exists, it has to be equal to $U^G$, so we denote the latter by $T_k$ and show that this does the job. Write $T_K=T_k\otimes_kK$. Clearly $T_K\subseteq U$, so we have to show the converse inclusion. Assume that it does not hold.

Now $T_k$ is a sub-$k$-vector space of $V_k$, so we can choose a basis $(t_1,\ldots ,t_m)$ of $T_k$ and complete it to a basis $(t_1,\ldots ,t_m,t_{m+1},\ldots ,t_n)$ of $V_k$. Now by assumption $U\smallsetminus T_K$ is nonempty, and every element of this set can be uniquely written as
\begin{equation*}
u=\sum_{i=1}^n\lambda_it_i,\quad \lambda_i\in K.
\end{equation*}
Choose then a $u\in U\smallsetminus T_K$ for which the number of nonzero coefficients $\lambda_i$ in this representation becomes minimal. Then $u-\sum_{i=m+1}^n\lambda_it_i$ has at most as many nonzero coefficients, i.e. we can assume that $u$ is of the form
\begin{equation*}
u=\sum_{i=m+1}^n\lambda_it_i.
\end{equation*}
By assumption at least one of $\lambda_{m+1},\ldots ,\lambda_n$ is nonzero (say $\lambda_{i_0}$), so we may as well multiply the whole expression by $\lambda_{i_0}^{-1}$ and thereby assume that $\lambda_{i_0}=1$. Then for any element $\sigma\in G$,
\begin{equation*}
u-u^{\sigma}=\sum_{m+1\le i\le n,\, i\neq i_0}(\lambda_i-\lambda_i^{\sigma })t_i
\end{equation*}
is by construction either zero or still an element of $U\smallsetminus T_K$, but with less nonzero coefficients than $u$, so by the properties of $u$ the second possibility is ruled out and $u-u^{\sigma}$ must be zero. But this holds for any $\sigma\in G$, so $u\in T_K$, contradiction.
\end{proof}
\begin{Corollary}\label{DescentOfInvariantSubmodules}
Let $k\subseteq K$ be a Galois field extension, with $k$ perfect. Let $V_k$ be a $(\pi ,k)$-module and let $U$ be a sub-$(\pi ,K)$-module of $V_K$ such that $U^{\sigma }=U$ for all $\sigma\in \Aut (K|k)$. Then there exists a unique sub-$(\pi ,k)$-module $T_k\subseteq V_k$ with $T_K=U$.
\end{Corollary}
\begin{proof}
By Lemma \ref{DescentForVectorSpaces}, $T_k=U^G$ is a sub-$k$-vector space of $V_k$ such that $T_K=U$, and since the operation of $\pi$ is defined over $k$, this construction obviously respects the $(\pi ,k)$-module structures.
\end{proof}
\begin{SemisimplicityandBaseChange}
Now comes the central statement of this section: semisimplicity is stable under base change, in both directions.
\end{SemisimplicityandBaseChange}
\begin{Proposition}\label{SemisimplicityInvariantUnderBC}
Let $k\subseteq K$ be a field extension in characteristic zero, and let $V$ be a $(\pi ,k)$-module. Then $V$ is a semisimple $(\pi ,k)$-module if and only if $V\otimes_kK$ is a semisimple $(\pi, K)$-module.
\end{Proposition}
\begin{proof}
This is a special case of \cite[Chap. V, Prop. 1.2]{Hochschild81}.
\end{proof}
\begin{SimplicityandBaseChange}
What about simple modules? Clearly if $V\otimes_kK$ is simple then $V$ is simple, but the converse does not necessarily hold --- we will see examples below. It does hold if both fields are algebraically closed. To see this, we need some elementary remarks concerning the structure of semisimple modules.
Given the nice behaviour of spaces of homomorphisms (and hence of endomorphism rings), we get from this:
\end{SimplicityandBaseChange}
\begin{Corollary}
Let $k\subseteq K$ be an extension of algebraically closed fields and let $V$ be a semisimple $(\pi ,k)$-module. Then $V$ is a simple $(\pi ,k)$-module if and only if $V_K$ is a simple $(\pi, K)$-module.
\end{Corollary}
\begin{proof}
Algebraically closed fields are perfect, so if $V$ is semisimple then also $V_K$ is semisimple by Proposition \ref{SemisimplicityInvariantUnderBC}. Now Proposition \ref{GoodbehaviourOfRepsUnderBC}(iii) yields that
\begin{equation*}
(\End_{k[\pi ]}V)\otimes_kK \simeq\End_{K[\pi ]}(V\otimes_kK)
\end{equation*}
as $K$-vector spaces, so using Proposition \ref{SimpleIffDimEndIsOne} we see that $V$ is simple if and only if $V_K$ is simple.
\end{proof}
\begin{Corollary}
Let $V_k$ be a $(\pi ,k)$-module and let $H_k\subseteq V_k$ be the largest semisimple sub-$(\pi ,k)$-module, i.e. the sum of all simple sub-$(\pi ,k)$-modules. Then $H_K=H_k\otimes_kK$ is also the largest semisimple sub-$(\pi, K)$-module of $V_K$.
\end{Corollary}
\begin{proof}
By Proposition \ref{SemisimplicityInvariantUnderBC}, $H_K$ is semisimple, so it is obviously contained in the largest semisimple submodule. For the converse implication, let $U\subseteq V_K$ be a semisimple sub-$(\pi, K)$-module. Then the sub-$(\pi ,K)$-module
\begin{equation*}
GU=\sum_{\sigma\in G}U^{\sigma }
\end{equation*}
is stabilized by $G$ and semi-simple (as a sum of simple $(\pi ,K)$-modules, so it is of the form $T_k\otimes_kK$ for some semisimple sub-$(\pi ,k)$-representation $T_k$ of $V_k$ which is therefore contained in $H_k$. Hence $U\subseteq GU\subseteq H_K$.
\end{proof}

\subsection{Local Systems of $\bbq$-Vector Spaces}

Let us apply the considerations in the previous section to the case $k=\bbq$. So let $V$ be a $(\pi ,\bbq)$-module and let $W$ be its maximal semisimple submodule. Then $W\otimes\overline{\bbq }$ is also the largest semisimple submodule of $V\otimes\overline{\bbq }$. We obtain a decomposition into isotypical components over $\overline{\bbq }$:
\begin{equation*}
W\otimes\overline{\bbq }=\bigoplus_{\nu =1}^nV_{\nu },\quad V_{S_{\nu }}\simeq S_{\nu }^{m_{\nu }}.
\end{equation*}
Here $S_1,\ldots ,S_n$ are pairwise non-isomorphic simple $(\pi ,\overline{\bbq })$-modules, and $m_{\nu }>0$. Now the $S_{\nu }\otimes_{\overline{\bbq }}\bbc$ are also simple $(\pi, \bbc )$-modules, and thus
\begin{equation*}
W\otimes\bbc =\bigoplus_{\nu =1}^n(V_{\nu }\otimes_{\overline{\bbq }}\bbc),\quad V_{\nu }\otimes_{\overline{\bbq }}\bbc \simeq (S_{\nu }\otimes_{\bbq }\bbc )^{m_{\nu }}
\end{equation*}
is \textit{the} isotypcial decomposition of $W\otimes\bbc$. In other words, all simple submodules which occur in $V\otimes\bbc$ are already defined over $\overline{\bbq }$. In fact they can be defined over a number field. Before showing this we introduce some reminders about number fields.
\begin{SomeRemarksonGaloisTheory}
We are going to study how $(\pi, \bbq )$-modules behave under base extension. For this purpose we have to fix some notations concerning field extensions.

By $\overline{\bbq }$ we always mean \textit{the} algebraic closure of $\bbq$ in $\bbc$ (not some abstract algebraic closure obtained by the axiom of choice), i.e. the field of all algebraic numbers in $\bbc$. A \textit{number field} is a finite field extension of $\bbq$, not necessarily contained in $\bbc$. The \textit{degree} of a number field $k$ is the degree of the field extension $k|\bbq$. If $k$ is a number field of degree $d$, it has precisely $d$ distinct embeddings
\begin{equation*}
\sigma_1,\ldots ,\sigma_d:k\to\overline{\bbq }\subset\bbc .
\end{equation*}
Often we will encounter the situation where $k$ is given as some subfield of $\bbc$. In this case there is a ``preferred embedding'', namely the identical embedding; we then enumerate the embeddings in such a way that this is $\sigma_1$.
\end{SomeRemarksonGaloisTheory}
\begin{Definition}\label{DefinitionGaloisClosure}
Let $k$ be a number field. By its \textup{Galois closure} we mean the subfield $k^{\text{Galois}}$ of $\overline{\bbq }\subset\bbc$ generated by the images of all embeddings $k\to\bbc$.
\end{Definition}
This is again a number field, and equal to the Galois closure of each $\sigma_j(k)$ in $\overline{\bbq }$ in the usual sense.

By the \textit{absolute Galois group} of a number field $k$ which is given as a subfield of $\bbc$ we mean the group $G_k=\Gal (\overline{\bbq }|k)$. This is, with the Krull topology, a profinite group. It is furthermore an open subgroup of $G_{\bbq }$, and this assignment gives a bijection between finite subextensions of $\overline{\bbq }|\bbq$ and open subgroups of $G_{\bbq }$.
\begin{Proposition}
Let $k$ be a number field which is contained in $\bbc$. Then the map
\begin{equation*}
\alpha :G_k\backslash G_{\bbq }\to\Hom (k,\overline{\bbq })=\Hom (k,\bbc ),\quad G_k\cdot \sigma\mapsto \sigma|_k
\end{equation*}
defines a bijection between the quotient $G_k\backslash G_{\bbq }$ and the set of embeddings $k\to\bbc$.
\end{Proposition}
\begin{proof}
Let $\sigma_1,\sigma_2\in G_{\bbq }$. Then their restriction to $k$ is the same if and only if the restriction of $\sigma_1\sigma_2^{-1}$ to $k$ is the identity, hence if and only if $\sigma_1$ and $\sigma_2$ are in the same left coset modulo $G_k$. Hence the map is injective. It is also clearly surjective: any embedding $k\to\overline{\bbq}$ can, by virtue of Zorn's lemma, be extended to some automorphism of $\overline{\bbq }$.
\end{proof}
\begin{TheFieldofModuli}
Let $S$ be a simple $(\pi ,\overline{\bbq })$-module, and set
\begin{equation*}
G_S=\{\sigma\in G_{\bbq }\, |\, S^{[\sigma ]}\simeq S\} .
\end{equation*}
\end{TheFieldofModuli}
\begin{Lemma}\label{GSisOpen}
Assume that $\pi$ is finitely generated. Then the subgroup $G_S\subseteq G_{\bbq }$ is open.
\end{Lemma}
\begin{proof}
A subgroup of $G_{\bbq }$ is open if and only if it contains $G_k$ for some finite subextension $k|\bbq$ of $\overline{\bbq }$. Now let $\pi$ be generated by $\gamma_1,\ldots ,\gamma_d$, and let $S$ correspond to a representation
\begin{equation*}
\varrho :\pi\to\GL_m(\overline{\bbq }).
\end{equation*}
Let $k$ be the number field generated by the entries of the matrices $\varrho (\gamma_1),\ldots ,\varrho (\gamma_d)$. Then the image of $\varrho$ is contained in $\GL_m(k)$, which means that $S$ is isomorphic to $T\otimes_k\overline{\bbq }$ for some $(\pi ,k)$-module $T$. But then $G_k\subseteq G_S$.
\end{proof}
\begin{Definition}
Let $S$ and $\pi$ as above. The \textup{field of moduli} of $S$ is the number field $k_S=\overline{\bbq }^{G_S}$.
\end{Definition}
Note that the field of moduli really is a number field by Lemma \ref{GSisOpen}.
\begin{Proposition}
Let $V$ be a $(\pi ,\bbq )$-module and let $S$ be a simple $(\pi ,\overline{\bbq })$-module which occurs \textup{with multiplicity one} in $V\otimes\overline{\bbq}$. Then $S$ is defined over $k_S$, i.e. there is a unique simple sub-$(\pi ,k_S)$-module $T$ of $V\otimes k_S$ with $T\otimes_{k_S}\overline{\bbq }=S$.
\end{Proposition}
\begin{proof}
For any $\sigma\in G_S$, the submodule $S^{\sigma }$ of $V\otimes\overline{\bbq }$ is isomorphic to $S$ and hence, because of multiplicity one, must be actually \textit{equal} to $S$. So $S$ is invariant under $G_S$, and using Corollary \ref{DescentOfInvariantSubmodules} we see that it is defined over $k_S$.
\end{proof}
These considerations can be carried out in a very similar way when $\overline{\bbq }$ is replaced by $\bbc$. Every field automorphism of $\bbc$ restricts to an automorphism of $\overline{\bbq }$, hence we have a well-defined group homomorphism $\Aut\bbc\to\Aut\overline{\bbq }=G_{\bbq }$. It is surjective by an application of Zorn's lemma.
For a simple $(\pi ,\bbc )$-module $S$ we can define
\begin{equation*}
H_S=\{ \sigma\in\Aut\bbc\, |\, S^{[\sigma ]}\simeq S\} .
\end{equation*}
\begin{Definition}
Let $S$ be a simple $(\pi ,\bbc )$-module. Then its \textup{field of moduli} $k_S$ is the fixed field $\bbc^{H_S}\subseteq\bbc$.
\end{Definition}
In general we cannot say much about this field. But when $S$ is defined over $\overline{\bbq }$, say $S=T\otimes_{\overline{\bbq}}\bbc$ for some simple $(\pi ,\overline{\bbq})$-module $T$, then $H_S$ is the preimage of $G_T$ under the restriction homomorphism $\Aut\bbc\to G_{\bbq}$. In particular the field of moduli of $S$ is the same as the field of moduli of $T$, in particular a number field. Hence also the Proposition has a direct analogue:
\begin{Proposition}
Let $V$ be a $(\pi ,\bbq )$-module and let $S$ be a simple $(\pi ,\bbc )$-module which occurs \textup{with multiplicity one} in $V\otimes\bbc$. Then the moduli field $k_S\subset\bbc$ is a number field, and there is a unique simple sub-$(\pi ,k_S)$-module $T$ of $V\otimes k_S$ with $T\otimes_{k_S}\bbc =S$.\hfill $\square$
\end{Proposition}

\subsection{The Local System of a Teichm\"{u}ller Disk}

Recall that for a translation surface $(X,\omega )$ one obtains a family of curves $f:\mathrsfs{X}\to C=\Gamma\backslash\bbh$. We shall apply the results of the previous section to the local system $\mathcal{J}=R^1f_{\ast}\bbq$ on $C=\Gamma\backslash\bbh$. We first take a look at its monodromy representation.
\begin{MonodromyRepresentation}
Recall that the fibre over the base point $\mathrm{i}\in\bbh$ is identified with $X$. Let $c_0\in C$ be the image of $\mathrm{i}$, so that we can also identify the fibre of $f$ over $c_0$ with $X$. Hence we can identify the fibre of the local system $(R^1f_{\ast }\bbq )_{c_0}$ with the cohomology group $H^1(X,\bbq )$.
Now let $a\in\pi_1(C,c_0)$ be represented by a path $\alpha :[0,1]\to C$; lift this to a path $\tilde{\alpha }:[0,1]\to\bbh$ with $\tilde{\alpha }(0)=\mathrm{i}$. Recall that under our identification of the fundamental group with $\Gamma$, the element $a$ corresponds to the unique element $D_a\in\Gamma$ with $D_a(\tilde{\alpha }(1))=\mathrm{i}$. By construction of the mirror Veech group we can write
\begin{equation}\label{DeckToVeech}
D_a=
\begin{pmatrix}
-1&0\\
0&1
\end{pmatrix}
\cdot
A
\cdot
\begin{pmatrix}
-1&0\\
0&1
\end{pmatrix}
\end{equation}
for some $A$ in the Veech group $\SL (X,\omega )$. This means that there exists some affine map $\varphi\in\Aff^+(X,\omega )$ with linear part $A$. This map can for our purposes now better be viewed as an isomorphism of translation surfaces
\begin{equation*}
\varphi :(X,\omega )\overset{\simeq }{\to }A^{-1}\cdot (X,\omega ).
\end{equation*}
This is precisely the isomorphism which identifies the two fibres $(X,\omega )$ and $A^{-1}(X,\omega )$ of $\tilde{f}$ --- the first is the fibre over $\mathrm{i}$, the second is the fibre over $D_a^{-1}(\mathrm{i})=A(\mathrm{i})$ (here the operation of the Veech group on the upper half plane is as in Proposition \ref{ActionOfVeechGroupOnVariousThings}). Hence if $[a]:H^1(X,\bbz )\to H^1(X,\bbz )$ is the monodromy action of $a$ and if we, for one moment, write $A^{-1}\cdot X$ for the Riemann surface underyling $A^{-1}\cdot (X,\omega )$, we get a commutative diagram:
\begin{equation*}
\xymatrix{
H^1(X,\bbz ) \ar[r]^{[a]} \ar@{=}[d]
& H^1(X,\bbz )\\
H^1(X,\bbz ) \ar[r]_{\mathrm{id}}
& H^1(A^{-1}\cdot X,\bbz ) \ar[u]_{\varphi^{\ast }}
}
\end{equation*}
So we have shown the following:
\end{MonodromyRepresentation}
\begin{Proposition}
The monodromy representation of $\pi_1(C,c_0)$ on $(R^1f_{\ast }\bbz )_{c_0}=H^1(X,\bbz )$ is given as follows: let $a\in\pi_1(C,c_0)$; identify $\pi_1(C,c_0)$ with $\Gamma$ as above, giving an element $D_a\in\Gamma$. Write this as in (\ref{DeckToVeech}) for some $A$ in the Veech group. This $A$ is the linear part of some affine map $\varphi\in\Aff^+(X,\omega )$. Then the monodromy operation of $a$ on $H^1(X,\bbz )$ is the map $\varphi^{\ast}$.\hfill $\square$
\end{Proposition}

\begin{TheCanonicalSubvariation}
We have seen that the canonical subspace $S_{\omega }\subseteq H^1(X,\bbr )$ is preserved by the action of the affine group, hence it is also preserved by the monodromy operation, and we can write down this operation explicitly. So:
\end{TheCanonicalSubvariation}
\begin{PropositionDefinition}
There is a unique sub-local system of $\bbr$-vector spaces $\mathcal{S}_{\omega }\subseteq R^1f_{\ast }\bbr$, called the \textup{canonical subsystem}, whose fibre in $c_0$ is the canonical subspace $S_{\omega }\subseteq H^1(X,\bbr )$. In the $\bbr$-basis $(\Re\omega ,\Im\omega )$ of $S_{\omega }$, the monodromy operation of an element $G\in\Gamma =\pi_1(C,c_0)$ is given by the matrix
\begin{equation*}
\begin{pmatrix}
-1&0\\
0&1
\end{pmatrix}
\cdot G\cdot
\begin{pmatrix}
-1&0\\
0&1
\end{pmatrix}.
\end{equation*}
Hence in the basis $(-\Re\omega ,\Im\omega )$, the monodromy operation of $\Gamma$ on $S_{\omega }$ becomes the identity inclusion $\Gamma\hookrightarrow\SL_2(\bbr )$.\hfill $\square$
\end{PropositionDefinition}
We can describe this subsystem in a more revealing way. Namely consider its lift to $\bbh$; this is a subsystem of the trivial local system $R^1\tilde{f}_{\ast }\bbr$ on $\bbh$ (this is the lift of $R^1f_{\ast}\bbr$ along the projection $\bbh\to\Gamma\backslash\bbh$). Now let $\tau\in\bbh$ be any point, and write $\tau =A\cdot\mathrm{i}$. Then the fibre $\mathrsfs{X}_{\tau }$ is the Riemann surface $X'$ underlying $(X',\omega ')=A\cdot (X,\omega )$. But then we have
\begin{equation*}
\begin{pmatrix}
\Re\omega '\\
\Im\omega '
\end{pmatrix}
=A\cdot
\begin{pmatrix}
\Re\omega \\
\Im\omega
\end{pmatrix}
\end{equation*}
so $\Re\omega '$ and $\Im\omega '$ span the same real vector space as $\Re\omega$ and $\Im\omega$; in other words: $S_{\omega '}=S_{\omega }$. So the fibre of the canonical subsystem at any point in $\bbh$ is the canonical subspace of any of the one-forms corresponding to that point.
\begin{MFofCanonicalSubvariation}
In the case where the Veech group is a large enough, we can say much more about the local system.
\end{MFofCanonicalSubvariation}
\begin{Proposition}
Assume that the Veech group of $(X,\omega )$ is non-elementary. Then the canonical subsystem $\mathcal{S}_{\omega }$ is simple, and its multiplicity in $R^1f_{\ast }\bbr$ is one.
\end{Proposition}
\begin{proof}
Any finite index subgroup of a non-elementary group is non-elementary. Hence also $\Gamma$ is non-elementary.

The canonical subsystem has rank two. If it were not simple, its monodromy would be contained in $\bbr^{\times }\times\bbr^{\times }$, in particular be commutative. But the monodromy representation of $\mathcal{S}_{\omega }$ is faithful, and $\Gamma$ is noncommutative. Hence $\mathcal{S}_{\omega }$ is simple.

Now choose some hyperbolic element in $\Gamma$. It operates on a one-dimensional subspace of $S_{\omega }$ with the unique maximal eigenvalue (see Prop. \ref{EigenvaluesOfHyperbolicInVGActingOnH}). Hence $\mathcal{S}_{\omega }$ appears only once in $R^1f_{\ast }\bbr$.
\end{proof}
\begin{Proposition}
Assume that $\SL (X,\omega )$ is non-elementary and contains a parabolic element (this condition is always satisfied if $(X,\omega )$ is Veech). Then the moduli field of $\mathcal{S}_{\omega }\otimes_{\bbr }\bbc \subseteq R^1f_{\ast }\bbc$ is equal to the trace field $k$ of $\SL (X,\omega )$.
\end{Proposition}
\begin{proof}
We have seen that $\SL (X,\omega )$ can be conjugated into $\SL_2(k)$, hence the moduli field is contained in the trace field. The other inclusion always holds.
\end{proof}
In particular, the canonical subsystem is defined over the trace field. It hence has well-defined Galois conjugates $\mathcal{S}_{\omega }^{\sigma }$ for every $\sigma :K\to\bbc$.

\subsection{Variations of (Pseudo-) Hodge Structure}

We now study how the Hodge decomposition (or rather the Hodge filtration) behaves in families.
\begin{Definition}
Let $X$ be a complex space and let $K$ be a subring of $\bbr $. A \textup{variation of (pure) $K$-Hodge structure (short: VHS) $\mathbb{V}$ of weight $k$} on $X$ consists of a local system $\mathcal{V}_K$ of finitely generated free $K$-modules and an exhausting filtration of the holomorphic vector bundle $\mathrsfs{V}=\mathcal{V}_K\otimes_K\mathrsfs{O}_X$ by holomorphic sub-vector bundles $F^p\mathrsfs{V}$ for $p\in\bbz$, such that the following conditions hold:
\begin{enumerate}
\item For every point $x\in X$, identifying the fibre $\mathrsfs{V}_x$ with $\mathcal{V}_{K ,x}\otimes_K\bbc$, the data $(\mathcal{V}_{\bbz ,x},F^p\mathrsfs{V}_x)$ define a pure $K$-Hodge structure of weight $k$. We denote this pure Hodge structure by $\mathbb{V}_x$.
\item ``Griffiths' transversality'': if $\nabla$ denotes the unique integrable connection on $\mathrsfs{V}$ for which the sections of $\mathcal{V}_K$ are flat, then $\nabla F^p\mathrsfs{V}\subseteq F^{p-1}\mathrsfs{V}\otimes\Omega_X^1$.
\end{enumerate}
\end{Definition}

Note that a VHS on $\Spec\bbc$ is the same as a pure Hodge structure. Easy examples of VHS are constant variations of Hodge structure (which of course do not really deserve the name): let $M$ be a pure Hodge structure of weight $k$ and let $X$ be any complex space. Then we construct the constant VHS $\mathbb{V}=\underline{M}_X$ (often by abuse of notation just denoted by $M$) as follows:
 The underlying local system of $K$-modules $\mathcal{V}_K$ is the constant sheaf $M_K$. Note that then $\mathrsfs{V}$ is the constant vector bundle constructed from the vector space $M_{\bbc }$, and we let $F^p\mathrsfs{V}$ be the constant subbundle constructed from the sub-vector space $F^pM_{\bbc}$. This is easily checked to define a pure VHS on $X$; it is effective if and only if $M$ is.

\begin{Definition}
Let $X$ be a complex space and let $\mathbb{V}$ be a pure variation of $K$-Hodge structure on $X$. Then a \textup{polarization} of $\mathbb{V}$ is a morphism of VHS $\mathbb{V}\otimes\mathbb{V}\to K(-1)$ such that the induced morphism of pure Hodge structures $\mathbb{V}_x\otimes\mathbb{V}_x\to K(-1)$ is a polarization for every $x\in X$.
\end{Definition}
We now extend our notion of ``Jacobian type'' to variations:
\begin{Definition}
Let $X$ be a complex space, and let $K$ be a subring of $\bbr$. An \textup{$K$-VHS of Jacobian type of rank $g$} on $X$ is a pure variation of $K$-Hodge structure $\mathbb{J}$ on $X$ of weight one together with a principal polarization $S:\mathbb{V}\otimes\mathbb{V}\to K(-1)$, such that every fibre $\mathbb{J}_x$ with the induced polarization is a Hodge structure of Jacobian type of rank $g$.
\end{Definition}
When $f:\mathrsfs{X}\to C$ is a family of curves of genus $g$, we can construct a canonical variation of $K$-Hodge structure of Jacobian type of rank $g$ which we denote by $R^1f_{\ast }K(0)$. Its underlying local system is $R^1f_{\ast }K$, and the Hodge filtration and the polarization are defined fiberwise. The connection $\nabla$ is then called the \textit{Gauss-Manin connection} of the family. That this really defines a polarized variation of Hodge structure is \cite[Corollary 10.32]{PetersSteenbrink}.
\begin{HodgeSubbundles}
So let $K$ be a subring of $\bbr$ and let $\mathbb{V}$ be an $K$-variation of Hodge structures of weight $k$ on $S$, with associated holomorphic vector bundle $\mathrsfs{V}=\mathcal{V}_K\otimes_K\mathrsfs{O}_S$. Denote its total space by $V\to S$. The connection can then also be interpreted as a connection of the smooth vector bundle underlying $V$, with the additional nice property that it sends holomorphic sections to holomorphic $V$-valued differential forms.

 The Hodge subbundles $F^pV$ of $V$ are holomorphic subbundles, and their complex conjugates (with respect to the real structure given by $\mathbb{V}_K$) are antiholomorphic subbundles of $V$. Their intersections can still be understood as smooth subbundles, so upon defining
\begin{equation*}
V^{p,q}=F^pV\cap\overline{F^qV}
\end{equation*}
we get a ``Hodge decomposition'' in the $\mathrsfs{C}^{\infty }$ category:
\begin{equation}\label{HodgeDecompositionOfVectorBundles}
V=\bigoplus_{p+q=k}V^{p,q}
\end{equation}
\textit{as smooth vector bundles.} This decomposition induces a Hodge decomposition on every fibre of $\mathcal{V}_K$, and we can reconstruct the whole variation of Hodge structures from the local system $\mathcal{V}_K$ plus the decomposition. So the obvious question is, starting from a local system and a smooth decomposition (\ref{HodgeDecompositionOfVectorBundles}), when do these data define a variation of Hodge structures?
\end{HodgeSubbundles}
\begin{Proposition}\label{VHSasVPsHS}
Let $S$ be a complex manifold, let $\mathcal{V}_K$ be a local system of finitely generated free $K$-modules on $S$. Denote by $\mathcal{V}_{\bbc }=\mathcal{V}_K\otimes_K\bbc$ the associated complex local system which can in turn be interpreted as a holomorphic vector bundle $V$ with a flat holomorphic connection $\nabla$. Furthermore let there be given a decomposition (\ref{HodgeDecompositionOfVectorBundles}) of smooth complex vector bundles. Then these data define a variation of Hodge structure if and only if all of the following conditions are satisfied:
\begin{enumerate}
    \item For every point $s\in S$, the fibres of (\ref{HodgeDecompositionOfVectorBundles}) define a Hodge decomposition on $(\mathcal{V}_K)_s$.
    \item The subbundles $F^pV=\bigoplus_{r\ge p}V^{r,k-r}$ are holomorphic (equivalently, the subbundles $\overline{F^qV}=\bigoplus_{r\ge q}V^{k-r,r}$ are antiholomorphic).
    \item The $\mathrsfs{C}^{\infty}$ formulation of Griffiths' transversality: denote by $TM$ the $\mathrsfs{C}^{\infty }$ tangential bundle of $S$. Then $\nabla$ sends (smooth sections of) $V^p$ to $V^{p-1}\otimes TM$, and $\overline{F^qV}$ to $\overline{F^{q-1}V}$.
\end{enumerate}
\end{Proposition}
\begin{VariationsofPsHS}This proposition makes it clear how to extend the concept of variations of Hodge structure to pseudo-Hodge structures:
\end{VariationsofPsHS}
\begin{Definition}\label{DefinitionOfVPsHS}
Let $S$ be a complex manifold and let $K$ be a subring of $\bbc$. A \textup{variation of pseudo-Hodge structure} $\mathbb{V}$ (over $K$) on $S$ consists of the following data:
\begin{enumerate}
    \item a local system of finitely generated free $K$-modules $\mathcal{V}_K$ on $S$ --- write $\mathcal{V}_{\bbc }=\mathcal{V}_K\otimes_K\bbc$ and interpret this as a holomorphic vector bundle $V$ on $S$ with integrable holomorphic connection $\nabla$,
    \item a decomposition of smooth vector bundles
    \begin{equation}
    V=\bigoplus_{p\in\bbz }V^p,
    \end{equation}
\end{enumerate}
subject to the following axioms:
\begin{enumerate}
    \item the subbundles $F^pV=\bigoplus_{r\ge p}M^r$ are holomorphic,
    \item the subbundles $\overline{F}^qV=\bigoplus_{r\le -q}V^q$ are antiholomorphic and
    \item Griffiths' transversality as in Proposition \ref{VHSasVPsHS}, (iii).
\end{enumerate}
\end{Definition}
Note that since there is in general no complex conjugation defined on $V$, the notation $\overline{F}^qV$ is purely symbolical and intended to remind the reader of the classical situation. Also note that a variation of pseudo-Hodge structure $\mathbb{V}$ on $X$ really defines a pseudo-Hodge structure $\mathbb{V}_s$ for every $s\in S$ by ``taking fibres''. We leave out the usual abstract nonsense (definition of morphisms etc.) and rather explain how to obtain a variation of pseudo-Hodge structure from a variation of Hodge structure.

If $K$ is a subring of $\bbr$ and $\mathbb{V}$ is a variation of $K$-Hodge structure of weight $k$ on $X$, Proposition \ref{VHSasVPsHS} shows that we can define an associated variation of $K$-pseudo-Hodge structure $\mathbb{V}^{\psi}$ on $X$ along the lines of the forgetful functor from Hodge structures to pseudo-Hodge structures described above. Of course Definition \ref{DefinitionOfVPsHS} is just made in the most reasonable way to make this possible. We now give an explicit description of $\mathbb{V}^{\psi}$.

The underlying local system is the same as that of $\mathbb{V}$, i.e. $\mathbb{V}_K^{\psi}=\mathbb{V}_K$. The holomorphic bundle $V$ is decomposed into smooth subbundles as described in (\ref{HodgeDecompositionOfVectorBundles}), and we set $V^p=V^{p,k-p}$. Note that then the definition of the Hodge filtration for the variation of Hodge structures and the variation of pseudo-Hodge structures agree; yet the image of the complex conjugate $\overline{F^pV}$ of the Hodge subbundles with respect to the $K$-structure is (in the pseudo-Hodge notation) not $\overline{F}^qV$ but $\overline{F}^{q-k}V$. This may be confusing, but it is the price we have to pay for ignoring weights on the pseudo-Hodge side. This will not play a r\^{o}le in subsequent considerations.

\begin{Definition}
Let $K$ be a subring of $\bbc$ which is stable under complex conjugation, let $X$ be a complex manifold and let $\mathbb{V}$ be a variation of $K$-pseudo-Hodge structure on $X$. A \textup{polarization} of $\mathbb{V}$ is a map of local systems $H:\mathbb{V}_K\times\mathbb{V}_K\to K$ which defines a polarization on each fibre $\mathbb{V}_s$.
\end{Definition}
The discussion of the relation between this notion of polarization and the corresponding notion for variations of Hodge structure is completely parallel to the case over a point, so we omit it.

\subsection{Structure Theorems for Variations of Hodge Structure}

\begin{AThmofSchmid}
A complex manifold $X$ is called \textit{compactifiable} if there exists a compact complex manifold $\overline{X}$ and a normal crossings divisor $D\subset\overline{X}$ such that $X\simeq\overline{X}\smallsetminus D$. By Hironaka's theorem on the resolution of singularities, every smooth complex algebraic variety is compactifiable. On a compactifiable complex manifold, every plurisubharmonic function which is bounded above is constant. In particular a bounded domain in $\bbc^n$ is never compactifiable. We also see that a Riemann surface is compactifiable if and only if it is of finite type.

There are strong results about variations of Hodge structure which need that the base manifold be compactifiable. Their proofs can all be reduced to the following theorem of Wilfried Schmid, see \cite[Thm. 7.22]{Schmid73}:
\end{AThmofSchmid}
\begin{Theorem}[Schmid]
Let $X$ be a compactifiable complex manifold, let $\mathbb{V}$ be a polarizable variation of $\bbc$-pseudo-Hodge structure on $X$ and let $e\in H^0(X,\mathcal{V}_{\bbc })$. Viewing this $e$ as a flat section of the associated smooth vector bundle, its components in the Hodge decomposition are also flat.
\end{Theorem}
Schmid in loc.cit. only formulates this theorem for variations of Hodge structure which are defined over the integers, but his proof does not make any use of the assumption that the underlying local system be defined over $\bbz$, nor of the Hodge symmetry, and hence applies without changes to complex variations of pseudo-Hodge structure. Note that Schmid includes polarizability in the axioms for a variation of Hodge structure (contrary to us) and does make use of that assumption during the proof.

This theorem might not seem very surprising on first sight, but it has several surprising consequences.\footnote{We also note that it fails for $X=\Delta^{\ast }$, the pointed unit disk, since these consequences do no longer hold there} The first is just a mild reformulation:
\begin{Corollary}
Let $X$ and $\mathbb{V}$ be as in the theorem, and assume that $X$ is connected. If a global flat section $e\in H^0(X,\mathcal{V}_{\bbc })$ is of pure Hodge type $(p,q)$ at some point of $X$, then it is of this Hodge type everywhere.\hfill $\square $
\end{Corollary}
This directly implies the next very strong result. Note that if $\mathbb{V}$ is a variation of $\bbq$-Hodge structure and if $x\in X$, then the underlying local system $\mathcal{V}_{\bbq }$ gives rise to a linear representation of $\pi_1(X,x)$ on the $\bbq $-vector space $\mathcal{V}_{\bbq ,x}$. This operation does in general \textit{not} preserve the Hodge structure, i.e. the endomorphisms of $\mathcal{V}_{\bbq ,x}$ induced by the elements of the fundamental group are in general not morphisms of Hodge structure.\footnote{Note that this does not contradict the Rigidity Theorem. Namely to prove that the Hodge structure is preserved we would need the Rigidity Theorem not for $X$ but for its universal covering space which may fail to be compactifiable. In fact in all cases we consider, this universal covering space is the unit disk for which all these theorems are wrong} But:
\begin{Corollary}[Rigidity Theorem]
Let $X$ be as in the previous corollary, let $K\subseteq\bbc$ be a subring stable under complex conjugation, and let $\mathbb{V}$ and $\mathbb{W}$ be polarizable variations of $K$-pseudo-Hodge structure on $X$. Let $x\in X$, and let $f_x:\mathbb{V}_x\to\mathbb{W}_x$ be a morphism of Hodge structures such that the underlying map of $K$-modules is equivariant for the action of $\pi_1(X,x)$. Then $f_x$ extends uniquely to a morphism of variations of $K$-pseudo-Hodge structure $f:\mathbb{V}\to\mathbb{W}$.
\end{Corollary}
\begin{proof}
The equivariance means that $f_x$ extends to a morphism of local systems $f:\mathcal{V}_K\to\mathcal{W}_K$. We have to show that it is in fact a morphism of variations of pseudo-Hodge structure. But that is equivalent to showing that it is a flat section of the variation of pseudo-Hodge structure $\underline{\Hom }(\mathbb{V},\mathbb{W})$ which is everywhere of Hodge type $(0,0)$. But it is so at $x$ (since it there preserves the pseudo-Hodge structures), so it is everywhere by Schmid's theorem.
\end{proof}
\begin{SemisimplicityofVHS}
There are several questions around semisimplicity one can ask about variations of Hodge structure: under which conditions are variations of Hodge structures semisimple as such? And when the underlying local system? There is a satisfactory answer (at least for our purposes) to the first question:
\end{SemisimplicityofVHS}
\begin{Corollary}\label{PolarizableQVHSAreSemisimple}
Let $X$ be a complex manifold and let $K\subseteq\bbr$ be a field. Then the category of polarizable $\bbq$-variations of Hodge structure on $X$ is semisimple.
\end{Corollary}
\begin{proof}
Let $\mathbb{V}$ be a polarizable $K$-variation of Hodge structure on $X$ and let $\mathbb{W}$ be a sub-variation of $\mathbb{V}$. Denote the underlying local systems of $K$-vector spaces by $\mathcal{W}\subseteq\mathcal{V}$. Choose a polarization $S$ of $\mathbb{V}$; then the orthogonal complement $\mathcal{W}^{\perp}$ of $\mathcal{W}$ in $\mathcal{V}$ with respect to $S$ is a well-defined local subsystem, and at every point $x\in X$ its fibre underlies by Proposition \ref{NicePropsOfPolarizableHS} a sub-Hodge structure of $\mathbb{V}_x$. Hence $\mathcal{W}^{\perp}$ underlies a sub-variation of Hodge structure $\mathbb{W}^{\perp}$, and we have $\mathbb{V}=\mathbb{W}\oplus\mathbb{W}^{\perp}$.
\end{proof}
\begin{SemisimplicityofMonodromy}
The other question is more delicate. In particular we need to assume that the variations of Hodge structure in question are defined over the integers (or at least can be embedded into some that are).
\end{SemisimplicityofMonodromy}
\begin{Theorem}[Deligne, Schmid]\label{MonodromyRepOfVHSOverZIsSemisimple}
Let $X$ be a compactifiable complex manifold, and let $\mathbb{V}$ be a variation of $\bbz$-Hodge structure on $X$. Then the local system of $\bbq$-vector spaces $\mathcal{V}_{\bbq }$ is semisimple.
\end{Theorem}
\begin{proof}
There is a proof for this statement in \cite[sect. 4.2]{Deligne71}, which is formulated as a conditional proof with the condition being precisely that Schmid's theorem (which was not yet known at that time) holds.
\end{proof}
\begin{Comparison}
Now consider some polarizable variation of $\bbq$-Hodge structure $\mathbb{V}$ on a compactifiable complex manifold $X$ such that the underlying local system $\mathcal{V}_{\bbq }$ can be defined over $\bbz $. We do not need the exact form of this $\bbz $-structure, just its existence, in order to be able to apply Theorem \ref{MonodromyRepOfVHSOverZIsSemisimple}. Then there are several ways of decomposing $\mathbb{V}$, forgetting different kinds of structure:
\end{Comparison}
\begin{enumerate}
\item First of all we may decompose $\mathbb{V}$ as a variation of $\bbq$-Hodge structure: it is semisimple by Corollary \ref{PolarizableQVHSAreSemisimple}.
\item Considering the underlying local system of $\bbq$-vector spaces $\mathcal{V}_{\bbq}$, each simple summand of the decomposition (i) further decomposes into simple local systems of $\bbq$-vector spaces by Theorem \ref{MonodromyRepOfVHSOverZIsSemisimple}.
\item Finally the simple summands in (ii) further decompose into simple local systems of complex vector spaces after tensoring with $\mathbb{C}$. This gives a decomposition of $\mathcal{V}_{\bbc }$.
\item But we may also start again from (i) and forget the $\bbq$-structure first, in order to get a variation of $\bbc$-pseudo-Hodge structure $\mathbb{V}_{\bbc }^{\psi }$. We will show below that this also is semisimple, and in a sense to be made precise, its decomposition into simple parts is governed by the decomposition (iii).
\end{enumerate}
\begin{MonodromyGovernsTheHD}
So the conditions to work with are as follows: $X$ is a compactifiable complex manifold, and $\mathbb{V}$ is a polarizable variation of $\bbc$-pseudo-Hodge structure on $X$ such that the underlying local system of $\bbc$-vector spaces is semisimple. This is for example true if $\mathbb{V}$ is a sub-variation of a variation of pseudo-Hodge structure which comes from a polarizable variation of $\bbz$-Hodge structure, as we have seen. Astonishingly, the underlying local system then almost completely determines the variation of pseudo-Hodge structure!

The next lemma and the next proposition are due to Deligne, see \cite{Deligne87}.
\end{MonodromyGovernsTheHD}
\begin{Lemma}[Deligne]
Let $X$ be a compactifiable complex manifold and let $\mathcal{V}$ be a \textup{semisimple} local system of complex vector spaces on $X$ such that there exists \textup{some} polarizable $\bbc$-pseudo-Hodge structure on $\mathcal{V}$. Decompose $\mathcal{V}$ as
\begin{equation}\label{DecompositionOfVInSimples}
\mathcal{V}=\bigoplus_{i=1}^n\mathcal{S}_i\otimes W_i
\end{equation}
with pairwise non-isomorphic simple local systems of $\bbc$-vector spaces $\mathcal{S}_i$ and nonzero complex vector spaces $W_i$. Then for each $i$ there exists some structure of a polarizable variation of $\bbc $-pseudo-Hodge structure on $\mathcal{S}_i$.
\end{Lemma}
Note that the polarizablility of the variations on $\mathcal{S}_i$ is the point which makes the whole lemma nontrivial. Namely every local system of complex vector spaces $\mathcal{T}$ underlies a trivial variation of $\bbc$-pseudo-Hodge structure by setting, say, $T^0=T$ and $T^p=0$ for all other $p$, but this variation need not be polarizable.
\begin{Proposition}[Deligne]
Let $X$ and $\mathcal{V}$ be as in the previous lemma, and choose some polarizable variation of $\bbc$-pseudo-Hodge structure on each $\mathcal{S}_i$ (writing it as $\mathbb{S}_i$). Then for every polarizable variation of $\bbc$-pseudo-Hodge structure on $\mathcal{V}$, say $\mathbb{V}$, there exist unique $\bbc$-pseudo-Hodge structures on the $W_i$ such that
\begin{equation*}
\mathbb{V}=\bigoplus_{i=1}^n\mathbb{S}_i\otimes W_i
\end{equation*}
as variations of $\bbc$-pseudo-Hodge structure.
\end{Proposition}
\begin{Corollary}
Let $X$ be a compactifiable complex manifold and let $\mathcal{S}$ be a simple local system of complex vector spaces on $X$. Then any two polarizable variations of pseudo-Hodge structure on $\mathcal{S}$ only differ by a shift of the grading.
\end{Corollary}
\begin{proof}
Evidently it suffices to assume that there exists some polarizable variation of pseudo-Hodge structure on $\mathcal{S}$, and then we can apply the previous proposition with $\mathcal{V}=\mathcal{S}$.
\end{proof}
\begin{Corollary}
Let $X$ be a compactifiable complex manifold, let $F\subseteq\bbr$ be a subfield and and let $\mathcal{S}$ be an absolutely simple\footnote{I.e. $\mathcal{S}_F\otimes_F\bbc$ is simple} local system of $F$-vector spaces. Then given an integer $w$, there is at most one polarizable variation of $F$-Hodge structure of weight $w$ on $\mathcal{S}$.\hfill $\square $
\end{Corollary}
\begin{proof}
Every polarizable variation of $F$-Hodge structure of weight $w$ on $\mathcal{S}$ defines in an obvious manner a polarizable variation of $\bbc$-pseudo-Hodge structure on $\mathcal{S}\otimes_K\mathbb{C}$ to which we can apply the previous corollary. Now a variation of $\bbc$-pseudo-Hodge structure on $\mathcal{S}\otimes_K\bbc$, even if only given up to a shift of the grading\footnote{Here is the point where we use that we have fixed a weight.}, uniquely determines the variation of $F$-Hodge structure, and even the ambiguity of a shift of the grading gets lost when fixing the weight.
\end{proof}

\subsection{The Variation of Hodge Structure over a Teichm\"{u}ller Disk}

Now we finally apply the various findings of this chapter to the variation of Hodge structure over a Teichm\"{u}ller disk or curve. So let $(X,\omega )$ be a translation surface. Recall that this defines holomorphic isometric embeddings $f^{\omega }:\Delta\to\mathrsfs{T}_g$ and $h^{\omega }:\bbh\to\mathrsfs{T}_g$, which differ from each other by composition with $C:\bbh\to\Delta$ as in section 3.2. This amounts to a family $\tilde{f}:\tilde{\mathrsfs{X}}\to\bbh$. For a suitable finite index subgroup $\Gamma\subseteq\Gamma (X,\omega )$, which may well be the trivial group, we also get a family of Riemann surfaces $f:\mathrsfs{X}\to C=\Gamma\backslash\bbh$ whose lift to $\bbh$ is the family $\tilde{f}$. We get then a variation of rational Hodge structures of Jacobian type $\mathbb{J}=R^1f_{\ast }\bbq (0)$ on the quotient $C=\Gamma\backslash\bbh$, and its lift $\tilde{\mathbb{J}}=R^1\tilde{f}_{\ast }\bbq (0)$ to the upper half plane $\bbh$. We now summarize what we can say about these variations.

\begin{TheGeneralCase}
In general we can only say the following:
\end{TheGeneralCase}
\begin{Proposition}
Let $(X,\omega )$ be any translation surface, with family $\tilde{f}:\tilde{\mathrsfs{X}}\to\Gamma\backslash\bbh$ as above. Then the canonical subsystem $\mathcal{S}_{\omega }\subseteq R^1f_{\ast }\bbr$ is in fact a sub-variation of Hodge structure $\mathbb{S}_{\omega }\subseteq \tilde{\mathbb{J}}\otimes\bbr =R^1f_{\ast }\bbr$.
\end{Proposition}
\begin{proof}
This can be checked fiberwise. But we have already seen that the canonical subspace at every point is a sub-Hodge structure.
\end{proof}
\begin{NonelementaryWithParabolic}
Assume now that the finite index subgroup $\Gamma\subseteq\Gamma (X,\omega )$ can be chosen non-elementary and containing a parabolic element. This is for example possible when $(X,\omega )$ is a Veech surface. Then the moduli field $k$ of the canonical subsystem $\mathcal{S}_{\omega }\otimes_{\bbr }\bbc \subseteq\mathcal{J}\otimes\bbc$ is equal to the trace field of $\Gamma$. Denote by $K=k^{\text{Galois}}\subset\bbc$ the Galois closure of $k$. Then both the canonical subvariation and each of its Galois conjugates $\mathcal{S}_{\omega }^{\sigma }$ occur with multiplicity one in $\mathcal{J}\otimes\bbc$; their sum is defined over $\bbq$. Yet we can, under these conditions, not prove that the Galois conjugates also define sub-variations of pseudo-Hodge structure. This changes when we assume that $(X,\omega )$ is a Veech surface:
\end{NonelementaryWithParabolic}
\begin{TheVeechGroupALattice}
Assume now that $(X,\omega )$ is a Veech surface, i.e. that $\Gamma (X,\omega )$, and hence also $\Gamma$, is a lattice. This is equivalent to $C$ being of finite type, which in turn is equivalent to $C$ being compactifiable as a complex manifold. To begin with, in this case the trace field of $\Gamma$ is equal to the trace field of $\SL (X,\omega )$, in particular it only depends on $(X,\omega )$ and not on a particular choice of $\Gamma$. Now since all the Galois conjugates appear with multiplicity one\footnote{Note that we need not use that the canonical subsystem itself underlies a variation of Hodge structure --- this follows already by its having multiplicity one!}, we can conclude that they all underlie sub-$K$-variations of pseudo-Hodge structure. In particular their sum is a sub-$\bbq$-variation of Hodge structure of $\mathbb{J}$. There is a way to describe this structure closer to the intuition of abelian varieties, using \textit{real multiplication}, see the next chapter.

In fact we can get rid of the ``pseudo'': M. M\"{o}ller has shown in \cite[Prop. 2.6]{Moeller05b} that the trace field of a lattice Veech group is always totally real; we reproduce his argument in chapter 7. Thus $K$ is contained in $\bbr$, and we can talk about ``honest'' Hodge structures.
\end{TheVeechGroupALattice}

\subsection{Ahlfors' Variational Formula}

Now we have talked about the global structure of the family of Jacobians over a Teichm\"{u}ller curve; we also need some information on how it behaves locally. For this we first quote \textit{Ahlfors' variational formula:} let $X$ be a Teichm\"{u}ller marked surface; then its Jacobian $J(X)$ is a Hodge structure of Jacobian type (with our definition), whose underlying abelian group has a ``standard'' symplectic basis induced by the Teichm\"{u}ller marking. We refer to the period matrix of this Jacobian with respect to this basis simply as ``the period matrix of $X$'' and denote it by $\Pi (X)$. Also denote the adapted basis of $\Omega^1(X)$ by $\omega_1,\ldots ,\omega_g$, so that we have
\begin{equation*}
\Pi_{ij}(X)=\int_{b_i}\omega_j=\int_{b_j}\omega_i.
\end{equation*}
Now let $\mu\in\mathcal{B}^{<1}(X)$; then $X_{\mu }$ is also a Teichm\"{u}ller marked Riemann surface. The formula now gives an estimate on how $\Pi (X_{\mu })$ behaves when $\mu$ is small.
\begin{Theorem}[Ahlfors]
With the above notation (for fixed $X$ and varying $\mu$), one has
\begin{equation*}
\Pi_{ij}(X)-\Pi_{ij}(X_{\mu })=\int_X(\omega_i\otimes\omega_j)\mu +O(\lVert\mu\rVert_{\infty }^2).
\end{equation*}
\end{Theorem}
\begin{proof}
This formula is originally due to Ahlfors, see \cite[p. 56]{Ahlfors60}, albeit in a somewhat different language. A proof of the formulation given here can be found in \cite[section 4.1.6]{Nag88}. Our formula is precisely formula (1.25) in loc.cit.
\end{proof}
Now let $(X,\omega )$ be a Teichm\"{u}ller marked translation surface, with associated Teichm\"{u}ller disk $f^{\omega }:\Delta\to\mathrsfs{T}_g$ or $h^{\omega }:\bbh\to\mathrsfs{T}_g$. As above, we obtain variations of Hodge structure of Jacobian type $\tilde{\mathbb{J}}^0$ on $\Delta$ and $\tilde{\mathbb{J}}$ on $\bbh$, such that the latter is the pullback of the former along $C:\bbh\to\Delta$.
\begin{Corollary}
With the above notation, we have
\begin{equation*}
\frac{\de }{\de t}|_{t=0}\Pi_{ij}(f^{\omega }(t))=\int_X(\omega_i\otimes\omega_j)\frac{\overline{\omega }}{\omega }
\end{equation*}
and
\begin{equation*}
\frac{\de }{\de t}|_{t=\mathrm{i}}\Pi_{ij}(h^{\omega }(t))=\frac{\mathrm{i}}{2}\int_X(\omega_i\otimes\omega_j)\frac{\overline{\omega }}{\omega }.
\end{equation*}
\end{Corollary}
\begin{proof}
For the difference quotient we compute:
\begin{equation*}
\frac{\Pi_{ij}(f^{\omega }(t))-\Pi_{ij }(f^{\omega }(0))}{t}=\frac{\Pi_{ij}(X_{t\cdot\overline{\omega }/\omega })-\Pi_{ij }(X)}{t}=\int_X(\omega_i\otimes\omega_j)\frac{\overline{\omega }}{\omega }+O(t).
\end{equation*}
The second formula follows by noting that $h^{\omega }=f^{\omega }\circ C$ and applying the chain rule (one has $C'(\mathrm{i})=\mathrm{i}/2$).
\end{proof}

\newpage
\section{Real Multiplication}

In this chapter we finally turn to real multiplication on Jacobian varieties. By \textit{real multiplication} on an abelian variety $A$, one usually means a monomorphism $K\to\End A\otimes\bbq$ with certain properties, where $K$ is a totally real number field. For the situation we are interested in it is more convenient to use a more flexible convention where the field is also allowed to operate on an abelian subvariety. Since we had to consider abelian varieties only up to isogeny before, we also do this here, which allows us to avoid number-theoretic subtleties associated with orders in number fields.

There are deep relations between properties of translation surface and the existence of real multiplication (in this wider sense) on the Jacobian of the underlying Riemann surface. Namely for a Veech surface $(X,\omega )$, the results of section 6.6 can be reformulated in the following vein: one has real multiplication by the trace field of $\SL (X,\omega )$ on the Jacobian of $X$, with $\omega$ as an eigenform. The real multiplication structure is uniquely characterized by the operation of the trace field on the canonical subspace. Furthermore, this structure is preserved by the $\SL_2(\bbr )$-operation. This theorem, due to M\"{o}ller, rests upon the results of Deligne and Schmid discussed in the previous chapter.

Now in genus two, there is an elementary argument due to McMullen which shows that one does not need $(X,\omega )$ to be Veech. Hence the following question is natural: can this elementary argument perhaps be extended to higher genera? The answer, provided by Theorem 8.3.3, is: no. In all genera $g\ge 3$, for suitable $k$ (e.g. $k$ coprime to $2g+1$, which can always be achieved) the Jacobian of $W_g$ carries a real multiplication structure by the trace field of $\SL (W_g,\omega_k)$ with $\omega_k$ as an eigenform, but which is not preserved by the $\SL_2(\bbr )$-action.

\subsection{Real Multiplication}

\begin{DefinitionOfRM}
In analogy to the more classical notion of complex multiplication on elliptic curves, real multiplication is a slogan for certain endomorphism structures on abelian varieties (up to isogeny). One is very naturally led to real multiplication when considering large families of self-adjoint endomorphisms:
\end{DefinitionOfRM}
\begin{Proposition}\label{SelfAdjointFieldActionsOnVHSMustBeTotallyReal}
Let $X$ be a complex manifold, let $\mathbb{V}$ be a rational variation of Hodge structures of Jacobian type, of rank $g$, let $K$ be a number field of degree $d$ and let $\varrho :K\to\End\mathbb{V}$ be a ring homomorphism whose image consists entirely of self-adjoint maps (with respect to the polarization). Then $K$ is totally real, and $d\le g$.
\end{Proposition}
\begin{proof}
Let $a\in K$ be a primitive element, and let $f\in\bbq [x]$ be the minimal polynomial of $a$ over $\bbq$. Its roots are precisely the Galois conjugates of $a$. Since $K$ is a field, the homomorphism $\varrho$ must be injective, and consequently the minimal polynomial of the endomorphism $\varrho (a)$ is also $f$. Its roots are the eigenvalues of $\varrho (a)$, but the eigenvalues of a self-adjoint map are always real. So all Galois conjugates of $a$ must be real. Also we see that $\varrho (a)$ has precisely $d$ distinct eigenvalues, which implies that $d\le g$.
\end{proof}
The definition of real multiplication that we shall work with is the following:
\begin{Definition}
Let $J$ be a polarized $\bbq$-Hodge structure of Jacobian type of rank $g$ and let $K$ be a totally real number field of degree $d\le g$. Then a \textup{real multiplication structure by $K$ on $J$} consists of the following data:
\begin{enumerate}
\item a splitting $J=E\oplus F$ of $\bbq$-Hodge structures, orthogonal with respect to the polarization, such that $E$ has rank $d$;
\item a ring homomorphism $K\to\End E$ such that every element of $K$ is sent to a self-adjoint endomorphism of $E$.
\end{enumerate}
\end{Definition}
For the special case when $J$ is actually defined over $\bbz$, i.e. the Hodge structure corresponding to an abelian variety $A$, and when $d=g$, we retrieve the classical definition of real multiplication. Namely we have $J=E$, and the subring
\begin{equation*}
\mathfrak{o}=\{ a\in K : \varrho (a)(J_{\bbz })\subseteq J_{\bbz }\}
\end{equation*}
is an \textit{order} in $K$, i.e. finitely generated as a group and generating $K$ as a $\bbq$-vector space. Then $\varrho$ defines an injection $\mathfrak{o}\to\End A$.

We analogously define real multiplication for variations of Hodge structure:
\begin{Definition}
Let $X$ be a complex manifold, let $\mathbb{J}$ be a polarized variation of $\bbq$-Hodge structure of Jacobian type of rank $g$ on $X$, and let $K$ be a totally real number field of degree $d\le g$. Then a \textup{real multiplication structure by $K$ on $\mathbb{J}$} consists of the following data:
\begin{enumerate}
\item a splitting $\mathbb{J}=\mathbb{E}\oplus\mathbb{F}$ of variations of $\bbq$-Hodge structure, orthogonal with respect to the polarization, such that $\mathbb{E}$ has rank $d$;
\item a ring homomorphism $K\to\End \mathbb{E}$ such that every element of $K$ is sent to a self-adjoint endomorphism of $\mathbb{E}$.
\end{enumerate}
\end{Definition}
Note that if $\mathbb{J}$ is a variation of Hodge structures with real multiplication in the above sense, we get a real multiplication structure on every fibre $\mathbb{J}_x$. This uniquely determines the real multiplication structure on the whole of $\mathbb{J}$ (as soon as $X$ is connected), in the following sense:
\begin{Lemma}
Let $\mathbb{J}$ be a polarized variation of $\bbq$-Hodge structure of Jacobian type on a connected complex manifold $X$, and for some point $x\in X$ let there be given a real multiplication structure by a totally real number field $K$ on the fibre $\mathbb{J}_x$. Then there is \textup{at most one} real multiplication structure on $\mathbb{V}$ whose restriction to $\mathbb{J}_x$ is the given one.
\end{Lemma}
\begin{proof}
Let $\mathbb{J}$ be a polarized variation of $\bbq$-Hodge structure of Jacobian type of rank $g$ on $X$ together with a real multiplication structure in the above sense. We show that it is possible to reconstruct this structure from the real multiplication structure on the fibre $\mathbb{J}_x$; this will prove the lemma.

So write $J=\mathbb{J}_x$, $E=\mathbb{E}_x$ and $F=\mathbb{F}_x$. The sub-variations of Hodge structure $\mathbb{E}$ and $\mathbb{F}$ are uniquely determined by the underlying local systems $\mathcal{E},\mathcal{F}\subseteq\mathcal{V}_{\bbq }$, and these in turn are, by the usual monodromy correspondence, uniquely determined by their fibres $\mathcal{E}_x=E_{\bbq }$ and $\mathcal{F}_x=F_{\bbq }$. Hence the splitting $\mathbb{J}=\mathbb{E}\oplus\mathbb{F}$ can be reconstructed from the splitting $J=E\oplus F$.

Also by the monodromy correspondence, any endomorphism of a local system is uniquely determined by the induced endomorphism on one fibre. Hence the map $K\to\End\mathbb{E}$ is uniquely determined by the map $K\to\End E$.
\end{proof}
Hence if we are given a polarized variation of Hodge structure together with a real multiplication structure on one fibre, we can ask whether it extends to the whole variation, and if it does, the extension is unique.

\begin{OrthogonalDecomposition}
Real multiplication as above induces an orthogonal decomposition over the Galois closure of the involved field (see Definition \ref{DefinitionGaloisClosure}):
\end{OrthogonalDecomposition}
\begin{Proposition}
Let $\mathbb{J}$ be a rational variation of Hodge structure of Jacobian type of rank $g$, let $K$ be a totally real number field of degree $d\le g$ and let the splitting $\mathbb{J}=\mathbb{E}\oplus\mathbb{F}$ and the homomorphism $\varrho :K\to\End\mathbb{E}$ define a real multiplication structure on $\mathbb{J}$. Then we have a unique orthogonal decomposition
\begin{equation*}
\mathbb{J}\otimes K^{\text{Galois}}=\bigoplus_{j=1}^d\mathbb{S}^{\sigma_j}
\end{equation*}
of $K^{\text{Galois}}$-variations of Hodge structure with the following properties:
\begin{enumerate}
\item $\mathbb{S}^{\sigma_j}$ has rank one for every $1\le j\le d$,
\item the action $\varrho$ of $K$ on $\mathbb{E}$ respects this decomposition and
\item every $a\in K$ operates on $\mathbb{S}^{\sigma_j}$ by multiplication with $a^{\sigma_j}$.
\end{enumerate}
\end{Proposition}
\begin{proof}
We only need to show the existence of such a decomposition; uniqueness then follows automatically. So let $a\in K$ be a primitive element, i.e. such that $K=\bbq (a)$. Note that then the elements $\sigma_1(a),\ldots ,\sigma_d(a)$ of $K^{\text{Galois}}$ are pairwise different. Let then $\mathbb{S}^{\sigma_j}$ be the kernel of the morphism of $K^{\text{Galois }}$-variations of Hodge structure
\begin{equation*}
\varrho (a)-\sigma_j(a)\cdot\mathrm{id }:\mathbb{E}\otimes K^{\text{Galois }}\to \mathbb{E}\otimes K^{\text{Galois}}.
\end{equation*}
Since $\varrho (a)$ is self-adjoint, the $\mathbb{S}^{\sigma_j}$ are mutually orthogonal. It remains to show that they all have rank at least one, i.e. are not reduced to zero, since then by dimension reasons they must have rank exactly one.

For this note that the minimal polynomial of the endomorphism $\varrho (a)$ is the same as the minimal polynomial of the algebraic number $a$. Its roots are precisely the $a^{\sigma_j}$. By the Cayley-Hamilton theorem then all $\sigma_j(a)$ are eigenvalues of $\varrho (a)$, hence the $\mathbb{S}^{\sigma_j}$ are nonzero.
\end{proof}

\begin{ConstructingRMFromASubsystem}
By the deep structure results on variations of Hodge structure over compactifiable complex manifolds, we can give a partial inverse to this construction:

Let $X$ be a compactifiable complex manifold and let $\mathbb{J}$ be a rational variation of Hodge structure of Jacobian type on $X$, of rank $g$. Assume in addition that $\mathbb{J}$ can be defined over $\bbz$, i.e. there is some principally polarized integral variation of Hodge structure $\mathbb{J}_{\bbz }$ such that $\mathbb{J}\simeq\mathbb{J}_{\bbz }\otimes\bbq$, respecting the polarizations. This is e.g. the case for variations of the form $R^1p_{\ast }\bbq (0)$ where $p$ is a family of algebraic curves. Let $\mathcal{J}$ be the underlying local system of $\bbq$-vector spaces; by Theorem \ref{MonodromyRepOfVHSOverZIsSemisimple}, $\mathcal{J}$ is semisimple.

Assume further that there exists a number field $K$, which we for simplicity assume to be contained in $\bbc$, and a simple local system of $K$-vector spaces $\mathcal{S}$ which occurs \textit{with multiplicity one} in the semisimple local system $\mathcal{J}\otimes K$. Let $\mathrm{id} =\sigma_1,\ldots ,\sigma_d$ be the embeddings of $K$ into $\bbc$, then the Galois conjugates $$\mathcal{S}\otimes K^{\text{Galois }}=\mathcal{S}^{\sigma_1},\ldots ,\mathcal{S}^{\sigma_d}$$
all appear with multiplicity one in $\mathcal{J}\otimes K^{\text{Galois}}$, and are all mutually orthogonal. Furthermore by Proposition ... they all underly uniquely defined subvariations of $K^{\text{Galois}}$-pseudo-Hodge structure $\mathbb{S}^{\sigma_j}$. Their sum is defined over $\bbq$, say equal to $\mathbb{E}\otimes K^{\text{Galois}}$ for some rational subvariation $\mathbb{E}\subseteq\mathbb{J}$. Denote its orthogonal complement by $\mathbb{F}$.

Then we can define a ring homomorphism
\begin{equation*}
\varrho :K\to\End_{K^{\text{Galois}}}(\mathbb{E}\otimes K^{\text{Galois}})=\prod_{j=1}^d\End\mathbb{S}^{\sigma_j}
\end{equation*}
by sending $a\in K$ to $(\sigma_1(a),\ldots ,\sigma_d(a))$. By Galois invariance, $\varrho$ in fact has image contained in $\End_{\bbq }\mathbb{E}$, and the image clearly consists only of self-adjoint endomorphisms. For further reference:
\end{ConstructingRMFromASubsystem}
\begin{Proposition}
Under the above assumptions, $K$ is totally real.
The decomposition $\mathbb{J}=\mathbb{E}\oplus\mathbb{F}$ together with the homomorphism $\varrho :K\to\End\mathbb{E}$ as constructed above define a real multiplication datum by $K$ on $\mathbb{J}$.
\end{Proposition}
\begin{proof}
The first part follows from Proposition \ref{SelfAdjointFieldActionsOnVHSMustBeTotallyReal}; the second part has already been shown.
\end{proof}

\begin{RMandTCurves}
We now take up again the discussion of section 7.6. Assume that $(X,\omega )$ is a Veech surface. Then the construction just explained can be applied to the variation of Hodge structure $\mathbb{J}=R^1f_{\ast }\bbq (0)$ of the associated family $f:\mathrsfs{X}\to C$. We summarize
\end{RMandTCurves}
\begin{Theorem}[M\"{o}ller]\label{RMForVeechSurfaces}
Let $(X,\omega )$ be a Veech surface. Then the trace field $K$ of the Veech group $\SL (X,\omega )$ is totally real. The Jacobian $J(X)$ has a canonical real multiplication structure given by $J(X)=E\oplus F$ and $\varrho :K\to\End E$.

This real multiplication structure is uniquely characterized by the following properties: $E$ is the smallest sub-$\bbq$-Hodge structure containing the canonical subspace, and $a\in K$ operates on the canonical subspace by multiplication with $a$.
\end{Theorem}
This theorem is essentially \cite[Theorem 2.7]{Moeller05b}.

\subsection{Real Multiplication and the $\SL_2(\bbr)$-Action}

In this section we shall explain known results about the interplay between real multiplication on the Jacobian and the $\SL_2(\bbr)$-action on the moduli space of abelian differentials.

\begin{TransportingRealMultiplication}
Let $X$ be a closed Riemann surface, and assume we are given some real multiplication structure
\begin{equation*}
J(X)=E\oplus F,\quad \varrho :K\to\End E
\end{equation*}
on the Jacobian of $X$. Let $\omega$ be an abelian differential on $X$. Then we obtain a Teichm\"{u}ller disk $h^{\omega }:\bbh\to\mathrsfs{T}_g$ and an associated rational variation of Hodge structure $\tilde{\mathbb{J}}$ of Jacobian type on $\bbh$. We can identify the fibre of $\tilde{\mathbb{J}}$ over $\mathrm{i}\in\bbh$ with $J(X)$, and the question is now: does the real multiplication structure on $J(X)$ extend to a real multiplication structure (which is then necessarily unique) on $\tilde{\mathbb{J}}$? If it does, we say that \textit{the real multiplication structure is preserved by the $\SL_2(\bbr )$-action on $(X,\omega )$}. Note that the real multiplication structure on the Jacobian for a Veech surface $(X,\omega )$, as specified in Theorem \ref{RMForVeechSurfaces}, is preserved by the $\SL_2(\bbr )$-action because of its very construction.

This condition can also be expressed directly in terms of the $\SL_2(\bbr )$-action. The real multiplication structure can be interpreted as a splitting $H^1(X,\bbq )=E_{\bbq}\oplus F_{\bbq}$ and a homomorphism $\varrho :K\to\End_{\bbq }E_{\bbq }$ with certain properties. Now for any $A\in\SL_2(\bbr )$ we can identify the topological surface $Y$ underlying $A\cdot (X,\omega )$ with $X$. Hence we also have a splitting of $H^1(Y,\bbq )$ and an action of $K$ on one part of it. The condition to be satisfied is then that these data also define a real multiplication structure on $J(Y)$.
\end{TransportingRealMultiplication}
\begin{Eigenforms}
Assume now that the real multiplication structure is preserved. Then as above we get an orthogonal decomposition of $\tilde{\mathbb{J}}$ which in turn induces an orthogonal decomposition of the associated variation of $\bbr$-Hodge structure:
\begin{equation*}
\tilde{\mathbb{J}}\otimes\bbr =\bigoplus_{j=1}^d(\mathbb{S}^{\sigma_j}\otimes\bbr )\oplus (\mathbb{F}\otimes\bbr ).
\end{equation*}
Here the $\mathbb{S}^{\sigma_j}\otimes\bbr$ are simple variations of $\bbr$-Hodge structure. On the other hand we have the canonical subvariation $\mathbb{S}_{\omega }$ which is also simple. Hence if it is contained in $\mathbb{E}$, i.e. the part of $\tilde{\mathbb{J}}$ where something interesting happens, it has to be equal to one of the $\mathbb{S}^{\sigma_j}$.

But this in turn means that $\omega$ already has to lie in $E\subseteq J(X)$ and that it is an \textit{eigenform} for the real multiplication structure, i.e. $\varrho (K)\omega\subset\bbc\omega$. So the following picture will appear:

For every $A\in\SL_2(\bbr )$, the real multiplication structure carries over to the Jacobian $J(X')$, where $A\cdot (X,\omega )=(X',\omega ')$. Furthermore $\omega '$ is also an eigenform for this structure, for the same eigencharacter as $\omega$.
\end{Eigenforms}

\begin{VeechSurfaces}
We can rephrase Theorem \ref{RMForVeechSurfaces} in this context as follows:
\end{VeechSurfaces}
\begin{Corollary}
Let $X$ be a closed Riemann surface and let $\omega$ be an abelian differential on it, such that the Veech group $\SL (X,\omega )$ is a lattice. Then $\omega$ is an eigenform for the real multiplication structure on $J(X)$ by the trace field constructed above, and this real multiplication structure is preserved by the $\SL_2(\bbr )$-action on $(X,\omega )$.\hfill $\square$
\end{Corollary}
\begin{GenusTwo}
There is also a complete picture of when real multiplication is transported along the orbit of an eigenform in genus two: always! To understand why, let us first remark that real multiplication structures in genus two cannot be very complicated. The number field has to have degree at most two, hence it is either $\bbq$ or a real quadratic field.

Let $X$ be a closed Riemann surface of genus two and let $\omega$ be an abelian differential on $X$. A real multiplication structure by $\bbq$ on $J(X)$ is then nothing but an orthogonal decomposition of rational Hodge structures $J(X)=E\oplus F$, where each summand has rank one. Furthermore, $\omega$ is an eigenform for this real multiplication structure if and only if $\omega\in E^{1,0}$. This is equivalent to $E\otimes\bbr$ being equal to the canonical subspace in $H^1(X,\bbr )$. Since $F$ is the orthogonal complement of $E$, this means that we can reconstruct the real multiplication structure on $X$ solely from $\omega$.
\end{GenusTwo}
\begin{Proposition}
Let $X$ be a closed Riemann surface of genus two, together with some real multiplication structure by $\bbq$ on $J(X)$. Let further $\omega$ be an abelian differential on $X$ which is an eigenform for this real multiplication structure. Then the real multiplication structure is preserved by the $\SL_2(\bbr )$-action on $(X,\omega )$.
\end{Proposition}
\begin{proof}
Let $J(X)=E\oplus F$ be the real multiplication structure, so that $\omega\in E^{1,0}$. Then $E\otimes\bbr =S_{\omega }\subset H^1(X,\bbr )$. Hence $E$ extends to a local subsystem $\mathcal{E}\subset\mathcal{J}_{\bbq }$ with the property that $\mathcal{E}\otimes\bbr =\mathcal{S}_{\omega }$. Hence $\mathcal{E}$ underlies a subvariation of $\bbq$-Hodge structure $\mathbb{E}$ (in other words, the canonical subvariation is defined over $\bbq$), and we can set $\mathbb{F}=\mathbb{E}^{\perp}$.
\end{proof}
Consider now a real quadratic number field $K=\bbq (\sqrt{d})$, where $d$ is a squarefree positive integer. Let $J$ be a rational Hodge structure of Jacobian type of rank two; then a real multiplication datum on $M$ is the same as a ring homomorphism $\varrho :K\to\End J$ whose image consists entirely of self-adjoint endomorphisms.
\begin{Proposition}
Let $X$ be a closed Riemann surface of genus two, let $\omega$ be an abelian differential and let $\varrho :K\to\End J(X)$ be a real multiplication structure by a real quadratic number field $K$, such that $\omega$ is an eigenform for this structure. Then the real multiplication structure is preserved by the $\SL_2(\bbr )$-action on $(X,\omega )$.
\end{Proposition}
\begin{proof}
We get a decomposition into eigenspaces $J(X)\otimes K=S^{\sigma_1}\oplus S^{\sigma_2}$. Here $S^{\sigma_1}$ and $S^{\sigma_2}$ are sub-$K$-Hodge structures, the decomposition is orthogonal, and $a\in K$ operates on $S^{\sigma_j}$ by multiplication with $\sigma_j(a)$. Here $\sigma_1$ and $\sigma_2$ are the two embeddings $K\to\bbr$.

Now by assumption $\omega$ is an eigenform, and without loss of generality we may assume that $\omega\in (S^{\sigma_1})^{1,0}$. Then we have $S_{\omega }=S^{\sigma_1}\otimes_K\bbr$. So by the same argument as above, the canonical subvariation is defined over $K$, i.e. equal to $\mathbb{S}^{\sigma_1}\otimes_K\bbr$ for some sub-variation of $K$-Hodge structure $\mathbb{S}^{\sigma_1}$ of $\tilde{\mathbb{J}}\otimes K$, of rank one. Let $\mathbb{S}^{\sigma_2}$ be its orthogonal complement. Then define a real multiplication structure on $\tilde{\mathbb{V}}$ by
\begin{equation*}
(\sigma_1,\sigma_2 ):K\to\End_K(\tilde{\mathbb{J}}\otimes K)=\End_K\mathbb{S}^{\sigma_1}\times\End_K\mathbb{S}^{\sigma_2}=K\times K.
\end{equation*}
This actually has image in $\End_{\bbq }\tilde{\mathbb{J}}$ and hence defines a real multiplication structure which extends the given one on $J(X)$.
\end{proof}
\begin{RMbytheTraceField}
If $(X,\omega )$ is a Veech surface, we get real multiplication by the trace field of the Veech group with $\omega$ as an eigenform, and this structure is transported by the $\SL_2(\bbr )$-action. In genus two, this is also true without the assumption of $(X,\omega )$ being Veech.
\end{RMbytheTraceField}
\begin{Lemma}\label{LemmaRMbyTraceField}
Let $X$ be a complex manifold, let $\mathbb{J}$ be a rational variation of Hodge structure of Jacobian type on $X$, of rank two. Let $\mathbb{S}$ be a sub-$\bbr$-variation of Hodge structure of $\mathbb{J}\otimes\bbr$, and let $T$ be an endomorphism of the local system of $\bbq$-vector spaces $\mathcal{J}$ underlying $\mathbb{J}$, which is self-adjoint with respect to the polarization and which takes the sub-variation $\mathbb{S}$ to itself. Then $T$ is actually an endomorphism of $\mathbb{J}$ (i.e. a morphism of variations of Hodge structure).
\end{Lemma}
\begin{proof}
Let $\mathbb{S}^{\perp}$ be the orthogonal complement of $\mathbb{S}$; this is a sub-$\bbr$-variation of Hodge structure of $\mathbb{J}\otimes\bbr$, and $\mathbb{J}\otimes\bbr =\mathbb{S}\oplus\mathbb{S}^{\perp}$. Since $T$ is self-adjoint, it also has to map $\mathbb{S}^{\perp}$ to itself. Since any self-adjoint endomorphism of a two-dimensional symplectic vector space is a constant times the identity, $T|_{\mathbb{S}}$ as well as $T|_{\mathbb{S}^{\perp}}$ are multiples of the identity and hence respect the Hodge decompositions. Hence also $T$ acting on the whole of $\mathbb{V}$ respects the Hodge decompositions.
\end{proof}
Stated in somewhat different language, but with the same proof up to translation, this lemma is due to McMullen, see \cite[Lemma 7.4]{McMullen03}.
\begin{Proposition}
Let $X$ be a closed Riemann surface of genus two, and let $\omega$ be an abelian differential on it. Assume that the trace field $K$ of the Veech group $\SL (X,\omega )$ is real quadratic. Then the canonical subvariation $\mathbb{S}_{\omega }$ as well as its orthogonal complement $\mathbb{S}_{\omega }^{\perp}$ are defined over $K$.
\end{Proposition}
\begin{proof}
Choose some $\varphi\in\Aff^+(X,\omega )$ whose linear part has irrational trace $a=\tr D\varphi$, hence $K=\bbq (a)$. Now recall that topologically the family $\tilde{f}:\tilde{X}\to\bbh$ is just a product $X\times\bbh\to\bbh$, and every fibre can be canonically identified with $X$ (just with a different complex structure). Extend $\varphi$ to a self-homeomorphism of $\mathrsfs{X}$ by setting $\varphi (x,\tau )=(\varphi (x),\tau )$. On every fibre $\mathrsfs{X}_{\tau }$, which is the Riemann surface underlying say $A\cdot (X,\omega )$, this induces again an affine map, with the same trace and the same action on cohomology. Hence we get an automorphism $\varphi^{\ast }$ of the local system $\mathcal{J}=R^1\tilde{f}_{\ast }\bbq$, symplectic with respect to the Poincar\'{e} pairing.

Set then $\Psi =\varphi^{\ast }+\varphi^{\ast ,-1}$. This is a self-adjoint map (by the same trick as above), and it stabilizes the sub-variation of real Hodge structure $\mathbb{S}_{\omega }$. Hence by Lemma \ref{LemmaRMbyTraceField} it is a morphism of integral variations of Hodge structure. It operates on $\mathbb{S}$ as $a$ times the identity and on $\mathbb{S}^{\perp}$ as some other number times the identity; this other number must then necessarily be the Galois conjugate $a'$ of $a$ (since the characteristic polynomial of $\Psi$ has integral coefficients).

Hence we can define $\mathbb{S}_{\omega ,K}$ as the kernel of the endomorphism $T-a\cdot \mathrm{id}$ of $\mathbb{V}\otimes K$, and analogously $\mathbb{S}_{\omega ,K}^{\perp}$ as the kernel of $T-a'\cdot\mathrm{id}$. These are sub-$K$-variations of Hodge structure whose base change to $\bbr$ are the canonical subvariation and its orthogonal complement.
\end{proof}
Now we can apply the usual trick to define real multiplication on $\mathbb{V}$:
\begin{Corollary}
Let $X$ be a closed Riemann surface of genus two and let $\omega$ be an abelian differential on it. Assume that the trace field of Veech group $\SL (X,\omega )$ is strictly larger than $\bbq$, i.e. is a real quadratic number field $K$. Then there is a unique real multiplication structure by $K$ on $J(X)$ such that $\omega$ is an eigenform for the identity character $K\to\bbc$. This structure is preserved by the $\SL_2(\bbr )$-action on $(X,\omega )$.
\end{Corollary}
\begin{proof}
The uniqueness is clear: if $\omega$ is an eigenform, its orthogonal complement has to be the eigenspace for the other embedding $K\to\bbc$. As to the existence, let $K$ act by the identical character on $\mathbb{S}_{\omega }$ and by the other character on $\mathbb{S}_{\omega}^{\perp}$.
\end{proof}
\begin{ImplicationsForThePeriodMatrix}
Let $X$ be a closed Riemann surface, now again of arbitrary genus $g$, with an abelian differential $\omega$. Let $K$ be a totally real number field of degree $d\le g$, and let
\begin{equation*}
J(X)=E\oplus F,\quad \varrho :K\to\End E
\end{equation*}
be a real multiplication datum with $\omega$ as an eigenform. Hence we have an orthogonal decomposition:
\begin{equation*}
E\otimes K^{\text{Galois}}=\bigoplus_{j=1}^dS^{\sigma_j}
\end{equation*}
where $\sigma_1,\ldots ,\sigma_d$ are the different embeddings $K\to\bbr$. Without loss of generality we may assume that $S_{\omega }=S^{\sigma_1}\otimes\bbr$.

Choose then a symplectic $K$-basis $(a_j,b_j)$ of the subspace of $H_1(X,K^{\text{Galois}})$ corresponding to $S^{\sigma_j}$, and a symplectic $K$-basis $(a_{d+1},\ldots ,a_g,b_{d+1},\ldots ,b_g)$ of the subspace corresponding to $F$. This defines altogether a symplectic $K$-basis $(a_1,\ldots ,a_g,b_1,\ldots ,b_g)$ of $H_1(X,K^{\text{Galois}})$. What does the associated period matrix look like, and how does it behave under the $\SL_2(\bbr )$-operation?

The basis of $\Omega^1(X)=J(X)^{1,0}$ we need for the computation of the period matrix has the following form:
\begin{equation*}
(c\cdot\omega ,\omega^{\sigma_2},\ldots \omega^{\sigma_d},\eta_{d+1},\ldots ,\eta_g).
\end{equation*}
Here $c\neq 0$ is an unimportant constant, $\omega^{\sigma_j}$ is a generator of $(S^{\sigma_j})^{1,0}$, and the $\eta_j$ are elements of $F^{1,0}$. From this we deduce: the period matrix has the form
\begin{equation}
\begin{pmatrix}
D&0\\
0&P
\end{pmatrix}
\end{equation}
where $D$ is a $d\times d$ diagonal matrix with nonzero entries, and $P$ is any matrix.
\begin{Proposition}\label{PeriodMatrixForTransportedRM}
With notation as above, the period matrix of $J(X)$ with respect to the basis $(a_1,\ldots ,a_g,b_1,\ldots ,b_g)$ has the form
\begin{equation*}
\begin{pmatrix}
D&0\\
0&P
\end{pmatrix}
\end{equation*}
where $D$ is a $d\times d$ diagonal matrix with nonzero entries. If the real multiplication structure is preserved by the $\SL_2(\bbr )$-action on $(X,\omega )$, the period matrix of $\tilde{\mathbb{J}}_{\tau }$ for any $\tau\in\bbh$ with respect to the same basis has the form
\begin{equation*}
\begin{pmatrix}
D(\tau )&0\\
0&P(\tau )
\end{pmatrix}
\end{equation*}
where $D=\operatorname{diag} (f_1(\tau ),\ldots ,f_d(\tau ))$. Here the $f_j$ are holomorphic functions on $\bbh$ with values in $\bbh$, in particular they are never zero.
\end{Proposition}
\begin{proof}
It only remains to show that the $f_j$ have image in $\bbh$, but this follows from the fact that period matrices always have positive definite imaginary part.
\end{proof}
\end{ImplicationsForThePeriodMatrix}

\subsection{The Jacobians of the Wiman Curves}

We are going to construct a remarkable family of real multiplication structures on the Jacobians of the Wiman curves. \begin{DecompositionOfJacobian}
Consider the automorphism $\varphi$ of $X$ which is given in affine coordinates by $\varphi (x,y)=(\zeta x,y)$ with $\zeta =\exp\frac{2\pi\mathrm{i}}{2g+1}$. This has order exactly $2g+1=n$ in the automorphism group of $X$. It induces by functoriality an automorphism $\varphi^{\ast }$ of the Hodge structure $J(X)$; this is also of order precisely $2g+1$ (because the Torelli group is torsion free). The \emph{endomorphism}
\begin{equation*}
\Psi =\varphi^{\ast }+\varphi^{\ast ,-1}
\end{equation*}
of $J(X)$ shall be of greater importance for us, since it has some nice properties:
\end{DecompositionOfJacobian}
\begin{Lemma}
The endomorphism $\Psi$ is self-adjoint with respect to the polarization
\end{Lemma}
\begin{proof}
The endomorphism $\varphi^{\ast }$ is symplectic, i.e. $S(\varphi^{\ast }v,\varphi^{\ast }w)=S(v,w)$ for $S$ the polarization of $J(X)$. This is true for any endomorphism of the first cohomology of a compact orientable surface which is induced by an orientation-preserving self-homeomorphism of that surface. Analogously $\varphi^{\ast ,-1}$ is symplectic. Hence the equations
\begin{equation*}
S(\varphi^{\ast }v,w)=S(v,\varphi^{\ast ,-1}w)\text{ and }S(\varphi^{\ast ,-1}v,w)=S(v,\varphi^{\ast}w);
\end{equation*}
adding them, we get that $\Psi =\varphi^{\ast}+\varphi^{\ast ,-1}$ is self-adjoint.
\end{proof}
We now get a nice eigenspace decomposition for $\Psi$ acting on $J(X)$ \textit{as a Hodge structure}; this decomposition is defined over the field $K_n=\bbq (\cos\frac{2\pi }{n})$. So let $J(X)\otimes K_n$ be the induced polarized $K_n$-Hodge structure, and for any integer $1\le k\le g$ (not necessarily prime to anything) let $T^k$ be the kernel of the morphism of $K_n$-Hodge structures
\begin{equation*}
\Psi -2\cos\frac{2k\pi}{2g+1}\cdot\mathrm{id} :J(X)\otimes K_n\to J(X)\otimes K_n.
\end{equation*}
Since Hodge structures form an abelian category, this is a sub-Hodge structure of $J(X)\otimes K_n$.
The Hodge type $(1,0)$ component of $T^k$ consists precisely of the holomorphic one-forms $\omega$ on $X$ satisfying $$\varphi^{\ast}\omega +\varphi^{-1,\ast}\omega =2\cos\frac{2k\pi}{2g+1}\cdot \omega .$$
A simple computation confirms that $\omega_k$ satisfies this equation, so $T^k$ has at least rank one. Since the $T^k$, as eigenspaces of a self-adjoint endomorphism, are mutually orthogonal, we also find that the $T^k$ must have rank precisely one.
\begin{PropositionDefinition}
Let $1\le t\le g$ be a divisor of $n=2g+1$. Then the sub-Hodge structure
\begin{equation*}
\bigoplus_{1\le k\le g\atop\gcd (k,n)=t}T^k
\end{equation*}
is defined over $\bbq$, say equal to $E^t$ for some sub-$\bbq$-Hodge structure $E^t\subseteq J(X)$.
\end{PropositionDefinition}
\begin{proof}
It is Galois-invariant.
\end{proof}
Hence we get an orthogonal decomposition
\begin{equation*}
J(W_g)=\bigoplus_{1\le t\le g\atop t|n}E^t.
\end{equation*}
The endomorphism $\Psi$ respects the direct sum decomposition, in particular we get by restriction a self-adjoint endomorphism $\Psi^t$ of $E^t$. Consider now $\Psi^t$ as a $\bbq$-linear endomorphism of $E^t_{\bbq }$, forgetting for the moment its compatibility with the Hodge structure. Then by construction $\Psi^t$ is diagonalizable (over a finite extension of $\bbq $), and its eigenvalues are precisely the Galois conjugates of $\cos\frac{2\pi}{n_0}$, where $n_0=n/t$.. Hence it has the same minimal polynomial over $\bbq$ as $\cos\frac{2\pi}{n_0}$, and we get a well-defined algebra homomorphism
\begin{equation*}
\varrho :K_{n_0}=\bbq (\cos\frac{2\pi }{n_0})\to\End E^t,\quad \cos\frac{2\pi }{n_0}\mapsto\Psi^t .
\end{equation*}
Every endomorphism in its image is a polynomial in $\Psi^t$, hence self-adjoint. Thus the orthogonal splitting $J(W_g)=E^t\oplus (\text{the rest})$ together with the homomorphism $\varrho :K_{n_0}\to\End E^t$ defines a real multiplication structure on $J(W_g)$.

\begin{TheWimanCurvesasVeechSurfaces}
We now relate the explicitly constructed real multiplication on $J(W_g)$ to the general theory developed before. To start with, $(W_g,\omega_t)$ is a Veech surface, and the Veech group $\SL (W_g,\omega_t)$ has trace field equal to $K_{n_0}$, where again $n_0=n/t$. This gives a real multiplication structure on $J(W_g)$ by $K_{n_0}$. Its characterization in Theorem \ref{RMForVeechSurfaces} shows that this is the same as the structure constructed above explicitly.
\end{TheWimanCurvesasVeechSurfaces}
\begin{Theorem}
Let $1<t<g$ be a divisor of $n=2g+1$, let $rt$ be the maximal multiple of $t$ which is $\le g$, and let $t<k<rt$ be such that $\gcd (k,n)=t$. Then $\omega_k$ is an eigenform for the real multiplication structure on $J(W_g)$ described above. Furthermore this real multiplication structure satisfies the conclusions\footnote{Not the conditions!} of Theorem \ref{RMForVeechSurfaces} except for the last sentence, i.e. it is real multiplication by the trace field of $\SL (W_g,\omega_k)$ such that $E$ is the smallest sub-$\bbq$-Hodge structure containing the canonical subspace, and on the canonical subspace, the trace field operates by the identity character.

This real multiplication structure is \emph{not} preserved by the $\SL_2(\bbr )$-action.
\end{Theorem}
\begin{proof}
We have shown before that $\bbq (\cos\frac{2\pi }{n_0})$ is the trace field of $\SL (W_g,\omega_k)$. Hence we only need to show that the real multiplication structure is not preserved by the $\SL_2(\bbr )$-action.

Let $h^{\omega }:\bbh\to\mathrsfs{T}_g$ be the Teichm\"{u}ller disk belonging to $(X,\omega_k)$, and let $\tilde{\mathbb{V}}$ be the associated variation of Hodge structure on $\bbh$.

For every $1\le j\le g$, let $(a_j,b_j)$ be a symplectic basis of the subspace of $H_1(X,K)$ belonging to the subspace $T_K^j\subset H^1(X,K)$. Then, up to renumeration, $(a_1,\ldots ,a_g,b_1,\ldots ,b_g)$ is a symplectic basis of $H_1(X,K)$ as in Proposition \ref{PeriodMatrixForTransportedRM}. The corresponding basis of $\Omega^1(X)$ is precisely $(\omega_1,\ldots ,\omega_k)$. Denote the period matrix of $\tilde{\mathbb{V}}_{\tau }$ with respect to this basis by $\Pi (\tau )$. Then by Proposition \ref{PeriodMatrixForTransportedRM} we see:
\begin{enumerate}
\item $\Pi_{jj}(\mathrm{i})\in\bbh$ for all $j$ with $\gcd (j,n)=t$;
\item $\Pi_{j\ell }(\mathrm{i})=0$ whenever at least one of $j$ and $\ell$ has greatest common divisor $t$ with $n$.
\end{enumerate}
If the real multiplication structure were preserved by the $\SL_2(\bbr )$-action, these statements would hold if $\mathrm{i}$ were replaced by an arbitrary $\tau\in\bbh$. But now the Ahlfors-Rauch formula tells us that
\begin{equation*}
\frac{\de\Pi_{\mu\nu }(\tau )}{\de\tau }|_{\tau =\mathrm{i}}=\int_X\omega_{\mu}\omega_{\nu}\cdot\frac{\overline{\omega_k}}{\omega_k}.
\end{equation*}
Choose $\ell$ such that $1\le k-\ell t<k+\ell t\le g$ and such that at least one of $k-\ell t$, $k+\ell t$ has greatest common divisor $t$ with $2g+1$. This is always possible since one can always achieve either $k-\ell t=t$ or $k+\ell t=rt$. Then
\begin{equation*}
\frac{\de\Pi_{k-\ell t,k+\ell t}(\tau )}{\de\tau }|_{\tau =\mathrm{i}}=\int_X\omega_{k-\ell t}\omega_{k+\ell t}\cdot\frac{\overline{\omega_k}}{\omega_k}=\int_X|\omega_k|^2>0
\end{equation*}
(where we have used $\omega_{k-a}\omega_{k+a}=\omega_k^2$),
contradiction.
\end{proof}
We summarize: for every $(W_g,\omega_k)$ with $g\ge 3$ and $1<k<g$ we have a real multiplication structure on $J(W_g)$ by the trace field of $\SL (W_g,\omega_k)$, with $\omega_k$ as eigenform for the identity character. In other words, we have the conclusion of Theorem \ref{RMForVeechSurfaces}, except for the stability under the $\SL_2(\bbr )$-action.

Now for every divisor $t$ of $n$ with $1\le t\le g$, we consider the set of those $1\le k\le g$ with $\gcd (k,n)=t$. Then for the minimal such $k$ (this is $t$) and for the maximal such $k$, we know that $(W_g,\omega_k)$ is a Veech surface and that the real multiplication structure is preserved by the $\SL_2(\bbr )$-action. For all other $k$, the real multiplication structure is not stable under the $\SL_2(\bbr )$-operation. In particular we also see that $(W_g,\omega_k)$ cannot be Veech for these $k$.

Finally let us note that this construction gives a counterexample to an extension of M\"{o}ller's theorem to non-Veech surfaces for every genus $g\ge 3$: we only need to find $1<k<g$ such that $k$ and $2g+1$ are coprime. This is always possible, take e.g. $k=2$. For those $g$ for which $n=2g+1$ is a prime number, we get real multiplication on the Jacobian of $W_g$ in the classical sense, again not being transported by the $\SL_2(\bbr )$-action.

This series of counterexamples is not completely new. In \cite{McMullen03}, McMullen studied the special case $g=3$, $k=2$ and sketched a proof how to use Ahlfors' variational formula in order to see that a given real multiplication structure (in the classical sense) on $J(W_3)$ with $\omega_2$ as an eigenform (in our notation) is not preserved by the $\SL_2(\bbr )$-action. This is Theorem 7.5 in op.cit.

M\"{o}ller in \cite[Remark 2.8]{Moeller05b} considered the same special case and remarked that the trace field of $\SL (W_3,\omega_2)$ contains $\bbq (\cos 2\pi /7)$. He also suggested in op.cit., Remark 2.9 to study the general case $(W_g,\omega_k)$ as possible counterexamples, but erroneously stated (without proof) that for $1<k<g$ one always gets a counterexample.

\end{document}